\documentclass{amsart}

\usepackage[pdfborder={0 0 0}, draft=false, pagebackref, unicode=true]{hyperref}

\usepackage{esint,amssymb,mathrsfs,mathtools,amsthm,tikz}
\usetikzlibrary{shadings}

\newtheorem{thm}[equation]{Theorem}
\newtheorem{lem}[equation]{Lemma}
\newtheorem{cor}[equation]{Corollary}
\newtheorem{prp}[equation]{Proposition}
\theoremstyle{definition}
\newtheorem{defn}[equation]{Definition}
\newtheorem{rmk}[equation]{Remark}

\numberwithin{equation}{section}

\usepackage{cleveref}
\crefname{equation}{}{}
\crefname{lem}{}{}
\crefrangelabelformat{equation}{(#3#1#4--#5#2#6)}

\newcommand\dist{\mathop{\mathrm{dist}}\nolimits}
\newcommand\Div{\mathop{\mathrm{div}}\nolimits}
\newcommand\Tr{\mathop{\mathrm{Tr}}\nolimits}
\newcommand\supp{\mathop{\mathrm{supp}}\nolimits}
\newcommand\diam{\mathop{\mathrm{diam}}\nolimits}
\newcommand\esssup{\mathop{\mathrm{ess\,sup}}}
\newcommand\R{\mathbb{R}}
\newcommand\Z{\mathbb{Z}}
\newcommand\N{\mathbb{N}}
\newcommand\1{\mathbf{1}}

\newcommand\corkscrew{\varkappa}

\newcommand\dmn{{\mathfrak{n}}}  
\newcommand\pdmn{{\mathfrak{n}}} 
\newcommand\bdmn{{\mathfrak{d}}} 
\newcommand\dmnMinusOne{{\mathfrak{n}-1}}
\newcommand\pdmnMinusOne{{(\mathfrak{n}-1)}}
\newcommand\pbdmn{{\mathfrak{d}}}
\newcommand\bpdmn{{\mathfrak{d}}}

\title{The Poisson problem in domains with Ahlfors regular boundary}

\author{Ariel Barton, Svitlana Mayboroda, and Alberto Pacati}


\begin{document}

\begin{abstract}
We establish well posedness of the Poisson problem in weak local John domains, for linear second order elliptic equations with real coefficients, and with data in weighted Lebesgue spaces with a very broad range of acceptable parameters.
\end{abstract}

\maketitle

\setcounter{tocdepth}{2}

\tableofcontents

\setcounter{tocdepth}{3}

\sloppy

\section{Introduction}

Let us first set up the framework for our results and then turn to details and to the (extensive) history of the subject.

The present paper and its companion \cite{BarMP25pB} establish what one could see as an extrapolation mechanism to obtain solvability of the Poisson problem for the second order elliptic partial differential equation $-\Div A\nabla u=h$, in weighted Sobolev spaces, with Dirichlet boundary conditions, in the full generality of (possibly weighted) elliptic coefficients and Ahlfors-David regular domains.

Poisson's equation $-\Delta u=h$, for a given function~$h$ and an unknown function~$u$, was historically one of the first differential equations to be studied, due in large part to its many applications in physics and beyond. Well-posedness results in various function spaces with progressively rougher assumptions on the coefficients and domains were obtained throughout the 20th century. The development of harmonic analysis and geometric measure theory in the past 20--30 years, however, have pushed the boundaries of the subject towards the sharp, best possible (or rather worst possible) conditions. In particular, equivalence of well-posedness of the Dirichlet problem in $L^p$ to the weak local John condition along with the uniform rectifiability of the boundary was established \cite{AzzHMMT20}, weighted elliptic equations for domains with boundaries of an arbitrary dimension have been treated, and the sharp range of $A$ for which either of the above is possible is rather well understood now, through the variations of the so-called Dahlberg-Kenig-Pipher condition, cf.\ \cite{KenP01, HofMMTZ21}. We refer the reader to Section~\ref{sec:history} and the discussion below for more detailed references and even more so to our recent note \cite{BarMP25pD} where we tried to survey the results of particular relevance to the present context. These results correspond to ``zero'' smoothness, around $L^p$ boundary data. Some (comparatively few) of them have been also extended to the regularity problem, that is, to the Dirichlet problem with data of smoothness 1 in some sense; see, for example, \cite{KenP93, Shen2019,
DindosHofmannPipher2023,
MouPT22p}. The present paper and \cite{BarMP25pB} address the intermediate range, establishing well-posedness for the Poisson problem which would correspond to Dirichlet boundary data of smoothness strictly between 0 and 1 and to solutions in weighted Sobolev spaces, in the full generality of elliptic operators and (properly open and connected) Ahlfors-David regular domains.

Our results are far from unprecedented---see \cite{JerK95,MayMit04A, BarM16A, MazMS10} and many predecessors discussed below. However, there is a significant gap between the previous knowledge and the present paper, and perhaps more importantly, a completely different, extrapolation rather than interpolation paradigm. With the exception of relatively smooth coefficients and domains, the literature devoted to well-posedness of the Poisson problem in spaces of fractional smoothness pertains to (at worst) coefficients satisfying strong regularity properties, typically either lying in the Saranson space $VMO$ or constant in a specified direction, and at worst Lipschitz boundaries or boundaries with a well-controlled unit normal. As we mentioned above, this is very far from the modern limitations of the subject. The reason for this shortcoming was not lack of attention. Most of the existing methods relied, in a sense, on interpolation mechanisms, {\emph{interpolating}} between the Dirichlet problem of smoothness 0 and the regularity problem of smoothness 1 and/or on the accompanying method of layer potentials. In the present work we develop a different approach, allowing one to draw an extended range of well-posedness (cf.\ Figure~\ref{fig:extrapolation}) knowing only {\emph{one}} Dirichlet end-point result. In fact, we even give a default range of well-posedness (cf.\ Figure~\ref{fig:arbitrary}) knowing {\emph{no}} $L^p$ Dirichlet endpoint. We treat general elliptic coefficients, including possibly weighted ones, and general John domains with boundary of any dimension $0<\bdmn<n$, but even in the classical setting $\bdmn=n-1$ with Dahlberg-Kenig-Pipher coefficients these results are new. Yet, the value of our theorems is in their versatility: in many borderline settings, it is not known whether the regularity problem is well-posed and in some cases it is even known that it is not. Beyond DKP coefficients, even the $L^p$ Dirichlet problem need not be well-posed. Our theorems provide a general umbrella or rather extension for all of the above.

In a forthcoming paper \cite{BarMP25pB} we will consider the Poisson-Dirichlet problem
\begin{equation*}
-\Div A\nabla u=h\text{ in }\Omega,\quad
u=f\text{ on }\partial\Omega
\end{equation*}
for specified $f$, $\Omega$, $A$, and~$h$. In this monograph we will confine our attention to the Poisson problem with homogeneous boundary data
\begin{equation}\label{eqn:Poisson:introduction} -\Div A\nabla u=h\text{ in }\Omega,\quad
u=0\text{ on }\partial\Omega
\end{equation}
for specified $\Omega$, $A$, and~$h$.

The principal condition we will impose on the domain is Ahlfors regularity of the boundary. This condition essentially asserts only that the boundary is $\bdmn$-dimensional for some consistent value of~$\bdmn$.

\begin{defn}\label{dfn:Ahlfors}
A nonempty, possibly unbounded set $\Gamma \subset\R^\dmn$ is said to be $\pbdmn$-Ahlfors regular if there is a constant $A>0$ such that, if $\xi\in\Gamma $ and $r\in\R$ with $0<r\leq \diam \Gamma $, then
\begin{equation*}\frac{1}{A}r^\bdmn\leq \sigma(B(\xi,r)\cap\Gamma )\leq Ar^\bdmn\end{equation*}
where $\sigma$ denotes the $\bdmn$-dimensional Hausdorff measure on subsets of~$\R^\dmn$.
\end{defn}

We will consider the full range $0<\bdmn<\dmn$, where $\bdmn$ need not be an integer. Our work includes significant new results in the classical case $\bdmn=\dmnMinusOne$, but also contains results beyond this case. Thus our work is a part of the recent surge of interest in elliptic boundary value problems in domains with lower-dimensional boundary; see \cite{DavFM19A,DavFM19B,MayZ19,DavEM21, DavFM21,FenMZ21,FenMZ21,DavM22,Fen22,DavFM20p}.

With the exception of the more general paper \cite{DavFM20p}, these papers consider boundary value problems in the domain $\Omega$, where $\partial\Omega$ is $\bdmn$-Ahlfors regular for some $0<\bdmn<\dmnMinusOne$, for the operator $L=-\Div A\nabla$, where $A$ is a measurable $\dmn\times\dmn$ matrix-valued function with real coefficients satisfying the weighted ellipticity conditions
\begin{align}
\label{eqn:elliptic:introduction}
{\xi}\cdot A(x)\xi
&\geq
\delta(x)^{\bdmn+1-\pdmn} \lambda\|\xi\|^2
,\\
\label{eqn:elliptic:bounded:introduction}
|\eta\cdot A(x)\xi|
&\leq \delta(x)^{\bdmn+1-\pdmn} \Lambda\|\xi\|\,\|\eta\| \end{align}
for all $\eta$, $\xi\in\R^\dmn$ and all $x\in\R^\dmn\setminus\partial\Omega$, where
\begin{equation*}\delta(x)=\dist(x,\partial\Omega)=\inf_{\xi\in\partial\Omega} |x-\xi|. \end{equation*}
This pair of conditions was introduced in \cite{DavFM19A}. The classical case considers matrices that are bounded and positive definite, uniformly in $\Omega$ (or $\R^\dmn$), that is, the $\bdmn=\dmnMinusOne$ case. However, it is well known that lower-dimensional sets are invisible to the operator $L$ in the classical case; for example, if $\partial\Omega$ is $\bdmn$-Ahlfors regular, then the Dirichlet problem with continuous boundary values is not well posed if $\bdmn\leq \dmn-2$ (by contrast, it is well posed for any $\dmn-2<\bdmn<\dmn$). The weight $\delta^{\bdmn+1-\pdmn}$ magnifies the effect of the boundary, increasing the visibility of the lower-dimensional boundary to the operator~$L$. Once again, we underline that our results are new in the classical setting, without any degeneracy, as well.

Our first main result, Theorem~\ref{thm:Poisson}, will be established without any further assumptions on the operator~$L$ beyond the ellipticity conditions (\ref{eqn:elliptic:introduction}--\ref{eqn:elliptic:bounded:introduction}), the condition that the input $u$ to $L$ be a scalar-valued function (not a vector-valued function), and the condition that $A$ have real entries. Systems of elliptic equations $L\vec u=\vec h$ or equations with complex coefficients have been investigated in the literature (see \cite{AusQ02,MazMS10,HofMayMou15,BarM16A,Bar20p}), but the methods of the present paper rely on many techniques and results only available in the case of real equations, such as harmonic measure and the maximum principle.

We will, however, need to impose some additional geometric conditions, of topological nature. This is classical in the subject and of the nature of minimal possible. First, domains containing an open neighborhood of infinity (that is, domains with bounded complements) introduce many technicalities relating to behavior at infinity that are not present in the case of bounded domains, or even unbounded domains with unbounded boundary. Thus, we have chosen to confine ourselves to domains that are either bounded or unbounded with unbounded boundary; this case preserves much of the possible local behavior. Observe that we may state this condition more concisely as the condition that $\Omega$ has an unbounded complement, or the condition that $\diam(\Omega)=\diam(\partial\Omega)$.

Second, in the case $\dmnMinusOne\leq \bdmn<\dmn$, we will impose the weak local John condition on our domains. This condition has appeared recently in the literature (see for example \cite{AzzHMMT20}) and is extremely general. We will give a precise definition below (Definition~\ref{dfn:local:John}); loosely, the weak local John condition means that, for every point $x$ in~$\Omega$, a substantial portion of $\partial\Omega$ is close to~$x$ and also nontangentially accessible from~$x$; this means that a substantial portion of the boundary is able to have significant influence on the point~$x$, and so the condition $u=0$ on $\partial\Omega$ in the Poisson problem~\eqref{eqn:Poisson:introduction} can have a significant effect on the value of~$u(x)$.

In the case $\dmnMinusOne<\bdmn<\dmn$, we will impose interior corkscrew and Harnack chain conditions. We conjecture that these conditions are probably unnecessary but extending the background to even more general domains would perhaps make the length of the manuscript insurmountable---see Remark~\ref{Rem1.13}.

\begin{thm}\label{thm:Poisson}
Let $\Omega\subseteq\R^\dmn$, $\dmn\geq 2$, be a connected open set. Suppose that $\partial\Omega$ is $\bdmn$-Ahlfors regular for some real number $\bdmn$ with $0<\bdmn<\dmn$, and that $\R^\dmn\setminus\Omega$ is unbounded.

If $\bdmn=\dmnMinusOne$, we additionally require that $\Omega$ satisfies the weak local John condition. If $\dmnMinusOne<\bdmn<\dmn$, we require $\Omega$ to satisfy the interior corkscrew and Harnack chain conditions.

Let $L=-\Div A\nabla$ be a linear differential operator associated to real, not necessarily symmetric coefficients $A:\Omega\mapsto\R^{\dmn\times\dmn}$ that satisfy the positive definiteness and boundedness conditions \textup{(\ref{eqn:elliptic:introduction}--\ref{eqn:elliptic:bounded:introduction})}.

There are then numbers $\mathfrak{a}$, $\mathfrak{a}^*\in (0,1]$ and $p^-$, $p^+$ with $p^-<2<p^+$ such that, if $(s,1/p)$ lies in the open pentagon indicated in Figure~\ref{fig:arbitrary}, and if $p^-<\beta<p^+$, then
for all
$\vec H:\Omega\mapsto\R^\dmn$ such that the right hand side of the bound~\eqref{eqn:Besov:estimate} is finite, there is a unique solution $u$ to the Poisson problem
\begin{equation}\label{eqn:Poisson}
\left\{\begin{aligned}
Lu&=-\Div(\delta^{\bdmn+1-\pdmn}\vec H)&&\text{in }\Omega,\\
u&=0&&\text{on }\partial\Omega\text{ in the averaged sense,}
\end{aligned}\right.\end{equation}
that satisfies
\begin{multline}\label{eqn:Besov:estimate}
\int_\Omega
\biggl(\fint_{B(x,\delta(x)/2)} |\nabla u|^\beta \biggr)^{p/\beta}
\delta(x)^ {\bdmn-\pdmn+p-ps}
\,dx
\\\leq
C\int_\Omega
\biggl(\fint_{B(x,\delta(x)/2)}|\vec H|^\beta \biggr)^{p/\beta}
\delta(x)^ {\bdmn-\pdmn+p-ps}
\,dx
\end{multline}
for some constant $C$ depending on $p$, $s$, $\beta$, $A$, and~$\Omega$, but not on~$\vec H$.

If $0<s<\mathfrak{a}$, $p^-<\beta<p^+$, and the right-hand side of the bound~\eqref{eqn:Besov:estimate:infinity} is finite, then there is a unique solution to the problem~\eqref{eqn:Poisson} that satisfies the estimate
\begin{equation}\label{eqn:Besov:estimate:infinity}
\sup_{x\in\Omega}
\biggl(\fint_{B(x,\delta(x)/2)}|\nabla u|^\beta\biggr)^{1/\beta}
\delta(x)^ {1-s}
\leq
C\sup_{x\in\Omega}
\biggl(\fint_{B(x,\delta(x)/2)}|\vec H|^\beta\biggr)^{1/\beta}
\delta(x)^ {1-s}
.\end{equation}

Furthermore, we have that $\mathfrak{a}\geq \max(\alpha,1-\bdmn)$ and $\mathfrak{a}^*\geq \max(\alpha^*,1-\bdmn)$, where $\alpha$ is as in the boundary De Giorgi-Nash estimate (Lemma~\ref{lem:boundary:DGN} below), and $\alpha^*$ is as in Lemma~\ref{lem:boundary:DGN} with $L$ replaced by~$L^*$.

Finally, $p^+$ may be taken to be the number $p^+_L$ in Meyers's interior reverse Hölder estimate \eqref{eqn:interior:Meyers} for~$L$, and $p^-$ may be taken to be the Hölder conjugate of the number $p^+_{L^*}$ in Meyers's reverse Hölder estimate for~$L^*$.
\end{thm}

\def\figurealpha{0.3}
\def\figuredimen{2}
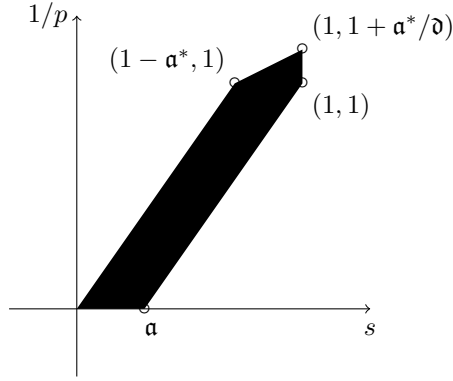
\begin{figure}
\begin{tikzpicture}[scale=3]
\begin{scope}[shift={(0,0)}]
\draw [->] (-0.3,0)--(1.3,0) node [below ] {$\vphantom{1}s$};
\draw [->] (0,-0.3)--(0,1.3) node [ left] {$1/p$};
\fill  (1,1) -- (\figurealpha,0) -- (0,0) -- (1-\figurealpha,1) --(1,1+\figurealpha/\figuredimen) -- cycle;
\node at (1,1) {$\circ$};
\node [below right] at (1,1) {$(1,1)$};
\node at (1,1+\figurealpha/\figuredimen) {$\circ$};
\node [above right] at (1,1+\figurealpha/\figuredimen) {$(1,1+\mathfrak{a}^*/\bdmn)$};
\node at (1-\figurealpha,1) {$\circ$};
\node [above left] at (1-\figurealpha,1) {$(1-\mathfrak{a}^*,1)$};
\node [below] at (\figurealpha,0) {$\phantom{1}\mathfrak{a}$}; \node at (\figurealpha,0) {$\circ$};
\end{scope}
\end{tikzpicture}
\caption{Values of the parameters $(s,1/p)$ such that the estimate~\eqref{eqn:Besov:estimate} on solutions to the Poisson problem is valid. Because $\mathfrak{a}^*\geq 1-\bdmn$, the illustrated region is convex.}
\label{fig:arbitrary}
\end{figure}

\begin{rmk} The choice of function spaces to measure regularity of solutions may seem complicated, but it is, in fact, natural and well-aligned with the existing results. Barring the extra averaging in $L^\beta$, needed to overcome lack of local regularity, these are the traditional choices of weighted Sobolev spaces in this context---see Section~\ref{sec:history}. In a forthcoming paper \cite{BarMP25pB}, we will generalize to the Poisson-Dirichlet problem
\begin{equation*}
\left\{\begin{aligned}
Lu&=-\Div(\delta^{\bdmn+1-\pdmn}\vec H)&&\text{in }\Omega,\\
u&=f&&\text{on }\partial\Omega\text{ in the averaged sense,}
\end{aligned}\right.\end{equation*}
with the estimate
\begin{multline*}
\int_\Omega
\biggl(\fint_{B(x,\delta(x)/2)} |\nabla u|^\beta \biggr)^{p/\beta}
\delta(x)^ {\bdmn-\pdmn+p-ps}
\,dx
\\\leq
C\|f\|_{\dot B^{s,p}_p(\Omega)}^p
+
C\int_\Omega
\biggl(\fint_{B(x,\delta(x)/2)}|\vec H|^\beta \biggr)^{p/\beta}
\delta(x)^ {\bdmn-\pdmn+p-ps}
\,dx
\end{multline*}
where the spaces $\dot B^{s,p}_p(\Omega)$ are generalizations of the well known homogeneous Besov spaces on~$\R^\bdmn$ and are $p<\infty$ analogues of the space $\dot C^s(\partial\Omega)$ of Hölder continuous functions. Extension theorems (to be proven in \cite{BarMP25pB} and analogous to those of \cite{Usp61,Nik77B,Sha85,NikLM88,MitT06,Kim07,MazMS10,DavFM21}) indicate that there is a deep and natural connection between boundary data in ${\dot B^{s,p}_p(\Omega)}$ and solutions $u$ such that the norm on the left-hand side of the previous equation (or formula~\eqref{eqn:Besov:estimate}) is finite. It is for this reason that we have chosen to work with the complicated expression $\delta^{\bdmn-\pdmn+p-ps}$, rather than a simpler weight such as $\delta^\theta$ or $\delta^{p\theta}$.
\end{rmk}

Expressed algebraically, the pair $(s,1/p)$ satisfies the conditions of Theorem~\ref{thm:Poisson} if either
\begin{align}
\label{eqn:thm:Dirichlet:quasi}
\frac{\bdmn}{\bdmn+\mathfrak{a}^*}&<p\leq 1&\text{and}&&\frac{\bdmn}{p}+1-\bdmn-\mathfrak{a}^*&<s<1  \end{align}
or
\begin{align}
\label{eqn:thm:Dirichlet:Banach}
1&\leq p\leq \infty &\text{and}&&\frac{1-\mathfrak{a}^*}{p}&<s<\mathfrak{a}+\frac{1-\mathfrak{a}}{p}.
\end{align}

By $u=0$ on~$\partial\Omega$ in the averaged sense, we mean that
\begin{equation}\label{eqn:trace}
\lim_{r\to 0^+} \frac{1}{r^\dmn}\int_{B(\xi,r)\cap \Omega} |u|=0
\end{equation}
for $\sigma$-almost every $\xi\in\partial\Omega$.

Recall that $Lu=-\Div(A\nabla u)$ and that $A$ satisfies the ellipticity conditions~\cref{eqn:elliptic:introduction,eqn:elliptic:bounded:introduction}, meaning that the eigenvalues of~$A$ are bounded above and below by multiples of $\delta^{1+\bdmn-\pdmn}$.
We prefer to state Poisson's equation as $Lu=-\Div(A\nabla u)=-\Div(\delta^{1+\bdmn-\pdmn}\vec H)$, rather than $Lu=f$ or $Lu=-\Div\vec\Phi$, because with this formulation the two quantities $\vec H$ and $\nabla u$ lie in the same function spaces.

The  weak local John, interior corkscrew, and Harnack chain conditions are given by Definitions~\ref{dfn:local:John}, \ref{dfn:iCS}, and~\ref{dfn:Harnack} below.
We observe that if $\bdmn<\dmnMinusOne$ and $\partial\Omega$ is $\bdmn$-Ahlfors regular, then by \cite[Lemmas 2.1 and~11.6]{DavFM21} $\Omega$ necessarily satisfies the interior corkscrew and Harnack chain conditions. As noted in Remark~\ref{rmk:higher:John:corkscrew}, any domain satisfying the interior corkscrew and Harnack chain conditions is a weak local John domain. Thus all domains $\Omega$ considered in Theorem~\ref{thm:Poisson} are weak local John domains, and we will use this condition extensively in the proof; this condition is omitted from the statement of Theorem~\ref{thm:Poisson} in the $\bdmn<\dmnMinusOne$ case, when it follows automatically from the Ahlfors condition.

We comment on the additional geometric conditions required in the case $\dmnMinusOne<\bdmn<\dmn$.
The case $\bdmn=\dmnMinusOne$ is the classical case and has been studied intensively, and many results are known in the generality of weak local John domains and beyond. The case $\bdmn<\dmnMinusOne$ has been of considerable interest in recent years, and as noted above automatically enjoys the interior corkscrew and Harnack chain conditions.
The case $\bdmn>\dmnMinusOne$, by contrast, has received little study. Indeed the principal work treating this case, \cite{DavFM20p}, is primarily concerned with domains of mixed codimension (in particular with so-called ``sawtooth domains'' arising from domains with lower dimensional boundary), and produces results for the case $\bdmn>\dmnMinusOne$ almost incidentally. The paper \cite{DavFM20p} imposes the interior corkscrew and Harnack chain conditions on all domains of study because they hold automatically in the higher codimensional case of their principal interest; thus, many important known results in the literature, which we will cite and use, are in the $\bdmn>\dmnMinusOne$ case known only with these additional geometric assumptions.

\begin{rmk}\label{Rem1.13}
The Harnack chain and interior corkscrew conditions will not be used directly in any of the arguments of the present paper. They will be used indirectly only in that the boundary de Giorgi-Nash estimate (Lemma~\ref{lem:boundary:DGN} below), well posedness of the continuous Dirichlet problem (Lemma~\ref{lem:cts} below), and the existence of the Green's function and its properties (formulas~\cref{eqn:green:W12,eqn:green:solution,eqn:green:fundamental,eqn:green:boundary,eqn:green:symmetric,eqn:green:poisson:nodiv,eqn:green:Poisson,eqn:green:near:size,eqn:green:upper:bound,eqn:green:positive} below) for $\bdmn>\dmnMinusOne$ were established in \cite{DavFM20p}, and thus are known only in the case of Harnack chains and interior corkscrews. We observe that if Lemmas \ref{lem:boundary:DGN} and~\ref{lem:cts} and formulas~\cref{eqn:green:W12,eqn:green:solution,eqn:green:fundamental,eqn:green:boundary,eqn:green:symmetric,eqn:green:poisson:nodiv,eqn:green:Poisson,eqn:green:near:size,eqn:green:upper:bound,eqn:green:positive} could be proven in the case of weak local John domains for $\dmnMinusOne<\bdmn<\dmn$, then Theorem~\ref{thm:Poisson} and~\ref{thm:Poisson:Lq} would hold in this generality as well.
\end{rmk}

To state our next theorem, recall that if $\Omega$ and $L$ are as in Theorem~\ref{thm:Poisson}, it is well known that for every function $f$ defined on $\partial\Omega$ that is Lipschitz continuous and compactly supported, there is a unique solution $v$ to the Dirichlet problem
\begin{equation}\label{eqn:Dirichlet:cts}
\left\{\begin{gathered}\begin{aligned}
Lv&=0\quad\text{in }\Omega,\\
v&=f\quad\text{on }\partial\Omega,\\
\sup_\Omega v&= \sup_{\partial\Omega} f, \quad
\inf_\Omega v= \inf_{\partial\Omega} f,\\
\end{aligned}\\
\int_\Omega |\nabla v|^2\,\delta^{\bdmn+1-\pdmn}<\infty,\\
v\text{ is continuous on $\overline\Omega$}
.\end{gathered}
\right.\end{equation}
See Lemma~\ref{lem:cts} below; note that the proof consists almost entirely of citations of known results in the literature.

In many cases, solutions to the problem~\eqref{eqn:Dirichlet:cts} are known to satisfy further estimates. In particular, if $f\in L^q(\partial\Omega)$ for some $1<q<\infty$, then the solution $v$ may often be shown to satisfy the nontangential estimate
\begin{equation}
\label{eqn:N:bound}
\|Nv\|_{L^q(\partial\Omega)}\leq C_q\|f\|_{L^q(\partial\Omega)}\end{equation}
where $N$ is the nontangential maximal function given by \begin{equation}\label{eqn:N:intro}
Nv(\xi)=N_{a,\Omega} v(\xi)
=\esssup\{|v(x)|:x\in\Omega,\> |x-\xi|<(1+a)\dist(x,\partial\Omega)\}
\end{equation}
for some positive number~$a$. This is the rigorous estimate accompanying the aforementioned well-posedness of the Dirichlet problem in $L^q$.

\begin{thm}\label{thm:Poisson:Lq}
Suppose that the conditions of Theorem~\ref{thm:Poisson} are satisfied.

Suppose furthermore that there is a $q\in (1,\infty)$ such that the $L^q$-Dirichlet problem is well posed. That is, suppose that there exist $q\in (1,\infty)$, $a>0$, and $C_q<\infty$ such that if $f$ is Lipschitz continuous and compactly supported, then the solution $v$ to the problem~\eqref{eqn:Dirichlet:cts} also satisfies the estimate~\eqref{eqn:N:bound}.

Then there is a number $\mathfrak{b}^*\in (0,1]$ such that the conclusions of Theorem~\ref{thm:Poisson} are valid for $(s,1/p)$ in the region indicated on the left in Figure~\ref{fig:extrapolation}.

If instead there is a $q^*\in (1,\infty)$ such that the solution $v$ to the problem~\eqref{eqn:Dirichlet:cts}, with $L$ replaced by~$L^*$, satisfies the estimate~\eqref{eqn:N:bound} (with $q$ replaced by~$q^*$), then there is a $\mathfrak{b}\in (0,1]$ such that the conclusions of Theorem~\ref{thm:Poisson} are valid for $(s,1/p)$ in the region indicated on the right in Figure~\ref{fig:extrapolation}.

Furthermore, the numbers $\mathfrak{b}^*$ or $\mathfrak{b}$ satisfy the estimates
\begin{equation}\label{eqn:b:lower}
\mathfrak{b}\geq \alpha,
\qquad
\mathfrak{b}^*\geq \alpha^*,
\qquad
\mathfrak{b}\geq 1+\frac{\bdmn}{q^*}-\bdmn,
\qquad
\mathfrak{b}^*\geq 1+\frac{\bdmn}{q}-\bdmn
\end{equation}
where $\alpha$ and $\alpha^*$ are the boundary De Giorgi-Nash exponents for $L$ and~$L^*$.

(In the event that $q$ and $q^*$ both exist, the conclusions of Theorem~\ref{thm:Poisson} are valid for $(s,1/p)$ in the region indicated on the bottom in Figure~\ref{fig:extrapolation}.)
\end{thm}

\def\figureq{0.2}
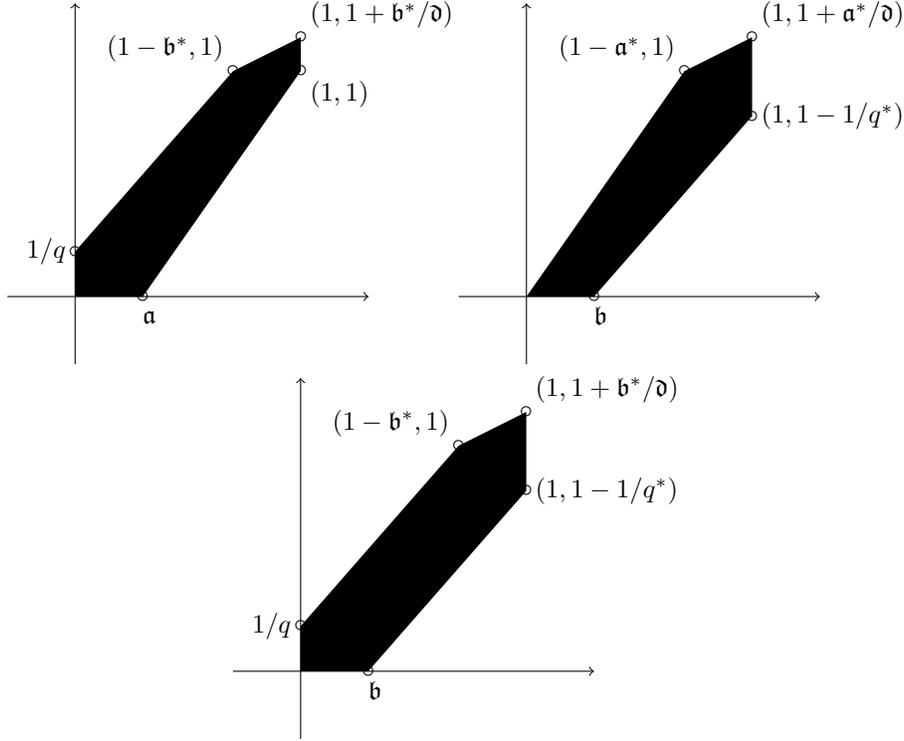
\begin{figure}
\begin{tikzpicture}[scale=3]
\begin{scope}[shift={(0,0)}]
\draw [->] (-0.3,0)--(1.3,0);
\draw [->] (0,-0.3)--(0,1.3);
\node at (1,1) {$\circ$};
\node [below right] at (1,1) {$(1,1)$};
\node at (1,1+\figurealpha/\figuredimen) {$\circ$};
\node [above right] at (1,1+\figurealpha/\figuredimen) {$(1,1+\mathfrak{b}^*/\bdmn)$};
\node at (1-\figurealpha,1) {$\circ$};
\node [above left] at (1-\figurealpha,1) {$(1-\mathfrak{b}^*,1)$};
\node [left] at (0,\figureq) {$1/q$}; \node at (0,\figureq) {$\circ$};
\node [below] at (\figurealpha,0) {$\phantom{1}\mathfrak{a}$}; \node at (\figurealpha,0) {$\circ$};
\fill
  (0,0) --(0,\figureq) --(1-\figurealpha,1) --(1,1+\figurealpha/\figuredimen) -- (1,1) --(\figurealpha,0) --cycle;
\end{scope}
\begin{scope}[shift={(2,0)}]
\draw [->] (-0.3,0)--(1.3,0);
\draw [->] (0,-0.3)--(0,1.3);
\node at (1,1-\figureq) {$\circ$};
\node [right] at (1,1-\figureq) {$(1,1-1/q^*)$};
\node at (1,1+\figurealpha/\figuredimen) {$\circ$};
\node [above right] at (1,1+\figurealpha/\figuredimen) {$(1,1+\mathfrak{a}^*/\bdmn)$};
\node at (1-\figurealpha,1) {$\circ$};
\node [above left] at (1-\figurealpha,1) {$(1-\mathfrak{a}^*,1)$};
\node [below] at (\figurealpha,0) {$\phantom{1}\mathfrak{b}$}; \node at (\figurealpha,0) {$\circ$};
\fill
  (1,1) --(1,1-\figureq) --(\figurealpha,0)--(0,0)
--(1-\figurealpha,1) --(1,1+\figurealpha/\figuredimen) --cycle;
\end{scope}
\end{tikzpicture}\\
\begin{tikzpicture}[scale=3]
\draw [->] (-0.3,0)--(1.3,0);
\draw [->] (0,-0.3)--(0,1.3);
\node at (1,1-\figureq) {$\circ$};
\node [right] at (1,1-\figureq) {$(1,1-1/q^*)$};
\node at (1,1+\figurealpha/\figuredimen) {$\circ$};
\node [above right] at (1,1+\figurealpha/\figuredimen) {$(1,1+\mathfrak{b}^*/\bdmn)$};
\node at (1-\figurealpha,1) {$\circ$};
\node [above left] at (1-\figurealpha,1) {$(1-\mathfrak{b}^*,1)$};
\node [left] at (0,\figureq) {$1/q$}; \node at (0,\figureq) {$\circ$};
\node [below] at (\figurealpha,0) {$\phantom{1}\mathfrak{b}$}; \node at (\figurealpha,0) {$\circ$};
\fill
  (1,1) --(1,1-\figureq) --(\figurealpha,0) --(0,0)
--(0,\figureq) --(1-\figurealpha,1) --(1,1+\figurealpha/\figuredimen) --cycle;
\end{tikzpicture}
\caption{Values of the parameters $(s,1/p)$ such that the estimate~\eqref{eqn:Besov:estimate} is valid, given solvability of the $L^q$-Dirichlet problem for $L$ (on the left) or the $L^{q^*}$-Dirichlet problem for $L^*$ (on the right) with nontangential estimates.  Because $\mathfrak{b}^*\geq 1-\bdmn+\bdmn/q$, the illustrated region is convex.
In the common case that $q$ and $q^*$ both exist, the estimate~\eqref{eqn:Besov:estimate} is valid for values of $(s,1/p)$ as in the bottom picture.
}

\label{fig:extrapolation}
\end{figure}

\begin{rmk}Our proofs of Theorems \ref{thm:Poisson} and~\ref{thm:Poisson:Lq} are valid for any real number $\bdmn\in (0,\dmn)$; in particular, we do not require that $\bdmn$ be an integer.

However, most of the results in the literature that establish the estimate~\eqref{eqn:N:bound}, for certain domains and coefficients, require that $\partial\Omega$ be uniformly rectifiable, which requires in particular that $\bdmn$ be an integer.

Thus, Theorem~\ref{thm:Poisson} does not require that $\bdmn$ be an integer, and neither does the proof of Theorem~\ref{thm:Poisson:Lq}, but most known applications of Theorem~\ref{thm:Poisson:Lq} do. A notable exception is the ``magic exponent'' of \cite{DavEM21}.
\end{rmk}

The outline of this paper is as follows. In Section~\ref{sec:history} we will review the history of the Poisson problem.

In Section~\ref{sec:function} we will discuss function spaces on domains, in particular the weighted averaged Lebesgue and Sobolev spaces used in Theorem~\ref{thm:Poisson}.
In Section~\ref{sec:trace:zero} we will discuss a particular topic in function theory, namely trace results (and in particular will elaborate on the condition $u=0$ on $\partial\Omega$ of Theorems \ref{thm:Poisson} and~\ref{thm:Poisson:Lq}).

In Section~\ref{sec:elliptic} we will review the basic theory of solutions to elliptic differential equations and some known results which we intend to use.

In Section~\ref{sec:beta} we will discuss more deeply the theory of local integrability of solutions to elliptic equations. This will lay the framework for an expanded range of~$\beta$ beyond the single case $\beta=2$ in Theorems \ref{thm:Poisson} and~\ref{thm:Poisson:Lq}.

In Section~\ref{sec:elliptic:boundary} we will return to the literature on solutions to elliptic differential equations. In particular, some such results are easy to prove in the case $\dmn=\bdmn+1=2$, but the authors have not been able to find them in the literature in the generality of weak local John domains; we will thus use the results of Section~\ref{sec:elliptic} to prove certain results in Section~\ref{sec:beta}, and then use the results of Section~\ref{sec:beta} to prove the results of Section~\ref{sec:elliptic:boundary} in this special case.

We will prove the $p\leq 1$, $\mathfrak{a}^*=\alpha^*$ case of Theorem~\ref{thm:Poisson} in Section~\ref{sec:atom}. We will complete the proof (that is, prove the $p>1$ case and the $\mathfrak{a}^*=\max(\alpha,1-\bdmn)$ case) in Section~\ref{sec:thm:Poisson}, and will prove Theorem~\ref{thm:Poisson:Lq} in Section~\ref{sec:Poisson:Lp}.

Finally, in Appendix~\ref{sec:Lp}, we address the small constant case that establishes the bound~\eqref{eqn:N:bound} for any specified $q$ for certain coefficients and domains.

\subsection{The history of the Poisson problem and weighted averaged spaces}\label{sec:history}

We now review the history of the Poisson problem. A straightforward argument involving the Lax-Milgram lemma (see Definition~\ref{dfn:Lax-Milgram} below) establishes that there is a unique solution $u$ to the Poisson problem
\begin{equation*}Lu=-\Div(\delta^{1+\bdmn-\pdmn}\vec H)\text{ in } \Omega,
\qquad
u=0\text{ on }\partial\Omega\end{equation*}
 that satisfies the estimate
\begin{equation*}\int_\Omega |\nabla u|^2\delta^{1+\bdmn-\pdmn} \leq C\int_\Omega |\vec H|^2 \delta^{1+\bdmn-\pdmn}\end{equation*}
for every $\vec H$ such that the right hand side is finite.
In the classical codimension 1 case this reduces to the estimate $\|\nabla u\|_{L^2(\Omega)}\leq C\|\vec H\|_{L^2(\Omega)}$. (This is in fact true considerably beyond the scope of the present paper, allowing for complex coefficients, systems, fourth-order and higher-order differential equations, and so forth.)

Much of the existing literature on the Poisson problem generalizes the estimate $\|\nabla u\|_{L^2(\Omega)}\leq C\|\vec H\|_{L^2(\Omega)}$ to the estimate $\|\nabla u\|_{L^p(\Omega)}\leq C\|\vec H\|_{L^p(\Omega)}$ for some $p\in (1,\infty)$, or to the weighted estimate $\|\nabla u\|_{L^p(\Omega,\omega)}\leq C\|\vec H\|_{L^p(\Omega,\omega)}$ for some weight~$\omega$. See, for example,  \cite{DiF96,KinZ01,AusQ02,ByuW04,Byu05} (the unweighted case) and \cite{MitT06,MazMS10,Phu11,AdiP16,AdiMP21,YanYY22} (the weighted case). These papers all treat coefficients with some continuity properties beyond those of Theorem~\ref{thm:Poisson}, most commonly coefficients lying  in or near the space $VMO$. In this case, and in sufficiently nice domains and for weights satisfying strong Muckenhoupt type conditions, the weighted $L^p$ estimate holds for a very broad range of~$p$. For general coefficients, by contrast, the $L^p$ or weighted $L^p$ estimate holds only for $p$ very near~$2$.

The particular weight $\delta^{\bdmn-\pdmn+p-ps}$ of Theorem~\ref{thm:Poisson} (and studied in particular in \cite{MitT06,MazMS10}) is inspired by, first, the obvious importance of the distance to the boundary, and, second, existing trace and extension theorems. One can extend from the Poisson problem of the present paper
\begin{equation*}
-\Div A\nabla u=\Div(\delta^{1+\bdmn-\pdmn}\vec H)\text{ in }\Omega,\quad
u=0\text{ on }\partial\Omega
\end{equation*}
to the Dirichlet problem
\begin{equation*}
-\Div A\nabla v=0\text{ in }\Omega,\quad
v=f\text{ on }\partial\Omega
\end{equation*}
by letting $\Phi$ be an arbitrary function with $\Phi=f$ on~$\partial\Omega$; letting $u$ be the solution to the Poisson problem with data $\vec H=\delta^{\pdmn-\bdmn-1}A\nabla\Phi$ yields that $v=\Phi-u$ is a solution to the Dirichlet problem. Thus there is a natural interest in the Poisson problem with data $\vec H=\delta^{\pdmn-\bdmn-1}A\nabla\Phi$, where $\Phi$ is an extension of a function on the boundary.

It was shown in \cite{Usp61} that if $\Omega=\R^\dmn_+$ is a half space, (and so in particular $\partial\Omega$ is $\bdmn$-Ahlfors regular for $\bdmn=\dmnMinusOne$,) and if $f$ lies in the boundary Besov space $\dot B^{s,p}_p(\partial\Omega)$, then $f$ may be extended to a function $\Phi$ in a weighted Sobolev space, that is, with
\begin{equation}\label{eqn:Uspenskii}
\|\Phi\|_{L^p(\Omega;\delta^{\bdmn-\pdmn+p-ps})}=
\biggl(\int_{\Omega} |\nabla \Phi|^p\delta^{\bdmn-\pdmn+p-ps}\biggr)^{1/p}\approx\|f\|_{\dot B^{s,p}_p(\partial\Omega)}.\end{equation}
This result has been generalized in succeeding years. If $\Omega$ is a bounded Lipschitz domain, then such extensions were constructed in \cite[Proposition~4.1]{MitT06} and \cite[Theorem~2.10]{Kim07}. If $\partial\Omega$ is merely $\bdmn$-Ahlfors regular, then extensions of functions in $\dot B^{s,p}_p(\partial\Omega)$ were constructed in \cite{JonW84}; the monograph \cite{JonW84} establishes membership of the extensions in a different function space, but in the case $\bdmn-\pdmn+p-ps=0$ their results coincide with the bound~\eqref{eqn:Uspenskii}. If $\Omega$ is as in Theorem~\ref{thm:Poisson}, then by \cite{DavFM21,DavFM20p} a function $f\in\dot B^{1/2,2}_{2}(\partial\Omega)$ may be extended to a function $\Phi\in L^2(\Omega;\delta^{1+\bdmn-\pdmn})$ satisfying the $p=2$, $s=1/2$ case of the estimate~\eqref{eqn:Uspenskii}.  Most of the work in our forthcoming paper \cite{BarMP25pB} consists of constructing extensions of functions in $\dot B^{s,p}_p(\partial\Omega)$, where $\partial\Omega$ is merely $\bdmn$-Ahlfors regular, that satisfy the estimate~\eqref{eqn:Uspenskii} for a broad range of $p$ and~$s$.

Thus it is natural to seek solutions $u$ to the Poisson (or Poisson-Dirichlet) problem that satisfy $\int_\Omega |\nabla u|^p\delta^{\bdmn-\pdmn+p-ps}<\infty$. Beyond the Lax-Milgram construction mentioned above and used extensively in \cite{DavFM21,DavFM20p}, this has thus far been done only in the codimension $1$ case $\bdmn=\dmnMinusOne$.

We mention in particular the papers \cite{MitT06,MazMS10}, in which the Poisson-Dirichlet problem with weighted estimates
\begin{equation}\label{eqn:MazMS10}\left\{\begin{aligned}
Lu&=-\Div \vec H\phantom{f} \text{ in }\Omega,\\
u&=f\phantom{-\Div \vec H}\text{ on }\partial\Omega,\\
\int_\Omega |\nabla u|^p\delta^{p-1-ps}&\leq
C\|f\|_{\dot B^{s,p}_p(\partial\Omega)}^p +C\int_\Omega |\vec H|^p\delta^{p-1-ps}
\end{aligned}\right.\end{equation}
was shown to be well posed for a broad range of $p$ and~$s$ under certain conditions on $A$ and~$\Omega$.

The paper \cite{MitT06} considered the case where $\Omega$ is a bounded Lipschitz domain and where $A$ is real, symmetric, and satisfies the Dini-type condition
\begin{equation*}|A(x)-A(y)|\leq \omega(|x-y|), \qquad \int_0^1 \frac{\sqrt{\omega(t)}}{t}\,dt<\infty.\end{equation*}
Under these circumstances they were able to establish well posedness of the problem~\eqref{eqn:MazMS10} for a range of $s$ and $p$ as in Figure~\ref{fig:extrapolation} (in fact with $q$, $q^*<2$).

The paper \cite{MazMS10} considered the case where $\Omega$ is a bounded Lipschitz domain, and moreover where $A$ and the unit outward normal $\nu$ to $\partial\Omega$ both lie in or near the space $VMO$. If $A$ and $\nu$ lie in $VMO$, then the Poisson-Dirichlet problem~\eqref{eqn:MazMS10} is well posed for all $0<s<1$ and all $1<p<\infty$; conversely, for any such $s$ and~$p$, if $A$ and $\nu$ are sufficiently close to $VMO$ (depending on $s$, $p$ and other parameters) then the problem~\eqref{eqn:MazMS10} is well posed.

Note that in both cases, the coefficients $A$ display regularity properties beyond boundedness and ellipticity. For coefficients with different regularity properties, or with no regularity properties, the estimate $\nabla u\in L^p(\Omega;\delta^{p-1-ps})$ is not suitable. For arbitrary coefficients the solutions even to $Lu=0$ can display bad behavior, in particular with the gradient $\nabla u$ only locally in $L^{2+\varepsilon}$ rather than locally in $L^p$.

Following the treatment of the Laplacian in Lipschitz domains \cite{JerK95,MayMit04A}, the monograph \cite{BarM16A} (by the first and second authors of the present paper) considered the Poisson-Dirichlet (and Poisson-Neuman) problems in the domain above a Lipschitz graph for coefficients constant in a direction transverse to the boundary.
To require only local $L^\beta$ estimates for $\beta$ sufficiently close to~$2$, we modified the problem~\eqref{eqn:MazMS10} by replacing the estimate on $u$ therein with
\begin{multline}\label{eqn:BarM16A}
\int_\Omega \biggl(\fint_{B(x,\delta(x)/2)}|\nabla u|^\beta \biggr)^{p/\beta} \delta(x)^{p-1-ps}\,dx
\\\leq
C\|f\|_{\dot B^{s,p}_p(\partial\Omega)}^p +C\int_\Omega \biggl(\fint_{B(x,\delta(x)/2)}|\vec H|^\beta \biggr)^{p/\beta} \delta(x)^{p-1-ps}\,dx\end{multline}
for $\beta$ sufficiently close to~$2$. This was inspired by a similar modification $\widetilde N$ of the nontangential maximal operator $N$ given by formula~\eqref{eqn:N:intro} and pioneered in \cite{KenP93} for the same reason.
We were able to establish well posedness of the Poisson-Dirichlet problem for such real coefficients for a range of $s$, $p$ as in Figure~\ref{fig:extrapolation}, using known results for the $L^q$-Dirichlet problem for such coefficients, as in Theorem~\ref{thm:Poisson:Lq} of the present paper.

Before concluding this brief history of the Poisson problem, (the interested reader may consult \cite{BarMP25pD} for a more detailed exposition,) we would like to mention one further class of results concerning the Poisson problem with rough variable coefficients. As noted above, the estimates in the problems~\eqref{eqn:MazMS10}, \eqref{eqn:BarM16A}, or Theorem~\ref{thm:Poisson} are well adapted for extending the Poisson problem to the Poisson-Dirichlet problem with boundary data~$f$ in the Besov space $\dot B^{s,p}_p(\partial\Omega)$. The literature on the Dirichlet problem without Poisson data, by contrast, contains many more results for boundary data $f$ in Lebesgue spaces $L^q(\partial\Omega)$ or boundary Sobolev spaces $\dot W^{1,q}(\partial\Omega)$.\footnote{Recently some authors have sought to generalize from the boundary Sobolev spaces $\dot W^{1,q}(\partial\Omega)$, which are only meaningful in domains with relatively nice boundary, to the Haj{\l}asz spaces $\dot M^{1,q}(\partial\Omega)$.} In particular, in many situations the estimate~\eqref{eqn:N:bound} has been established; it is this fact that makes Theorem~\ref{thm:Poisson:Lq} of use. In some situations there also exist estimates on solutions $u$ to the Dirichlet problem~\eqref{eqn:Dirichlet:cts} of the form
\begin{equation*}\|\mathcal{A}_2(\delta\nabla u)\|_{L^p(\partial\Omega)} \leq C\|f\|_{L^p(\partial\Omega)}
\quad\text{or}\quad
\|\widetilde N(\nabla u)\|_{L^p(\partial\Omega)} \leq C\|f\|_{\dot W^{1,p}(\partial\Omega)}
\end{equation*}
where $\widetilde{N}$ is the modification of the nontangential maximal operator $N$ pioneered in \cite{KenP93}, and where $\mathcal{A}_2$ is the Lusin area integral common in the theory. The second inequality above is the aforementioned regularity problem treated, e.g., in \cite{KenP93, Shen2019,
DindosHofmannPipher2023, MouPT22p}.

In the papers \cite{HofMayMou15,Bar20p}, the authors established a number of estimates on the solution $w$ to  the global Poisson problem
\begin{equation*}Lw=-\Div \vec H\quad\text{in }\R^\dmn\end{equation*}
for a function $\vec H$ supported in~$\Omega$, and for $L$ and $\Omega$ as in \cite{BarM16A}, that is, for $t$-independent coefficients~$L$ in the domain above a Lipschitz graph. The function $w$ then need not be zero on~$\partial\Omega$, but rather can be extended to a function that satisfies $Lw=0$ in $\R^\dmn\setminus\overline\Omega$. Some representative estimates are
\begin{align*}
\|\widetilde N(\nabla w)\|_{L^p(\partial\Omega)}
+\|\Tr w\|_{\dot W^{1,p}(\partial\Omega)}
&\leq C \|\mathcal{A}_2(\vec H)\|_{L^p(\partial\Omega)}
,\\
\|\mathcal{A}_2(\delta\nabla w)\|_{L^p(\partial\Omega)}
+\|\Tr w\|_{L^p(\partial\Omega)}
&\leq C \|\widetilde{\mathfrak{C}}_1(\delta\vec H)\|_{L^p(\partial\Omega)}
\end{align*}
where $\Tr w$ represents the Dirichlet boundary values (that is, the boundary trace) of~$w$, and where $\widetilde{\mathfrak{C}}_1$ is the Carleson-type operator given by
\begin{equation*}
\widetilde{\mathfrak{C}}_1 F(\xi)=
\sup_{r>0} \frac{1}{r^{\dmn-1}} \int_{B(\xi,r)\cap\Omega}
\biggl(\fint_{B(x,\delta(x)/2)} |F|^2\biggr)^{1/2} \frac{dx}{\delta(x)}.\end{equation*}
Although this is not noted in \cite{HofMayMou15,Bar20p}, given well posedness of the Dirichlet problem, we can correct $w$ to take zero boundary values, or any desired boundary values, by subtracting an appropriate solution to the Dirichlet problem; this yields solutions $u$ to the Dirichlet-Poisson problem
\begin{equation*}Lu=-\Div\vec H\text{ in }\Omega,\quad u=f\text{ on }\partial\Omega\end{equation*}
that satisfy the estimates
\begin{align}
\label{eqn:Poisson:regularity}
\|\widetilde N(\nabla u)\|_{L^p(\partial\Omega)}
&\leq C \|\mathcal{A}_2(\vec H)\|_{L^p(\partial\Omega)}
+C\|f\|_{\dot W^{1,p}(\partial\Omega)}
\quad\text{or}
\\
\label{eqn:Poisson:Dirichlet}
\|\mathcal{A}_2(\delta\nabla u)\|_{L^p(\partial\Omega)}
&\leq C \|\widetilde{\mathfrak{C}}_1(\delta\vec H)\|_{L^p(\partial\Omega)}
+C\|f\|_{L^p(\partial\Omega)}
.\end{align}

The paper \cite{MouPT22p} also considered the Poisson-Dirichlet problem with boundary data in $L^p(\partial\Omega)$ or in the Sobolev space $\dot W^{1,p}(\partial\Omega)$, or in the more general Haj\l asz space $\dot M^{1,p}(\partial\Omega)$ (with a slightly different set of estimates); unlike \cite{HofMayMou15,Bar20p}, the authors of \cite{MouPT22p} worked directly with the Poisson-Dirichlet problem in~$\Omega$ rather than working with the global Poisson problem. The paper \cite{MouPT22p}, like many recent results, considers much more general domains, and much more general coefficients, than \cite{HofMayMou15,Bar20p}. Indeed, the bulk of their argument applies in all domains with $\pdmnMinusOne$-Ahlfors regular boundary and unbounded complement satisfying the corkscrew condition.

This form of the Dirichlet-Poisson problem, with boundary data in a Lebesgue, Sobolev, or Haj\l asz space, is attractive due to the greater interest in the theory in Lebesgue, Sobolev, or Haj\l asz spaces as compared to Besov spaces. However, we observe that in above estimates, the functions $\nabla u$ and $\vec H$ lie in different spaces. In particular, the estimate on $\vec H$ is often more stringent than the estimate on~$\nabla u$. In this paper, we have chosen to focus on the form of the Poisson problem given in Theorem~\ref{thm:Poisson} (directly analogous to \eqref{eqn:MazMS10} or~\eqref{eqn:BarM16A}), because in this case $\nabla u$ and $\vec H$ lie in the same space as warranted by the equation.

\vskip .5cm {\bf Acknowledgements.} Svitlana Mayboroda is partially supported by the Simons Initiative of Geometry of Flows. Svitlana Mayboroda and Alberto Pacati are both partially supported by the Simons Collaboration on Localisation of Waves, 563916 SM.

\section{Notation and function spaces}\label{sec:function}

In this section we will define the norms on functions that we intend to use in this paper. We will then use several tools, most notably the Whitney decomposition, to establish certain properties of the weighted averaged Lebesgue and Sobolev spaces used in Theorem~\ref{thm:Poisson}.

\subsection{Definitions}

We begin with some standard notation.

We will say that $A\approx B$ if there are positive constants $c$ and $C$ such that $cA\leq B\leq CA$; we will refer to $c$ and $C$ as the implicit constants in such an inequality.

If $E\subset\R^\dmn$ then $|E|$ denotes the Lebesgue measure of~$E$, while if $\Gamma\subset\R^\dmn$ is $\bdmn$-Ahlfors regular we will let $\sigma$ denote the $\bdmn$-dimensional Hausdorff measure restricted to~$\Gamma$.

If $\Psi\subseteq\R^\dmn$ is a set of positive $\dmn$-dimensional Lebesgue measure, we let $L^p(\Psi)$ be the standard Lebesgue space with respect to Lebesgue measure; conversely, if $\Psi$ is $\bdmn$-Ahlfors regular we let $L^p(\Psi)$ be the standard Lebesgue space with respect to the $\bdmn$-dimensional Hausdorff measure~$\sigma$.

We will let $\fint$ denote the averaged integral, that is,
\begin{equation*}\fint_E F=\frac{1}{|E|} F,\quad \fint_\Psi F\,d\sigma=\frac{1}{\sigma(\Psi)} \int_\Psi F\,d\sigma\end{equation*}
whenever $E$ is a set of positive finite Lebesgue measure and $\Psi$ is (a subset of) a $\bdmn$-Ahlfors regular set with $0<\sigma(\Psi)<\infty$.

If $Q\subset\R^\dmn$ is a cube and $K>0$, then we let $x_Q$ denote its midpoint, $\ell(Q)$ denote its side-length, and $KQ$ its $K$-fold dilate; that is,
\begin{equation*}\ell(Q)=|Q|^{1/\pdmn},
\quad x_Q=\fint_Q x\,dx,\quad
KQ=\{x_Q+K(y-x_Q):y\in Q\}.\end{equation*}

We let $L^p_{loc}(\Psi)$ be the set of functions $u$ such that, if $K\subset\Psi$ is compact, then $u\in L^p(K)$ (with the same measure as used in $L^p(\Psi)$).

Let $W\subset\R^\dmn$ be open. We let $\dot W^{1,p}(W)$ be the standard homogeneous Sobolev space with norm
\begin{equation*}\|u\|_{\dot W^{1,p}(W)}=\|\nabla u\|_{L^p(W)}.
\end{equation*}
Elements of $\dot W^{1,p}(W)$ are then not functions, but equivalence classes of functions defined up to adding constants.

We define $\dot W^{1,p}_{loc}(W)$ analogously to $L^p_{loc}(W)$, as the set of all functions $u$ such that $u\in \dot W^{1,p}(V)$ for every open set $V$ such that the closure $\overline V$ is both compact and contained in~$W$.

Throughout $\Omega$ will denote an open set in $\R^\dmn$. In many cases we will need to impose additional assumptions on~$\Omega$. We will adopt the shorthand
\begin{equation}
\label{eqn:delta}
\delta(x)=\dist(x,\partial\Omega),
\qquad
\Delta(x,r)=B(x,r)\cap\partial\Omega\end{equation}
for all $x\in\R^\dmn$, all $E\subset\R^\dmn$, and all $r>0$.

The definition~\eqref{eqn:N:intro} of the nontangential maximal function $N$ of Theorem~\ref{thm:Poisson:Lq} may be written as
\begin{equation*}
Nu(\xi)=N_a u(\xi)=N_{a,\Omega}(\xi)
=\esssup_{\gamma_{a,\Omega}(\xi)} |u|\end{equation*}
where the nontangential cone $\gamma_{a,\Omega}(\xi)$ is given by
\begin{equation}\label{eqn:cone}
\gamma_a(\xi)=\gamma_{a,\Omega}(\xi)=
\{x\in\Omega:|x-\xi|<(1+a)\, \delta(x)\}\end{equation}
where $a>0$ is a fixed positive constant. When no ambiguity will arise we will omit the $\Omega$ subscript.

\begin{rmk}\label{rmk:N:aperture}%
In many familiar geometric situations, it is well known that, if $0<a<b$ and $1\leq q\leq\infty$, then
\begin{equation*}\|N_a F\|_{L^q(\partial\Omega)} \leq \|N_b F\|_{L^q(\partial\Omega)}\leq C_{a,b,q,\bdmn,A} \|N_a F\|_{L^q(\partial\Omega)}.\end{equation*}
We will now prove it in the case where $\Omega\subset\R^\dmn$ and $\partial\Omega$ is $\bdmn$-Ahlfors regular, $0<\bdmn<\dmn$.
Observe that
\begin{equation*}\{\xi\in\partial\Omega:N_aF(\xi)>\lambda\}= \bigcup_{\substack{x\in \Omega\\|F(x)|>\lambda}}
B(x,(1+a)\delta(x))\cap\partial\Omega.\end{equation*}
Apply the Vitali covering lemma to the collection of balls on the right hand side. Then
\begin{equation*}\{x\in \Omega:|F(x)|>\lambda\} \subset \bigcup_{x\in\mathcal{I}}
B(x,(5+5a)\delta(x))\cap\partial\Omega\end{equation*}
where the balls $B(x,(1+a)\delta(x))$ for $x\in\mathcal{I}$ are pairwise disjoint. Then
\begin{equation*}\sigma(\{\xi\in\partial\Omega:N_aF(\xi)>\lambda\})
\geq \sum_{x\in\mathcal{I}} \sigma(B(x,(1+a)\delta(x))\cap\partial\Omega). \end{equation*}
For each $x\in\Omega$ there is a $\xi_x\in\partial\Omega$ with $|x-\xi_x|=\delta(x)$. We have that
\begin{equation*}\Delta(\xi_x,a\delta(x))\subset B(x,(1+a)\delta(x)) \subset B(\xi,(2+a)\delta(x))\end{equation*}
and so
\begin{equation*}\sigma(\{\xi\in\partial\Omega:N_aF(\xi)>\lambda\})
\geq \sum_{x\in\mathcal{I}} \sigma(\Delta(\xi_x,a\delta(x))). \end{equation*}
Conversely,
\begin{align*}\{\xi\in\partial\Omega:N_bF(\xi)>\lambda\}
&= \bigcup_{\substack{z\in \Omega\\|F(z)|>\lambda}}
B(z,(1+b)\delta(z))\cap\partial\Omega
\\&\subset \bigcup_{x\in \mathcal{I}} \bigcup_{z\in B(x,(5+5a)\delta(x))} B(z,(1+b)\delta(z))\cap\partial\Omega
.\end{align*}
If $z\in B(x,(5+5a)\delta(x))\subset B(\xi_x,(6+5a)\delta(x))$ then $\delta(z)\leq |z-x|+\delta(x)< (6+5a)\delta(x)$, and so
\begin{align*}\{\xi\in\partial\Omega:N_bF(\xi)>\lambda\}
&\subset \bigcup_{x\in \mathcal{I}}
 \Delta(\xi_x,C_{a,b}\delta(x))
.\end{align*}
Thus
\begin{align*}
\sigma(\{\xi\in\partial\Omega:N_bF(\xi)>\lambda\})
&\leq \sum_{x\in \mathcal{I}}
 \sigma(\Delta(\xi_x,C_{a,b}\delta(x)))
.\end{align*}
By the Ahlfors regularity of the boundary,
\begin{equation*}\sigma(\Delta(\xi_x,C_{a,b}\delta(x)))
\leq C_{a,b,A}\sigma(\Delta(\xi_x,a\delta(x)))\end{equation*}
and so
\begin{equation*}\sigma(\{\xi\in\partial\Omega:N_bF(\xi)>\lambda\})
\leq C_{a,b,A} \sigma(\{\xi\in\partial\Omega:N_aF(\xi)>\lambda\}).
\end{equation*}
A standard argument yields the bound $\|N_b F\|_{L^q(\partial\Omega)}\leq C_{a,b,A,q}\|N_a F\|_{L^q(\partial\Omega)}$.

Thus, the exact value of the aperture parameter $a$ in the nontangential maximal function is not crucial; in particular, the bound~\eqref{eqn:N:bound} is either true for all choices of $N=N_a$ (with $C$ depending on~$a$) or for no choices of~$N$.
\end{rmk}

We now introduce some shorthand notation for the norms on $u$ and $\vec H$ in Theorem~\ref{thm:Poisson}.

\begin{defn}\label{dfn:weighted:averaged}
Let $\Omega$ be an open set.
Recall that $\delta(y)=\dist(y,\partial\Omega)$.

Let $0<c<1$ and $1\leq \beta\leq\infty$.
The Whitney averaging operator is given by
\begin{align*}
\mathcal{W}_{c,\beta} u(y) &=
\biggl(\fint_{B(y,c\delta(y))} |u|^\beta \biggr)^{1/\beta}
\end{align*}
for $y\in\Omega$ and $u:\Omega\to\R$,
with the obvious modifications in the case $\beta=\infty$.

Let $0<p\leq\infty$ and let $s\in\R$.
The weighted averaged Lebesgue space ${L^{p}_{av,c,\beta,s}(\Omega)}$ is the set of all $f\in L^\beta_{loc}(\Omega)$ such that
\begin{equation*}\|f\|_{L^{p}_{av,c,\beta,s}(\Omega)}
= \biggl(\int_{\Omega}
\mathcal{W}_{c,\beta} u(y)^p \, \delta(y)^{\bdmn-\pdmn+p-ps}\,dy\biggr)^{1/p}\end{equation*}
is finite.

The weighted averaged Sobolev space $\dot W^{1,p}_{av,c,\beta,s}(\Omega)$ is the set of all $f\in \dot W^{1,\beta}_{loc}(\Omega)$ such that
\begin{equation*}\|f\|_{\dot W^{1,p}_{av,c,\beta,s}(\Omega)}
= \|\nabla f\|_{L^p_{av,c,\beta,s}(\Omega)}
=
\biggl(\int_{\Omega}
\mathcal{W}_{c,\beta} (\nabla u)(y)^p\,
\delta(y)^{\bdmn-\pdmn+p-ps}\,dy\biggr)^{1/p}\end{equation*}
is finite.

We will occasionally need local versions of these spaces. If $\Psi\subset\Omega$ is open, we will let ${L^{p}_{av,c,\beta,s}(\Psi;\Omega)}$ denote the set of all $f\in L^\beta_{loc}(\Omega)$ such that
\begin{equation*}\|f\|_{L^{p}_{av,c,\beta,s}(\Psi;\Omega)}
= \biggl(\int_{\Psi}
\mathcal{W}_{c,\beta} u(y)^p \,\dist(y,\partial\Omega)^{\bdmn-\pdmn+p-ps}\,dy\biggr)^{1/p}
\end{equation*}
is finite. (We emphasize that we will use $\delta$ as shorthand for $\dist(y,\partial\Omega)$, not $\dist(y,\partial\Psi)$.) We will let $\dot W^{1,p}_{av,c,\beta,s}(\Psi;\Omega)$ denote the set of all $f$ such that $\|f\|_{\dot W^{1,p}_{av,c,\beta,s}(\Psi;\Omega)}
= \|\nabla f\|_{L^p_{av,c,\beta,s}(\Psi;\Omega)}$
is finite.
\end{defn}

We observe that Hölder's inequality immediately yields the embedding result
\begin{equation}\label{eqn:Lpav:Holder}
\|F\|_{L^p_{av,c,\beta,s}(\Omega)}\leq \|F\|_{L^p_{av,c,\rho,s}(\Omega)}\quad\text{whenever } 0<\beta<\rho\leq\infty
.\end{equation}

We will see (Lemma~\ref{lem:averaged:Whitney} below) that if $c$, $\widetilde c\in (0,1)$ then $L^p_{av,c,\beta,s}(\Omega)=L^p_{av,\widetilde c,\beta,s}(\Omega)$ with an equivalence of norms, and so in many cases the exact value of $c$ does not matter. We will thus employ the shorthand
\begin{equation*}L^p_{av,\beta,s}(\Omega)=L^p_{av,1/2,\beta,s}(\Omega),
\qquad
\dot W^{1,p}_{av,\beta,s}(\Omega) =\dot W^{1,p}_{av,1/2,\beta,s}(\Omega).\end{equation*}

\subsection{A lemma concerning Ahlfors regular sets}

In Section~\ref{sec:Whitney:Sobolev} we will prove some basic properties of the weighted averaged spaces $L^p_{av,\beta,s}$ and $\dot W^{1,p}_{av,\beta,s}$ that will be of use throughout the paper. We will need the following lemma (as well as some tools of Sections \ref{sec:dfn:Whitney} and~\ref{sec:Whitney:properties}).  This lemma lets us estimate the measure of sets near~$\partial\Omega$. We remark that the $\varpi=1$ case of formula~\eqref{eqn:delta:in:ball} is \cite[formula~(2.6)]{DavFM21}.
\begin{lem}\label{lem:close:to:Ahlfors}Let $\Gamma\subset\R^\dmn$ be $\bdmn$-Ahlfors regular for some $0<\bdmn<\dmn$ with Ahlfors constant~$A$. If $\xi\in\Gamma$ and $0<r\leq R<2\diam\Gamma$, then
\begin{equation}
\label{eqn:close:to:Ahlfors}
|\{y\in B(\xi,R):\dist(y,\partial\Omega)<r\}|
\approx R^\bdmn r^{\dmn-\pbdmn}
\end{equation}
where the implicit constants depend only on~$A$, $\bdmn$ and~$\dmn$.

Furthermore, if $\xi\in \Gamma$, $0<R<\infty$, and $\varpi>0$, then
\begin{equation}\label{eqn:delta:in:ball}
\int_{B(\xi,R)} \dist(y,\Gamma)^{\pbdmn-\pdmn+\varpi}\,dy
\approx R^{\pbdmn+\varpi}
\end{equation}
where the implicit constants depend only on~$A$, $\bdmn$, $\dmn$, and~$\varpi$.
\end{lem}

We remark that for most of this paper, we are concerned with a domain~$\Omega$ whose boundary $\partial\Omega$ is $\bdmn$-Ahlfors regular, and thus we will apply the lemma with $\Gamma=\partial\Omega$. The bound
\begin{equation*}
\int_{B(\xi,R)\cap\Omega} \dist(y,\partial\Omega)^{\pbdmn-\pdmn+\varpi}\,dy
\leq C R^{\pbdmn+\varpi}
\end{equation*}
is valid in all such domains, but we remark that the reverse inequality (which is inherently present in the bound~\eqref{eqn:delta:in:ball}) requires the domain of integration to be all of $B(\xi,R)$ rather than merely $B(\xi,R)\cap\Omega$.

\begin{proof}[Proof of Lemma~\ref{lem:close:to:Ahlfors}]
Throughout this proof let $\delta(y)=\dist(y,\Gamma)$.

We begin with the bound~\eqref{eqn:close:to:Ahlfors}. If $r>R/3$ then the estimate follows because $B(\xi,r)\subset\{y\in B(\xi,R):\delta(y)<r\}\subset B(\xi,R)$. Thus we may assume that $r\leq R/3$.

We apply the Vitali covering lemma to find points $x_j\in\Delta(\xi,R-r)$ such that
\begin{equation*}\Delta(\xi,R-r)\subset \bigcup_{j=1}^M B(x_j,r)\end{equation*}
and such that the balls $B(x_j,r/5)$ are pairwise disjoint.
We apply the Vitali covering lemma again to find points $x_j\in \Delta(\xi,R+r)$ for $M+1\leq j\leq N$ such that
\begin{equation*}\Delta(\xi,R+r)\subset \bigcup_{j=1}^N B(x_j,r)\end{equation*}
and such that the balls $B(x_j,r/5)$ are still pairwise disjoint.

If $y\in B(x_j,r)$ then $\delta(y)<r$. Conversely, if $y\in B(\xi,R)$ and $\delta(y)<r$ then $y\in B(x,r)$ for some $x\in \Delta(\xi,R+r)$. Thus
\begin{equation*}\bigcup_{j=1}^M B(x_j,r)
\subset \{y\in B(\xi,R):\delta(y)<r\}
\subset \bigcup_{j=1}^N B(x_j,2r)
\end{equation*}
and so
\begin{equation*}
M\beta_\dmn (r/5)^\dmn
\leq |\{y\in B(\xi,R):\delta(y)<r\}|
\leq N\beta_\dmn (2r)^\dmn\end{equation*}
where $\beta_\dmn$ is the volume of the unit ball in~$\R^\dmn$.

Observe that $r\leq R$ and $x_j\in B(\xi,R+r)$, and so $B(x_j,r/5)\subset B(\xi,R+6r/5)\subseteq B(\xi,(11/5)R)$. Because the sets $\Delta(x_j,r/5)$ are pairwise disjoint, we have that
\begin{equation*}
N\frac{1}{A}\biggl(\frac{r}{5}\biggr)^\bdmn
\leq
\sum_{j=1}^N \sigma(\Delta(x_j,r/5))
=
\sigma\biggl(\bigcup_{j=1}^N \Delta(x_j,r/5)\biggr)
\leq \sigma\Bigl(\Delta\Bigl(\xi,\frac{11}{5}R\Bigr)\Bigr) \leq A\frac{11^\bdmn}{5^\bdmn}R^\bdmn.\end{equation*}
Similarly,
\begin{equation*}
\frac{1}{A}(R/2)^\bdmn
\leq \sigma(\Delta(\xi,R/2))
\leq \sigma(\Delta(\xi,R-r))
\leq \sum_{j=1}^M \sigma(\Delta(x_j,r))
\leq MAr^\bdmn.\end{equation*}
Combining these estimates establishes the bound~\eqref{eqn:close:to:Ahlfors}.

We now turn to the estimate~\eqref{eqn:delta:in:ball}.
If $R\geq2\diam(\Gamma)$, then because $\bdmn+\varpi>0$ we have that
\begin{equation*}\int_{B(\xi_0,R)\setminus B(\xi_0,(3/2)\diam(\Gamma))} \delta^{\bdmn-\pdmn+\varpi} \approx \int_{B(\xi_0,R)\setminus B(\xi_0,(3/2)\diam(\Gamma))} |x-\xi_0|^{\bdmn-\pdmn+\varpi}\,dx\approx R^{\bdmn+\varpi}\end{equation*} and so we may assume $R<2\diam(\Gamma)$. Thus we may use the bound~\eqref{eqn:close:to:Ahlfors}.

If $\pbdmn-\pdmn+\varpi\geq 0$ then $\delta\leq R$ in $B(\xi,R)$ and by the bound~\eqref{eqn:close:to:Ahlfors} there is a $\varepsilon>0$ depending on $\dmn$, $\bdmn$ and~$A$ such that $\delta>\varepsilon R$ in a set of measure comparable to~$R^\dmn$. The bound~\eqref{eqn:delta:in:ball} follows immediately.

We are left with the case $R<2\diam(\Gamma)$ and $\pbdmn-\pdmn+\varpi<0$. By definition of the Lebesgue integral
\begin{align*}
\int_{B(\xi,R)} \delta^{{\pbdmn-\pdmn+\varpi}}
&= \int_0^\infty |\{x\in B(\xi,R):\delta(x)^{\pbdmn-\pdmn+\varpi}>\lambda\}|\,d\lambda
\\&= \int_0^\infty |\{x\in B(\xi,R):\delta(x)<\lambda^{-1/(\pdmn-\pbdmn-\varpi)}\}| \,d\lambda
.\end{align*}
By the bound~\eqref{eqn:close:to:Ahlfors},
\begin{align*}
\int_{B(\xi,R)} \delta^{{\pbdmn-\pdmn+\varpi}}
&\approx \int_0^\infty
\min (R^\dmn,
R^\bdmn
\lambda^{-(\dmn-\pbdmn)/{(\pdmn-\pbdmn-\varpi)}})
\,d\lambda
.\end{align*}
Because $0<\varpi<\dmn-\pbdmn$, the exponent is less than $-1$ and so the integral converges. A straightforward computation completes the proof.
\end{proof}

\subsection{Whitney decompositions of open sets}\label{sec:dfn:Whitney}

We define Whitney decompositions of open sets as follows.

\begin{defn}\label{dfn:Whitney}
A dyadic cube in $\R^\dmn$ is a set of the form \begin{equation*}Q=[j_12^{k},(j_1+1)2^k)\times\dots\times[j_\dmn2^k,(j_\dmn+1)2^k)\end{equation*} for some integers $k$, $j_1,\dots,j_\dmn$. The side length of such a cube is
$\ell(Q)=2^k=|Q|^{1/\dmn}$. Observe that the set of all dyadic cubes of side length $2^k$ forms a partition of~$\R^\dmn$.

Suppose that $Q$ is a dyadic cube in $\R^\dmn$. We let $P(Q)$ denote its dyadic parent, that is, the unique dyadic cube that satisfies both $Q\subset P(Q)$ and $\ell(P(Q))=2\ell(Q)$.

If $\Omega\subset\R^\dmn$ is a nonempty proper open subset of~$\R^\dmn$, and if $0<\kappa<\infty$, then the Whitney decomposition $\mathcal{G}_\kappa=\mathcal{G}_\kappa(\Omega)$ of $\R^\dmn\setminus\Omega$ is given by
\begin{equation*}\mathcal{G}_\kappa(\Omega)= \{Q\text{ dyadic}: \dist(Q,\partial\Omega)>\kappa\diam Q,\> \dist(P(Q),\partial\Omega) \leq \kappa \diam P(Q)\}.\end{equation*}
\end{defn}

We will frequently need the closest point to a cube; thus, if $Q\in \mathcal{G}_\kappa$, we let $\xi_Q$ satisfy
\begin{equation}\label{eqn:cube:point}
\xi_Q\in\partial\Omega,\qquad \dist(\xi_Q,Q)=\dist(\partial\Omega,Q).
\end{equation}
Because $\partial\Omega$ is closed, $\xi_Q$ must exist.

\subsection{Basic properties of Whitney decompositions}
\label{sec:Whitney:properties}

In this section we establish a few straightforward bounds on Whitney cubes.
Observe that if $x\in\Omega$ then $x\in Q$ for exactly one $Q\in\mathcal{G}_\kappa(\Omega)$.
If $Q\in \mathcal{G}_\kappa$,
then
\begin{equation*}\dist(Q,\partial\Omega)-\diam Q \leq\dist(P(Q),\partial\Omega)\leq \kappa\diam(P(Q))=2\kappa\diam Q ,\end{equation*}
and so
\begin{equation}\label{eqn:Whitney}
\kappa\diam Q < \dist(Q,\partial\Omega)\leq (2\kappa+1)\diam Q .
\end{equation}
Thus $\diam Q\approx \dist(Q,\partial\Omega)=\dist(Q,\xi_Q)$ with implicit constants depending only on~$\kappa$.

We will need the following covering lemma for the boundary and Whitney cubes.

\begin{lem}\label{lem:Whitney:boundary}
Let $\Omega\subsetneq\R^\dmn$ be a nonempty proper open subset and let $\kappa>0$, $\tau>0$.
For each $R\in \mathcal{G}_\kappa(\Omega)$, let \begin{equation*}\Delta_R=\Delta(\xi_R,\tau\diam R)=B(\xi_R, \tau\diam R)\cap\partial\Omega\end{equation*}
where $\xi_R$ is as in formula~\eqref{eqn:cube:point}.

Then there is a $C$ depending only on $\kappa$, $\tau$ and $\dmn$ such that, if $x\in\partial\Omega$ and $j\in \Z$, then
\begin{equation*}\{R\in \mathcal{G}_\kappa(\Omega):\ell(R)=2^j,\>x\in \Delta_R\}\end{equation*}
contains at most $C$ elements.
\end{lem}

\begin{proof} Fix $x\in\partial\Omega$ and $j\in\Z$.
If $R\in\mathcal{G}_\kappa$, $\ell(R)=2^j$, and $x\in\Delta_R$, then for all $y\in R$ we have that
\begin{equation*}|y-x|\leq \diam(R)+\dist(R,\xi_R)+|\xi_R-x|
\leq (2\kappa+2+\tau)\diam R =(2\kappa+2+\tau)\sqrt{\dmn}2^j\end{equation*}
by the bound~\eqref{eqn:Whitney},
and so $R\subset B(x,(2\kappa+2+\tau)\sqrt{\dmn} 2^j)$. Because each such $R$ has volume $2^{j\dmn}$ and distinct cubes in $\mathcal{G}_\kappa$ are disjoint, there are at most
$\beta_\dmn (2\kappa+2+\tau)^\dmn\sqrt{\dmn}^\dmn$ such cubes, where $\beta_\dmn$ is the volume of the unit ball in~$\R^\dmn$.
\end{proof}

We now prove the following lemma concerning dilations of Whitney cubes.
\begin{lem}\label{lem:dilate:Whitney}
Let $K\geq 1$, let $\kappa>0$, let $\Omega\subset\R^\dmn$ be a nonempty proper open set, and let $\mathcal{G}_\kappa=\mathcal{G}_\kappa(\Omega)$ be as in Definition~\ref{dfn:Whitney}.

If $Q\in\mathcal{G}_\kappa$, then
\begin{align}
\label{eqn:whitney:dilate:distance:min}
\dist(KQ,\partial\Omega)&>\biggl(\kappa-\frac{K-1}{2}\biggr)\diam Q,
\\
\label{eqn:whitney:dilate:distance:max}
\sup_{x\in KQ} |x-\xi_Q|&\leq\biggl(\frac{K+3}{2}+2\kappa\biggr) \diam(Q)\end{align}
where $\xi_Q$ is as in formula~\eqref{eqn:cube:point}.

If $Q$, $R\in\mathcal{G}_\kappa$, then
\begin{equation}
\label{eqn:Whitney:dilate:distance}
\biggl(\kappa-\frac{K-1}{2}\biggr)\diam(Q)
-\biggl(\frac{K+3}{2}+2\kappa\biggr)\diam(R)
\leq\dist(KQ,KR)\end{equation}
and so if $K<2\kappa+1$ and $KQ\cap KR\neq \emptyset$, then
\begin{equation*}
\dist(Q,\partial\Omega) \approx \diam Q\approx \diam R \approx \dist(R,\partial\Omega)\end{equation*}
with the implicit constants depending only on $\kappa$ and~$K$.
In particular, for each~$Q$ there are at most $C_{\dmn,K,\kappa}$ such cubes~$R$.
\end{lem}
\begin{proof}
Observe that if $x\in KQ$ then $\dist(x,Q)\leq \frac{K-1}{2}\diam Q$. By the triangle inequality and the bound~\eqref{eqn:Whitney},
\begin{align*}|x-\xi_Q|
&\leq \dist(x,Q)+\diam(Q)+\dist(Q,\xi_Q)
\\&\leq \frac{K-1}{2}\diam Q +\diam(Q)+(2\kappa+1)\diam(Q).\end{align*}
This establishes the bound~\eqref{eqn:whitney:dilate:distance:max}.
Again by the triangle inequality,
\begin{equation*}\dist(Q,\partial\Omega)\leq \dist(Q,x)+\dist(x,\partial\Omega)\leq
\frac{K-1}{2}\diam Q+\dist(x,\partial\Omega),\end{equation*}
which by the definition of Whitney decomposition gives the bound~\eqref{eqn:whitney:dilate:distance:min}.

If $Q$, $R\in \mathcal{G}_\kappa$ and $x\in KR$, then
\begin{align*}
\dist(KQ,\partial\Omega)
&\leq\dist(KQ,\xi_R)\leq \dist(KQ,x)+|x-\xi_R|
\\&\leq \dist(KQ,x)+\biggl(\frac{K+3}{2}+2\kappa\biggr)\diam(R)
.\end{align*}
Taking the infimum over all such~$x$ and applying the bound~\eqref{eqn:whitney:dilate:distance:min} establishes the bound~\eqref{eqn:Whitney:dilate:distance}.
\end{proof}

\subsection{The basic lemma concerning the Whitney decomposition and weighted averaged Sobolev spaces}

In this section we will prove the following lemma. This gives an alternative characterization (in terms of Whitney cubes) of the weighted averaged spaces of Definition~\ref{dfn:weighted:averaged}.

\begin{lem}\label{lem:averaged:Whitney}Let $\Omega\subset\R^\dmn$ be a nonempty proper open subset of~$\R^\dmn$. Let $0<c<1$ and let $0<\kappa<\infty$. Let $\mathcal{G}_\kappa$ be as in Definition~\ref{dfn:Whitney} and let $1\leq K<2\kappa+1$.

Let $1\leq\beta\leq\infty$, $0< p\leq\infty$ and $s\in\R$. Let $F\in L^\beta_{loc}(\Omega)$. Then
\begin{equation}
\label{eqn:averaged:Whitney}
\|F\|_{L^p_{av,c,\beta,s}(\Omega)}
\approx
\biggl(\sum_{Q\in\mathcal{G}_\kappa} \biggl(\fint_{KQ} |F|^\beta\biggr)^{p/\beta} \ell(Q)^{\bdmn+p-ps}\biggr)^{1/p}
\end{equation}
where the implicit constants depend only on $p$, $s$, $\kappa$, $K$, $\dmn$, $\bdmn$, and~$c$.
\end{lem}

We divide the proof into three steps.

\subsubsection{The upper bound on \texorpdfstring{$\|F\|_{L^p_{av,c,\beta,s}(\Omega)}$}{the norm of F}}

In this case it suffices to consider the case $K=1$, as
\begin{equation*}\sum_{Q\in\mathcal{G}_\kappa} \biggl(\fint_{Q} |F|^\beta\biggr)^{p/\beta} \ell(Q)^{\bdmn+p-ps}
\leq
K^{\pdmn p/\beta}\sum_{Q\in\mathcal{G}_\kappa} \biggl(\fint_{KQ} |F|^\beta\biggr)^{p/\beta} \ell(Q)^{\bdmn+p-ps}.\end{equation*}
Observe that if $x\in\Omega$, then
\begin{equation*}
\biggl(\int_{B(x,c\delta(x))} |F|^\beta\biggr)^{1/\beta}
\leq
\biggl(\sum_{\substack{Q\in\mathcal{G}_\kappa\\Q\cap B(x,c\delta(x))\neq \emptyset}}
\int_{Q} |F|^\beta\biggr)^{1/\beta} .\end{equation*}

Now, if $Q\in\mathcal{G}_\kappa$, $x\in\Omega$, and $B(x,c\delta(x))\cap Q\neq \emptyset$, we may compute a number of useful facts relating $\delta(x)$, $\diam(Q)$, and~$Q$.

Let $y\in B(x,c\delta(x))\cap Q$. Recalling the bound~\eqref{eqn:Whitney}, we have that
\begin{equation*}
\kappa\diam Q <\dist(Q,\partial\Omega)\leq \delta(y)<\delta(x)+c\delta(x)=(1+c)\delta(x)\end{equation*}
and, conversely,
\begin{equation*}\delta(x)\leq \delta(y)+|x-y|\leq \dist(Q,\partial\Omega)+\diam Q +c\delta(x)\end{equation*}
and so
\begin{equation*}(1-c)\delta(x)
\leq (2\kappa+2)\diam Q
.\end{equation*}
Thus, $\delta(x)\approx \diam(Q)$.
Furthermore, $\dist(x,Q)\leq |x-y|$ and so
\begin{equation*}Q\subset \overline{B(x,c\delta(x)+\diam Q)}.\end{equation*}
Because the cubes are pairwise-disjoint and $\delta(x)$ and $\diam Q$ are comparable, this means that there are at most $C_{\dmn,\kappa,c}$ cubes $Q\in\mathcal{G}_\kappa$ with $Q\cap B(x,c\delta(x))\neq \emptyset$.

Thus,
\begin{equation*}
\mathcal{W}_{c,\beta}F(x)
=
\biggl(\fint_{B(x,c\delta(x))} |F|^\beta\biggr)^{1/\beta}
\leq
C
\biggl(\sum_{\substack{Q\in\mathcal{G}_\kappa\\Q\cap B(x,c\delta(x))\neq \emptyset}}
\fint_{Q} |F|^\beta \biggr)^{1/\beta}
.\end{equation*}
Because there are at most $C$ such cubes~$Q$, we have that
\begin{equation*}
\mathcal{W}_{c,\beta}F(x)
\leq
C
\sum_{\substack{Q\in\mathcal{G}_\kappa\\Q\cap B(x,c\delta(x))\neq \emptyset}}
\biggl(\fint_{Q} |F|^\beta \biggr)^{1/\beta}
.\end{equation*}

We now bound $\|F\|_{L^\infty_{av,c,\beta,s}(\Omega)}$. By Definition~\ref{dfn:weighted:averaged}, because $\delta(x)\approx\ell(Q)$ for all such cubes~$Q$, and again because there are at most $C$ such~$Q$, we have that
\begin{align*}
\|F\|_{L^\infty_{av,c,\beta,s}(\Omega)}
&=\sup_{x\in\Omega}
\mathcal{W}_{c,\beta} F(x)\,\delta(x)^{1-s}
\leq
C_{\dmn,\kappa,c,s}
\sup_{Q\in\mathcal{G}_\kappa}
\biggl(
\fint_{Q} |F|^\beta\biggr)^{1/\beta}
\ell(Q)^{1-s}
.\end{align*}
We now bound $\|F\|_{L^p_{av,c,\beta,s}(\Omega)}$ for $p$ finite. Again by Definition~\ref{dfn:weighted:averaged}, we have that
\begin{align*}
\|F\|_{L^p_{av,c,\beta,s}(\Omega)}^p
&=
\int_\Omega\mathcal{W}_{c,\beta} F(x)^p\,\delta(x)^{\bdmn-\pdmn+p-ps}\,dx
\\&\leq
C
\int_\Omega
\sum_{\substack{Q\in\mathcal{G}_\kappa\\Q\cap B(x,c\delta(x))\neq \emptyset}}
\biggl(
\fint_{Q} |F|^\beta\biggr)^{p/\beta}
\delta(x)^{\bdmn-\pdmn+p-ps}\,dx
.\end{align*}
Interchanging the order of summation and integration, we have that
\begin{align*}
\|F\|_{L^p_{av,c,\beta,s}(\Omega)}^p
&\leq
C
\sum_{{Q\in\mathcal{G}_\kappa}}
\biggl(\fint_{Q} |F|^\beta\biggr)^{p/\beta}
\int_{\{x\in\Omega:Q\cap B(x,c\delta(x))\neq \emptyset\}}
\delta(x)^{\bdmn-\pdmn+p-ps}\,dx
.\end{align*}
Observe that $\delta(x)$ is bounded above and below by multiples of $\ell(Q)$ and also the set $\{x\in\Omega:Q\cap B(x,c\delta(x))\neq \emptyset\}$ is contained in a ball whose radius is again a multiple of~$\ell(Q)$, and so we have that
\begin{align*}
\|F\|_{L^p_{av,c,\beta,s}(\Omega)}
&\leq
C_{\dmn,p,\kappa,c}
\biggl(\sum_{{Q\in\mathcal{G}_\kappa}}
\biggl(\fint_{Q} |F|^\beta\biggr)^{p/\beta}
\ell(Q)^{\bdmn+p-ps}\biggr)^{1/p}
\end{align*}
for all $0<p\leq\infty$.
This concludes the first step.

\subsubsection{The lower bound on \texorpdfstring{$\|F\|_{L^p_{av,c,\beta,s}(\Omega)}$}{the norm of F} for \texorpdfstring{$\kappa$}{κ} large enough}\label{sec:step2:Whitney:Sobolev}

Suppose that $\kappa\geq K/c$. Then if $x\in Q\in\mathcal{G}_\kappa$, then
\begin{equation*}\diam (KQ) <\frac{K}{\kappa}\dist(Q,\partial\Omega)\leq \frac{K}{\kappa}\delta(x)\leq c\delta(x)\end{equation*}
and so $KQ\subset B(x,c\delta(x))$. Thus, if $1\leq \beta\leq\infty$, then
\begin{align*}\biggl(\int_{KQ} |F|^\beta\biggr)^{1/\beta}
&\leq \inf_{x\in Q}
\biggl(\int_{B(x,c\delta(x))} |F|^\beta\biggr)^{1/\beta}
.\end{align*}
Furthermore, $\ell(Q)\leq \ell(KQ)<\diam(KQ)\leq c\delta(x)$ and $\delta(x)\leq\diam Q +\dist(Q,\partial\Omega)\leq C_\kappa \ell(Q)$.
Thus, if $\theta\in\R$, then
\begin{align*}\biggl(\fint_{KQ} |F|^\beta\biggr)^{1/\beta} \ell(Q)^{\theta}
&\leq \inf_{x\in Q}
C
\biggl(\fint_{B(x,c\delta(x))} |F|^\beta\biggr)^{1/\beta} \delta(x)^{\theta}
=C
\inf_{x\in Q}
\mathcal{W}_{c,\beta} F(x)\delta(x)^{\theta}
.\end{align*}

Recall the definition~\ref{dfn:weighted:averaged} of $L^p_{av,c,\beta,s}(\Omega)$.
If $p=\infty$ we take $\theta=1-s$, and use the fact that the infimum is bounded by the essential supremum, to see that
\begin{align*}
\sup_{Q\in\mathcal{G}_\kappa}
\biggl(\fint_{KQ} |F|^\beta\biggr)^{1/\beta} \ell(Q)^{1-s}
&\leq
\sup_{Q\in\mathcal{G}_\kappa}
C
\esssup_{x\in Q}
\mathcal{W}_{c,\beta} F(x) \,
\delta(x)^{1-s}
=
C
\|F\|_{L^\infty_{av,c,\beta,s}(\Omega)}
.\end{align*}

If $p<\infty$, then we use the fact that the infimum is bounded by the average to see that
\begin{align*}\biggl(\sum_{Q\in\mathcal{G}_\kappa}
\biggl(\fint_{KQ} |F|^\beta\biggr)^{p/\beta} \ell(Q)^{\theta}\ell(Q)^{\dmn}
\biggr)^{1/p}
&\leq
C
\biggl(\sum_{Q\in\mathcal{G}_\kappa}
\int_Q  \mathcal{W}_{c,\beta} F(x)^p \delta(x)^{\theta}
\,dx
\biggr)^{1/p}
\\&=
C
\biggl(
\int_\Omega  \mathcal{W}_{c,\beta} F(x)^p \delta(x)^{\theta}
\,dx
\biggr)^{1/p}
\end{align*}
and taking $\theta=\bdmn-\pdmn+p-ps$ yields that
\begin{align*}
\biggl(\sum_{Q\in\mathcal{G}_\kappa}
\biggl(\fint_{KQ} |F|^\beta\biggr)^{p/\beta} \ell(Q)^{\bdmn+p-ps}
\biggr)^{1/p}
&\leq
C
\|F\|_{L^p_{av,c,\beta,s}(\Omega)}
\end{align*}
for all $0<p\leq\infty$, as long as $\kappa\geq K/c$.

\subsubsection{Arbitrary \texorpdfstring{$\kappa$}{κ}}

Suppose now that $0<\kappa<K/c$. Recall that we have further assumed that $K<2\kappa+1$. We have established that
\begin{align*}
\biggl(\sum_{Q\in\mathcal{G}_{K/c}}
\biggl(\fint_{KQ} |F|^\beta\biggr)^{p/\beta} \ell(Q)^{\bdmn+p-ps}
\biggr)^{1/p}
&\leq
C
\|F\|_{L^p_{av,c,\beta,s}(\Omega)}
.\end{align*}

If $Q\in \mathcal{G}_{K/c}$, then $\dist(Q,\partial\Omega)>{(K/c)}\diam Q>\kappa\diam Q$ and so $Q\subseteq R$ for some $R\in\mathcal{G}_\kappa$.
But if $Q\subseteq R\in\mathcal{G}_\kappa$ and $Q\in\mathcal{G}_{K/c}$, then
\begin{equation*}(2K/c+1)\diam Q\geq \dist(Q,\partial\Omega)\geq \dist(R,\partial\Omega)>\kappa\diam R \end{equation*}
and so $\diam R\geq\diam Q\geq \frac{c\kappa}{2K+c}\diam R$. In particular, if $R\in\mathcal{G}_{\kappa}$ then there are at most $C_{\dmn,\kappa,K,c}$ cubes $Q\in\mathcal{G}_{K/c}$ with $Q\subseteq R$.

To complete the proof, it suffices to show that
\begin{equation*}\biggl(\sum_{R\in\mathcal{G}_{\kappa}}
\biggl(\fint_{KR} |F|^\beta\biggr)^{p/\beta} \ell(R)^{\bdmn+p-ps}
\biggr)^{1/p}
\leq
C \biggl(\sum_{Q\in\mathcal{G}_{K/c}}
\biggl(\fint_{KQ} |F|^\beta\biggr)^{p/\beta} \ell(Q)^{\bdmn+p-ps}
\biggr)^{1/p}\end{equation*}
and so we wish to bound averages over $KR$ for $R\in \mathcal{G}_\kappa$.

Suppose that $R\in \mathcal{G}_\kappa$. By Lemma~\ref{lem:dilate:Whitney}, we have that if $K<2\kappa+1$, then there are at most $C_{\dmn,K,\kappa}$ cubes $R'\in \mathcal{G}_\kappa$ such that $KR\cap KR'\neq \emptyset$. Furthermore, $\diam(R)\approx\diam(R')$ for all such~$R'$. Let $\mathcal{H}_{K/c}(R)= \{Q\in \mathcal{G}_{K/c}: Q\subset R'$ for some $R'\in \mathcal{G}_\kappa$ with $KR\cap KR'\neq \emptyset\}$. By the above argument, $\mathcal{H}_{K/c}(R)$ contains at most $C$ cubes, all of diameter comparable to~$\diam(R)$. Thus if $\theta\in\R$ then
\begin{equation*}\biggl(\fint_{KR} |F|^\beta\biggr)^{1/\beta}\ell(R)^\theta
\leq C \biggl(\sum_{Q\in \mathcal{H}_{K/c}(R)}
\fint_{Q} |F|^\beta\biggr)^{1/\beta}\ell(R)^{\theta}
\end{equation*}
and because $\mathcal{H}_{K/c}(R)$ contains at most $C$ elements, all of diameter comparable to~$\diam(R)$, we have that
\begin{equation*}\biggl(\fint_{KR} |F|^\beta\biggr)^{q/\beta}\ell(R)^\theta
\leq C_q \sum_{Q\in \mathcal{H}_{K/c}(R)}\biggl(
\fint_{Q} |F|^\beta\biggr)^{q/\beta}\ell(Q)^{\theta}\end{equation*}
for all $0<q<\infty$.

If $p=\infty$ then we take $q=1$ and $\theta=1-s$ and use the fact that $\mathcal{H}_{K/c}(R)$ contains at most $C$ cubes~$Q$ to see that
\begin{equation*}\sup_{R\in\mathcal{G}_\kappa}
\biggl(\fint_{KR} |F|^\beta\biggr)^{1/\beta}
\ell(R)^{1-s}
\leq
C_{\dmn,\kappa,K,c,s}
\sup_{\substack{Q\in\mathcal{G}_{K/c}}} \biggl(
\fint_Q |F|^\beta\biggr)^{1/\beta}
\ell(Q)^{1-s}
.\end{equation*}
The right hand side is at most $C\|F\|_{L^\infty_{av,c,\beta,s}(\Omega)}$ by Section~\ref{sec:step2:Whitney:Sobolev}.

If $0<p<\infty$ we take $q=p$ and sum to see that
\begin{align*}
\sum_{R\in \mathcal{G}_\kappa}
\biggl(\fint_{KR} |F|^\beta\biggr)^{p/\beta}
\ell(R)^{\theta}
&\leq
C_{\dmn,\kappa,K,c,p,\theta}
\sum_{R\in \mathcal{G}_\kappa}
\sum_{\substack{Q\in\mathcal{H}_{K/c}(R)}} \biggl(
\fint_Q |F|^\beta\biggr)^{p/\beta}
\ell(Q)^{\theta}
.\end{align*}
If $Q\in \mathcal{H}_{K/c}(R)$ then $\diam(Q)\approx\diam(R)$ and $\dist(R,Q)\leq C\diam(Q)$, and so there can be at most $C$ cubes $R\in\mathcal{G}_\kappa$ with $Q\in \mathcal{H}_{K/c}(R)$. Thus we may take $\theta=\bdmn+p-ps$ and change the order of summation and see that
\begin{align*}
\sum_{R\in \mathcal{G}_\kappa}
\biggl(\fint_{KR} |F|^\beta\biggr)^{p/\beta}
\ell(R)^{\bdmn+p-ps}
&\leq
C_{\dmn,\kappa,K,c,p,\theta}
\sum_{\substack{Q\in\mathcal{G}_{K/c}}} \biggl(
\fint_Q |F|^\beta\biggr)^{p/\beta}
\ell(Q)^{\bdmn+p-ps}
.\end{align*}
Again the right hand side is bounded by Section~\ref{sec:step2:Whitney:Sobolev}, as desired.

\subsection{Further results concerning the Whitney decomposition and weighted averaged Sobolev spaces}\label{sec:Whitney:Sobolev}
In this section we collect some corollaries of Lemma~\ref{lem:averaged:Whitney} that will be of use throughout the paper.

We have the following embedding result.

\begin{cor}\label{cor:weighted:embedding}Let $\Omega\subset\R^\dmn$ be a nonempty open proper subset. Let $0<c<1$.
Let $1\leq \beta\leq\infty$, $0<p<q\leq\infty$, and $s\in\R$.
Then
\begin{equation*}L^p_{av,c,\beta,s}(\Omega)\subset L^q_{av,c,\beta,s+\bdmn/q-\bdmn/p}(\Omega)\end{equation*}
and, if $F\in L^p_{av,c,\beta,s}(\Omega)$, then
\begin{equation*}\|F\|_{L^q_{av,c,\beta,s+\bdmn/q-\bdmn/p}(\Omega)}\leq C_{\dmn,\bdmn,c,p,q,s}\|F\|_{L^p_{av,c,\beta,s}(\Omega)}.
\end{equation*}
\end{cor}

\begin{proof}
By Lemma~\ref{lem:averaged:Whitney}, we have that
\begin{align*}\|F\|_{L^p_{av,c,\beta,s}(\Omega)}
&\approx
\biggl(\sum_{Q\in\mathcal{G}_\kappa} \biggl(\fint_Q |F|^\beta\biggr)^{p/\beta} \ell(Q)^{\bdmn+p-ps}\biggr)^{1/p}
\\&=
\biggl(\sum_{Q\in\mathcal{G}_\kappa}
\biggl[\biggl(\fint_Q |F|^\beta\biggr)^{1/\beta} \ell(Q)^{\bdmn/p+1-s}\biggr]^p\biggr)^{1/p}
\end{align*}
where the implicit constants depend only on $p$, $s$, $\kappa$, $\dmn$, $\bdmn$, and~$c$.

By the standard embedding result for sequence spaces,
\begin{multline*}
\biggl(\sum_{Q\in\mathcal{G}_\kappa}
\biggl[\biggl(\fint_Q |F|^\beta\biggr)^{1/\beta} \ell(Q)^{\bdmn/p+1-s}\biggr]^p\biggr)^{1/p}
\\
\geq
\biggl(\sum_{Q\in\mathcal{G}_\kappa}
\biggl[\biggl(\fint_Q |F|^\beta\biggr)^{1/\beta} \ell(Q)^{\bdmn/p+1-s}\biggr]^q\biggr)^{1/q}
\\=
\biggl(\sum_{Q\in\mathcal{G}_\kappa}
\biggl(\fint_Q |F|^\beta\biggr)^{q/\beta}
\ell(Q)^{\bdmn+q -q(\bdmn/q-\bdmn /p+s)}
\biggr)^{1/q}.
\end{multline*}
Another application of Lemma~\ref{lem:averaged:Whitney} completes the proof.
\end{proof}

The $\beta=p$ case of Lemma~\ref{lem:averaged:Whitney} immediately yields the following corollary relating averaged spaces to unaveraged spaces.

\begin{cor}\label{cor:averaged:is:Wp}Let $\Omega\subset\R^\dmn$ be a nonempty open proper subset. Let $0<c<1$.
Let $0<p\leq\infty$ and $\theta\in\R$. Let $F\in L^p_{loc}(\Omega)$. Then
\begin{equation*}\int_\Omega \biggl(\fint_{B(x,c\delta(x))} |F|^p\biggr)\,\delta(x)^\theta\,dx
\approx
\int_\Omega |F|^p \delta^\theta
\end{equation*}
where the implicit constants depend only on $p$, $\theta$, $\dmn$, and~$c$.
\end{cor}

We will conclude this section with the following lemma. This lemma may be viewed as an extension of Corollary~\ref{cor:averaged:is:Wp} to the range $\beta\leq p$; it will also provide some precise local control.

\begin{lem}\label{lem:averaged:to:unaveraged}
Let $\Omega\subset\R^\dmn$ be a nonempty open proper subset, and suppose that $\partial\Omega$ is $\bdmn$-Ahlfors regular for some $0<\bdmn<\dmn$. Let $x\in\partial\Omega$, $0<c<1$, and let $0\leq\varrho<r<2\diam(\partial\Omega)$. Define \begin{equation*}A=\Omega\cap B(x,r)\setminus B(x,\varrho),
\qquad \Psi=\Omega\cap B\biggl(x,\frac{1}{1-c}r\biggr)\setminus B\biggl(x,\frac{1}{1+c}\varrho\biggr).\end{equation*}

Let $\varsigma$, $s\in\R$, and suppose that either $1\leq \beta=p\leq \infty$ and $\varsigma\leq s$, or $1\leq \beta<p\leq \infty$ and
$\varsigma<s$.

Let $F\in L^\beta_{loc}(\Omega)$. Then we have that
\begin{align*}
\biggl(\int_{A} |F|^\beta \delta^{\bdmn-\pdmn+\beta-\beta\varsigma}\biggr)^{1/\beta}
&\leq
C
r^{\bdmn/\beta-\bdmn/p-\varsigma+s}
\|F\|_{L^p_{av,c,\beta,s}(\Psi;\Omega)}
\end{align*}
provided the right hand side is finite, where $C$ is a constant depending only on $A$, $\dmn$, $\bdmn$, $c$, and~$\theta$.
\end{lem}

\begin{proof}
If $p=\beta=\infty$ the result is obvious. Thus we may assume that $\beta<\infty$.

Recall that
\begin{equation*}\|F\|_{L^p_{av,c,\beta,s}(\Psi;\Omega)}=\biggl(\int_{\Psi} \mathcal{W}_{c,\beta}F(z)^p \delta(z)^{\bdmn-\pdmn+p-ps}\,dz\biggr)^{1/p}. \end{equation*}
By Corollary~\ref{cor:averaged:is:Wp} (with $p$ replaced by~$\beta$),
\begin{equation*}
\int_{A} |F|^\beta \delta^{\bdmn-\pdmn+\beta-\beta\varsigma}
\leq C\int_{\Omega} \biggl(\fint_{B(z,c\delta(z))} \1_{A} |F|^\beta\biggr) \delta(z)^{\bdmn-\pdmn+\beta-\beta\varsigma}\,dz.\end{equation*}
If the inner integral is nonzero, then $B(z,c\delta(z))\cap B(x,r)\neq \emptyset$ and so
\begin{equation*}|x-z|<r+c\delta(z)\leq r+c|x-z|.\end{equation*}
Thus $z\in B(x,\frac{1}{1-c}r)$. Furthermore, $B(z,c\delta(z))\not \subseteq B(x,\varrho)$ and so
\begin{equation*}\varrho\leq|x-z|+c\delta(z)\leq (1+c)|x-z|\end{equation*}
and so $z\notin B(x,\frac{1}{1+c}\varrho)$.
Thus the inner integral is zero if $z\notin\Psi$ and so
\begin{align*}
\int_{A} |F|^\beta \delta^{\bdmn-\pdmn+\beta-\beta\varsigma}
&\leq
C
\int_{\Psi} \biggl(\fint_{B(z,c\delta(z))} |F|^\beta\biggr) \delta(z)^{\bdmn-\pdmn+\beta-\beta\varsigma}\,dz
\\&
=
C
\int_{\Psi} (\mathcal{W}_{c,\beta} F)^\beta\delta^{\bdmn-\pdmn+\beta-\beta\varsigma}
.\end{align*}
If $p=\beta$, then we are done because $\varsigma\leq s$, $\delta\leq r/(1-c)$, and so $\delta^{\beta-\beta\varsigma}=\delta^{p-p\varsigma}\leq r^{ps-p\varsigma}\delta^{p-ps}$ in $B(x,r/(1-c))$.
Otherwise, by Hölder's inequality,
\begin{align*}
\int_{A} |F|^\beta \delta^{\bdmn-\pdmn+\beta-\beta\varsigma}
&\leq
C
\biggl(\int_{\Psi} (\mathcal{W}_{c,\beta} F)^p\delta^{\bdmn-\pdmn+p-ps}
\biggr)^{\beta/p}
\biggl(\int_{\Psi}
\delta^{\bdmn-\pdmn+\varpi}\biggr)^{1-\beta/p}
\end{align*}
where $\varpi>0$ if $s>\varsigma$ and $p>\beta$ (with a suitable modification in the case $p=\infty$). Thus we may apply Lemma~\ref{lem:close:to:Ahlfors} to complete the proof.
\end{proof}

\section{Boundary values}
\label{sec:trace:zero}

Recall that Theorems~\ref{thm:Poisson} and~\ref{thm:Poisson:Lq} concern functions $u$ defined in a domain~$\Omega$ that are zero on~$\partial\Omega$ in the sense that formula~\eqref{eqn:trace} is true for almost every $\xi\in\partial\Omega$.
In this section we will prove some important results for functions in $\dot W^{1,p}_{av,\beta,s}$ that satisfy this condition.

\subsection{Definitions: weak local John domains and traces}

We will not work in the full generality of domains with Ahlfors regular boundaries; we will additionally need our domains to satisfy the weak local John condition (as required by Theorem~\ref{thm:Poisson}), which we now define.

\begin{defn}\label{dfn:local:John}
Let $\Omega\subset\R^\dmn$ be open.

Let $x\in\partial\Omega$ and let $y\in\Omega$. We say that $\zeta$ is a $\lambda$-carrot path from $x$ to $y$ if $\zeta$ is a connected rectifiable path, which we require to be parameterized by arc length, such that $\zeta:[0,T]\to\Omega\cup\{x\}$ is continuous, $\zeta(0)=x$, $\zeta(T)=y$, and $\delta(\zeta(t))=\dist(\zeta(t),\partial\Omega)\geq \lambda t$ for all $t\in [0,T]$.

Suppose that $\partial\Omega$ is $\bdmn$-Ahlfors regular for some $0<\bdmn<\dmn$.
Let $y\in\Omega$. Let $\theta$, $\lambda\in (0,1]$ and let $N\geq 2$. We say that $y\in\Omega$ is a $(\theta,\lambda,N)$-weak local John point if there is a Borel set $F\subset B(y,N\delta(y))\cap\partial\Omega$ such that $\sigma(F)\geq \theta\sigma( B(y,N\delta(y))\cap\partial\Omega )$ and such that, if $x\in F$, then there is a $\lambda$-carrot path from $x$ to~$y$.

We say that $\Omega$ is a weak local John domain if there exist constants $\theta\in (0,1]$, $\lambda\in (0,1]$ and $N\geq 2$ such that every $y\in\Omega$ is a $(\theta,\lambda,N)$-weak local John point.
\end{defn}

\begin{rmk}\label{rmk:carrot:Lusin}Suppose that there is a $\lambda$-carrot path of length~$T$ from $x$ to~$y$. If $t\in (0,T]$, then $(1/\lambda)\delta(\zeta(t))\geq t$ by definition of carrot path. But $t\geq |\zeta(t)-x|$ by definition of the length of a path. Thus by the definition~\eqref{eqn:cone} of a nontangential cone, $\zeta((0,T])\subset \gamma_a(x)$ whenever $a>\frac{1-\lambda}{\lambda}$.
\end{rmk}

Certain of the cited results of Section~\ref{sec:elliptic:boundary} will require the stronger conditions of interior corkscrews and Harnack chains.

\begin{defn}[Interior corkscrew condition]\label{dfn:iCS}
We say that $\Omega$ satisfies the {interior corkscrew condition} if there is a constant $\corkscrew\in (0,1)$ such that, if $x\in \partial\Omega$ and $0<r<\diam\Omega$, then there is a point $A_r(x)$ (called a corkscrew point) such that
\begin{equation*}
A_r(x)\in B(x,r)\quad\text{and}\quad B(A_r(x), \corkscrew r)\subset\Omega.
\end{equation*}
\end{defn}

\begin{defn}[Harnack chains]\label{dfn:Harnack}
Let $\Omega\subset\R^\dmn$. We say that $\Omega$ satisfies the \emph{Harnack chain condition} if there are constants $M\geq 1$ and $m\in (0,1)$ such that if $x$, $y\in \Omega$, then there is an integer~$N$, depending only on the ratio $|x-y|/\delta$ for $\delta=\min(\dist(x,\partial\Omega),\dist(y,\partial\Omega))$, such that there is a chain of open balls $\{B_k\}_{k=1}^N$ with $x\in B_1$, $y\in B_N$, $B_k\cap B_{k+1}\neq\emptyset$ for all $1\leq k<N$, and $m\diam(B_k)\leq \dist(B_k,\partial\Omega)\leq M\diam(B_k)$.
\end{defn}

\begin{rmk}\label{rmk:higher:John:corkscrew}
By connecting corkscrew points with Harnack chains we can construct carrot paths to a point $y$ in the domain from \emph{any} nearby boundary point~$x$. Thus any domain that satisfies both the interior corkscrew and Harnack chain conditions is a weak local John domain.

If $0<\bdmn<\dmnMinusOne$, then the only possible choice for~$\Omega$ is $\Omega=\R^\dmn\setminus\partial\Omega$ (that is, a lower dimensional boundary cannot divide $\R^\dmn$ into multiple open connected components).
Furthermore, if $\bdmn<\dmnMinusOne$ then, by \cite[Lemmas 2.1 and~11.6]{DavFM21}, $\Omega$ satisfies both the interior corkscrew and Harnack chain conditions, and so $\Omega$ is necessarily a weak local John domain.

If $\bdmn\geq\dmnMinusOne$, then the weak local John condition is not guaranteed by the Ahlfors condition. However, it is often a necessary condition in the theory; see, for example, \cite{AzzHMMT20,HofMMTZ21}, where it is shown that given Ahlfors regularity, interior corkscrews, and certain natural conditions on the coefficients~$A$, the weak local John condition is a necessary condition for well posedness of the $L^q$ Dirichlet problem (that is, for the estimate~\eqref{eqn:N:bound}).

The authors conjecture that the interior corkscrew and Harnack chain conditions in Theorem~\ref{thm:Poisson} can be relaxed to the weak local John condition even in the lower codimensional case $\dmnMinusOne<\bdmn<\dmn$, but proving this is beyond the scope of this paper.
\end{rmk}

Recall that in Theorem~\ref{thm:Poisson}, we say that $u=0$ on $\partial\Omega$ in the averaged sense if formula~\eqref{eqn:trace} holds, that is, if
\begin{equation*}\lim_{r\to 0^+} \frac{1}{r^n} \int_{B(\xi,r)\cap\Omega} |u|=0\end{equation*}
for $\sigma$-almost every $\xi\in\partial\Omega$. We will use the shorthand
\begin{equation}\label{dfn:Tr}
\Tr u(\xi)=0 \quad\text{if}\quad \lim_{r\to 0^+} \frac{1}{r^n} \int_{B(\xi,r)\cap\Omega} |u|=0.\end{equation}
Another notion of boundary value (boundary trace) will also be useful for us. We now define this notion of boundary trace.

\begin{defn}\label{dfn:trace}
Let $\Omega\subset\R^\dmn$ be open and let $x\in\partial\Omega$. Let $u\in L^1_{loc}(\Omega)$. If $a>0$ and $c\in (0,1)$, then we define the boundary trace $\Tr_{a,c} u(x)$ by
\begin{equation*}\Tr_{a,c} u(x)=\Tr_{a,c,\Omega} u(x)=U\quad\text{if}\quad x\in \overline{\gamma_a(x)}\quad\text{and}\quad
\lim_{\substack{y\to x\\y\in\gamma_a(x)}} \fint_{B(y,c\delta(y))} |u-U|=0.\end{equation*}
(As usual, if there is no such $U$ then we say that $\Tr_{a,c} u(x)$ does not exist. If $x\not\in \overline{\gamma_a(x)}$, that is, if there is some ball $B(x,r)$ such that $\delta(y)\leq|y-x|/(1+a)$ for all $y\in \Omega\cap B(x,r)$, then we again say that $\Tr_{a,c} u(x)$ does not exist.)
\end{defn}

If $\Tr u(\xi)=0$ in the sense of formula~\eqref{dfn:Tr}, and if $\xi\in \overline{\gamma_a(\xi)}$, then clearly $\Tr_{a,c}u(\xi)=0$. Thus if $\Tr u=0$ almost everywere on~$\partial\Omega$ (that is, $u=0$ almost everywhere on $\partial\Omega$ in the sense of Theorem~\ref{thm:Poisson}), then $\Tr_{a,c} u(x)=0$ for almost every $x\in\partial\Omega$ with $x\in\overline{\gamma_a(x)}$.

Under some conditions on~$u$ and~$\Omega$, we will prove the converse, that is, if $\Tr_{a,c}u(x)=0$ for almost all such nontangentially accessible~$x\in\partial\Omega$ then $\Tr u(\xi)=0$ for $\sigma$-almost every~$\xi\in\partial\Omega$, in Lemmas \ref{lem:boundary:Poincare} and~\ref{lem:trace:zero}.

\subsection{Bounding a function using a carrot path}
\label{sec:trace:Lusin}

The following result allows us to bound a function on an interior ball, given existence of a carrot path and an area integral estimate on the gradient.

\begin{lem}\label{lem:trace:Lusin}
Let $\Omega\subset\R^\dmn$ be open.
Suppose that $x\in\partial\Omega$, $y\in\Omega$, and that there is a $\lambda$-carrot path from $x$ to~$y$ for some $\lambda\in (0,1]$.

Let $a>\frac{1-\lambda}{\lambda}$; by Remark~\ref{rmk:carrot:Lusin} we have that $x\in \overline{\gamma_{a}(x)}$.

If $0<c<\widetilde c<1$, then there is an $\widetilde a>0$ depending only on $\lambda$, $c$, and~$\widetilde c$ with the following significance. If $1\leq \beta\leq \infty$, $0<p\leq\infty$, $\varepsilon\geq0$, and either $\varepsilon>0$ or $p\leq 1$, and if $u\in \dot W^{1,\beta}_{loc}(\Omega)$ is such that $\Tr_{a,c}u(x)$ exists, then
\begin{multline*}
\biggl(\fint_{B(y,c\delta(y))} |u-\Tr_{a,c}u(x)|^\beta\biggr)^{1/\beta}
\\\leq
C\delta(y)^{\varepsilon }
\biggl(\int_{\gamma_{\widetilde a}(x)\cap B(x,\delta(y)/\lambda)}
\mathcal{W}_{\widetilde c,\beta} (\nabla u) ^p \,\delta^{p-p\varepsilon-\pdmn}\biggr)^{1/p}
\end{multline*}
where $\delta(z)=\dist(z,\partial\Omega)$ and for some $C$ depending only on $\dmn$, $\bdmn$, $\lambda$, $c$, $\widetilde c$, $p$, and~$\varepsilon$.
\end{lem}

The remainder of Section~\ref{sec:trace:Lusin} will be devoted to the proof of Lemma~\ref{lem:trace:Lusin}.

Without loss of generality we take $\Tr_{a,c} u(x)=0$.
Let $\zeta:[0,T]\to\R^\dmn$ be the assumed $\lambda$-carrot path from $x$ to~$y$. Then $|x-y|\leq T\leq \delta(y)/\lambda$. Let $\tau=\frac{1}{1+{{c}}\lambda}$. Define $y_k=\zeta(\tau^k T)$, so $y_0=y$ and $\lim_{k\to\infty} y_k=x$. We then have that
\begin{align*}
\delta(y_k)&\geq \lambda \tau^k T,&
|y_k-x|&\leq \tau^k T,&
|y_k-y_{k+1}|&\leq \tau^k (1-\tau)T
\end{align*}
and so
\begin{gather}
\label{eqn:carrot:Harnack:1}
\lambda\tau^k\delta(y)
\leq
\lambda\tau^k|y-x|
\leq \delta(y_k)\leq |y_k-x|\leq \frac{\tau^k}{\lambda}\delta(y)
.\end{gather}
It is straightforward to check that
\begin{equation}\label{eqn:carrot:Harnack:2}
|y_k-y_{k+1}|\leq {{c}} \min\bigl(\delta(y_k),\delta(y_{k+1})\bigr).\end{equation}

Let $B_k=B(y_k,c\delta(y_k))$.
By Remark~\ref{rmk:carrot:Lusin}, each $y_k$ is in $\gamma_a(x)$. By the definition of $\Tr_{a,c}$ in Definition~\ref{dfn:trace},
\begin{align*}
\lim_{k\to \infty} \biggl|\fint_{B_k} u\biggr|
&\leq \lim_{k\to \infty} \frac{C}{\delta(y_k)^\dmn}\int_{B_k} |u-\Tr_{a,c} u(x)|
=0
.\end{align*}
Thus because $\beta\geq 1$ we have that
\begin{equation*}\biggl(\fint_{B(y,c\delta(y))} |u|^\beta\biggr)^{1/\beta}
\leq \biggl(\fint_{B_0} |u-{\textstyle\fint_{B_0} u}|^\beta\biggr)^{1/\beta}
+ \sum_{k=1}^\infty \biggl|\fint_{B_{k}} u-\fint_{B_{k-1}}u\biggr|.\end{equation*}
By the Poincar\'e inequality,
\begin{equation*}\biggl(\fint_{B_0} |u-{\textstyle\fint_{B_0} u}|^\beta\biggr)^{1/\beta}
\leq C\delta(y)\biggl(\fint_{B_0} |\nabla u|^\beta\biggr)^{1/\beta}
.\end{equation*}
We wish to use the Poincar\'e inequality to bound the second term in the previous equation.
Let $\widehat y_k=\frac{1}{2}(y_k+y_{k-1})$. Then by the bound~\eqref{eqn:carrot:Harnack:2},
\begin{equation*}|\widehat y_k-y_k|=|\widehat y_k-y_{k-1}|=\frac{1}{2}|y_k-y_{k-1}|
\leq \frac{1}{2}c\min(\delta(y_k),\delta(y_{k-1}))\end{equation*}
and so $\widehat B_k=B(\widehat y_k,\frac{1}{2}c\min(\delta(y_k),\delta(y_{k-1}))) \subset B_k\cap B_{k-1}$. Thus
\begin{align*}
\biggl|\fint_{B_k} u-\fint_{B_{k-1}} u\biggr|
&\leq
\biggl|\fint_{B_k} u-\fint_{\widehat B_{k}} u\biggr|
+\biggl|\fint_{\widehat B_k} u-\fint_{B_{k-1}} u\biggr|
\\&=
\biggl|\fint_{\widehat B_{k}} \bigl({\textstyle\fint_{B_k} u}-u\bigr)\biggr|
+\biggl|\fint_{\widehat B_k} \bigl(u-{\textstyle\fint_{B_{k-1}} u}\bigr)\biggr|
\\&\leq
\fint_{\widehat B_{k}} \bigl|u-{\textstyle\fint_{B_k} u}\bigr|
+\fint_{\widehat B_k} \bigl|u-{\textstyle\fint_{B_{k-1}} u}\bigr|
\\&\leq
C\fint_{B_k} \bigl|u-{\textstyle\fint_{B_k} u}\bigr|
+C\fint_{B_{k-1} } \bigl|u-{\textstyle\fint_{B_{k-1}} u}\bigr|
\end{align*}
where in the last line we have used that $|B_k|\approx |\widehat B_k|\approx |B_{k-1}|$ with implicit constants depending on $\lambda$, $c$, and~$\dmn$. Then the Poincar\'e inequality yields that
\begin{equation*}\biggl|\fint_{B_k} u-\fint_{B_{k-1}} u\biggr|
\leq C\delta(y_k)\fint_{B_k} |\nabla u|+C\delta(y_{k-1})\fint_{B_{k-1}} |\nabla u|\end{equation*}
and so by Hölder's inequality
\begin{equation*}
\biggl(\fint_{B(y,c\delta(y))} |u|^\beta\biggr)^{1/\beta}
\leq C
\sum_{k=0}^\infty \delta(y_k)\biggl(\fint_{B_{k}} |\nabla u|^\beta\biggr)^{1/\beta}
= C
\sum_{k=0}^\infty \delta(y_k)\mathcal{W}_{c,\beta}(\nabla u)(y_k)
\end{equation*}
where $\mathcal{W}_{c,\beta}$ is as in Definition~\ref{dfn:weighted:averaged}.

If $p\leq 1$, then by the standard inclusion relations in sequence spaces and because $\delta(y_k)\leq C\delta(y)$, we have that if $\varepsilon\geq 0$ then
\begin{equation*}
\biggl(\fint_{B(y,c\delta(y))} |u|^\beta\biggr)^{1/\beta}
\leq
C\delta(y)^\varepsilon\Bigl(
\sum_{k=0}^\infty \delta(y_k)^{p-p\varepsilon}
\mathcal{W}_{c,\beta} (\nabla u)(y_k)^p\Bigr)^{1/p}
.\end{equation*}
Conversely, if $p>1$ and $\varepsilon>0$, then by Hölder's inequality
\begin{equation*}
\biggl(\fint_{B(y,c\delta(y))} |u|^\beta\biggr)^{1/\beta}
\leq C
\biggl(\sum_{k=0}^\infty \delta(y_k)^{p-p\varepsilon} \mathcal{W}_{c,\beta} (\nabla u) (y_k)^p\biggr)^{1/p}
\biggl(\sum_{k=0}^\infty \delta(y_k)^{\varepsilon p'} \biggr)^{1/p'}
.\end{equation*}
Because $\delta(y_k)\approx \tau^k\delta(y)$, we may apply convergence of geometric series to see that
\begin{equation}
\label{eqn:trace:Lusin:1}
\biggl(\fint_{B(y,c\delta(y))} |u|^\beta\biggr)^{1/\beta}
\leq C\delta(y)^{\varepsilon }
\biggl(\sum_{k=0}^\infty \delta(y_k)^{p-p\varepsilon}\mathcal{W}_{c,\beta} (\nabla u) (y_k)^p\biggr)^{1/p}
\end{equation}
if $0< p\leq 1$ and $\varepsilon\geq 0$, or if $1<p<\infty$ and $\varepsilon>0$.

The following lemma will allow us to pass from a sum over the discrete points $y_k$ to an integral over a nontangential cone.
\begin{lem}\label{lem:averaged:nearby}Let $\Omega$ be a nonempty proper open subset of~$\R^\dmn$.
Suppose that $0<c<\widetilde c<1$. Let $\upsilon=\frac{\widetilde c-c}{1+\widetilde c}$.

If $y_k\in\Omega$ and $|y_k-z|<\upsilon\delta(y_k)$, and if $F\in L^\beta_{loc}(\Omega)$, then
\begin{equation*}\mathcal{W}_{c,\beta} F(y_k)\leq C_{c,\widetilde c} \mathcal{W}_{\widetilde c,\beta} F(z)
\quad\text{and}\quad
\mathcal{W}_{c,\beta} F(z)\leq C_{c,\widetilde c} \mathcal{W}_{\widetilde c,\beta} F(y_k)
.\end{equation*}
\end{lem}
\begin{proof}
Observe that $0<\upsilon<1$. Suppose that $|y_k-z|<\upsilon\delta(y_k)$. Then
\begin{align}
\label{eqn:averaged:nearby:1}
\delta(z)&\geq \delta(y_k)-|y_k-z|> \delta(y_k)\frac{1+c}{1+\widetilde c}
,\\
\label{eqn:averaged:nearby:2}
\delta(z)&\leq \delta(y_k)+|y_k-z|
< \delta(y_k) \frac{1+2\widetilde c-c}{1+\widetilde c}
\end{align}
and so
\begin{align*}
|y_k-z|+c\delta(y_k)&< \widetilde c\frac{1+c}{1+\widetilde c}\delta(y_k)
\leq \widetilde c\delta(z)
,\\
|y_k-z|+c\delta(z)
&<
\delta(y_k) \frac{\widetilde c+2c\widetilde c-c^2}{1+\widetilde c}
=\biggl(\widetilde c -\frac{(\widetilde c-c)^2}{1+\widetilde c}\biggr)\delta(y_k)
<
\widetilde c\delta(y_k)
.\end{align*}
This yields that $B(y_k,c\delta(y_k))\subset B(z,\widetilde c\delta(z))$ and $B(z,c\delta(z))\subset B(y_k,\widetilde c\delta(y_k))$. By the upper and lower bounds on $\delta(z)$, $|B(y_k,c\delta(y_k))|\approx |B(z,\widetilde c\delta(z))|$ with implicit constants depending only on $\dmn$, $c$, and~$\widetilde c$. Recalling the definition of $\mathcal{W}_{c,\beta}$ completes the proof.
\end{proof}

We now return to the proof of Lemma~\ref{lem:trace:Lusin}. By the bound~\eqref{eqn:trace:Lusin:1} and Lemma~\ref{lem:averaged:nearby},
\begin{equation*}
\biggl(\fint_{B(y,c\delta(y))} |u|^\beta\biggr)^{1/\beta}
\leq C\delta(y)^{\varepsilon }
\biggl(\sum_{k=0}^\infty
\delta(y_k)^{p-p\varepsilon}
\inf_{z\in B(y_k,\upsilon\delta(y_k))}
\mathcal{W}_{\widetilde c,\beta} (\nabla u) (z)^p\biggr)^{1/p}.
\end{equation*}
By the bound~\eqref{eqn:carrot:Harnack:1}, $|y_k-x|\leq \delta(y)/\lambda$ and so $y_k\in B(x,\delta(y)/\lambda)$ for all~$k\geq 0$. Furthermore, $\delta(y_k)\leq |y_k-x|$ and so $\delta(y_k)\leq \delta(y)/\lambda$, and so $E_k=B(x,\delta(y)/\lambda)\cap B(y_k,\upsilon\delta(y_k))$ has volume comparable to $\delta(y_k)^\dmn$. Thus
\begin{equation*}
\biggl(\fint_{B(y,c\delta(y))} |u|^\beta\biggr)^{1/\beta}
\leq C\delta(y)^{\varepsilon }
\biggl(\sum_{k=0}^\infty
\delta(y_k)^{p-p\varepsilon-\pdmn}
\int_{E_k}
\mathcal{W}_{\widetilde c,\beta} (\nabla u) (z)^p\,dz\biggr)^{1/p}.
\end{equation*}
By the bounds~(\ref{eqn:averaged:nearby:1}--\ref{eqn:averaged:nearby:2}), if $z\in B(y_k,\upsilon\delta(y_k))\supseteq E_k$ then $\delta(z)\approx\delta(y_k)$, and so
\begin{equation*}
\biggl(\fint_{B(y,c\delta(y))} |u|^\beta\biggr)^{1/\beta}
\leq C\delta(y)^{\varepsilon }
\biggl(\sum_{k=0}^\infty
\int_{E_k}
\mathcal{W}_{\widetilde c,\beta} (\nabla u) (z)^p
\delta(z)^{p-p\varepsilon-\pdmn}
\,dz\biggr)^{1/p}.
\end{equation*}
We claim that the sets $\{E_k\}_{k=0}^\infty$ have bounded overlap. To see this, observe that if $E_k\cap E_\ell\neq \emptyset$, then $B(y_k,\upsilon\delta(y_k))\cap B(y_\ell,\upsilon\delta(y_\ell))\neq\emptyset$ and so \begin{equation*}|y_k-y_\ell|<\upsilon \delta(y_k)+\upsilon \delta(y_\ell).\end{equation*}
But $|y_k-y_\ell|\geq \delta(y_k)-\delta(y_\ell)$ and so $(1-\upsilon)\delta(y_k)<(1+\upsilon)\delta(y_\ell)$.
By the bound~\eqref{eqn:carrot:Harnack:1}
\begin{equation*}
(1-\upsilon)\lambda\tau^k \delta(y)
\leq
(1-\upsilon)\delta(y_k)<(1+\upsilon)\delta(y_\ell)
\leq (1+\upsilon) \frac{\tau^\ell}{\lambda}\delta(y)\end{equation*}
and so if $E_k\cap E_\ell\neq \emptyset$ then
$\frac{(1-\upsilon)\lambda^2}{1+\upsilon}
\leq \tau^{\ell-k}$. This puts a lower bound on $k-\ell$ in terms of $\lambda$, $c$ and~$\widetilde c$. By symmetry there is the same lower bound on~$\ell-k$, and so a point in $\bigcup_k E_k$ can be in $E_\ell$ for at most $C$ values of~$\ell$.

Thus if $E=\bigcup_{k=0}^\infty E_k$, then
\begin{equation*}
\biggl(\fint_{B(y,c\delta(y))} |u|^\beta\biggr)^{1/\beta}
\leq C\delta(y)^{\varepsilon }
\biggl(\sum_{k=0}^\infty
\int_{E}
\mathcal{W}_{\widetilde c,\beta} (\nabla u) (z)^p
\delta(z)^{p-p\varepsilon-\pdmn}
\,dz\biggr)^{1/p}.
\end{equation*}
It remains only to show that $E\subseteq \gamma_{\widetilde a}(x)\cap B(x,\delta(y)/\lambda)$.
By construction $E\subset B(x,\delta(y)/\lambda)$.
If $z\in B(y_k,\upsilon \delta(y_k))$, then by the bound~\eqref{eqn:averaged:nearby:1} and by the bound~\eqref{eqn:carrot:Harnack:1}
\begin{align*}
\delta(z)&\geq \frac{1+c}{1+\widetilde c}\delta(y_k),\\
|z-x|&\leq |z-y_k|+|y_k-x|\leq (\upsilon+1/\lambda^2)\delta(y_k).
\end{align*}
Combining these bounds yields that $z\in\gamma_{\widetilde a}(x)$ (and so $E\subset\gamma_{\widetilde a}(x)$), where ${\widetilde a}$ depends only on $\lambda$, $c$ and~$\widetilde c$.
This completes the proof of Lemma~\ref{lem:trace:Lusin}.

\subsection{A boundary Poincar\'e inequality}

We now use Lemma~\ref{lem:trace:Lusin} to prove a boundary Poincar\'e inequality for our weighted averaged Sobolev spaces.

\begin{lem}\label{lem:boundary:Poincare}
Let $\Omega\subset\R^\dmn$ be a weak local John domain, and let $\bdmn\in (0,\dmn)$, $\lambda\in (0,1]$, $\theta\in(0,1]$, and $N\geq 2$ be as in Definition~\ref{dfn:local:John} (that is, the definition of weak local John domain).

Let $\xi\in\partial\Omega$ and let $0<R<2\diam(\partial\Omega)$, and let $\Psi\subseteq \Omega\cap B(\xi,R)$.

For each $y\in \Psi$, let $F(y)$ be a subset of the set in the definition of weak local John point, that is,
\begin{equation*}F(y)\subseteq\{x\in {B(y,N\delta(y))}\cap\Gamma:
\text{there is a $\lambda$-carrot path from $x$ to~$y$}\}.\end{equation*}
We further require that there exists a $\vartheta$ with $0<\vartheta\leq\theta$ such that
\begin{equation*}\sigma(F(y))\geq \vartheta\sigma(B(y,N\delta(y))\cap\partial\Omega)\end{equation*}
for each $y\in\Psi$.

Let $C_1=N+1/\lambda+1$, and let $a>(1-\lambda)/\lambda$.
Let $s>0$, $1\leq p\leq\infty$, $0<c<\widetilde c<1$, and $1\leq \beta\leq \infty$.
Suppose that $u\in \dot W^{1,\beta}_{loc}(\Omega\cap B(\xi,2C_1R))$ and that $\Tr_{a,c} u$ exists $\sigma$-almost everywhere on $\bigcup_{y\in\Psi}F(y)$.
Then
\begin{multline}\label{eqn:boundary:Poincare} \int_{\Psi}
\biggl(\fint_{B(y,c\delta(y))} |u-{\textstyle \fint_{F(y)} \Tr_{a,c}u\,d\sigma}|^\beta\biggr)^{p/\beta}
\delta(y)^{\bdmn-\dmn-ps}\,dy
\\
\leq C
\int_{B(\xi,C_1R)\cap\Omega}\mathcal{W}_{\widetilde c,\beta}(\nabla u)^p\delta^{\bdmn-\dmn+p-ps}\end{multline}
where $C$ depends only on $a$, $c$, $p$, $s$, $\dmn$, $\bdmn$, $N$, $\lambda$, $\theta$, and~$\vartheta$.

In particular, if $\Tr_{a,c} u(x)=0$ for almost every $x\in \Delta(\xi,C_1R)$ that satisfies $x\in\overline{\gamma_a(x)}$, then
\begin{equation}\label{eqn:boundary:Poincare:zero}
\int_{B(\xi,R)\cap\Omega}
\mathcal{W}_{c,\beta}(u)^p
\delta^{\bdmn-\dmn-ps}
\leq C
\int_{B(\xi,C_1R)\cap\Omega}\mathcal{W}_{\widetilde c,\beta}(\nabla u)^p\delta^{\bdmn-\dmn+p-ps}.\end{equation}
\end{lem}

\begin{proof}
We begin by showing that the bound~\eqref{eqn:boundary:Poincare:zero} follows from the bound~\eqref{eqn:boundary:Poincare} under the stated assumptions. We simply choose $\Psi=B(\xi,R)\cap \Omega$ and let $F(y)$ be the (full) set in Definition~\ref{dfn:local:John}. Then if $y\in\Psi$ then $F(y)\subset B(\xi,R+N\delta(y))\subset B(\xi,(1+N)R)\subset B(\xi, C_1R)$. Furthermore, $F(y)\subseteq\partial\Omega$ and so $F(y)\subseteq \Delta(\xi,C_1R)$. Finally, by Remark~\ref{rmk:carrot:Lusin}, if $x\in F(y)$ then $x\in \overline{\gamma_{a}(x)}$, and so if $\Tr_{a,c}u(x)=0$ for almost every $x\in\Delta(\xi,C_1R)$ that satisfies $x\in\overline{\gamma_a(x)}$, we have that $\Tr_{a,c}u(x)=0$ almost everywhere on~$F(y)$, and so the average must be zero.

We now establish the bound~\eqref{eqn:boundary:Poincare}.
By definition of average,
\begin{equation*}\fint_{B(y,c\delta(y))} |u-{\textstyle \fint_{F(y)} \Tr_{a,c} u\,d\sigma}|^\beta
=
\fint_{B(y,c\delta(y))} |{\textstyle \fint_{F(y)} u(z)-\Tr_{a,c} u(x)\,d\sigma(x)}|^\beta\,dz
.\end{equation*}
Because $L^\beta$ is a Banach space,
\begin{multline*}\biggl(\fint_{B(y,c\delta(y))} |u-{\textstyle \fint_{F(y)} \Tr_{a,c} u\,d\sigma}|^\beta\biggr)^{1/\beta}
\\\leq
\fint_{F(y)}
\biggl(\fint_{B(y,c\delta(y))} |u(z)-\Tr_{a,c} u(x)|^\beta\,dz\biggr)^{1/\beta}
\,d\sigma(x)
.\end{multline*}
Thus by Lemma~\ref{lem:trace:Lusin}, if $\varepsilon>0$ and $c<\widetilde c<1$ then
\begin{multline*}\biggl(\fint_{B(y,c\delta(y))} |u-{\textstyle \fint_{F(y)} \Tr_{a,c} u\,d\sigma}|^\beta\biggr)^{1/\beta}
\\\leq
C\delta(y)^{\varepsilon }
\fint_{F(y)}
\biggl(\int_{\gamma_{\widetilde a}(x)\cap B(x,\delta(y)/\lambda)}
\mathcal{W}_{\widetilde c,\beta} (\nabla u)(z) ^p \,\delta(z)^{p-p\varepsilon-\pdmn}\,dz\biggr)^{1/p}
\,d\sigma(x)
.\end{multline*}
By Hölder's inequality and because $p\geq 1$,
{\multlinegap=0pt\begin{multline*}
\int_{\Psi}
\biggl(\fint_{B(y,c\delta(y))} |u-{\textstyle \fint_{F(y)} \Tr_{a,c} u\,d\sigma}|^\beta\biggr)^{p/\beta}
\delta(y)^{\bdmn-\dmn-ps}\,dy
\\\leq
C\int_{\Psi}
\fint_{F(y)}
\int_{\gamma_{\widetilde a}(x)\cap B(x,\delta(y)/\lambda)}
\mathcal{W}_{\widetilde c,\beta} (\nabla u)(z) ^p \,
\frac{\delta(y)^{\bdmn-\dmn-ps+p\varepsilon}} {\delta(z)^{-p+p\varepsilon+\pdmn}}
\,dz
\,d\sigma(x)
\,dy
.\end{multline*}}
If $(y,x,z)$ is in the region of integration, that is, if $y\in\Psi\subset B(\xi,R)$, $x\in F(y)\subset  {B(y,N\delta(y))}$, and $z\in B(x,\delta(y)/\lambda)$, then $|z-\xi|\leq |z-x|+|x-y|+|y-\xi|<\delta(y)/\lambda+N\delta(y)+R$. But $\delta(y)\leq |y-\xi|<R$ and so $|z-\xi|\leq C_1R$, where we recall that $C_1=(N+1+1/\lambda)$. Thus, if we change the order of integration, we find that
\begin{multline*}
\int_{\Psi}
\biggl(\fint_{B(y,c\delta(y))} |u-{\textstyle \fint_{F(y)} \Tr_{a,c} u\,d\sigma}|^\beta\biggr)^{p/\beta}
\delta(y)^{\bdmn-\dmn-ps}\,dy
\\\leq
C
\int_{B(\xi,C_1R)\cap\Omega}
\mathcal{W}_{c,\beta} (\nabla u) (z)^p \int_{\Psi}
\frac{\sigma(F(y,z))}{\sigma(F(y))}
\,\frac{\delta(y)^{\bdmn-\pdmn-ps+p\varepsilon}} {\delta(z)^{\dmn-p+p\varepsilon}}
\,dy
\,dz
\end{multline*}
where $F(y,z)=\{x\in F(y):z\in \gamma_{\widetilde a}(x)\cap B(x,\delta(y)/\lambda)\}$.

Define
\begin{equation*}I(z)=\int_{\Psi}
\frac{\sigma(F(y,z))}{\sigma(F(y))}
\,\delta(y)^{\bdmn-\pdmn-ps+p\varepsilon}
\,dy
\end{equation*}
for $z\in B(\xi, C_1R)\cap\Omega$. To complete the proof, we need only show that $I(z)\leq C\delta(z)^{\bdmn-ps+p\varepsilon}$.

If $F(y,z)\neq\emptyset$ then $|y-z|<(N+1/\lambda)\delta(y)$. Let \begin{equation*}E(z)=\{y\in\Omega:|y-z|<(N+1/\lambda)\delta(y)\}\end{equation*}
so
\begin{equation*}I(z)=\int_{E(z)\cap\Psi}
\frac{\sigma(F(y,z))}{\sigma(F(y))}
\,\delta(y)^{\bdmn-\pdmn-ps+p\varepsilon}
\,dy
.\end{equation*}
If $y\in E(z)$ then $\delta(y)\geq \delta(z)-|z-y|>\delta(z)-(N+1/\lambda)\delta(y)$ and so $\delta(y)>\delta(z)/C_1$.
Let
\begin{equation*}
E_j(z)=\{y\in E(z)\cap\Psi:2^{j}\delta(z)/C_1<\delta(y)\leq 2^{j+1}\delta(z)/C_1\}.\end{equation*}
We may thus write
\begin{align*}
I(z)
&= \sum_{j=0}^\infty \int_{E_j(z)}
\frac{\sigma(F(y,z))}{\sigma(F(y))}
\,\delta(y)^{\bdmn-\pdmn-ps+p\varepsilon}
\,dy
.\end{align*}
If $y\in E_j(z)$ then $|y-z|<(N+1/\lambda)\delta(y)<2^{j+1}\delta(z)$. Thus $E_j(z) \subset B(z, 2^{j+1}\delta(z))$ and so $|E_j(z)|\leq C_\dmn(2^j\delta(z))^\dmn$. Furthermore, if $y\in E_j(z)$ then $\delta(y)\approx 2^j\delta(z)$. Thus
\begin{align}\label{eqn:poincare:boundary:2}
I(z)
&\leq
C\sum_{j=0}^\infty
2^{j(\bdmn+p\varepsilon-ps)}
\,\delta(z)^{\bdmn+p\varepsilon-ps}
\sup_{y\in E_j(z)}\frac{\sigma(F(y,z))}{\sigma(F(y))}
.\end{align}
Recall that if $x\in F(y,z)$ then $z\in \gamma_{\widetilde a}(x)$. Thus $|x-z|\leq (1+{\widetilde a})\delta(z)$ and so $x\in B(z,(1+{\widetilde a})\delta(z))$. Furthermore, $\delta(z)\leq C_1R<2C_1\diam(\partial\Omega)$. Applying the Ahlfors regularity of~$\Gamma$, we have that
\begin{equation*}\sigma(F(y,z))\leq C\delta(z)^\bdmn.\end{equation*}
By assumption on~$F(y)$,
\begin{equation*}\sigma(F(y))\geq \frac{1}{C}\delta(y)^\bdmn\end{equation*}
and so
\begin{align*}
I(z)
&\leq
C\sum_{j=0}^\infty
2^{j(\bdmn+p\varepsilon-ps)}
\,\delta(z)^{\bdmn+p\varepsilon-ps}
\sup_{y\in E_j(z)}
\frac{\delta(z)^\bdmn}{\delta(y)^\bdmn}
\leq
C\sum_{j=0}^\infty
2^{j(p\varepsilon-ps)}
\,\delta(z)^{\bdmn+p\varepsilon-ps}
.\end{align*}
Choosing $0<\varepsilon<s$, the sum converges. This completes the proof.
\end{proof}

\begin{rmk}Even if the weak local John condition fails, the bound~\eqref{eqn:boundary:Poincare:zero} is still valid if $ps>\bdmn$ and $\Tr_{a,c} u(x)=0$ for every (not almost every) $x\in\Gamma$ with $x\in\overline{\gamma_a(x)}$.

To see this, observe that if $N>1$ then $F(y)$ is nonempty: there is at least one $x\in\Gamma$ with $|x-y|=\delta(y)$, and the ball $B(y,\delta(y))$ contains no points of~$\Gamma$, so the straight line path from $x$ to $y$ is a $\lambda$-carrot path for $\lambda=1$.

If we replace the surface measure $\sigma$ on $F(y)$ with any other finite measure~$\mu_y$ with $\mu_y(F(y))>0$, then the proof proceeds identically in the absence of the weak local John condition until the bound~\eqref{eqn:poincare:boundary:2}. This includes the sufficiency of the bound $I(z)\leq C\delta(z)^{\bdmn-ps+p\varepsilon}$.

Because $F(y,z)\subseteq F(y)$ we have that $\sup_{y\in E_j(z)}\frac{\mu_y(F(y,z))}{\mu_y(F(y))}\leq 1$ and so $I(z)
\leq
C\sum_{j=0}^\infty
2^{j(\bdmn+p\varepsilon-ps)}
\,\delta(z)^{\bdmn+p\varepsilon-ps}
$. If $ps>\bdmn$, and in particular if $p=\infty$, we may choose $\varepsilon>0$ small enough that the sum converges. Thus if $p>\bdmn/s$, then the bound~\eqref{eqn:boundary:Poincare:zero} is valid (if $\Tr_{a,c}u=0$ everywhere) even without the weak local John condition. However, the weak local John condition does appear to be necessary for general~$p$.
\end{rmk}

\subsection{The converse: a sufficient condition for trace zero}

The following result is a converse to the trace zero case of
Lemma~\ref{lem:boundary:Poincare}. Together these results provide in essence an alternative characterization of the trace of a $\dot W^{1,p}_{av,\beta,s}$-function.

To state this lemma, we provide a slightly stronger notion of trace.

\begin{defn}\label{dfn:trace:stronger}
Let $\Omega\subset\R^\dmn$ be open and let $x\in\partial\Omega$. Let $u\in L^1_{loc}(\Omega)$. We define the boundary trace $\Tr u(x)$ by
\begin{equation*}\Tr u(x)=\Tr_{\Omega} u(x)=U\quad\text{if}\quad
\lim_{r\to 0^+} \frac{1}{r^\dmn}\int_{B(x,r)\cap\Omega} |u-U|=0.\end{equation*}
(As usual, if there is no such $U$ then we say that $\Tr_{a,c} u(x)$ does not exist. If $x\not\in \overline{\gamma_a(x)}$, that is, if there is some ball $B(x,r)$ such that $\delta(y)\leq|y-x|/(1+a)$ for all $y\in \Omega\cap B(x,r)$, then we again say that $\Tr_{a,c} u(x)$ does not exist.)
\end{defn}

\begin{lem}\label{lem:trace:zero}
Let $\Omega\subset\R^\dmn$ and suppose that $\partial\Omega$ is $\bdmn$-Ahlfors regular for some $0<\bdmn<\dmn$. Let $\xi\in\partial\Omega$ and let $0<R<2\diam(\partial\Omega)$.

Let $1\leq\beta\leq\infty$, $1\leq p\leq \infty$, $0<c<1$, and $s\in\R$.

Suppose that $u\in L^p_{av,c,\beta,s+1}(B(\xi,R)\cap\Omega;\Omega)$, that is, that $u\in L^\beta_{loc}(\Omega)$ and satisfies
\begin{equation*}
\|u\|_{L^p_{av,c,\beta,s+1}(B(\xi,R)\cap\Omega;\Omega)}
=\biggl(
\int_{B(\xi,R)\cap\Omega}
\mathcal{W}_{c,\beta}(u)^p
\delta^{\bdmn-\dmn-ps}\biggr)^{1/p}
<\infty.\end{equation*}
If $s>-(\dmn-\bdmn)$, then $u\in L^1_{loc}(\overline\Omega\cap B(\xi,R))$. If $p=\infty$ and $s>0$, or if $p<\infty$ and $s\geq \bdmn/p$, then $\Tr u=0$ everywhere on~$\Delta(\xi,R)$. If $p<\infty$ and $s\geq 0$, then  $\Tr u=0$ $\sigma$-almost everywhere on~$\Delta(\xi,R)$.
\end{lem}
\begin{proof}
Let $x\in\Delta(\xi,R)$ and let $r>0$ be such that $B(x,\frac{1}{1-c}r)\subset B(\xi,R)$.
By Lemma~\ref{lem:averaged:to:unaveraged} with $\varrho=0$, $\beta=1$, $\varsigma=\bdmn+1-\pdmn$, and $s$ replaced by~$s+1$, and extending $u$ by zero if necessary to the other connected components of~$\R^\dmn\setminus\partial\Omega$, we have that
\begin{equation}\label{eqn:boundary:Poincare:3}
\frac{1}{r^\dmn}\int_{B(x,r)\cap\Omega} |u|
\leq Cr^{s-\bdmn/p}
\biggl(\int_{B(x,\frac{1}{1-c} r)\cap\Omega} \mathcal{W}_{c,1} (u)^p\,\delta^{\bdmn-\pdmn-ps}\biggr)^{1/p}
\end{equation}
provided the right hand side is finite and provided $1+s>\varsigma=\bdmn+1-\pdmn$.

The right hand side is finite by definition of $L^p_{av,c,\beta,s+1}$. This completes the proof that $u$ is locally integrable.

If $p=\infty$ and $s$ is positive, or if $p$ is finite and $s\geq \bdmn/p$, then the right hand side converges to zero as $r\to 0^+$. This implies that $\Tr u(x)=0$ for all such~$x$.
This completes the proof in this case.

Assume that $p<\infty$ and $s<\bdmn/p$ (that is, that $ps-\bdmn<0$). We need to show that $\Tr u(x)=0$ for almost every $x\in \Delta(\xi,R)$. Observe that if $r>0$, then there is a $\ell\in\Z$ with $2^{\ell-1}<r\leq 2^\ell$, and that
\begin{equation*}\frac{1}{r^n}\int_{B(x,r)\cap\Omega} |u|\leq \frac{2^\dmn}{2^{\ell\dmn}} \int_{B(x,2^\ell)\cap\Omega} |u|.\end{equation*}
Thus, it suffices to show that $\lim_{\ell\to -\infty} 2^{-\ell\dmn}\int_{B(x,2^\ell)\cap\Omega} |u|=0$ for almost every $x\in\partial\Omega$.

If $x\in\R^\dmn$, define
\begin{equation*}\Phi_L(x)=\biggl(\sum_{\ell=-\infty}^L \biggl(2^{-\ell\dmn}\int_{B(x,2^\ell)\cap\Omega} |u| \biggr)^p\biggl)^{1/p}.\end{equation*}
If $\Phi_L(x)$ is finite, then the sum converges and so $\lim_{\ell\to -\infty} 2^{-\ell\dmn}\int_{B(x,2^\ell)2^{-\ell\dmn}} |u|=0$. Thus, it suffices to prove that $\Phi_L$ is finite almost everywhere; we will do this by showing that $\Phi_L\in L^p_{loc}(\partial\Omega)$.

Let $\nu\in\N$ satisfy $2^\nu\geq \frac{1}{1-c}$. If $s>0$, let $\varsigma=s$; otherwise, let $\varsigma\in (-\pdmn+\bdmn,0)$. By the bound~\eqref{eqn:boundary:Poincare:3} with $r=2^\ell$ and with $s$ replaced by~$\varsigma$,
\begin{equation*}\Phi_L(x)
\leq
C\biggl(\sum_{\ell=-\infty}^L 2^{\ell p\varsigma-\ell\bdmn}
\int_{B(x, 2^{\ell+\nu})\cap\Omega} \mathcal{W}_{c,\beta} (u)^p\,\delta^{\bdmn-\pdmn-p\varsigma}\biggl)^{1/p}
.\end{equation*}
Changing the order of integration and summation yields that
\begin{equation*}
\Phi_L(x)^p
\leq
C
\int_{B(x, 2^{L+\nu})\cap\Omega} \mathcal{W}_{c,\beta} (u)(z)^p\,\delta(z)^{\bdmn-\pdmn-p\varsigma}
\sum_{\ell=\ell(z)}^L 2^{\ell p\varsigma-\ell\bdmn}
\,dz
\end{equation*}
where $\ell(z)$ is the smallest integer with $|x-z|<2^{\ell(z)+\nu}$. Because $p\varsigma<\bdmn$, the geometric series is dominated by its first term and so
\begin{equation*}
\Phi_L(x)^p
\leq
C
\int_{B(x, 2^{L+\nu})} \mathcal{W}_{c,\beta} (u)(z)^p\,\delta(z)^{\bdmn-\pdmn-p\varsigma}
|x-z|^{p\varsigma-\bdmn}
\,dz
.\end{equation*}

Fix some $\xi\in\partial\Omega$ and some $J\in\Z$ with $2^J<\diam(\partial\Omega)$. Then
\begin{multline*}
\int_{\Delta(\xi,2^J)} \Phi_L(x)^p\,d\sigma(x)
\\\leq
C \int_{\Delta(\xi,2^J)}
\int_{B(x, 2^{L+\nu})} \mathcal{W}_{c,\beta} (u)(z)^p\,\delta(z)^{\bdmn-\pdmn-p\varsigma}
|x-z|^{p\varsigma-\bdmn}
\,dz
\,d\sigma(x)
.\end{multline*}
Assume $L+\nu\leq J$.
Changing the order of integration yields that
\begin{multline*}
\int_{\Delta(\xi,2^J)} \Phi_L(x)^p\,d\sigma(x)
\\\leq
C
\int_{B(\xi, 2^{J+1})}
\mathcal{W}_{c,\beta} (u)(z)^p\,\delta(z)^{\bdmn-\pdmn-p\varsigma}
\int_{B(z,2^{L+\nu})\cap\partial\Omega}
|x-z|^{p\varsigma-\bdmn}
\,d\sigma(x)
\,dz
.\end{multline*}
The inner integral satisfies
\begin{equation*}\int_{B(z,2^{L+\nu})\cap\partial\Omega}
|x-z|^{p\varsigma-\bdmn}
\,d\sigma(x)
\leq C\sum_{k=K(z)}^{L+\nu} \sigma(B(z,2^k)\cap \partial\Omega) 2^{k(p\varsigma-\bdmn)}\end{equation*}
where $K(z)$ is the smallest integer with $2^{K(z)}\geq \delta(z)$. Applying the Ahlfors regularity of the boundary, we see that
\begin{equation*}\int_{B(z,2^{L+\nu})\cap\partial\Omega}
|x-z|^{p\varsigma-\bdmn}
\,d\sigma(x)
\leq C\sum_{k=K(z)}^{L+\nu}  2^{kp\varsigma}\end{equation*}
which converges to $C2^{Lp\varsigma}$ if $\varsigma>0$, and to $C\delta(z)^{p\varsigma}$ if $\varsigma<0$. (The reason to use $\varsigma$ rather than $s$ in the case $s=0$ is to avoid an inconvenient logarithmic term.) Thus, whether $s=0$ or $s>0$, we have that
\begin{equation*}
\int_{\Delta(\xi,2^J)} \Phi_L(x)^p\,d\sigma(x)
\leq
C2^{Lps}
\int_{B(\xi, 2^{J+1})}
\mathcal{W}_{c,\beta} (u)(z)^p\,\delta(z)^{\bdmn-\pdmn-ps}
\,dz
.\end{equation*}
This is finite, which completes the proof.
\end{proof}

\subsection{The Banach space of functions with trace zero}\label{sec:dense}

We now define
\begin{equation}
\label{eqn:W:traceless}
\widetilde W^{1,p}_{av,c,\beta,s}(\Psi;\Omega)
=
\dot W^{1,p}_{av,c,\beta,s}(\Psi;\Omega)\cap L^{p}_{av,c,\beta,s+1}(\Psi;\Omega)
\end{equation}
and as usual let $\widetilde W^{1,p}_{av,c,\beta,s}(\Omega)=\widetilde W^{1,p}_{av,c,\beta,s}(\Omega;\Omega)$.

The space $\widetilde W^{1,p}_{av,c,\beta,s}(\Psi;\Omega)$ is clearly a Banach space (that is, complete) under the norm $\|u\|_{\widetilde W^{1,p}_{av,c,\beta,s}(\Psi;\Omega)}= \|u\|_{\dot W^{1,p}_{av,c,\beta,s}(\Psi;\Omega)} +\|u\|_{L^{p}_{av,c,\beta,s+1}(\Psi;\Omega)}$. Thus it is a closed subset of $\dot W^{1,p}_{av,c,\beta,s}(\Psi;\Omega)$.

\begin{rmk}\label{rmk:traceless}
Consider the case $\Psi=\Omega$.
If $\Omega$ is a weak local John domain with unbounded complement, then by Lemma~\ref{lem:trace:zero} and by letting $R=\diam\Omega$ (or taking the limit as $R\to \infty=\diam\Omega$) in Lemma~\ref{lem:boundary:Poincare}, we see that $\widetilde W^{1,p}_{av,c,\beta,s}(\Omega)$ is the set of all $u\in\dot W^{1,p}_{av,c,\beta,s}(\Omega)$ with $\Tr u=0$ $\sigma$-almost everywhere on~$\partial\Omega$ and that $\|u\|_{\widetilde W^{1,p}_{av,c,\beta,s}(\Omega)}\approx \|u\|_{\dot W^{1,p}_{av,c,\beta,s}(\Omega)}$.
\end{rmk}

We will make relatively little use of the spaces $\widetilde W^{1,p}_{av,c,\beta,s}(\Psi;\Omega)$ for $\Psi\neq\Omega$.

We will need the fact that smooth compactly supported functions are dense in the Banach space $\widetilde W^{1,p}_{av,c,\beta,s}(\Omega)$. We now prove this fact.

\begin{prp}\label{prp:dense}
Let $\Omega$ be a weak local John domain with unbounded complement. Let $s>0$, $1\le p<\infty$, and $1\leq\beta<\infty$. Then $C^\infty_0(\Omega)$ is dense in $\widetilde W^{1,p}_{av,c,\beta,s}(\Omega)$, where $\widetilde W^{1,p}_{av,c,\beta,s}(\Omega)$ is given by formula~\eqref{eqn:W:traceless}.
\end{prp}

If $p=\infty$ then we have density in the weak-$\star$ topology; see Lemma~\ref{lem:dense:infinity} below.

We observe that in a weak local John domain with unbounded complement, Proposition~\ref{prp:dense} implies that $\Tr u=0$ almost everywhere on $\partial\Omega$ if and only if $u=0$ on $\partial\Omega$ in the sense of Sobolev spaces; this latter characterization is common in the literature.

The remainder of Section~\ref{sec:dense} will be devoted to the proof of Proposition~\ref{prp:dense}. The proof divides naturally into three parts. We will also prove a localized version.

We begin with the density of functions that are zero in a neighborhood of the boundary.

\begin{lem}\label{lem:dense:boundary}
Let $\Omega$ be a weak local John domain. Let $s>0$, $1\leq\beta\leq\infty$, and let $1\leq p\leq \infty$.

Let $0<r<\diam(\partial\Omega)$ and let $\eta_r:\Omega\to[0,1]$ be such that $\eta_r\equiv0$ on $\{y\in\Omega:\delta(y)\le r\}$, $\eta_r\equiv1$ on $\{y\in\Omega:\delta(y)\ge 2r\}$, and $|\nabla \eta_r|\leq 2 r^{-1}$.

Let $\Upsilon\subseteq\Omega$ be open. Let
$\varrho>0$. Let $\Psi\subsetneq\Upsilon$ be open and satisfy
\begin{equation*}
B(y,\varrho)\subseteq\Upsilon\text{ for all } y\in\Psi.\end{equation*}

If $0<c<\widetilde c<1$ and $u\in \widetilde W^{1,p}_{av,\widetilde c,\beta,s}(\Upsilon;\Omega)$, then
\begin{equation*}\|u-u\eta_r\|_{\widetilde W^{1,p}_{av,c,\beta,s}(\Psi;\Omega)}
\leq C\|u\|_{\widetilde W^{1,p}_{av,\widetilde c,\beta,s}(\Upsilon;\Omega)}\end{equation*}
for all $r<\varrho/C_2$, where $C_2$ depends only on $c$ and the constant $C_1$ in Lemma~\ref{lem:boundary:Poincare}.

If $p<\infty$ we additionally have that
\begin{equation*}\lim_{r\to 0^+}\|u-u\eta_r\|_{\widetilde W^{1,p}_{av,c,\beta,s}(\Psi;\Omega)}=0.\end{equation*}
\end{lem}

\begin{proof}
Clearly
\begin{equation*}\|(1-\eta_r )u\|_{L^p_{av,c,\beta,s+1}(\Psi;\Omega)}
\leq \|u\|_{L^p_{av,c,\beta,s+1}(\Psi;\Omega)}
\leq \|u\|_{\widetilde W^{1,p}_{av,c,\beta,s}(\Psi;\Omega)}\end{equation*}
and if $p<1$ then, by the dominated convergence theorem, the left hand side approaches zero as $r\to 0^+$. We need only consider the $\dot W^{1,p}_{av,c,\beta,s}$ norm of $(1-\eta_r)u$.

If $\Psi\subset\Omega$, then we have that
\begin{equation*}\|\nabla(u\eta_r)-\nabla u\|_{L^{p}_{av,c,\beta,s}(\Psi;\Omega)} \le\|(1-\eta_r)\nabla u\|_{L^p_{av, c, \beta, s}(\Psi;\Omega)}+\|u\nabla\eta_r\|_{L^p_{av, c, \beta, s}(\Psi;\Omega)}.\end{equation*}
We may deal with the first term $\|(1-\eta_r)\nabla u\|_{L^p_{av, c, \beta, s}(\Psi;\Omega)}$ as before. Therefore, it suffices to bound $\|u\nabla \eta_r\|_{L^p_{av, c, \beta, s}(\Omega)}$ (and show that it approaches zero if $p<\infty$).

By the Vitali covering lemma there exists an at most countable collection of points $\{x_i\}_{i\in I}\subset\partial\Omega$ such that $\partial\Omega\subset\bigcup_{i\in\N}B(x_i,r)$ and the balls $\{B(x_i,r/5)\}_{i\in I}$ are pairwise disjoint.

Notice that $\supp(\nabla \eta_r)\subseteq\{y\in\Omega:2r\ge\delta(y)\ge r\}$. Therefore, if $\delta(y)\ge 2r/(1-c)$ or $\delta(y)\le r/(1+c)$, then $\nabla \eta_r\equiv0$ on $B(y,c\delta(y))$.
Let $\Psi_r=\{y\in\Psi:r/(1+c)<\delta(y)<2r/(1-c)\}$.
Thus
\begin{align*}
\|u\nabla\eta_r\|_{L^p_{av,c,\beta,s}(\Psi;\Omega)} &=\biggl(\int_{\Psi_r}\left(\fint_{B\left(y,c{\delta(y)} \right)}|\nabla\eta_r|^\beta u^\beta\right)^{p/\beta}\delta(y)^{\bdmn-\dmn+p-ps}\,dy\biggr)^{1/p}
\end{align*}
Observe that in~$\Psi_r$, $|\nabla\eta_r|\delta$ is bounded by a constant independent of~$r$.
If $y$ is such that $\delta(y)\le 2r/(1-c)$ and $\xi_y\in\partial\Omega$ is such that $|\xi_y-y|=\delta(y)$, then there exists some $i\in I$ such that $\xi_y\in B(x_i,r)$, which means that $y\in B(x_i,r(3-c)/(1-c))$. Let $r_1=r(3-c)/(1-c)<\varrho/2$ and let $I_1=\{i\in I:x_i\in B(y,r_1)$ for some $y\in \Psi\}$. We may choose $C_2$ large enough (forcing $r$ small enough) that $B(x_i,C_1 r_1)\cap\Omega\subset\Upsilon$ for all $i\in I_1$. Thus
\begin{align*}
\|u\nabla\eta_r\|_{L^p_{av,1/2,\beta,s}(\Psi;\Omega)} &\leq C\biggl(\sum_{i\in I_1} \int_{B(x_i,r_1)\cap\Omega}\left(\fint_{B\left(y,c{\delta(y)}\right)}u^\beta\right)^{p/\beta}\delta(y)^{\bdmn-\dmn-ps}\,dy\biggr)^{1/p}
.\end{align*}
If $p=\infty$ we may take a supremum and an essential supremum in place of the sum and first integral.

By Lemma~\ref{lem:trace:zero} we have that $\Tr u=0$ on each $\Delta(x_i,C_1 r_1)$.
Thus Lemma~\ref{lem:boundary:Poincare}, applied on each ball $B(x_i,r_1)$, gives
\begin{align*}
\|u\nabla\eta_r\|_{L^p_{av,1/2,\beta,s}(\Psi;\Omega)}
&\leq C \biggl(\sum_{i\in I_1} \int_{B(x_i,C_1 r_1)\cap \Omega} \mathcal{W}_{\widetilde c,\beta}(\nabla u)(y)^p \delta(y)^{\bdmn-\dmn+p-ps}\,dy\biggr)^{1/p}.
\end{align*}
Because the balls $B(x_i,r/5)$ are pairwise disjoint, we have that each $y$ is in $B(x_i,C_1r_1)$ for at most $C$ values of~$i$, and so
\begin{align*}
\|u\nabla\eta_r\|^p_{L^p_{av,1/2,\beta,s}}
&\leq C \biggl( \int_{\{y\in\Upsilon:\delta(y)<C_1r_1\}} \mathcal{W}_{\widetilde c,\beta}(\nabla u)(y)^p \delta(y)^{\bdmn-\dmn+p-ps}\,dy\biggr)^{1/p}
.\end{align*}
This quantity is bounded by $\|u\|_{\dot W^{1,p}_{av,\widetilde c,\beta,s}(\Upsilon;\Omega)}$. Furthermore, if $p<\infty$ then an argument involving the dominated convergence theorem completes the proof that $\eta_ru\to u$ as $r\to 0^+$ in $\dot W^{1,p}_{av,c,\beta,s}(\Psi;\Omega)$.
\end{proof}

\begin{lem}\label{lem:dense:smooth}
Let $u$, $\Psi$, $\Upsilon$, and $\Omega$ be as in Lemma~\ref{lem:dense:boundary}. Suppose in addition that $\Upsilon$ is bounded. Let $s>0$, $1\leq \beta<\infty$, and $1\leq p<\infty$.

Then for every $\varepsilon>0$ there is a function $\varphi\in C^\infty(\R^\dmn)$ that is zero in a neighborhood of $\partial\Omega$ and satisfies
\begin{equation*}\|u-\varphi\|_{\dot W^{1,p}_{av,c,\beta,s}(\Psi;\Omega)}<\varepsilon.\end{equation*}
\end{lem}

\begin{proof}
We define
\begin{equation*}\Upsilon_{\widetilde c} = \bigcup_{y\in\Upsilon} B(y,\widetilde c\delta(y))\end{equation*}
and define $\Psi_c$ similarly. By definition of $\dot W^{1,p}_{av, \widetilde c,\beta,s}(\Upsilon)$, we have that $u\in \dot W^{1,\beta}_{loc}(\Upsilon_{\widetilde c})$ (and in particular is defined in this region).

Let $\eta_r$ be as in Lemma~\ref{lem:dense:boundary}. Let $c<c'<\widetilde<1$. Let
\begin{equation*}\widetilde\Upsilon=\Upsilon\cap \bigcup_{y\in\Psi} B(y,\varrho/2).\end{equation*}
By Lemma~\ref{lem:dense:boundary} with $\Psi$ replaced by $\widetilde\Upsilon$, there is an $r>0$ small enough that $\|u-u\eta_r\|_{\dot W^{1,p}_{av,c',\beta,s}(\widetilde\Upsilon;\Omega)} < \varepsilon/2$.

Define $\varphi$ by
\begin{equation*}\varphi(y)=\begin{cases} u(y)\eta_r(y), &y\in\Upsilon_{\widetilde c},\\ 0&\text{otherwise}\end{cases}.\end{equation*}
Then $\varphi=u\eta_r$ in $\widetilde\Upsilon_{c'}$ and so $\varphi\in \dot W^{1,p}_{av,c',\beta,s}(\widetilde\Upsilon;\Omega)$.

Let $K=\{y\in\overline{\Psi_c}:\delta(y)\geq r/2\}$. Then $K$ is a compact set, and so $\nabla (u\eta_r)\in L^\beta(K)$.

Let $\theta$ be smooth, nonnegative, supported in $B(0,1)$, and integrate to~$1$. Let $\theta_\delta(y)=\delta^{-\pdmn}\theta(y/\delta)$.

As is standard in smooth mollifier arguments, we have that $\nabla (\theta_\delta*(u\eta_r))\to \nabla (u\eta_r)$ as $\delta\to 0^+$ in $L^\beta(K)$. Furthermore $\eta_r(y)=0$ if $\delta(y)\leq r$ and so $\theta_\delta*(u\eta_r)=0$ in $\Psi_c\setminus K$. Thus $\nabla (\theta_\delta*(u\eta_r))\to \nabla (u\eta_r)$ in $\dot W^{1,p}_{av,c,\beta,s}(\Psi;\Omega)$, and so choosing a sufficiently small $\delta$ completes the proof.
\end{proof}

\begin{lem}\label{lem:dense:infinity}Let $\Omega$ be a weak local John domain with unbounded complement. Let $0<s<1$, $1\leq\beta\leq\infty$, and let $1\leq p\leq \infty$. Let $x_0\in\partial\Omega$.

For every $u\in \widetilde W^{1,p}_{av,c,\beta,s}(\Omega)$ there is a collection of functions $\{\psi_R\}_{R>1}$ such that $\supp\psi_R\subseteq\supp u$, $u(y)=\psi_R(y)$ for all $y\in\Omega$ with $|y-x_0|\leq \sqrt{R}$,
$0=\psi_R(y)$ for all $y\in\Omega$ with $|y-x_0|>R$, and such that $\|\psi_R\|_{\widetilde W^{1,p}_{av,c,\beta,s}(\Omega)} \leq C \|u\|_{\widetilde W^{1,p}_{av,c,\beta,s}(\Omega)}$, where $C$ does not depend on~$R$.

If $p<\infty$ we additionally have that
\begin{equation*}\lim_{R\to \infty}\|u-\psi_R\|_{\widetilde W^{1,p}_{av,c,\beta,s}(\Omega)}=0.\end{equation*}

\end{lem}

\begin{proof}
Without loss of generality we may take $x_0=0$. Take $\varphi\in C^{\infty}[0,\infty)$ with values in $[0,1]$ such that $\varphi\equiv1$ on $[0,1/2]$ and $\varphi\equiv 0$ on $[1,\infty)$. Let us define, for $R>1$, $\eta_R(x):=\varphi\left(\frac{\ln_+|x|}{\ln R}\right)$. Notice that $|\nabla\eta_R|\lesssim|x|^{-1}\ln(R)^{-1}$, that $\eta_R$ is supported in $\overline B(0,R)$, and that $\nabla\eta_R$ is supported in the annulus $\{x\in\R^\dmn:\sqrt{R}\le|x|\le R\}$.
We claim that $\psi_R=u\eta_R$ satisfies the conditions of the lemma.

As in the proof of Lemma~\ref{lem:dense:boundary},
we have that $\|(1-\eta_R)\nabla u\|_{L^p_{av,c,\beta,s}(\Omega)}\leq \|\nabla u\|_{L^p_{av,c,\beta,s}(\Omega)}$ for all $p$ and that this term approaches zero as $R\to \infty$ for all $p<\infty$.

Thus, it suffices to bound $\|u\nabla\eta_R\|_{L^p_{av,c,\beta,s}(\Omega)}$. Recall that
\begin{align*}
\|u\nabla\eta_R\|_{L^p_{av, c,\beta, s}(\Omega)}
&=\biggl(\int_{\Omega}
\mathcal{W}_{c,\beta}(u\nabla\eta_R)^p \delta^{\bdmn-\pdmn+p-ps}\biggr)^{1/p}.
\end{align*}
If $y\in B(x,c\delta(x))$, then by the triangle inequality we have that $|y|\ge|x|-|y-x|>(1-c)\delta(x)$ and so $|\nabla\eta_R(y)|\leq C(\ln R)^{-1} \delta(x)^{-1}$. Thus
\begin{align*}
\|u\nabla\eta_R\|_{L^p_{av, c,\beta, s}(\Omega)}
&\leq
C(\ln R)^{-1}
\biggl(\int_{\Omega}
\mathcal{W}_{c,\beta}(u)^p \delta^{\bdmn-\pdmn-ps}\biggr)^{1/p}.
\end{align*}
Applying Lemma~\ref{lem:boundary:Poincare} yields that
\begin{align*}
\|u\nabla\eta_R\|_{L^p_{av, c,\beta, s}(\Omega)}
&\leq
C(\ln R)^{-1}
\biggl(\int_{\Omega}
\mathcal{W}_{c,\beta}(\nabla u)^p \delta^{\bdmn-\pdmn+p-ps}\biggr)^{1/p}.
\end{align*}
Thus this quantity is bounded independently of $R$ and approaches zero as $R\to\infty$ for all~$p$. (The quantity $\|(1-\eta_R)\nabla u\|_{L^p_{av,c,\beta,s}(\Omega)}$ only approaches zero if $p<\infty$, and so the lemma statement cannot be strengthened.)
\end{proof}

Combining the above lemmas yields that any $u\in \widetilde W^{1,p}_{av,c,\beta,s}(\Omega)$ may be approximated by functions supported in compact subsets of~$\Omega$.

\begin{rmk}\label{rmk:dense:smooth}We may also use smooth mollification in conjunction with Lemma~\ref{lem:dense:boundary} to find smooth approximations to functions in $\dot W^{1,p}_{av,c,\beta,s}(\Psi;\Omega)$ for suitable~$\Psi$. \end{rmk}

\section{Elliptic differential equations: the interior}
\label{sec:elliptic}

From this point onwards we will work with an elliptic differential operator
\begin{equation*}L=-\Div A\nabla\end{equation*}
that satisfies the following ellipticity conditions and with the following standard weak formulation.
\begin{defn}\label{dfn:L}Let $\Lambda>\lambda>0$ and let $0<\bdmn<\dmn$. Let $\Omega\subset\R^\dmn$ be open and suppose that $\partial\Omega$ is $\bdmn$-Ahlfors regular. Recall that
\begin{equation*} \delta(x)=\dist(x,\partial\Omega).\end{equation*}
Let $A:\R^\dmn\to \R^{\pdmn\times\pdmn}$ be such that the ellipticity condition
\begin{equation}
\label{eqn:elliptic}
\lambda\delta(x)^{\bpdmn+1-\pdmn}\|\xi\|^2\leq  \,{\xi}\cdot A(x)\xi,\quad |\eta\cdot A(x)\xi|\leq \Lambda\delta(x)^{\bpdmn+1-\pdmn} \|\xi\|\,\|\eta\|
\end{equation}
is valid for all $\eta$, $\xi\in\R^\dmn$ and all $x\in\Omega$.

We say that $Lu=-\Div \vec H$ in an open set $\Psi\subseteq\Omega$ if
$\vec H\in L^1_{loc}(\Psi)$, $u\in W^{1,1}_{loc}(\Psi)$, and
\begin{equation}\label{eqn:L}
\int_\Psi\nabla\varphi\cdot A\nabla u=\int_\Psi \nabla\varphi\cdot \vec H\text{ for all }\varphi\in C^\infty_0(\Psi).
\end{equation}
\end{defn}

In this section we will list some known properties and constructions related to the elliptic operator~$L$.

\subsection{The Lax-Milgram solution to the Poisson problem}

We construct the solution to the Poisson problem in $\Omega$ using the Lax-Milgram lemma as follows. This construction is by now very standard in the case $\bdmn=\dmnMinusOne$ (and in other situations beyond the scope of this paper), and generalizes naturally to the higher codimensional case (see \cite[Lemma~9.1]{DavFM21} and \cite[Lemma~12.2]{DavFM20p}).

\begin{defn}\label{dfn:Lax-Milgram}Let $\bdmn$, $\dmn$, $L$, and~$\Omega$ be as in Definition~\ref{dfn:L}.
We additionally require $\Omega$ to be a weak local John domain with unbounded complement.

Let $\widetilde W^{1,2}_{av,c,1/2,2}(\Omega)$ be as in formula~\eqref{eqn:W:traceless}. As noted in Section~\ref{sec:dense}, $\widetilde W^{1,2}_{av,c,1/2,2}(\Omega)$ is a Banach space. By Corollary~\ref{cor:averaged:is:Wp} and Lemma~\ref{lem:boundary:Poincare}, if $0<c<1$ and if $\varphi\in \widetilde W^{1,2}_{av,c,1/2,2}(\Omega)$, we have that
\begin{equation*}\|\varphi\|_{\widetilde W^{1,2}_{av,c,1/2,2}(\Omega)}^2
\approx\|\varphi\|_{\dot W^{1,2}_{av,c,1/2,2}(\Omega)}^2
\approx \int_\Omega |\nabla \varphi|^2\delta^{\bdmn+1-\pdmn}
\end{equation*}
and so (up to equivalence of norms) $\widetilde W^{1,2}_{av,c,1/2,2}(\Omega)$ is a Hilbert space with inner product $\langle\varphi,\psi\rangle=\int_\Omega \overline{\nabla\varphi}\cdot \nabla \psi \, \delta^{\bdmn+1-\pdmn}$.

By the ellipticity condition~\eqref{eqn:elliptic}, we may apply the Lax-Milgram lemma and derive that, if
\begin{equation*}\int_{\Omega} |\vec H|^2\,\delta^{\bdmn+1-\pdmn}<\infty\end{equation*}
then there is a unique $u\in \widetilde W^{1,2}_{av,c,1/2,2}(\Omega)$ such that
\begin{equation}\label{eqn:weak}
\int_{\Omega} \nabla \varphi\cdot A\nabla u=
\int_{\Omega} \nabla\varphi\cdot (\delta^{\bdmn+1-\pdmn}\vec H)
\end{equation}
for all $\varphi\in \widetilde W^{1,2}_{av,c,1/2,2}(\Omega)$. Furthermore, $\|u\|_{\widetilde W^{1,2}_{av,c,1/2,2}(\Omega)}$ is comparable to the norm of $\varphi\mapsto \int_{\Omega} \nabla\varphi\cdot (\delta^{\bdmn+1-\pdmn}\vec H)$ as an operator on $\widetilde W^{1,2}_{av,c,1/2,2}(\Omega)$.

We call $u$ the solution to the Poisson problem with data~$\vec H$.
\end{defn}

By Lemma~\ref{lem:trace:zero} and the definition of $\widetilde W^{1,2}_{av,c,1/2,2}(\Omega)$, we have that $u$ satisfies
\begin{equation}\label{eqn:Dirichlet:Lax-Milgram}
Lu=-\Div(\delta^{\bdmn+1-\pdmn}\vec H) \text{ in } \Omega,\quad
\Tr u=0\text{ on }\partial\Omega
,\quad
\|u\|_{\dot W^{1,2}_{av,c,1/2,2}(\Omega)}
<\infty\end{equation}
and in addition
\begin{align*}
\|u\|_{\dot W^{1,2}_{av,c,1/2,2}(\Omega)}
&\leq
C\|\vec H\|_{L^2_{av,c,1/2,2}(\Omega)}
\end{align*}
for some constant $C$ depending only on the ellipticity constants~$\lambda$ and~$\Lambda$ in the condition~\eqref{eqn:elliptic}, the Ahlfors constant~$A$ in Definition~\eqref{dfn:Ahlfors}, and the dimension $\dmn$ of the ambient space and $\bdmn$ of the boundary.

We have assumed that $\R^\dmn\setminus\Omega$ is unbounded. In this case, by Proposition~\ref{prp:dense}, if formula~\eqref{eqn:L} is valid for all $\varphi\in C^\infty_0(\Omega)$ then it is valid for all $\varphi\in \widetilde W^{1,2}_{av,c,1/2,2}(\Omega)$. The Lax-Milgram lemma and Lemma~\ref{lem:boundary:Poincare} imply that $u$ is the only solution to the problem~\eqref{eqn:Dirichlet:Lax-Milgram}.

If $\R^\dmn\setminus\Omega$ is bounded, then formula~\eqref{eqn:L} for all $\varphi\in C^\infty_0(\Omega)$ does \emph{not} imply that formula~\eqref{eqn:L} is valid for all $\varphi\in \widetilde W^{1,2}_{av,c,1/2,2}(\Omega)$. This is one of the main technical difficulties in generalizing our results to the case where $\R^\dmn\setminus\Omega$ is bounded.

\subsection{Interior estimates}
\label{sec:interior}

Let $\Omega$, $A$ and~$L$ be as in Definition~\ref{dfn:L}. In this section we will discuss bounds on solutions to $Lu=0$ in $\Omega$ (or solutions to the Poisson problem) far from~$\partial\Omega$.

Suppose that $B(y,4r)\subset\Omega$. Define $\widetilde A(z)=\dist(y,\partial\Omega)^{\dmn-1-\pbdmn} A(z)$.
If $z\in B(y,2r)$ then $\dist(z,\partial\Omega)< \dist(y,\partial\Omega)+2r \leq (3/2)\dist(y,\partial\Omega)$ and $\dist(z,\partial\Omega)>\dist(y,\partial\Omega)-2r \geq \dist(y,\partial\Omega)/2$. Thus $\widetilde A$ satisfies
\begin{gather*}
(2/3)^{\dmn-1-\pbdmn}\lambda\|\xi\|^2
\leq \frac{[(2/3)\dist(z,\partial\Omega)]^{\dmn-1-\pbdmn}} {\dist(y,\partial\Omega)^{\dmn-1-\pbdmn}} \xi\cdot \widetilde A(z)\xi
\leq \xi\cdot \widetilde A(z)\xi
,\\
|\eta\cdot \widetilde A(z)\xi|
\leq
\frac{\dist(y,\partial\Omega)^{\dmn-1-\pbdmn}} {\dist(z,\partial\Omega)^{\dmn-1-\pbdmn}} \Lambda\|\xi\|\,\|\eta\|
\leq
2^{\dmn-1-\pbdmn}\Lambda\|\xi\|\,\|\eta\|
\end{gather*}
for all $\eta$, $\xi\in\R^\dmn$.
Thus $\widetilde A$ is elliptic in $B(y,2r)$ in the classical (codimension~$1$) sense, and so inherits classical interior regularity results. For ease of reference we list the relevant such results here.

If $\vec H\in L^2(B(y,2r))$, $B(y,4r)\subset\Omega$, and $Lu=-\Div(\delta^{1+\bdmn-\pdmn}\vec H)$ in $B(y,2r)$, then
\begin{equation*}-\Div \widetilde A\nabla u=-\Div\biggl( \biggl(\frac{\delta}{\delta(y)}\biggr)^{1+\bdmn-\pdmn} \vec H\biggr)\end{equation*}
in $B(y,2r)$, and so by the classic Caccioppoli inequality, \begin{align}
\label{eqn:interior:Caccioppoli:1}
\int_{B(y,r)}|\nabla u|^2
&\leq \frac{C}{r^2}\int_{B(y,2r)}|u|^2
+C\int_{B(y,2r)}|\vec H|^2
\end{align}
which we may rewrite as
\begin{align}
\label{eqn:interior:Caccioppoli}
\int_{B(y,r)}|\nabla u|^2 \delta^{\bdmn+1-\pdmn}
&\leq \frac{C}{r^2}\int_{B(y,2r)}|u|^2 \delta^{\bdmn+1-\pdmn}
+C\int_{B(y,2r)}
|\vec H|^2 \delta^{\bdmn+1-\pdmn}
\end{align}
recalling that $\delta(z)=\dist(z,\partial\Omega)$.

Similarly, by the classic result of Meyers \cite{Mey63}, there is a $p^+=p^+_L>2$ such that, if $2\leq\beta<p^+$, then there is a $C_{\beta}<\infty$ such that if $\vec H\in L^\beta(B(y,2r))$, $B(y,4r)\subset\Omega$, and $Lu=-\Div(\delta^{1+\bdmn-\pdmn}\vec H)$ in $B(y,2r)$, then
\begin{align}
\label{eqn:interior:Meyers}
\biggl(\fint_{B(y,r)}|\nabla u|^\beta \biggr)^{1/\beta}
&\leq C_{q,\beta}\biggl(\fint_{B(y,2r)} |\nabla u|^2\biggr)^{1/2}
+C_{q,\beta}\biggl(\fint_{B(y,2r)}|\vec H|^\beta \biggr)^{1/\beta}.
\end{align}
It is possible to bound $p^+-2$ from below by a positive number depending only on the ellipticity constants of~$\widetilde A$ (and thus only on the ellipticity constants of~$A$). In many cases, it is possible to improve beyond these estimates on~$p^+_L$; for example, if $A$ is constant\footnote{By the ellipticity condition~\eqref{eqn:elliptic}, this can only happen if $\bdmn=\dmnMinusOne$.} then $p^+_L=\infty$.

In the case where $\vec H=0$ and so $Lu=0$ in $B(y,2r)$, then we have the interior Moser condition  (see, for example, \cite[Theorem 4.1]{HanL11})
\begin{align}
\label{eqn:interior:Moser}
\sup_{B(y,r)}|u|
&\leq C \fint_{B(x,2r)}|u|,
\end{align}
and the interior De Giorgi-Nash condition (see \cite[Theorem~4.11]{HanL11})
\begin{align}\label{eqn:interior:DGN:l2}
\sup_{x,z\in B(y,r)}\frac{|u(x)-u(z)|}{|x-z|^\alpha}
&\leq \frac{C}{r^\alpha} \biggl(\fint_{B(x,2r)}|u|^2\biggr)^{1/2}.
\end{align}
Combining the estimates \eqref{eqn:interior:Moser} and~\eqref{eqn:interior:DGN:l2} yields the improved estimate
\begin{align}
\label{eqn:interior:DGN}
\sup_{x,z\in B(y,r)}\frac{|u(x)-u(z)|}{|x-z|^\alpha}
&\leq \frac{C}{r^\alpha} \fint_{B(x,2r)}|u|,
\end{align}
where $C$ and $\alpha$ are constants depending on $\dmn$, $\lambda$, and $\Lambda$, with $0<\alpha<1$.

We conclude this section with the following remark.
\begin{rmk}\label{rmk:maximum}
By the maximum principle for solutions to elliptic equations (see, for example, \cite[Theorem~8.19]{GilT01}) we have that if $Lu=0$ in~$\Omega$, and if $u$ has an interior maximum or minimum in~$\Omega$, then $u$ must be constant in~$\Omega$.
\end{rmk}

\section{The Poisson problem: local integrability, duality, interpolation, and well posedness}\label{sec:beta}

In this section, we will prove duality and interpolation properties for the spaces $L^p_{av,\beta,s}(\Omega)$. These properties have immediate consequences for well posedness of boundary value problems for various ranges of $p$ and~$s$. Thus, they will be useful in the proof of Theorem~\ref{thm:Poisson}, both directly and indirectly.

\subsection{Definition of well posedness}

As noted, properties of the spaces $L^p_{av,\beta,s}(\Omega)$ have immediate consequences for well posedness of the Poisson problem. To state these consequences, it will be useful to precisely define well posedness and related concepts.

Recall that if $p\geq 1$ then the space $\widetilde W^{1,p}_{av,c,\beta,s}(\Omega)$ is given by formula~\eqref{eqn:W:traceless} and coincides with the set of functions in $\dot W^{1,p}_{av,c,\beta,s}(\Omega)$ with trace zero. Suppose that $0<p<1$, that $0<s<1$, and that $s/\bdmn+1>1/p$. By Corollary~\ref{cor:weighted:embedding}, we have that $\dot W^{1,p}_{av,c,\beta,s}(\Omega)\subset \dot W^{1,1}_{av,c,\beta,s+\bdmn-\pbdmn/p}(\Omega)$, and so we define
\begin{equation}
\label{eqn:W:traceless:quasi}
\widetilde W^{1,p}_{av,c,\beta,s}(\Omega)=\widetilde W^{1,1}_{av,c,\beta,s+\bdmn-\pbdmn/p}(\Omega)\cap \dot{W}^{1,p}_{av,c,\beta,s}(\Omega).
\end{equation}
By Corollary~\ref{cor:weighted:embedding} and Lemma~\ref{lem:boundary:Poincare}, we have that if $u\in \dot{W}^{1,p}_{av,c,\beta,s}(\Omega)$ and $\Tr u=0$ almost everywhere on~$\partial\Omega$, then
\begin{align*}\|u\|_{\widetilde W^{1,1}_{av,c,\beta,s+\bdmn-\pbdmn/p}(\Omega)}
&=\|u\|_{\dot W^{1,1}_{av,c,\beta,s+\bdmn-\pbdmn/p}(\Omega)}
+\|u\|_{L^1_{av,c,\beta,s+\bdmn-\pbdmn/p+1}(\Omega)}
\\&\leq C\|u\|_{\dot W^{1,1}_{av,c,\beta,s+\bdmn-\pbdmn/p}(\Omega)}
\leq C\|u\|_{\dot W^{1,p}_{av,c,\beta,s}(\Omega)}
\end{align*}
and so $\widetilde W^{1,p}_{av,c,\beta,s}(\Omega)$ is the closed subspace of $\dot W^{1,p}_{av,c,\beta,s}(\Omega)$ consisting of functions with trace zero, and the norms in the two spaces are equivalent. (As usual, the exact value of the parameter~$c$ is irrelevant up to equivalence of norms.)

Throughout this section, if $\vec H\in L^2_{av,2,1/2}(\Omega)$ then we will let $u_{\vec H}$ be the solution to the Poisson problem with data~$\vec H$ given by Definition~\ref{dfn:Lax-Milgram}, that is, the unique function that satisfies $u_{\vec H}\in \widetilde W^{1,2}_{av,2,1/2}(\Omega)$ and $Lu_{\vec H}=-\Div(\delta^{\bdmn+1-\pdmn}\vec H)$.

\begin{defn}\label{dfn:Poisson:psb}Let $\Omega\subset\R^\dmn$ be a weak local John domain and suppose that $\partial\Omega$ is $\bdmn$-Ahlfors regular for some $0<\bdmn<\dmn$. Let $L$ be as in Definition~\ref{dfn:L} (that is, a divergence-form operator that satisfies the weighted ellipticity condition~\eqref{eqn:elliptic}).

Let $0<p\leq \infty$, $s>0$, and $1\leq
\beta\leq \infty$. If $p<1$ we additionally require $s/\bdmn+1>1/p$ so that $\widetilde W^{1,p}_{av,c,\beta,s}(\Omega)$ is defined.

We say that the $L^{p}_{av,\beta,s}(\Omega)$-Poisson problem for $L$ is solvable in~$\Omega$ if, for every $\vec H\in L^{p}_{av,\beta,s}(\Omega)$, there is a function $u$ that satisfies
\begin{equation}\label{poisson:equation}
\left\{\begin{aligned}
Lu&=-\Div(\delta^{\bdmn+1-\pdmn}\vec H) \text{ in } \Omega,\\
u&\in \widetilde W^{1,p}_{av,\beta,s}(\Omega)
\end{aligned}\right.\end{equation}
and furthermore for some (hence every) $c\in (0,1)$ there exists a constant $C$ (depending on~$c$) such that
\begin{equation}\label{poisson:estimate}
\|u\|_{\dot W^{1,p}_{av,c,\beta,s}(\Omega)}
\leq C\|\vec H\|_{L^p_{av,c,\beta,s}(\Omega)}.\end{equation}

We say that the $L^{p}_{av,\beta,s}(\Omega)$-Poisson problem for $L$ satisfies the uniqueness property in~$\Omega$ if, for any $\vec H\in L^p_{av,\beta,s}(\Omega)$, there is at most one solution $u\in \widetilde W^{1,p}_{av,\beta,s}(\Omega)$ to the problem~\eqref{poisson:equation}.

We say that the $L^{p}_{av,\beta,s}(\Omega)$-Poisson problem for $L$ is well posed in~$\Omega$ if it is both solvable and satisfies the uniqueness property.

We say that the $L^{p}_{av,\beta,s}(\Omega)$-Poisson problem for $L$ is compatibly solvable in~$\Omega$ if there is a $C<\infty$ and a $c\in (0,1)$ such that, for every $\vec H\in L^{p}_{av,\beta,s}(\Omega)\cap L^2_{av,2,1/2}(\Omega)$, the Lax-Milgram solution $u_{\vec H}\in \widetilde W^{1,2}_{av,2,1/2}(\Omega)$ given by Definition~\ref{dfn:Lax-Milgram} to the Poisson problem
satisfies the estimate~\eqref{poisson:estimate}.

Finally, we say that the $L^{p}_{av,\beta,s}(\Omega)$-Poisson problem for $L$ is compatibly well posed in~$\Omega$ if it is both compatibly solvable and well posed.
\end{defn}

\begin{rmk}\label{rmk:compatible:solvable}
A straightforward density argument shows that, under the above assumptions on $s$, $p$ and~$\beta$, if $L^p_{av,\beta,s}(\Omega)\cap L^2_{av,2,1/2}(\Omega)$ is dense in $L^p_{av,\beta,s}(\Omega)$, then compatible solvability of the $L^{p}_{av,\beta,s}(\Omega)$-Poisson problem implies solvability. This density result is true if $p<\infty$ and $\beta<\infty$ (in fact by Lemma~\ref{lem:averaged:Whitney} compactly supported functions in $L^2(\Omega)\cap L^\beta(\Omega)$ are dense), and so in this case compatible solvability implies solvability. We will show that compatible solvability implies solvability in the case $p=\infty$ in Lemma~\ref{lem:solv:infty}.
\end{rmk}

By Definition~\ref{dfn:Lax-Milgram} and Corollary~\ref{cor:averaged:is:Wp}, we have that the $L^{2}_{av,2,1/2}(\Omega)$-Poisson problem is well posed. (It is compatibly well posed by definition.) The goal of this paper is to establish well posedness of the $L^{p}_{av,\beta,s}(\Omega)$-Poisson problem for a broader range of $p$, $\beta$, and~$s$.

\subsection{Duality}\label{sec:duality}
We now establish a duality result.

Let us fix from now on a grid of Whitney cubes $\mathcal{G}=\mathcal{G}_1$ in~$\Omega$.
In this subsection we will show that the dual space of a weighted averaged space is again a weighted averaged space.

We start by defining
\begin{equation}
\langle F,G\rangle=\int_\Omega F\cdot G\,\delta^{1+\bdmn-\dmn}
\end{equation}
where $F$ and $G$ are vector-valued or scalar-valued functions.
Since the side-length of a Whitney cube is proportional to its distance to the boundary we have
\begin{align*}
|\langle F,G\rangle| &\le \int_{\Omega} |F\cdot G|\,\delta^{1+\bdmn-\dmn}
\approx \sum_{Q\in\mathcal{G}}\fint_Q |F\cdot G| \,\ell(Q)^{1+\bdmn}.
\end{align*}

By Hölder's inequality, if $1\leq \beta\leq\infty$ we have that
\begin{align*}
|\langle F,G\rangle|&\le \sum_{Q\in\mathcal{G}}\biggl(\fint_Q |F|^{\beta'}\biggr)^{1/\beta'}\ell(Q)^{s+\frac{\bdmn}{p'}}\biggl(\fint_Q |G|^\beta\biggr)^{1/\beta}\ell(Q)^{1-s+\frac{\bdmn}{p}}
.\end{align*}
Applying Hölder's inequality in sequence spaces shows that if $1\leq p\leq \infty$ then
\begin{align*}
|\langle F,G\rangle|&\le  \biggl(\sum_{Q\in\mathcal{G}}\biggl(\fint_Q |F|^{\beta'}\biggr)^{p'/\beta'}\ell(Q)^{\bdmn+p's}\biggr)^{1/p'}
\biggl(\sum_{Q\in\mathcal{G}}\biggl(\fint_Q|G|^\beta\biggr)^{p/\beta}\ell(Q)^{\bdmn+p-ps}\biggr)^{1/p}
\end{align*}
and so by Lemma~\ref{lem:averaged:Whitney}, if $F\in{L^{p'}_{av,\beta',1-s}(\Omega)}$ and $G\in{L^p_{av,\beta,s}(\Omega)}$ then
\begin{align}\label{byhold}
|\langle F,G\rangle|&\le C\|F\|_{L^{p'}_{av,c,\beta',1-s}(\Omega)} \|G\|_{L^p_{av,c,\beta,s}(\Omega)}
\end{align}
where $C$ depends only on $p$, $s$, $\dmn$, $\bdmn$, and~$c$.

We will now use the inner product to produce an upper bound on the norms.
\begin{thm}\label{dualspace}
Suppose that $1\le p\le\infty$, $1\leq\beta\leq\infty$, and $s\in\R$. If $F\in L^1_{loc}(\Omega)$, and if
\begin{equation*}M=\sup\bigl\{ |\langle F,G\rangle|:G\text{ is compactly supported, bounded and }\|G\|_{L^{p'}_{av,c,\beta',1-s}(\Omega)}=1\bigr\}\end{equation*}
is finite, then $F\in L^p_{av,c,\beta,s}(\Omega)$ and
\begin{equation*}\|F\|_{L^p_{av,c,\beta,s}(\Omega)} \leq CM\end{equation*}
provided the right hand side is finite.
\end{thm}
\begin{proof}
Fix some $F$ such that $M$ is finite. Let $Q\in\mathcal{G}$. Then
\begin{equation*}\biggl|\int_Q F\cdot G \,\delta^{1+\bdmn-\pdmn}\biggr|
=|\langle F, \1_QG\rangle|\leq M\end{equation*}
for all bounded functions $G$ with $1=\|\1_Q G\|_{L^{p'}_{av,c,\beta',1-s}(\Omega)} \approx \|G\|_{L^{\beta'}(Q)}\ell(Q)^{\bdmn/p'-\pdmn/\beta'+s}$. Here $\1_Q G=G$ in $Q$ and $\1_Q G=0$ outside of~$Q$. It is well known that this yields $F\in L^\beta(Q)$. (See, for example, \cite[Lemma~8.4]{RoyF10}.)

For each $Q\in\mathcal{G}$, let $G_Q$ be bounded, supported in~$Q$ and satisfy
\begin{equation*}\|G_Q\|_{L^{\beta'}(Q)}=1,\qquad \int_Q (\delta^{1+\bdmn-\pdmn}F)\cdot G_Q  \geq \frac{1}{2} \|F\delta^{1+\bdmn-\pdmn}\|_{L^\beta(Q)}. \end{equation*}
(If $\beta<\infty$ then we may instead require $\int_Q (\delta^{1+\bdmn-\pdmn}F)\cdot G_Q  =\|F\delta^{1+\bdmn-\pdmn}\|_{L^\beta(Q)}$.)
Because $\delta\approx \ell(Q)$ in $Q$ we have that
\begin{equation*}\int_Q (\delta^{1+\bdmn-\pdmn}F)\cdot G_Q  \geq \vartheta\,\ell(Q)^{1+\bdmn-\pdmn}\|F\|_{L^{\beta}(Q)} \end{equation*}
for some $\vartheta>0$.

Let $\{g_Q\}_{Q\in\mathcal{G}}$ be a sequence of nonnegative real numbers in $\ell^{p'}(\mathcal{G})$.
$\mathcal{G}$ is countable, so we may write $\mathcal{G}=\{Q_j\}_{j=1}^\infty$.
Let $H_k$ satisfy
\begin{equation*}H_k=\sum_{j=1}^k g_{Q_j}\,\1_{Q_j} G_{Q_j}\,\ell(Q_j)^{-\bdmn/p'+\pdmn/\beta'-s}.\end{equation*}
Then $H_k$ is bounded (as each $G_{Q_j}$ is bounded), compactly supported (in the union of finitely many cubes), and by Lemma~\ref{lem:averaged:Whitney}
\begin{equation*}\|H_k\|_{L^{p'}_{av,c,\beta', 1-s}(\Omega)}
\approx
\Bigl(\sum_{j=1}^k |g_{Q_j}|^{p'} \|G_{Q_j}\|_{L^{\beta'}(Q_j)} ^{p'} \Bigr)^{1/p'}\leq \|\{g_Q\}_{Q\in\mathcal{G}}\|_{\ell^{p'}(\mathcal{G})}.
\end{equation*}
Thus
\begin{equation*}|\langle F,H_k\rangle| \leq M \|H_k\|_{L^{p'}_{av,c,\beta', 1-s}(\Omega)}
\leq C M \|\{g_Q\}_{Q\in\mathcal{G}}\|_{\ell^{p'}(\mathcal{G})}.\end{equation*}
But by definition of $H_k$ and the inner product
\begin{equation*}\langle F,H_k\rangle
= \sum_{j=1}^k g_{Q_j}\,\ell(Q_j)^{-\bdmn/p'+\pdmn/\beta'-s}
\int_{Q_j}F \cdot G_{Q_j}\,\delta^{1+\bdmn-\pdmn}
\end{equation*}
and by choice of $G_{Q}$
\begin{equation*}\langle F,H_k\rangle
\geq \sum_{j=1}^k g_{Q_j}\,\ell(Q_j)^{\bdmn/p-\pdmn/\beta+1-s}
\vartheta\|F\|_{L^\beta(Q_j)}
.\end{equation*}
Thus, taking the supremum over~$k$, we have that
\begin{equation*}\sum_{Q\in\mathcal{G}} g_Q \,\ell(Q)^{\bdmn/p-\pdmn/\beta+1-s}
\|F\|_{L^\beta(Q)}
\leq CM\|\{g_Q\}_{Q\in\mathcal{G}}\|_{\ell^{p'}(\mathcal{G})}.\end{equation*}
This yields that the sequence
\begin{equation*}\{\ell(Q)^{\bdmn/p-\pdmn/\beta+1-s}
\|F\|_{L^\beta(Q)}\}_{Q\in\mathcal{G}}\end{equation*}
lies in $\ell^p(\mathcal{G})$ with norm at most $CM$, and so by Lemma~\ref{lem:averaged:Whitney} we have that $F\in L^p_{av,\beta,s}(\Omega)$ with norm at most~$CM$, as desired.
\end{proof}

We now apply this duality result in the context of the Poisson problem.

\begin{thm}\label{first:step:duality}
Let $1\leq p\leq \infty$, $1\leq\beta\leq\infty$ and $s\in\R$. Assume that the $L^{p}_{av,\beta,s}(\Omega)$-Poisson problem for $L$ is compatibly solvable. Then the $L^{p'}_{av,\beta',1-s}(\Omega)$-Poisson problem for $L^*$ is compatibly solvable.
\end{thm}
\begin{proof}
Let $u_{\vec H}\in \widetilde W^{1,2}_{av,2,1/2}(\Omega)$ denote the solution to the Poisson problem $Lu_{\vec H}=-\Div(\delta^{1+\bdmn-\pdmn}\vec H)$ given by Definition~\ref{dfn:Lax-Milgram}, and let $u_{\vec H}^*$ denote the corresponding solution to $L^*u^*_{\vec H}=-\Div(\delta^{1+\bdmn-\pdmn}\vec H)$.

Suppose that
$\vec F\in L^p_{av,\beta,s}(\Omega)\cap L^2_{av,2,1/2}(\Omega)$ and that $\vec H\in L^{p'}_{av,\beta',1-s}(\Omega)\cap L^2_{av,c,1/2}(\Omega)$. By Proposition~\ref{prp:dense} and the weak definition of $L$ and~$L^*$, we have that
\begin{equation*}\int_\Omega \nabla\varphi\cdot A\nabla u_{\vec F} = \int_\Omega \nabla \varphi\cdot \vec F\,\delta^{1+\bdmn-\pdmn},
\qquad
\int_\Omega A\nabla\varphi\cdot \nabla u^*_{\vec H} = \int_\Omega \nabla \varphi\cdot \vec H\,\delta^{1+\bdmn-\pdmn}\end{equation*}
for any $\varphi\in \widetilde W^{1,2}_{av,2,1/2}(\Omega)$. In particular, we may take $\varphi=u_{\vec F}$ or $\varphi=u^*_{\vec H}$ and see that
\begin{equation}
\label{eqn:dual}
\int_\Omega \nabla u^*_{\vec H}\cdot \vec F\,\delta^{1+\bdmn-\pdmn}
=
\int_\Omega \nabla u^*_{\vec H}\cdot A\nabla u_{\vec F}  = \int_\Omega \nabla u_{\vec F}\cdot \vec H\,\delta^{1+\bdmn-\pdmn}.\end{equation}

By assumption the $L^{p}_{av,\beta,s}(\Omega)$-Poisson problem for $L$ is compatibly solvable. That is,
\begin{equation*}\|u_{\vec F}\|_{\widetilde W^{1,p}_{av,c,\beta,s}(\Omega)}\leq C\|\vec F\|_{L^p_{av,c,\beta,s}(\Omega)}\end{equation*}
for all $\vec F\in L^p_{av,c,\beta,s}(\Omega)\cap L^2_{av,c,\beta,s}(\Omega)$. In particular, this is true whenever $\vec F$ is bounded and compactly supported.

By Theorem~\ref{dualspace} and the bound~\eqref{byhold}, we have that
\begin{equation*}\|u^*_{\vec H}\|_{\widetilde W^{1,p'}_{av,c,\beta',1-s}(\Omega)}\leq C\|\vec H\|_{L^{p'}_{av,c,\beta',1-s}(\Omega)}\end{equation*}
as desired.
\end{proof}

\subsection{Uniqueness}

We now use the above duality result to show that compatible solvability implies uniqueness, and thus compatible well posedness.

\begin{lem}\label{lem:uniqueness}
Let $\Omega\subset\R^\dmn$ be a weak local John domain.
Let $1\leq p\leq \infty$, $1<\beta<\infty$ and $0<s<1$. Assume that the $L^{p}_{av,\beta,s}(\Omega)$-Poisson problem in $\Omega$ for $L$ is compatibly solvable. Then the $L^{p}_{av,\beta,s}(\Omega)$-Poisson problem for $L$ has the uniqueness property.
\end{lem}

\begin{proof}
By linearity it suffices to show that if $u\in \widetilde W^{1,p}_{av,\beta,s}(\Omega)$ is a solution to $Lu=0$ in $\Omega$, then $\nabla u=0$ almost everywhere. In particular, it suffices to show that
\begin{equation*}\int_\Omega \nabla u\cdot \vec \Phi\,\delta^{1+\bdmn-\pdmn}=0\end{equation*}
for all bounded compactly supported functions~$\vec\Phi$.

By Theorem~\ref{first:step:duality}, the $L^{p'}_{av,\beta',1-s}(\Omega)$-Poisson problem for $L^*$ is compatibly solvable. Let $\vec \Phi$ be bounded and compactly supported (and thus in $L^2_{av,2,1/2}(\Omega)\cap L^{p'}_{av,\beta',1-s}(\Omega)$) and let $u^*_{\vec\Phi}$ be the solution to $L^*u^*_{\vec\Phi}=-\Div(\delta^{1+\bdmn-\pdmn}\vec\Phi)$.

By Proposition~\ref{prp:dense} and because $u\in \widetilde W^{1,p}_{av,c,\beta,s}(\Omega)$ for $\beta<\infty$, if $p<\infty$ then we may use $u$ as the test function in the definition~\eqref{eqn:L} of $L^*u^*_{\vec\Phi}$, and so
\begin{align}\label{eqn:uniqueness:proof}
\int_{\Omega}\nabla u\cdot \vec\Phi\,\delta^{1+\bdmn-\dmn}&= \int_{\Omega}\nabla u\cdot A^*\nabla u^*_{\vec\Phi}
= \int_{\Omega}A\nabla u\cdot \nabla u^*_{\vec\Phi}.\end{align}
We claim that this formula still holds if $p=\infty$. In this case $\mathcal{W}_{c,\beta}(\nabla u)\delta^{1-s} +\mathcal{W}_{c,\beta}(u)\delta^{-s}$ is bounded by Lemma~\ref{lem:averaged:nearby} and the definition of $\widetilde W^{1,\infty}_{av,\beta,s}$. Let $\{\eta_n\}_{n=1}^\infty$ be a sequence of smooth cutoff functions taking values in $[0,1]$ and compactly supported in~$\Omega$, with $\eta_n\to 1$ and $\nabla \eta_n \to 0$ pointwise in $\Omega$ and with $|\nabla \eta_n(x)|\leq C\delta(x)^{-1}$. Then
\begin{equation*}\mathcal{W}_{c,\beta}(\nabla (\eta_n u))\delta^{1-s}
\leq C \mathcal{W}_{c,\beta}(\nabla u)\delta^{1-s} +C\mathcal{W}_{c,\beta}(u)\delta^{-s}.\end{equation*}
Thus $\eta_n u$ is compactly supported in $\Omega$ and so may be approximated by smooth functions, and so
\begin{align*}
\int_{\Omega}\nabla (\eta_n u)\cdot \vec\Phi\,\delta^{1+\bdmn-\dmn}&= \int_{\Omega}\nabla (\eta_n u)\cdot A^*\nabla u^*_{\vec\Phi}
.\end{align*}
By the product rule $\nabla(\eta_n u)\to \nabla u$ pointwise.
By the dominated convergence theorem, we may take the limit as $n\to\infty$ and recover formula~\eqref{eqn:uniqueness:proof}.

Because $u^*_{\vec\Phi}\in \widetilde W^{1,p'}_{av,c,\beta',1-s}(\Omega)$, by a similar argument we may use $u^*_{\vec\Phi}$ as the test function in the weak definition of $Lu=0$ and see that
\begin{align*}
\int_{\Omega}\nabla u\cdot \vec\Phi\,\delta^{1+\bdmn-\dmn}&= 0\end{align*}
as desired.
\end{proof}

\subsection{The case \texorpdfstring{$p=\infty$}{p=∞}}

By Remark~\ref{rmk:compatible:solvable} and Lemma~\ref{lem:uniqueness}, we have that if $1\leq p<\infty$, then compatible solvability of the Poisson problem implies well posedness. We would like to extend to the endpoint $p=\infty$.

\begin{lem}\label{lem:solv:infty}
Let $1<\beta<\infty$ and $0<s<1$. If the $L^{\infty}_{av,\beta,s}(\Omega)$-Poisson problem is compatibly solvable, then it is well posed.
\end{lem}
\begin{proof}
Uniqueness is given by Lemma~\ref{lem:uniqueness}.

Choose some $\vec H\in L^\infty_{av,\beta,s}(\Omega)$. We must construct a solution $u\in \widetilde W^{1,\infty}_{av,\beta,s}(\Omega)$ with $Lu=-\Div (\delta^{1+\bdmn-\pdmn}\vec H)$ in~$\Omega$.

Recall that $\mathcal{G}$ is a grid of Whitney cubes in~$\Omega$. As in proof of Theorem~\ref{dualspace}, we may write $\mathcal{G}=\{Q_j\}_{j=1}^\infty$, and let $\vec H_k=\sum_{j=1}^k \1_{Q_j}\vec H$.
Then $\vec H_k$ is compactly supported, and by the dominated convergence theorem $\langle \vec F,\vec H_k\rangle\to\langle \vec F,\vec H\rangle$ for all $\vec F\in L^1_{av,\beta',1-s}(\Omega)$; that is, $\vec H_k\rightharpoonup \vec H$ weakly in $(L^1_{av,\beta',1-s}(\Omega))^\star$. We may then find $\vec \Phi_k$ such that $\vec \Phi_k$ is compactly supported, bounded, and still satisfies $\vec \Phi_k\rightharpoonup \vec H$ weakly in $(L^1_{av,\beta',1-s}(\Omega))^\star$.

Let $u_{\vec\Phi_j}\in \widetilde W^{1,2}_{av,2,1/2}(\Omega)$ be the solution to $Lu_{\vec\Phi_j}=-\Div(\vec\Phi_j\delta^{1+\bdmn-\dmn})$ given by Definition~\ref{dfn:Lax-Milgram}. By assumption of compatible solvability, $u_{\vec\Phi_j}\in \widetilde W^{1,\infty}_{av,\beta,s}(\Omega)$ and if $0<c<1$ then
\begin{equation*}
\|\nabla u_{\vec\Phi_j}\|_{L^\infty_{av,c,\beta,s}(\Omega)}\le C\|\vec\Phi_j\|_{L^\infty_{av,c,\beta,s}(\Omega)}\le C\|\vec H\|_{L^\infty_{av,c,\beta,s}(\Omega)}.
\end{equation*}
Lemma \ref{lem:boundary:Poincare} gives us
\begin{equation*}
\|u_{\vec\Phi_j}\|_{L^\infty_{av,c,\beta,s+1}(\Omega)}\le C\|\nabla u_{\vec\Phi_j}\|_{L^\infty_{av,c,\beta,s}(\Omega)}\le C\|\vec H\|_{L^\infty_{av,c,\beta,s}(\Omega)}.
\end{equation*}
By the weak$^{\star}$ compactness of bounded sequences, by passing to a subsequence we can find $u\in L^\infty_{av,\beta,s+1}(\Omega)$ and $\vec G\in L^\infty_{av,\beta,s}(\Omega)$ such that
$u_j\rightharpoonup u$ in $(L^1_{av,\beta',-s}(\Omega))^\star$ and $\nabla u_j\rightharpoonup\vec G$ in $(L^1_{av,\beta',1-s}(\Omega))^\star$. We now prove that $\nabla u=\vec G$; this will imply that $u\in \widetilde W^{1,\infty}_{av,\beta,s}(\Omega)$.
For any $\vec\Psi\in C^\infty_c(\Omega)$ we have that
\begin{align*}
\int_\Omega u\Div\vec\Psi&=\langle u,\Div(\vec\Psi)\delta^{-1-\bdmn+\dmn}\rangle
.\end{align*}
By weak convergence of $u_{\vec\Phi_j}$ to~$u$ we have that
\begin{align*}
\int_\Omega u\Div\vec\Psi&=\lim_{j\to\infty}\langle u_{\vec\Phi_j}, \Div(\vec\Psi)\delta^{-1-\bdmn+\dmn}\rangle
=\lim_{j\to\infty}\int_\Omega u_{\vec\Phi_j}\Div\vec\Psi
.\end{align*}
Integrating by parts and applying weak convergence of $\nabla u_{\vec\Phi_j}$ yields that
\begin{align*}
\int_\Omega u\Div\vec\Psi&
=-\lim_{j\to\infty}\int_\Omega \nabla u_{\vec\Phi_j}\cdot\vec\Psi
=-\int_\Omega \vec G\cdot\vec\Psi
.\end{align*}
Similarly, by the weak definition of~$L$, we have that $Lu=-\Div(\vec H\delta^{1+\bdmn-\dmn})$.

By the lower semicontinuity of the weak$^{\star}$ convergence and by the properties of $\vec\Phi_j$, we have that
\begin{equation*}
\|\nabla u\|_{L^\infty_{av,c,\beta,s}}\le C \sup_j \|\nabla u_{\vec\Phi_j}\|_{L^\infty_{av,c,\beta,s}}\le C\sup_j\|\vec\Phi_j\|_{L^\infty_{av,c,\beta,s}}=C\|\vec H\|_{L^\infty_{av,c,\beta,s}}.
\end{equation*}
Thus, compatible solvability implies solvability. This completes the proof.
\end{proof}

\subsection{An extended range of\texorpdfstring{~$\beta$}{ β}}

As noted in Definition~\ref{dfn:Lax-Milgram}, the Lax-Milgram lemma yields well posedness of the $L^2_{av,2,1/2}(\Omega)$-Poisson problem. The goal of this paper is to extend to the $L^p_{av,\beta,s}(\Omega)$-Poisson problem for a range of $p$, $\beta$, and~$s$. In this section we show how the bound~\eqref{eqn:interior:Meyers} immediately yields an extended range of~$\beta$.

\begin{lem}\label{lem:beta}Let $L$ and~$\Omega$ be as in Definition~\ref{dfn:Poisson:psb}. Let $p^+_L$ be as in the bound~\eqref{eqn:interior:Meyers}, let $p^+_{L^*}$ be as in the bound~\eqref{eqn:interior:Meyers} with $L$ replaced by~$L^*$, let
\begin{equation*}\frac{1}{p^-_{L^*}}+\frac{1}{p^+_{L^*}}=1,
\end{equation*}
and let $\beta$ satisfy $p^-_{L^*}<\beta<p^+_L$.

Then the $L^{2}_{av,\beta,1/2}(\Omega)$-Poisson problem for $L$ is compatibly well posed in~$\Omega$.

More generally, suppose that $s\in\R$, $0<p\leq\infty$, and the $L^{p}_{av,2,s}(\Omega)$-Poisson problem is compatibly well posed. Then the $L^{p}_{av,\beta,s}(\Omega)$-Poisson problem is compatibly well posed, and the constant $C$ in the estimate~\eqref{poisson:estimate} may be bounded depending only on $\dmn$, $\bdmn$,~$c$, the $L^{p}_{av,c,2,s}(\Omega)$ solvability constant, and the constant in the estimate~\eqref{eqn:interior:Meyers}.
\end{lem}

\begin{proof} Suppose first that $\beta>2$.
Let $c<\frac{1}{2}$ and $\vec H\in L^p_{av,\beta,s}\cap L^2_{av,\beta, s}$. By \eqref{eqn:interior:Meyers}, for any $2\le\beta<p^+_L$, and any $x\in\Omega$,
\begin{equation*}
\biggl(\fint_{B(x,c\delta(x))}|\nabla u|^\beta \biggr)^{1/\beta}
\leq C\biggl(\fint_{B(x,2c\delta(x))}|\nabla u|^2\biggr)^{1/2}
+C\biggl(\fint_{B(x,2c\delta(x))}|\vec H|^\beta \biggr)^{1/\beta},
\end{equation*}
where $C$ depends on $\dmn$, $\beta$, and $L$ only. We can then integrate over $\Omega$,  getting
\begin{align*}
\|\nabla u\|_{L^p_{av,c,\beta,s}(\Omega)}&\le C(\|\nabla u\|_{L^p_{av,2c,2,s}(\Omega)}+\|\vec H\|_{L^p_{av,2c,\beta,s}(\Omega)})\\
&\le C ( \|\vec H\|_{L^p_{av,2c,2,s}(\Omega)}+\|\vec H\|_{L^p_{av,2c,\beta,s}(\Omega)})\\
&\le C  \|\vec H\|_{L^p_{av,2c,\beta,s}(\Omega)},
\end{align*}
where in the last line we used the Hölder inequality.   Lemma~\ref{lem:averaged:Whitney} gives $ \|\vec H\|_{L^p_{av,2c,\beta,s}(\Omega)}\approx \|\vec H\|_{L^p_{av,c,\beta,s}(\Omega)}$, ending the proof in the case $\beta>2$.
We may complete the proof using Theorem~\ref{first:step:duality}.
\end{proof}

\subsection{Interpolation}\label{sec:interpolation}

The duality results for the spaces $L^p_{av,\beta,s}(\Omega)$ of Section~\ref{sec:duality} allow us to pass from the $L^p_{av,\beta,s}(\Omega)$-Poisson problem for~$L$ to the $L^{p'}_{av,\beta',1-s}(\Omega)$-Poisson problem for~$L^*$.

We can also establish interpolation results for the spaces $L^p_{av,\beta,s}(\Omega)$, which will yield interpolation results for the $L^p_{av,\beta,s}(\Omega)$-Poisson problem. This is the subject of the present section.

We will use the real interpolation method of Lions and Peetre, following \cite{BerL76}.

We say that two quasi-normed vector spaces $A_0$, $A_1$ are compatible if there exist linear continuous embeddings of $A_0$ and $A_1$ into some Hausdorff topological vector space $X$. Then $A_0\cap A_1$, $A_0+A_1$ may be defined in the natural way as quasi-normed vector spaces, and if $A_0$ and $A_1$ are complete then so are $A_0\cap A_1$ and $A_0+A_1$.
If $0<r\le\infty$ and $0<\gamma<1$, then we let the real interpolation space $(A_0,A_1)_{\gamma,r}$ denote $\{a\in A_0+A_1\,:\,\|a\|_{(A_0,A_1)_{\gamma,r}}<\infty\}$, where
\begin{equation*}
\|a\|_{(A_0,A_1)_{\gamma,r}}:=\left(\int_0^\infty\,(t^{-\gamma}\inf\{\|a_0\|_{A_0}+t\|a_1\|_{A_1}\,:\,a_0+a_1=a\})^r\,dt/t\right)^{1/r}\,.
\end{equation*}
An important property of this interpolation space is as follows.
\begin{thm}[{\cite[Theorem 3.11.2]{BerL76}}]\label{thm:BerL76}
Let $A_0,$ $A_1$ and $B_0$, $B_1$ be two compatible pairs of quasi-normed vector spaces. If $T:A_0+A_1\to B_0+B_1$ is a linear bounded operator such that $T(A_0)\subset B_0$ and $T(A_1)\subset B_1$, then $T((A_0,A_1)_{\gamma,r})\subset (B_0,B_1)_{\gamma,r}$, with
\begin{equation}
\|T\|_{(A_0,A_1)_{\gamma,r}\to (B_0,B_1)_{\gamma,r}}\le\|T\|^{1-\gamma}_{A_0\to B_0}\|T\|^{\gamma}_{A_1\to B_1}.
\end{equation}
\end{thm}
In this section we will apply the previous theorem to the
operator $T$ given by
\begin{equation*}T\vec H=\nabla u_{\vec H}\end{equation*}
where $u_{\vec H}$ is the solution to the Poisson problem with data~$\vec H$.

\begin{thm}\label{interpolation}
Let $\Omega\subset\R^\dmn$ be open with $\bdmn$-Ahlfors regular boundary.
Let $p_0$, $p_1\in(0,\infty)$, $1\le\beta\le\infty$, $s_0$, $s_1\in\R$, and $c$, $\gamma\in (0,1)$. Then
\begin{equation}
(L^{p_0}_{av,c,\beta,s_0}(\Omega),L^{p_1}_{av,c,\beta,s_1}(\Omega))_{\gamma,p_\gamma}=L^{p_\gamma}_{av,c,\beta,s_{\gamma}}(\Omega)
\end{equation}
where
\begin{equation*}s_\gamma=(1-\gamma)s_0+\gamma s_1\quad\text{and}\quad \frac{1}{p_\gamma}=\frac{1-\gamma}{p_0}+\frac{\gamma}{p_1}.\end{equation*}
\end{thm}
The requirement that $\partial\Omega$ be Ahlfors regular is not needed for the interpolation arguments, but the parameter $\bdmn$ appears in the definition of $L^{p}_{av,c,\beta,s}(\Omega)$ and so we do need $\bdmn$ to exist in order for the given spaces to be meaningful with the definitions of this paper.

In order to prove Theorem~\ref{interpolation} we need two interpolation results for spaces of sequences with values in a Banach space (see \cite[Theorem~5.6.2]{BerL76}).  For any Banach space $E$, for any $0<p\le\infty$, and any $s\in\R$, we denote as $\ell_s^p(E)$ the space of all sequences $a=\{a_j\}_{j\in\Z}\subset E$ such that $\|a\|_{\ell_s^p(E)}<\infty$, where
\begin{equation*}
\|a\|_{\ell_s^p(E)}=\Bigl(\sum_{j\in\Z}2^{j(\bdmn+p-ps)}\|a_j\|^p_E\Bigr)^{1/p},
\end{equation*}
and we denote as $\ell^p(E)$ the well known space of $p$-summable sequences with values in $E$ and indexes in $\Z$.
If $p_0$, $p_1\in (0,\infty]$ then
\begin{equation} \label{preint}
(\ell^{p_0}(E),\ell^{p_1}(E))_{\gamma,p_\gamma}=\ell^{p_\gamma}(E),
\end{equation}
and if $p_0$, $p_1\in (0,\infty)$ and $E_0$, $E_1$ are two compatible Banach spaces, then
\begin{equation} \label{preint1}
(\ell_{s_0}^{p_0}(E_0),\ell_{s_1}^{p_1}(E_1))_{\gamma,p_\gamma}=\ell_{s_\gamma}^{p_\gamma}((E_0,E_1)_{\gamma,p_\gamma})
\end{equation}
where $p_\gamma$ and $s_\gamma$ are as in Theorem~\ref{interpolation}.

We finally need a change of variables lemma.
\begin{lem}[{\cite[Theorem 6.4.2]{BerL76}}]\label{changevariables}
Suppose that $(A_0,A_1)$ and $(B_0,B_1)$ are two compatible couples and that there are linear operators $\mathcal{I}:A_0+A_1\to B_0+B_1$ and $\mathcal{P}:B_0+B_1\to A_0+A_1$ such that $\mathcal{P}\circ\mathcal{I}$ is the identity operator on $A_0+A_1$, and such that $\mathcal{I}:A_0\to B_0$, $\mathcal{I}:A_1\to B_1$, $\mathcal{P}:B_0\to A_0$, and $\mathcal{P}:B_1\to A_1$ are all bounded operators.

If $0<\gamma<1$ and $0<r\le\infty$, then
\begin{equation}
(A_0,A_1)_{\gamma,r}=\mathcal{P}((B_0,B_1)_{\gamma,r}).
\end{equation}
\end{lem}
\begin{proof}[Proof of Theorem \ref{interpolation}]
Let us fix a grid of Whitney cubes $\mathcal{G}=\mathcal{G}_1$ in~$\Omega$. Let $A_k$ denote $L^{p_k}_{av,\beta,s_k}(\Omega)$.
We let
\begin{equation*}E_k=\ell^{p_k}(L^\beta(R)),\qquad
B_k=\ell_{s_k}^{p_k}(E_k)=\ell_{s_k}^{p_k}(\ell^{p_k}(L^\beta(R)))\end{equation*}
where $R$ denotes the unit cube in~$\R^\dmn$ with center zero. Because $p_0$, $p_1\in (0,\infty)$, by formula~\eqref{preint1} we have that
\begin{equation*}(B_0,B_1)_{\gamma,p_\gamma} =
(\ell_{s_0}^{p_0}(E_0), \ell_{s_1}^{p_1}(E_1))_{\gamma,p_\gamma} =\ell_{s_\gamma}^{p_\gamma}((E_0,E_1)_{\gamma,p_\gamma})
\end{equation*}
and by formula~\eqref{preint} we have that
\begin{equation*}(E_0,E_1)_{\gamma,p_\gamma}=
(\ell^{p_0}(L^\beta(R)),\ell^{p_0}(L^\beta(R)))_{\gamma,p_\gamma}
=\ell^{p_\gamma}(
L^\beta(R)).\end{equation*}
Thus
\begin{equation*}(B_0,B_1)_{\gamma,p_\gamma} = \ell_{s_\gamma}^{p_\gamma}(\ell^{p_\gamma}(L^\beta(R))).\end{equation*}
To complete the proof we thus need only find bounded mappings $\mathcal{I}$ and $\mathcal{P}$ as in Lemma~\ref{changevariables}.

We define $\mathcal{I}$ and $\mathcal{P}$ as follows.
For each $j\in\Z$, let ${\mathcal{A}}^j:=\{Q\in\mathcal{G}:\ell(Q)=2^j\}$.
Notice that $\{{\mathcal{A}}^j\}_{j\in\Z}$ is a partition of $\mathcal{G}$. There is an injection $\iota_j:\mathcal{A}^j\to\Z$. (If $\Omega$ is bounded then $\mathcal{A}^j$ is finite and so $\iota$ cannot be required to be surjective.) If $k\in\iota_j(\mathcal{A}^j)$, then let $x_{j,k}$ be the midpoint of the (unique) cube $Q\in \iota_j^{-1}(\{k\})$, and let $\ell_{j,k}$ be its sidelength.

If $F\in L^{p}_{av,\beta,s}(\Omega)$, then $\mathcal{I}F=\{ \{ \mathcal{I}_{j,k}F\}_{k\in\Z}\}_{j\in\Z}$, where
\begin{equation*}\mathcal{I}_{j,k}F (x) = \begin{cases} 0, & k\notin \iota_j(\mathcal{A}^j),
\\
F(x_{j,k}+\ell_{j,k}x), & k\in \iota_j(\mathcal{A}^j).\end{cases}\end{equation*}
Let $\mathcal{P}(\{\{f_{j,k}\}_{k\in \Z}\}_{j\in\Z})$ be given by
\begin{equation*}\mathcal{P}(\{\{f_{j,k}\}_{k\in \Z}\}_{j\in\Z})(x)
=
\sum_{j\in\Z}\sum_{Q\in\mathcal{A}^j} \1_{Q} f_{j,\iota_j(Q)} ((x-x_{j,\iota_j(Q)})/\ell_{j,\iota_j(Q)}).\end{equation*}
It is straightforward to establish that $\mathcal{P}\circ\mathcal{I}$ is the identity mapping and that $\mathcal{P}$ and $\mathcal{I}$ are linear. We need only establish boundedness $\mathcal{I}:L^{p}_{av,\beta,s}(\Omega)
\to \ell_s^p(\ell^p(L^\beta(R)))$ and $\mathcal{P}:\ell_s^p(\ell^p(L^\beta(R)))\to L^{p}_{av,\beta,s}(\Omega)$.

But by definition of $\ell_s^p$ and~$\ell^p$,
\begin{align*}
\|\mathcal{I}F\|_{\ell_s^p(\ell^p(L^\beta(R)))}
&= \biggl(\sum_{j\in\Z} 2^{j(\bdmn+p-ps)} \|\{\mathcal{I}_{j,k}F\}_{k\in\Z}\|_{\ell^p(L^\beta(R))}^p \biggr)^{1/p}
\\&=
\biggl(\sum_{j\in\Z} 2^{j(\bdmn+p-ps)} \sum_{k\in \Z} \|\mathcal{I}_{j,k}F\|_{L^\beta(R)}^p \biggr)^{1/p}
.\end{align*}
The term $\|\mathcal{I}_{j,k}F\|_{L^\beta(R)}$ is zero if $k\notin\iota_j(\mathcal{A}^j)$. If $k=\iota_j(Q)$ then
\begin{equation*}\|\mathcal{I}_{j,k}F\|_{L^\beta(R)}=\ell(Q)^{-\pdmn/\beta} \biggl(\int_Q |F|^\beta\biggr)^{1/\beta}\end{equation*}
and so
\begin{align*}
\|\mathcal{I}F\|_{\ell_s^p(\ell^p(L^\beta(R)))}
&= \biggl(\sum_{j\in\Z} 2^{j(\bdmn+p-ps)} \sum_{Q\in\mathcal{A}^j} \|F\|_{L^\beta(Q)}^p \ell(Q)^{-\pdmn/\beta} \biggr)^{1/p}
\end{align*}
which by the definition of $\mathcal{A}^j$ is equivalent to the $L^p_{av,\beta,s}(\Omega)$-norm of~$F$.

Conversely, if $F=\mathcal{P}(\{\{f_{j,k}\}_{k\in \Z}\}_{j\in\Z})$ then by Lemma~\ref{lem:averaged:Whitney} and the definition of $\mathcal{A}^j$
\begin{align*}
\|F\|_{L^p_{av,\beta,s}(\Omega)}
\approx
\biggl(\sum_{j\in\Z}
\sum_{Q\in\mathcal{A}^j}
\|F\|_{L^\beta(Q)}^p\ell(Q)^{\bdmn+p-ps-p\pdmn/\beta}\biggr)^{1/p}
.\end{align*}
If $Q\in\mathcal{A}^j$ then $\ell(Q)^{-\pdmn/\beta}\|F\|_{L^\beta(Q)} =\|f_{j,\iota_j(Q)}\|_{L^\beta(R)}$
and so
\begin{align*}
\|F\|_{L^p_{av,\beta,s}(\Omega)}
\approx
\biggl(\sum_{j\in\Z}
\sum_{Q\in\mathcal{A}^j}
\|f_{j,\iota_j(Q)}\|_{L^\beta(R)}^p\ell(Q)^{\bdmn+p-ps} \biggr)^{1/p}
.\end{align*}
Recalling the definition of $\mathcal{A}^j$ and that $\iota_j:\mathcal{A}^j\to\Z$ is an injection, we have that
\begin{align*}
\|F\|_{L^p_{av,\beta,s}(\Omega)}
\leq C
\biggl(\sum_{j\in\Z}2^{j(\bdmn+p-ps)}
\sum_{k\in\Z}
\|f_{j,k}\|_{L^\beta(R)}^p\biggr)^{1/p}
=C\|\{\{f_{j,k}\}_{k\in \Z}\}_{j\in\Z}\|_{\ell_s^p(\ell^p(L^\beta(R)))}
\end{align*}
as desired.
\end{proof}

We can use this interpolation result to interpolate compatible well posedness.

\begin{thm}\label{thm:interp:sol}
Let $p_0$, $p_1\in [1,\infty)$, $s_0$, $s_1\in (0,1)$, $1<\beta<\infty$, and assume that if $i=0$ or $i=1$ then the $L^{p_i}_{av,\beta,s_i}(\Omega)$-Poisson problem for $L$ is compatibly well posed.

Then, for any $0<\gamma<1$, the $L^{p_\gamma}_{av,\beta,s_\gamma}(\Omega)$-Poisson problem for $L$ is well posed.
\end{thm}
\begin{proof}
Let $T$ be given by
\begin{equation*}T\vec H=\nabla u_{\vec H}\end{equation*}
where $u_{\vec H}$ is the solution to the Poisson problem with data~$\vec H$. By well posedness $T$ is well defined on $L^{p_i}_{av,\beta,s_i}(\Omega)$ for $i=0$ and $i=1$. We have that $L^2_{av,2,1/2}(\Omega)$ is dense in $L^{p_0}_{av,\beta,s_0}(\Omega)\cap L^{p_1}_{av,\beta,s_1}(\Omega)$ and so by compatible well posedness we have that if $\vec H\in L^{p_0}_{av,\beta,s_0}(\Omega)\cap L^{p_1}_{av,\beta,s_1}(\Omega)$ then the solution to the Poisson problem in $\widetilde W^{1,p_0}_{av,\beta,s_0}(\Omega)$ must equal the solution in $\widetilde W^{1,p_1}_{av,\beta,s_1}(\Omega)$, and so $T$ is well defined on $L^{p_0}_{av,\beta,s_0}(\Omega)+ L^{p_1}_{av,\beta,s_1}(\Omega)$.

Then by Theorems \ref{thm:BerL76} and~\ref{interpolation}, we have that $T$ is bounded on $L^{p_\gamma}_{av,\beta,s_\gamma}(\Omega)$. In particular, if $\vec H\in L^{p_\gamma}_{av,\beta,s_\gamma}(\Omega)\cap L^{2}_{av,2,1/2}(\Omega)$ then $T\vec H=u_{\vec H}$ and so we have the estimate $\|u_{\vec H}\|_{\widetilde W^{1,p_\gamma}_{av,\beta,s_\gamma}(\Omega)} \leq C\|\vec H\|_{L^{p_\gamma}_{av,\beta,s_\gamma}(\Omega)}$. Thus the $L^{p_\gamma}_{av,\beta,s_\gamma}(\Omega)$-Poisson problem for $L$ is compatibly solvable.
Remark~\ref{rmk:compatible:solvable}, Theorem~\ref{first:step:duality}, and Lemma~\ref{lem:uniqueness} yield compatible well posedness.
\end{proof}

\section{Elliptic differential equations: the boundary}
\label{sec:elliptic:boundary}

In this section we will continue the discussion (begun in Section~\ref{sec:elliptic}) of known results for solutions to $Lu=-\Div \vec\Phi$ (or $Lu=0$). In this section we will focus on results near the boundary of the domain and the Green's function; these estimates are somewhat more delicate than the interior estimates of Section~\ref{sec:elliptic}. In particular, certain known results are not available in the case $\dmn=\bdmn+1=2$ in the generality of weak local John domains, and so we will prove them here; we will need the duality results of Section~\ref{sec:beta} in certain proofs.

We begin with the boundary Caccioppoli inequality.

\begin{lem}[The boundary Caccioppoli inequality] \label{lem:Caccioppoli:boundary}
Let $\Omega\subset\R^\dmn$ be a weak local John domain. Suppose that $\partial\Omega$ is $\bdmn$-Ahlfors regular for some $0<\bdmn<\dmn$ and that $L$ is as in Definition~\ref{dfn:L}.

Let $\xi\in\partial\Omega$ and $0<r<\infty$. Suppose that
$\int_{B(\xi,2r)\cap\Omega} |\nabla u|^2\delta^{1+\bdmn-\pdmn}<\infty$, that $\Tr u=0$ on $\Delta(\xi,2r)$, and that $Lu=-\Div (\delta^{1+\bdmn-\pdmn}\vec H)$ in $B(\xi,2r)\cap\Omega$. Then
\begin{equation*}\int_{B(\xi,r)\cap\Omega} |\nabla u|^2\delta^{1+\bdmn-\pdmn}\leq \frac{C}{r^2}
\int_{B(\xi,2r)\cap\Omega} |u|^2\delta^{1+\bdmn-\pdmn}
+ C \int_{B(\xi,2r)\cap\Omega} |\vec H|^2\delta^{1+\bdmn-\pdmn}
.\end{equation*}
\end{lem}

\begin{proof} This is well known in the theory, but we will provide a sketch of the proof to ensure that it is valid with the notion of zero boundary values used in the present paper.

By Lemma~\ref{lem:dense:smooth} we may approximate $u$ in $\dot W^{1,2}_{av,2,1/2}(B(\xi,(3/2)r)\cap\Omega;\Omega)$ by smooth functions that are zero in a neighborhood of~$\partial\Omega$. Let $\{\eta_k\}_{k=1}^\infty$ be an approximating sequence. By Lemma~\ref{lem:boundary:Poincare}, we have that $\eta_k\to u$ in $L^2_{av,2,3/2}(B(\xi,(3/2)r)\cap\Omega;\Omega)$.
If $\varphi$ is a smooth cutoff function supported in $B(\xi,(3/2)r)$, then by the weak definition of $L$ we have that
\begin{equation*}\int_{B(\xi,(3/2)r)\cap\Omega} \delta^{1+\bdmn-\pdmn} \nabla (\varphi^2\eta_k)\cdot\vec H=\int_{B(\xi,(3/2)r)\cap\Omega} \nabla (\varphi^2\eta_k)\cdot A\nabla u.\end{equation*}
As in the standard proof of the Caccioppoli inequality, we may complete the proof using the ellipticity condition~\eqref{eqn:elliptic}, the product rule, Young's inequality, and by taking the limit as $k\to\infty$.
\end{proof}

We now state some pointwise estimates on solutions to $Lu=0$ near the boundary.

\begin{lem}[Boundary Hölder continuity of solutions]\label{lem:boundary:DGN}
Let $\Omega\subset\R^\dmn$ be as in Theorem~\ref{thm:Poisson}. Let $L$ be an operator as in Definition~\ref{dfn:L}.

There exists a constant $\alpha>0$ depending only on the standard parameters with the following significance.

Let $\xi\in\partial\Omega$ and let $r\in(0,\diam\partial\Omega)$. Let $u$ be such that
\begin{align*}&\int_{B(\xi,2r)\cap\Omega} |\nabla u|^2 \delta^{-\pdmn+\bdmn+1}<\infty,
\\
Lu&=0 \text{ in }B(\xi,2r)\cap\Omega,
\\
\Tr u&=0 \text{ on }\Delta(\xi,2r).\end{align*}
Then for all $s$ with $0<s\leq r/2$ we have that
\begin{equation*}
\sup_{B(\xi,s)\cap\Omega} |u|\leq C\frac{s^\alpha}{r^\alpha} \frac{1}{r^{1+\bdmn} } \int_{B(\xi,r)\cap\Omega} |u| \delta^{1+\bdmn-\pdmn}
\approx
\frac{s^\alpha}{r^\alpha} \frac{1}{\int_{B(\xi,r)} \delta^{1+\bdmn-\pdmn}} \int_{B(\xi,r)\cap\Omega} |u| \delta^{1+\bdmn-\pdmn}
.\end{equation*}
\end{lem}
\begin{proof}
If $\bdmn<\dmnMinusOne$ then \cite[Theorem~3.1]{DavFM19B} or \cite[Lemma~8.16]{DavFM21} yields the estimate
\begin{equation}\label{eqn:boundary:DGN:proof}
\sup_{B(\xi,s)} |u|\leq C\frac{s^\alpha}{r^\alpha} \biggl(\frac{1}{\int_{B(\xi,r)}\delta^{1+\bdmn-\pdmn} } \int_{B(\xi,r)\cap\Omega} |u|^2 \delta^{1+\bdmn-\pdmn}\biggr)^{1/2}.\end{equation}
We may improve to the desired bound using \cite[Lemmas 2.3 and~8.13]{DavFM21}.

Similarly, if $\dmnMinusOne<\bdmn<\dmn$ then the result is \cite[Lemmas~11.20 and~11.32]{DavFM20p}.

If $\bdmn=\dmnMinusOne$ then the desired estimate is well known at least in domains with nice boundary. We state in particular that if $\dmn\geq 3$ then \cite[Theorems 4.4 and~4.14]{Mou23} yield the desired estimate with only the requirement that $\partial\Omega$ satisfy Wiener's capacity density condition; it is well known (see for example \cite[Section~3]{Zha18} and \cite[Lemma~3.27]{HofLMN17}) that Ahlfors regularity suffices to imply the capacity density condition.

We are left with the case $\bdmn=\dmnMinusOne$ and $\dmn=2$.  We will establish the lemma in this case in Section~\ref{sec:1:2}.
\end{proof}

We also have the following interior/boundary Moser estimate. (The boundary Moser estimate is in many sources a step in the proof of Lemma~\ref{lem:boundary:DGN}, but in our context it is simpler to present the following estimate as a corollary of Lemma~\ref{lem:boundary:DGN} rather than a separate citation.)
\begin{lem}[Boundary/interior Moser estimates on solutions]\label{lem:boundary:Moser}
Let $\Omega\subset\R^\dmn$ be as in Theorem~\ref{thm:Poisson}. Let $L$ be an operator as in Definition~\ref{dfn:L}.

Let $y\in \Omega$ and let $r>0$. Suppose that $u\in \dot W^{1,2}_{av,2,1/2}(\Omega\cap B(y,2r);\Omega)$, that $Lu=0$ in $B(y,2r)\cap \Omega$, and that $\Tr u=0$ on $B(y,2r)\cap\Omega$. Then
\begin{equation*}\sup_{B(y,r/12)}|u|\leq \frac{C}{r^{1+\bdmn}}\int_{B(y,r)\cap\Omega} |u|\,\delta^{1+\bdmn-\pdmn}\end{equation*}
for some constant $C$ depending only on the standard parameters.
\end{lem}

\begin{proof}
Let $\xi\in\partial\Omega$ satisfy $|y-\xi|=\delta(y)$. Suppose that $\delta(y)<r/3$. Let $R=(5/6)r$. Then $B(\xi,2R)=B(\xi,(5/3)r)\subset B(y,2r)$ and $B(y,r/12)\subset B(\xi,r/3+r/12)=B(\xi,R/2)$, and so the result follows from Lemma~\ref{lem:boundary:DGN}.

Suppose that $\delta(y)\geq r/3$. Then $B(y,r/3)\subset\Omega$, and so by the interior Moser estimate~\eqref{eqn:interior:Moser}
\begin{equation*}\sup_{B(y,r/12)} |u(y)|\leq C\fint_{B(y,r/6)} |u|\end{equation*}
and the result follows because $r/6\leq \delta(y)/2$ and so $\delta\approx\delta(y)$ in $B(y,r/6)$.
\end{proof}

We will need well posedness of the continuous Dirichlet problem.

\begin{lem}\label{lem:cts}
Let $\Omega\subset\R^\dmn$ be as in Theorem~\ref{thm:Poisson} and let $L$ be an operator as in Definition~\ref{dfn:L}.

If $f:\partial\Omega\to\R$ is Lipschitz and compactly supported, then there is a unique solution $u$ to the problem~\eqref{eqn:Dirichlet:cts}.
\end{lem}

\begin{proof}
The uniqueness of solutions follows from the maximum principle if $\Omega$ is bounded, or from Lemma~\ref{lem:boundary:DGN} if $\Omega$ is unbounded.

If $\bdmn<\dmnMinusOne$ then the existence of the solution $u$ is \cite[Lemmas 8.16 and~9.4]{DavFM21}.
If $\bdmn>\dmnMinusOne$ then the existence of the solution $u$ is \cite[Lemma 12.13]{DavFM20p}.

If $\bdmn+1=\dmn\geq 3$, then by \cite[Theorem~4.15]{Mou23} the solution to the Dirichlet problem in $\dot W^{1,2}(\Omega)=\dot W^{1,2}_{av,2,1/2}(\Omega)$ given by the Lax-Milgram lemma is continuous up to the boundary. In this case we can use Corollary~\ref{cor:averaged:is:Wp} and Lemmas \ref{lem:boundary:Poincare} and~\ref{lem:boundary:Moser} to show that $\lim_{x\to\infty} u(x)=0$, and the estimate $\sup_\Omega u=\sup_{\partial\Omega}f$ (or $\inf_\Omega u=\inf_{\partial\Omega}f$) follows from the maximum principle.

As in Lemma~\ref{lem:boundary:DGN}, we have not been able to find the result in the literature in the case $\bdmn+1=\dmn=2$ in the generality of unbounded weak local John domains, and so we will provide a proof in Section~\ref{sec:1:2}.
\end{proof}

\subsection{The Green's function}\label{sec:Green}

Let $\Omega$ be as in Theorem~\ref{thm:Poisson}, that is, let $\Omega\subseteq\R^\dmn$ be a weak local John domain with unbounded complement and with $\bdmn$-Ahlfors regular boundary, $0<\bdmn<\dmn$; if $\bdmn>\dmnMinusOne$ we also require the interior corkscrew and Harnack chain conditions. Let $L$ be given by Definition~\ref{dfn:L} for some real coefficients~$A$ that satisfy the ellipticity condition~\eqref{eqn:elliptic}.

The Green's function for $L$ in $\Omega$ is a function $G^L_\Omega$ defined on $\overline\Omega\times\overline\Omega\setminus\Delta$, where $\Delta$ denotes the diagonal $\Delta=\{(x,x):x\in\Omega\}$, with a number of useful properties.

We will rely on the construction in \cite{DavFM21} in the case $\bdmn<\dmnMinusOne$, in \cite{DavFM20p} in the case $\bdmn>\dmnMinusOne$, and the construction in \cite{HofK07} in the case $\bdmn+1=\dmn\geq 3$. (We remark that the Green's function in the codimension 1 case in our situation, that is, for real possibly nonsymmetric coefficients, was constructed first in \cite{GruW82} in bounded domains; it is relatively straightforward to generalize the results of \cite{GruW82} to unbounded domains but we have found it easier to work with \cite{HofK07}, which works explicitly in potentially unbounded domains.)

We will construct the fundamental solution in the case $\bdmn=1$, $\dmn=2$ in Section~\ref{sec:1:2} and confirm that the conditions \cref{eqn:green:W12,eqn:green:fundamental,eqn:green:solution,eqn:green:boundary,eqn:green:symmetric,eqn:green:poisson:nodiv,eqn:green:Poisson,eqn:green:near:size,eqn:green:upper:bound,eqn:green:positive} are all valid in this case.

The properties of the Green's function that we will require are as follows.

Fix some $y\in \Omega$. We let $G^L_y(x)=G^L_{\Omega,y}(x)=G^L_\Omega(x,y)$.
By \cite[formula~(10.7)]{DavFM21}, \cite[formula~(14.62)]{DavFM20p}, or \cite[formula~(4.43)]{HofK07},
\begin{equation}\label{eqn:green:local:upper:bound}
\nabla G^L_y \in L^1(B(y,\delta(y)/2)).
\end{equation}
By \cite[formula~(10.5)]{DavFM21}, \cite[formula~(14.66)]{DavFM20p}, or
\cite[formula~(4.40)]{HofK07},
\begin{equation}\label{eqn:green:W12}\int_{\Omega\setminus B(y,r)} |\nabla G^L_{y}|^2 \delta^{1+\bdmn-\pdmn}<\infty.\end{equation}
for all $r>0$.

By \cite[formula~(10.4)]{DavFM21}, \cite[formula~(14.63)]{DavFM20p}, or
\cite[formula~(4.39)]{HofK07},
if $\varphi\in C^\infty_0(\Omega)$, then
\begin{equation}
\label{eqn:green:fundamental}
\varphi(y)=\int_\Omega A\nabla G^L_{y} \cdot\nabla\varphi.
\end{equation}
By the weak definition of~$L$, this means that
\begin{equation}\label{eqn:green:solution}
L G^L_{y}=0\text{ in }\Omega\setminus\{y\}
\end{equation}
in the sense of Definition~\ref{dfn:L}. The bound~\eqref{eqn:interior:DGN} yields continuity of $G^L_y$ in $\Omega\setminus\{y\}$.

By \cite[Lemma~5.5 and formula~(10.2)]{DavFM21}, \cite[Lemma~9.18 and for\-mula (14.61)]{DavFM20p}, or \cite[formula~(4.40)]{HofK07},
if $\eta$ is smooth, and if for some $r>0$ we have that $\eta\equiv 1$ outside $B(y,2r)$ and $\eta\equiv 0$ in $B(y,r)$, then $(1-\eta)G^L_y$ lies in the closure in $\dot W^{1,2}_{av,2,1/2}(\Omega)$ of $C^\infty_0(\Omega)$. The results \cite[formula~(10.2)]{DavFM21} and \cite[for\-mula~(14.61)]{DavFM20p} in fact state that $\Tr ((1-\eta)G^L_y)=0$, where $\Tr$ is given by formula~\eqref{dfn:Tr}; this is still true in the setting of \cite{HofK07} by Lemmas \ref{lem:boundary:Poincare} and~\ref{lem:trace:zero} above.

Thus by Lemma~\ref{lem:boundary:DGN} and the bound~\eqref{eqn:green:W12},
\begin{equation}\label{eqn:green:boundary}
G^L_{y} \text{ is continuous on }\overline\Omega\setminus\{y\}\text { and } G^L_{y}=0 \text{ on } \partial\Omega.
\end{equation}

By \cite[formula~(10.6)]{DavFM21}, \cite[formula~(14.64)]{DavFM20p}, or
\cite[formula~(4.37)]{HofK07},
we have the bound
\begin{equation}\label{eqn:green:upper:bound}
G^L_y(x)\leq C|x-y|^{1-\pbdmn}
\quad\text{if }x, y\in\Omega \text{ and } \frac{1}{4}\delta(y)\leq |x-y|
.\end{equation}

By \cite[formula~(10.6)]{DavFM21},
\cite[formula~(14.65)]{DavFM20p},
or
\cite[formula~(4.29)]{HofK07},
if $|x-y|\leq \delta(y)/2$, then
\begin{equation}\label{eqn:green:near:size}
G^L_{y}(x)
\leq
\begin{cases}
C_\gamma\delta(y)^{\dmn-\bpdmn-1}\Bigl(\frac{\delta(y)}{|x-y|}\Bigr)^\gamma
, & \dmn=2,\\\noalign{\smallskip}
C\delta(y)^{\dmn-\bpdmn-1}|x-y|^{2-\pdmn}, & \dmn\geq 3
\end{cases}
\end{equation}
for every~$\gamma>0$. (In the special case $\bdmn+1=\dmn=2$, for the sake of simplicity we will prove this estimate only for the particular choice $\gamma=\alpha^*/2$, where $\alpha^*$ is the number in Lemma~\ref{lem:boundary:DGN} with $L$ replaced by~$L^*$. This is not the best possible result in the dimension 2 case (and indeed the result \cite[formula~(14.65)]{DavFM20p} is a logarithmic estimate instead), but it suffices for our purposes.)

By \cite[Lemma~10.7]{DavFM21}, \cite[Lemma~14.91]{DavFM20p}, or
\cite[formula~(4.37)]{HofK07},
if $f\in C^\infty_0(\Omega)$, and if we let
\begin{equation*}u(x)=\int_\Omega G^L_y(x)\,f(y)\,dy\end{equation*}
then $u\in \widetilde W^{1,2}_{av,2,1/2}(\Omega)$ and
\begin{equation}
\label{eqn:green:poisson:nodiv}
\int_\Omega A\nabla u\cdot \nabla \varphi=\int_\Omega f\varphi\end{equation}
for all $\varphi\in \widetilde W^{1,2}_{av,2,1/2}(\Omega)$.

By \cite[Lemma~10.6]{DavFM21} and \cite[Lemma 14.78]{DavFM20p},
if $\bdmn\neq\dmnMinusOne$, and if $L^*=-\Div A^*\nabla$ is the adjoint operator to $L=-\Div A\nabla$, then
\begin{equation}\label{eqn:green:symmetric}
G_y^L(x)=G_x^{L^*}(y).
\end{equation}
This is still true if $\bdmn=\dmnMinusOne$, $\dmn\geq 3$. To see this, first note that by \cite[formulas (4.52--4.53)]{HofK07}) $G_y^L(x)$ is continuous viewed as a function on $\Omega\times\Omega\setminus D$, where $D$ is the diagonal $D=\{(x,x):x\in\Omega\}$.
Let $f$, $g\in C^\infty_0(\Omega)$ with disjoint support and let
\begin{equation*}v(x)=\int_\Omega G_y^L(x)\,f(y)\,dy,\qquad w(y)=\int_\Omega G_x^{L^*}(y)\,g(x)\,dx.\end{equation*}
Then by \eqref{eqn:green:poisson:nodiv} we have that $v$ and $w$ are both in $\widetilde W^{1,2}_{av,2,1/2}(\Omega)$ and
\begin{equation*}\int_\Omega A\nabla v\cdot \nabla \varphi=\int_\Omega f\varphi,\qquad \int_\Omega A^*\nabla w\cdot \nabla \psi=\int_\Omega g\psi\end{equation*}
for all $\varphi$, $\psi\in \widetilde W^{1,2}_{av,2,1/2}(\Omega)$. In particular we may choose $\varphi=w$ and $\psi=v$ to see that
\begin{equation*}\int_\Omega fw=\int_\Omega A\nabla v\cdot \nabla w= \int_\Omega A^*\nabla w\cdot \nabla v=\int_\Omega gv\end{equation*}
and so
\begin{equation*}\int_\Omega f(y)\int_\Omega G_x^{L^*}(y)\,g(x)\,dx\,dy
=\int_\Omega g(x)\int_\Omega G_y^L(x)\,f(y)\,dy\,dx.\end{equation*}
By taking $f$ and $g$ to be suitable approximations of the identity and invoking the continuity of $G^L$ (and $G^{L^*}$) on $\Omega\times\Omega\setminus D$, we derive formula~\eqref{eqn:green:symmetric}.

Letting $f=-\Div(\delta^{\bdmn+1-\pdmn}\vec H)$ in formula~\eqref{eqn:green:poisson:nodiv} for some $\vec H$ such that the given quantity is in~$C^\infty_0(\Omega)$, (in particular, for $\vec H$ bounded and compactly supported,) and exploiting the symmetry property~\eqref{eqn:green:symmetric}, we see that if
\begin{equation*}
u(x)=\int_\Omega \nabla G^{L^*}_{x} \cdot \delta^{\bdmn+1-\pdmn} \vec H
\end{equation*}
then $u\in \widetilde W^{1,2}_{av,2,1/2}(\Omega)$ and
\begin{equation*}\int_\Omega A\nabla u\cdot \nabla \varphi=\int_\Omega \delta^{\bdmn+1-\pdmn}\vec H\cdot\nabla\varphi
.\end{equation*}
Thus, if $u_{\vec H}$ is the unique solution in $\widetilde W^{1,2}_{av,2,1/2}(\Omega)$ to the Poisson problem $Lu=-\Div(\delta^{\bdmn+1-\pdmn}\vec H)$ given by Definition~\ref{dfn:Lax-Milgram}, we must have that $u=u_{\vec H}$. By density, if $\vec H\in L^2_{av,2,1/2}(\Omega)$ and if $\vec H=0$ in $B(x,r)$ for some $r>0$, then
\begin{equation}
\label{eqn:green:Poisson}
u_{\vec H}(x)=\int_\Omega \nabla G^{L^*}_{x} \cdot \delta^{\bdmn+1-\pdmn} \vec H
.\end{equation}

Finally,
\begin{equation}\label{eqn:green:positive}
G^L_y(x)\geq 0\text{ for all $x$, $y\in\Omega$ with $x\neq y$.}
\end{equation}
This is noted explicitly in \cite{DavFM20p,DavFM21}. It is not noted in \cite{HofK07} (and in fact is not meaningful in that context, as \cite{HofK07} is written in the generality of elliptic systems and so $G^L_y$ is a potentially complex vector-valued function and not a real-valued function).

However, the estimate~\eqref{eqn:green:positive} is straighforward to establish. By formulas \eqref{eqn:green:poisson:nodiv} and~\eqref{eqn:green:symmetric}, if $f\in C^\infty_0(\Omega)$ is nonnegative then
\begin{equation*}u(x)=\int_\Omega G^{L^*}_x(y)\,f(y)\,dy\end{equation*} is an element of $\widetilde W^{1,2}_{av,c,2,1/2}(\Omega)$ and a solution to $Lu=0$ near~$\partial\Omega$. Thus by Lemmas \ref{lem:trace:zero} and~\ref{lem:boundary:DGN}, $u$ is continuous up to the boundary and zero on the boundary. By the bound~\eqref{eqn:green:upper:bound}, and because $\bdmn=\dmnMinusOne\geq 2$, we have that $u(x)\to 0$ as $|x|\to \infty$. Finally, because $f\geq 0$, by formula~\eqref{eqn:green:poisson:nodiv} we have that $u$ is a supersolution in all of~$\Omega$, and so by the maximum (or rather minimum) principle (see, for example, \cite[Theorem~8.1]{GilT01}) we have that $u(x)\geq 0$. Thus by continuity of $G^{L^*}_x$ we must have that $G^L_y(x)=G^{L^*}_x(y)\geq 0$ in~$\Omega$.

\subsection{The special case \texorpdfstring{$\bdmn=1$, $\dmn=2$}{d=1, n=2}}
\label{sec:1:2}

In this section we will prove Lemmas \ref{lem:boundary:DGN} and~\ref{lem:cts} and construct the Green's function in the special case $\bdmn=1$, $\dmn=2$, as we have not found these results in the literature in the generality of weak local John domains.

We begin with the following boundary form of Meyers's reverse Hölder inequality. This inequality (even in boundary form) is well known in the theory, but we provide a proof to ensure that it is valid in the generality of weak local John domains and with our desired form

\begin{lem}\label{lem:Meyers:boundary}%
Let $\Omega\subset\R^\dmn$, $\dmn\geq 2$, be a weak local John domain. Suppose that $\partial\Omega$ is $\pdmnMinusOne$-Ahlfors regular and that $L$ is as in Definition~\ref{dfn:L} with $\bdmn=\dmnMinusOne$.

Then there is a $\beta>2$, depending only on the standard parameters, with the following significance.

Let $\xi\in\partial\Omega$ and $0<r<\diam \partial\Omega $. Suppose that
$\int_{B(\xi,2r)\cap\Omega} |\nabla u|^2<\infty$, that $\Tr u=0$ on $\Delta(\xi,2r)$, and that $Lu=-\Div \vec H$ in $B(\xi,2r)\cap\Omega$.
If $\vec H\in L^\beta(B(\xi,2r)\cap\Omega)$, then
$u\in \dot W^{1,\beta}(B(\xi,r)\cap\Omega)$ and
\begin{equation*}
\biggl(\frac{1}{r^\dmn}\int_{B(\xi,r)\cap\Omega} |\nabla u|^\beta\biggr)^{1/\beta}
\leq
C\biggl(
\frac{1}{r^\dmn}\int_{B(\xi,2r)\cap\Omega} |\nabla u|^2\biggr)^{1/2}
+C\biggl(\frac{1}{r^\dmn}\int_{B(\xi,r)\cap\Omega} |\vec H|^\beta\biggr)^{1/\beta}
.\end{equation*}
\end{lem}
\begin{proof}
By the Caccioppoli inequality, if $x\in \Delta(\xi,r)$ and $\varrho$ is small enough, then
\begin{equation*}\biggl(\fint_{B(x,\varrho)\cap\Omega}|\nabla  u|^2\biggr)^{1/2}\leq \frac{C}{\varrho}
\biggl(\fint_{B(x,2\varrho)\cap\Omega} |u|^2\biggr)^{1/2}+\biggl(\fint_{B(x,2\varrho)\cap\Omega}|\vec H|^2\biggr)^{1/2}.\end{equation*}

We extend $u$ by zero to $B(\xi,2r)\setminus\Omega$. By Lemma~\ref{lem:dense:smooth}, we have that the extension is a Sobolev function in $B(\xi,(3/2)r)$. The standard Sobolev embedding theorem in balls (see for example \cite[Section 5.6.3]{Eva98}) implies that
\begin{equation*} \frac{1}{\varrho}\biggl(\fint_{B(x,2\varrho)} |u|^2\biggr)^{1/2}
\leq
C\biggl(\fint_{B(x,2\varrho)} |\nabla u|^p\biggr)^{1/p}
+
\frac{C}{\varrho}\biggl(\fint_{B(x,2\varrho)} |u|^p\biggr)^{1/p}\end{equation*}
for some $p<2$; by Hölder's inequality we may require that $p>1$.

If $1+ps>0$ then
\begin{equation*}
\fint_{B(x,2\varrho)} |u|^p
\leq (2\varrho)^{1+ps}\fint_{B(x,2\varrho)} |u|^p \delta^{-1-ps}
\end{equation*}
By Lemma~\ref{lem:boundary:Poincare} and Corollary~\ref{cor:averaged:is:Wp}, we have that if $s>0$ then
\begin{equation*}
\fint_{B(x,2\varrho)} |u|^p
\leq C\varrho^{1+ps}\fint_{B(x,3C_1\varrho)} |\nabla u|^p \delta^{p-1-ps}
.\end{equation*}
We may take $s=1-1/p>0$ (this is the reason for our requirement that $p>1$). Thus
\begin{equation*}
\biggl(\fint_{B(\xi,\varrho)\cap\Omega}|\nabla  u|^2\biggr)^{1/2}
\leq
C\biggl(\fint_{B(\xi,3C_1\varrho)} |\nabla u|^p\biggr)^{1/p}
+\biggl(\fint_{B(\xi,2\varrho)\cap\Omega}|\vec H|^2\biggr)^{1/2}.\end{equation*}

Thus $\nabla u$ satisfies a reverse Hölder inequality in each $B(x,3C_1\varrho)$ for $\varrho$ small enough. A similar and simpler argument shows that $\nabla u$ satisfies a reverse Hölder inequality in balls contained in $B(\xi,2r)\cap\Omega$. Thus we may apply the self improvement properties of reverse Hölder inequalities (see, for example, \cite[Chapter~V]{Gia83}) and a covering argument to see that
\begin{equation*}\biggl(\frac{1}{r^\dmn}\fint_{B(\xi,r)}|\nabla  u|^\beta\biggr)^{1/\beta}\leq C\biggl(\frac{1}{r^\dmn}\int_{B(\xi,2r)}|\nabla  u|^2\biggr)^{1/2}
+C\biggl(\frac{1}{r^\dmn}\fint_{B(\xi,2r)}|\vec H|^\beta\biggr)^{1/\beta}
\end{equation*}
for some $\beta>2$.
\end{proof}

As a corollary we can establish well posedness of the Poisson problem in the case $p=\beta=2$, $s=1-1/\beta$; we will use this in the proof of Lemma~\ref{lem:cts} and the construction of the Green's function later in this section.

\begin{cor}\label{cor:Meyers} Let $\Omega\subset\R^\dmn$ be a weak local John domain with $\pdmnMinusOne$-Ahlfors regular boundary and let $L$ be as in Definition~\ref{dfn:L}. The $L^\beta(\Omega)=L^{\beta}_{av,\beta,1-1/\beta}(\Omega)$-Poisson problem for $L$ is compatibly well posed in~$\Omega$ for all $\beta$ sufficiently close to~$2$.
\end{cor}
\begin{proof}
By Remark~\ref{rmk:compatible:solvable} and Lemma~\ref{lem:uniqueness} it suffices to establish compatible solvability. Let $\vec H\in L^2(\Omega)\cap L^\beta(\Omega)$; by Corollary~\ref{cor:averaged:is:Wp}, this means $\vec H\in L^2_{av,2,1/2}(\Omega)\cap L^\beta(\Omega)$. Let $u=u_{\vec H}$ be the solution to the Poisson problem given by Definition~\ref{dfn:Lax-Milgram}.
By the boundary Meyers estimate (Lemma~\ref{lem:Meyers:boundary}) we have that for some ${\widetilde\beta}>2$,
\begin{equation*}\biggl(\int_{\Omega\cap B(\xi,r)} |\nabla u|^{\widetilde\beta}\biggr)^{1/{\widetilde\beta}}
\leq C r^{\pdmn/{\widetilde\beta}-\pdmn/2}\biggl(\int_{\Omega\cap B(\xi,2r)} |\nabla u|^2\biggr)^{1/2}
+ C \biggl(\int_{\Omega\cap B(\xi,2r)} |\vec H|^{\widetilde\beta}\biggr)^{1/{\widetilde\beta}}.\end{equation*}
If $\diam(\partial\Omega)=\infty$ we may take the limit as $r\to \infty$. Otherwise, we may apply the Caccioppoli inequality (which does not require $r<\diam(\Omega)$) to remove the term involving $|\nabla u|^2$. In either case, we have that
\begin{equation*}\|u\|_{\dot W^{1,{\widetilde\beta}}_{av,c,{\widetilde\beta},1-1/{\widetilde\beta}}(\Omega)}
=\|\nabla u\|_{L^{\widetilde\beta}(\Omega)}\leq C\|\vec H\|_{L^{\widetilde\beta}(\Omega)}=C\|\vec H\|_{L^{\widetilde\beta}_{av,c,{\widetilde\beta},1-1/{\widetilde\beta}}(\Omega)}\end{equation*}
which yields compatible solvability of the $L^{\widetilde\beta}(\Omega)$-Poisson problem.

We now turn to the case $2<\beta<\widetilde\beta$.
A careful reading of the proof of Lemma~\ref{lem:Meyers:boundary} (including of the works cited therein) will yield that Lemma~\ref{lem:Meyers:boundary} is valid for all $\beta$ with $2<\beta<\widetilde\beta$, and so we have compatible solvability of the $L^{\beta}(\Omega)$-Poisson problem for such~$\beta$. Alternatively, the reader may prefer to recall that if $\bdmn=\dmnMinusOne$, then $L^\beta(\Omega)=L^\beta_{av,\beta,1-1/\beta}$, and so we can apply Theorem~\ref{thm:interp:sol} to derive compatible solvability of the $L^{\beta}(\Omega)$-Poisson problem from that of the $L^2$ and $L^{\widetilde\beta}(\Omega)$-Poisson problem.

Finally, if $\widetilde\beta^*>2$ is as in Lemma~\ref{lem:Meyers:boundary} with $L$ replaced by~$L^*$, and if a prime denotes the Hölder conjugate, then $(\widetilde \beta^*)'<2$. If $(\widetilde \beta^*)'<\beta<2$, then by the above argument the $L^{\beta'}(\Omega)$-Poisson problem for $L^*$ is compatibly solvable.
By Theorem~\ref{dualspace}, we have compatible solvability for the $L^{\beta}(\Omega)$-Poisson problem for~$L$, as desired.
\end{proof}

\begin{proof}[Proof of Lemma~\ref{lem:boundary:DGN}]
The estimate
\begin{equation*}
\sup_{B(\xi,s)} |u|\leq Cs^\alpha
\biggl( \int_{B(\xi,2r/3)\cap\Omega} |\nabla u|^\beta \biggr)^{1/\beta},\end{equation*}
for $\alpha=1-\dmn/\beta$,
follows from Morrey's inequality (see for example \cite[Section 5.6.3]{Eva10}) applied to the extension of~$u$ by zero considered in the proof of Lemma~\ref{lem:Meyers:boundary}, which by  Lemma~\ref{lem:Meyers:boundary} satisfies $\nabla u\in L^\beta(B(\xi,r))$ for some $\beta>2=\dmn$.

By Lemma~\ref{lem:boundary:DGN} and the boundary Caccioppoli inequality, we may improve to
\begin{equation*}
\sup_{B(\xi,s)} |u|\leq C\frac{s^\alpha}{r^\alpha}
\biggl( \fint_{B(\xi,3r/4)\cap\Omega} |u|^2 \biggr)^{1/2},\end{equation*}
which is the $\bdmn=1$, $\dmn=2$ case of the estimate~\eqref{eqn:boundary:DGN:proof} in the proof of Lemma~\ref{lem:boundary:DGN}.

This yields a reverse Hölder type estimate on $u$ in balls centered on the boundary. The corresponding estimate in interior balls is the bound~\eqref{eqn:interior:DGN}.
We may complete the proof (that is, replace the $L^2$ estimate on~$u$ by a $L^1$ estimate) using standard arguments; see, for example, the end of the proof of \cite[Chapter~IV, Theorem~1.1]{HanL11}, the proof of \cite[Lemma~8.13]{DavFM21}, or \cite[Lemma~33]{Bar16} (inspired by \cite[Section 9, Lemma 2]{FefS72}) applied to the aforementioned extension of $u$ by zero.
\end{proof}

\begin{proof}[Proof of Lemma~\ref{lem:cts}]
Let $\Omega\subset\R^2$ be a weak local John domain with $1$-Ahlfors regular boundary and unbounded complement. Let $f:\partial\Omega\to\R$ be a compactly supported Lipschitz function. Our goal is to establish unique solvability of the continuous Dirichlet problem~\eqref{eqn:Dirichlet:cts}.

By the Tietze extension theorem, there is a compactly supported Lipschitz function $F:\R^2\to\R$ with $F=f$ on~$\partial\Omega$.

Let $\beta$ satisfy $2<\beta<\infty$ and the conditions of Corollary~\ref{cor:Meyers}, so that the $L^\beta(\Omega)$-Poisson problem for $L$ is compatibly well posed. Let $v$ be the solution to the Poisson problem $Lv=-\Div (A\nabla F)$ in~$\Omega$; by compatible well posedness
\begin{equation*}\|v\|_{\widetilde W^{1,2}(\Omega)} \leq C\|\nabla F\|_{L^2(\R^2)},
\qquad
\|v\|_{\widetilde W^{1,\beta}(\Omega)} \leq C\|\nabla F\|_{L^{\beta}(\R^2)}.\end{equation*}
We may extend $v$ by zero to a function in $\dot W^{1,\beta}(\R^2)$. As in the proof of Lemma~\ref{lem:boundary:DGN}, by Morrey's inequality we have that $v$ is Hölder continuous. Then $u=F-v$ is continuous on $\overline\Omega$ (in fact on~$\R^2$), lies in $\dot W^{1,2}(\Omega)$, and satisfies $Lu=0$ in $\Omega$ and $u=f$ on~$\partial\Omega$. We need only establish the upper and lower bounds on~$u$ to show that $u$ is a solution to the continuous Dirichlet problem~\eqref{eqn:Dirichlet:cts} and establish uniqueness.

If $\Omega$ is bounded then we are done by the maximum principle. Otherwise, let $u_1$ be any function with the properties we have established, that is, $Lu_1=0$ in~$\Omega$, $u_1\in \dot W^{1,2}(\Omega)$, $u_1$ is continuous on~$\overline\Omega$, and $u_1=f$ on~$\partial\Omega$.
Then $v_1=u_1-F$ lies in $\widetilde W^{1,2}(\Omega)$ by Lemma~\ref{lem:boundary:Poincare}.

Assume without loss of generality that $0\in \partial\Omega$ and let $R>0$ be such that $\supp F\subset B(0,R)$. By Lemma~\ref{lem:boundary:Moser}, if $x\in \Omega\setminus B(0,2R)$, then
\begin{equation*}|v_1(x)|=|u_1(x)|\leq \frac{C}{|x|^2}\int_{B(x,|x|/2)\cap\Omega} |v_1|\end{equation*}
and by Hölder's inequality
\begin{equation*}|v_1(x)|\leq  \frac{C}{|x|^2}
\biggl(\int_{B(x,|x|/2)\cap\Omega} |v_1|^2\delta^{-2}\biggr)^{1/2}
\biggl(\int_{B(x,|x|/2)\cap\Omega} \delta^{2}\biggr)^{1/2}.\end{equation*}
Because $0\in \partial\Omega$ we have that $\delta(y)\leq |y|$ for all~$y$, and so
\begin{equation*}|v_1(x)|\leq  C
\biggl(\int_{\Omega\setminus B(0,|x|/2)} |v|^2\delta^{-2}\biggr)^{1/2}.\end{equation*}
By Lemma~\ref{lem:boundary:Poincare},
\begin{equation*}\biggl(\int_{\Omega} |v_1|^2\delta^{-2}\biggr)^{1/2}
\leq C\biggl(\int_{\Omega} |\nabla v_1|^2\biggr)^{1/2}<\infty\end{equation*}
and so
\begin{equation*}\lim_{|x|\to\infty} \biggl(\int_{\Omega\setminus B(0,|x|/2)} |v_1|^2\delta^{-2}\biggr)^{1/2}=0.\end{equation*}
Thus $v_1(x)\to 0$ (and therefore $u_1(x)\to 0$) as $|x|\to\infty$. We can now establish the upper and lower bounds on~$u$ and show that $u_1=u$ in $\Omega$ using the maximum principle.
\end{proof}

\subsubsection{Construction of the Green's function}\label{sec:green:2:1}

In a weak local John domain with unbounded complement, and in the special case $\bdmn+1=\dmn=2$, the Green's function is in fact straightforward to construct.

Let $\Omega\subset\R^2$ be a weak local John domain with $1$-Ahlfors regular boundary and let $L$ be as in Definition~\ref{dfn:L}. For the rest of Section~\ref{sec:green:2:1}, we will let $\beta$ satisfy $2<\beta\leq 4/(2-\alpha^*)$ and the conditions of Corollary~\ref{cor:Meyers} with $L$ replaced by $L^*$; that is, we will let $\beta>2$ be such that the $L^\beta(\Omega)$-Poisson problem for $L^*$ is compatibly well posed. (The requirement $\beta\leq 4/(2-\alpha^*)$ is imposed to aid in the proof of the estimate~\eqref{eqn:green:near:size}, which we need to hold for $\gamma=\alpha^*/2$.)

If $\vec H\in L^{\beta}(\Omega)$, let $u^*_{\vec H}$ be the solution to the $L^{\beta}(\Omega)$-Poisson problem for~$L^*$ with data~$\vec H$. Then
\begin{equation*}\|u^*_{\vec H}\|_{\widetilde W^{1,\beta}(\Omega)}\leq C\|\vec H\|_{L^{\beta}(\Omega)}
.\end{equation*}
If $y\in\Omega$ and $0<r\leq\delta(y)/2$ then
\begin{equation*}|u^*_{\vec H}(y)|
\leq \sup_{z\in B(y,r)} |u^*_{\vec H}(y)-u^*_{\vec H}(z)|
+ \fint_{B(y,r)} |u^*_{\vec H}|.\end{equation*}
Recall $\beta>2=\dmn$. As in the proof of Lemma~\ref{lem:boundary:DGN}, by Morrey's inequality we have that
\begin{equation}\label{eqn:2d:cts}
\sup_{z\in B(y,r)} |u^*_{\vec H}(y)-u^*_{\vec H}(z)|
\leq Cr^{1-\pdmn/\beta} \|u^*_{\vec H}\|_{\dot W^{1,\beta}(\R^\dmn)}
\leq Cr^{1-\pdmn/\beta} \|\vec H\|_{L^{\beta}(\Omega)}
\end{equation}
where we have extended $u^*_{\vec H}$ by zero in the case where $\R^\dmn\setminus\Omega$ has a nonempty interior.
By Hölder's inequality,
\begin{equation*}\fint_{B(y,r)} |u^*_{\vec H}|
\leq \delta(y)r^{-\pdmn/\beta}\biggl(\int_{B(y,r)} |u^*_{\vec H}|^{\beta}\delta^{-\beta}\biggr)^{1/\beta}\end{equation*}
and by Lemma~\ref{lem:boundary:Poincare} and Corollary~\ref{cor:averaged:is:Wp}
\begin{equation*}\fint_{B(y,r)} |u^*_{\vec H}|
\leq C\delta(y)r^{-\pdmn/\beta}\biggl(\int_{\Omega} |\nabla u^*_{\vec H}|^{\beta}\biggr)^{1/\beta}
\leq C\delta(y)r^{-\pdmn/\beta} \|\vec H\|_{L^{\beta}(\Omega)}.
\end{equation*}
Taking $r=\delta(y)/2$ yields
\begin{equation}\label{eqn:2d:pointwise}
|u^*_{\vec H}(y)|\leq C\delta(y)^{1-\pdmn/\beta} \|\vec H\|_{L^{\beta}(\Omega)}.\end{equation}
The operator $\vec H\mapsto u^*_{\vec H}(y)$ is clearly linear, and by the above estimate is bounded on $L^\beta(\Omega)$, and so by the Riesz representation theorem there is a unique $\vec g_y^L\in L^{\beta'}(\Omega)$ such that
\begin{equation}\label{eqn:2d:green:Poisson}
\|\vec g_y^L\|_{L^{\beta'}(\Omega)} \leq C\delta(y)^{1-\pdmn/\beta}
\quad\text{and}\quad
u^*_{\vec H}(y)=\int_\Omega \vec g_y^L\cdot\vec H
\text{ for all $\vec H\in L^{\beta}(\Omega)$.}
\end{equation}
In light of formula~\eqref{eqn:green:Poisson}, we expect $\vec g_y^L$ to be the gradient of the Green's function $G_y^L$.

We now modify this construction to produce the function~$G_y^L$.
If $\varphi\in \widetilde W^{1,{\beta'}}(\Omega)=\widetilde W^{1,{\beta'}}_{av,{\beta'},1-1/{\beta'}}(\Omega)$, then by Lemma~\ref{lem:boundary:Poincare} $\varphi \in L^{\beta'}_{av,{\beta'},2-1/{\beta'}}(\Omega)$. If $f$ lies in the dual space $L^{\beta}_{av,\beta,-1/\beta}(\Omega)$ then the linear operator
$T(\varphi)=\int_\Omega f\varphi$
satisfies
\begin{equation*}
|T(\varphi)|=\biggl|\int_\Omega f\varphi\biggr|\leq \biggl(\int_\Omega f^{\beta}\delta^{\beta}\biggr)^{1/\beta}
\biggl(\int_\Omega \varphi^{{\beta'}}\delta^{-{\beta'}}\biggr)^{1/{\beta'}}
\leq C\|f\|_{L^{\beta}_{av,\beta,-1/\beta}(\Omega)}
\|\varphi\|_{\widetilde W^{1,{\beta'}}(\Omega)}.\end{equation*}
Thus $T$ is a bounded linear operator on $\widetilde W^{1,{\beta'}}(\Omega)$, and so by the Hahn-Banach theorem and the Riesz representation theorem there is at least one $\vec H_f\in L^{\beta}(\Omega)$ with
\begin{equation}
\label{eqn:2d:step}\|\vec H_f\|_{L^{\beta}(\Omega)}
\leq C\|f\|_{L^{\beta}_{av,\beta,-1/\beta}(\Omega)}
\end{equation}
and
\begin{equation*}
\int_\Omega f\varphi=T(\varphi)=\int_\Omega \vec H_f\cdot\nabla\varphi
\text{ for all $\varphi\in \widetilde W^{1,{\beta'}}(\Omega)$.}\end{equation*}
The mapping $f\mapsto \vec H_f$ is not well-defined; there may be many possible values of~$\vec H_f$. However, the mapping $f\mapsto u^*_{\vec H_f}$ is well defined. To see this, observe that
\begin{equation}\label{eqn:2d:nodiv:1}\int_\Omega f\varphi=\int_\Omega \vec H_f\cdot\nabla\varphi
= \int_\Omega \nabla \varphi \cdot A^*\nabla u_{\vec H_f}^*\text{ for all $\varphi\in \widetilde W^{1,{\beta'}}(\Omega)$}.\end{equation}
If $v$ is another function in $\widetilde W^{1,\beta}(\Omega)$ with $\int_\Omega \nabla\varphi\cdot A^*\nabla v=\int_\Omega f\varphi$ for all $\varphi\in \widetilde W^{1,{\beta'}}(\Omega)$, then $\int_\Omega \nabla\varphi\cdot A^*\nabla v=\int_\Omega \nabla\varphi\cdot \vec H_f$ for all $\varphi\in \widetilde W^{1,{\beta'}}(\Omega)$, and so $v$ is a solution to the $L^{\beta}$-Poisson problem for $L^*$ with data $\vec H_f$. By the uniqueness property for the $L^\beta$-Poisson problem, we must have that $v=u^*_{\vec H_f}$.

If $y\in\Omega$ then the operator $f\mapsto u^*_{\vec H_f}(y)$ is clearly linear, and by the bounds \eqref{eqn:2d:pointwise} and~\eqref{eqn:2d:step} satisfies
\begin{equation*}
|u^*_{\vec H_f}(y)|
\leq C\delta(y)^{1-\pdmn/\beta}\|\vec H_f\|_{L^\beta(\Omega)}
\leq C\delta(y)^{1-\pdmn/\beta} \|f\|_{L^\beta_{av,\beta,-1/\beta}(\Omega)}
.\end{equation*}
Thus, again by the Riesz representation theorem, there is a unique (that is, well-defined) function $G_y^L$ such that
\begin{equation}
\label{eqn:green:2d:directsize}
\|G_y^L\|_{L^{\beta'}_{av,\beta',2-1/\beta'}(\Omega)}
\leq C\delta(y)^{1-\pdmn/\beta},
\end{equation}
and
\begin{equation}\label{eqn:2d:nodiv:2}
u^*_{\vec H_f}(y)=\int_\Omega G_y^L\,f
\quad\text{for all }f\in L^{\beta}_{av,\beta,-1/\beta}(\Omega).
\end{equation}

The function $G_y^L$ is the Green's function. Observe that if $\vec H$ is smooth and compactly supported in~$\Omega$, then
\begin{equation*}\int_{\Omega} G_y^L (-\Div \vec H)=u_{\vec H}(y)=\int_\Omega \vec g_y^L\cdot \vec H\end{equation*}
and so $\vec g_y^L$ is the weak derivative of $G_y^L$. Thus $G_y^L\in L^{\beta'}_{av,{\beta'},2-1/{\beta'}}(\Omega)\cap \dot W^{1,\beta'}(\Omega)=\widetilde W^{1,\beta'}(\Omega)$.

We now must establish the conditions \cref{eqn:green:W12,eqn:green:local:upper:bound,eqn:green:solution,eqn:green:fundamental,eqn:green:boundary,eqn:green:symmetric,eqn:green:poisson:nodiv,eqn:green:Poisson,eqn:green:near:size,eqn:green:upper:bound,eqn:green:positive}. We remark that the condition~\eqref{eqn:green:local:upper:bound} follows immediately from the fact that $G_y^L\in\widetilde W^{1,\beta'}(\Omega)$.

\subsubsection{The bound~\texorpdfstring{\eqref{eqn:green:W12}}{(\ref*{eqn:green:W12})}}

To establish the bound~\eqref{eqn:green:W12} we must improve from a $\dot W^{1,\beta'}$ estimate to a $\dot W^{1,2}$ estimate.
Because the $L^{\beta}(\Omega)$-Poisson problem for $L^*$ is compatibly well posed, we have that if $y\in\Omega$, $0<r<\delta(y)/2$, and $\vec H\in L^2(\Omega\setminus B(y,r))\cap L^{\beta}(\Omega\setminus B(y,r))$, then $u^*_{\vec H}\in \widetilde W^{1,2}(\Omega) \cap \widetilde W^{1,\beta}(\Omega)$. Then
\begin{equation*}u^*_{\vec H}(y)=\int_{\Omega\setminus B(y,r)} \nabla G_y^L\cdot \vec H\end{equation*}
and by the interior Moser estimate~\eqref{eqn:interior:Moser} and Hölder's inequality
\begin{equation*}|u^*_{\vec H}(y)|\leq C\fint_{B(y,r)} |u^*_{\vec H}|
\leq C\frac{\delta(y)}{r}\biggl(\int_{B(y,r)} |u^*_{\vec H}|^2\delta^{-2}\biggr)^{1/2}
.\end{equation*}
By Lemma~\ref{lem:boundary:Poincare}
\begin{equation*}|u^*_{\vec H}(y)|
\leq C\frac{\delta(y)}{r}\biggl(\int_{B(y,r)} |\nabla u^*_{\vec H}|^2\biggr)^{1/2}
\leq C\frac{\delta(y)}{r}\|\vec H\|_{L^2(\Omega\setminus B(y,r))}.
\end{equation*}
Thus
\begin{equation*}\biggl|\int_{\Omega\setminus B(y,r)} \nabla G_y^L\cdot \vec H\biggr|\leq C\frac{\delta(y)}{r}\|\vec H\|_{L^2(\Omega\setminus B(y,r))}\end{equation*}
for all $\vec H\in L^2(\Omega\setminus B(y,r))$,
and so $\nabla G_y^L\in L^2(\Omega\setminus B(y,r))$. This establishes the bound~\eqref{eqn:green:W12}.

\subsubsection{Formulas~\texorpdfstring{\cref{eqn:green:fundamental,eqn:green:solution,eqn:green:boundary}}{(\ref*{eqn:green:fundamental}--\ref*{eqn:green:boundary})}}
If $\eta$ is smooth and compactly supported, then $\eta\in \widetilde W^{1,\beta}(\Omega)$. Taking $u_1=\eta$ and $\vec H_1=A^*\nabla \eta$, we have that $u_1\in \widetilde W^{1,\beta}(\Omega)$, $\vec H_1\in L^{\beta}(\Omega)$ and
\begin{equation*}\int_\Omega\nabla \varphi\cdot A^*\nabla u_1=\int_\Omega\nabla\varphi\cdot \vec H_1\end{equation*}
for all $\varphi\in C^\infty_0(\Omega)$, and so $u_1$ must be the solution to the $L^{\beta}(\Omega)$-Poisson problem for $L^*$ with data~$\vec H_1$. Thus
\begin{equation*}\eta(y)=u_1(y)=u^*_{\vec H_1}(y)=\int_\Omega \vec g_y^L\cdot \vec H_1 = \int_\Omega\nabla G_y^L\cdot A^*\nabla\eta\end{equation*}
and so formula~\eqref{eqn:green:fundamental} is valid. As in Section~\ref{sec:Green}, formula~\eqref{eqn:green:fundamental} immediately yields formula~\eqref{eqn:green:solution}, that is, that $LG_y^L=0$ in $\Omega\setminus\{y\}$. We have that $G_y^L\in \widetilde W^{1,{\beta'}}(\Omega)$ and so by Lemma~\ref{lem:trace:zero} $\Tr G_y^L=0$ almost everywhere on~$\partial\Omega$. Thus by  Lemma~\ref{lem:boundary:DGN} and the bound~\eqref{eqn:green:W12} we have that $G_y^L$ is continuous at the boundary (and zero on the boundary), while the interior Hölder continuity \eqref{eqn:interior:DGN} yields that $G_y^L$ is continuous everywhere in~$\Omega\setminus\{y\}$. This is formula~\eqref{eqn:green:boundary}.

\subsubsection{The pointwise bounds \texorpdfstring{\cref{eqn:green:upper:bound,eqn:green:near:size}}{(\ref*{eqn:green:upper:bound}) and~(\ref*{eqn:green:near:size})}}

We begin with the bound~\eqref{eqn:green:upper:bound}. Let $x$, $y\in \Omega$ with $x\neq y$. By the local boundedness of solutions (Lemma~\ref{lem:boundary:Moser}), we have that
\begin{equation*}|G_y^L(x)|\leq \frac{C}{|x-y|^2} \int_{B(x,|x-y|/2)\cap\Omega} |G_y^L|\end{equation*}
(regardless of the comparative sizes of $|x-y|$ and $\delta(y)$).

By Hölder's inequality we have that if $\beta>1$ then
\begin{equation*}|G_y^L(x)|\leq \frac{C}{|x-y|^2}
\biggl(\int_{B(x,|x-y|/2)} |G_y^L|^{\beta'}\delta^{-{\beta'}}\biggr)^{1/{\beta'}}
\biggl(\int_{B(x,|x-y|/2)} \delta^{\beta}\biggr)^{1/\beta}.\end{equation*}
If $z\in B(x,|x-y|/2)$ then $\delta(z)\leq |z-x|+|x-y|+\delta(y)<C(|x-y|+\delta(y))$. Thus
\begin{equation*}|G_y^L(x)|\leq \frac{C(|x-y|+\delta(y))}{|x-y|^{2/{\beta'}}}
\biggl(\int_{B(x,|x-y|/2)} |G_y^L|^{\beta'}\delta^{-{\beta'}}\biggr)^{1/{\beta'}}
.\end{equation*}
By the bound~\eqref{eqn:green:2d:directsize}  we have that
\begin{equation*}
|G_y^L(x)|
\leq C(|x-y|+\delta(y))\frac{\delta(y)^{2/{\beta'}-1}}{|x-y|^{2/{\beta'}}}
.\end{equation*}
Recall that ${\beta'}<2$ and so $2/{\beta'}-1>0$. Thus $|G_y^L(x)|\leq C$ whenever $|x-y|\geq\delta(y)/4$, which establishes the bound~\eqref{eqn:green:upper:bound} with $\bdmn=1$. In fact this yields a stronger estimate: $\lim_{x\to\infty} G_y^L(x)=0$.

We are left with the bound~\eqref{eqn:green:near:size}. (Note that the above argument would yield the estimate~\eqref{eqn:green:near:size} with $\gamma=2/\beta'>1$; however, we require $\gamma$ small and so we must do a more complicated analysis.) Let $x_n$, $n\geq 2$, satisfy $|x_n-y|=2^{-n}\delta(y)$. We have that $|G_y^L(x_2)|\leq C$ by the above analysis.

Now	let $x$ satisfy $2^{-n-1}\delta(y)\leq |x-y|\leq 2^{-n}\delta(y)$, $n\geq 2$. We have established the bound~\eqref{eqn:green:solution} and so we know $LG_y^L=0$ in $\Omega\setminus\{y\}$. By the bound~\eqref{eqn:interior:Moser} we have that
\begin{align*}
|G_y^L(x)-G_y^L(x_n)|
&\leq
|G_y^L(x)-g|+|g-G_y^L(x_n)|
\\&\leq
\frac{C}{(2^{-n}\delta(y))^2} \int_{B(y,2^{1-n}\delta(y))\setminus B(y,2^{-n-2}\delta(y))}
|G_y^L-g|
\end{align*}
for any $g\in\R$. Choosing $g$ appropriately, by the Poincar\'e inequality
\begin{equation*}|G_y^L(x)-G_y^L(x_n)|
\leq
\frac{C}{(2^{-n}\delta(y))} \int_{B(y,2^{1-n}\delta(y))\setminus B(y,2^{-n-2}\delta(y))}
|\nabla G_y^L|.\end{equation*}
By Hölder's inequality
\begin{equation*}|G_y^L(x)-G_y^L(x_n)|
\leq
\frac{C}{(2^{-n}\delta(y))^{1-2/\beta}} \biggl( \int_{B(y,2^{1-n}\delta(y))\setminus B(y,2^{-n-2}\delta(y))}
|\nabla G_y^L|^{\beta'}\biggr)^{1/\beta'}.\end{equation*}
By the bound~\eqref{eqn:2d:green:Poisson},
\begin{equation*}|G_y^L(x)-G_y^L(x_n)|
\leq
C2^{n(1-2/\beta)}.\end{equation*}
Observe that this estimate is also valid with $x$ replaced by $x_{n+1}$.
Summing, we see that
\begin{align*}
|G_y^L(x)|
&
\leq |G_y^L(x)-G_y^L(x_n)|
+\sum_{k=2}^{n-1} |G_y^L(x_k)-G_y^L(x_{k+1})|
+|G_y^L(x_2)|
\\&\leq
C2^{n(1-2/\beta)}+C \sum_{k=2}^{n-1}2^{k(1-2/\beta)}+C.\end{align*}
Because $\beta>2$ the geometric series may be bounded by the $k=n-1$ term, and so
\begin{equation*}|G_y^L(x)|\leq
C2^{n(1-2/\beta)}
\leq C\biggl(\frac{\delta(y)}{|x-y|}\biggr)^{1-2/\beta}
\end{equation*}
whenever $|x-y|\leq \delta(y)/4$.

This establishes the bound~\eqref{eqn:green:near:size} with $\gamma=1-2/\beta$; because we required $\beta\leq4/(2-\alpha^*)$, this establishes the bound~\eqref{eqn:green:near:size} for some $\gamma\in (0,\alpha^*/2]$. But $|x-y|<\delta(y)$, and so we may (for simplicity) increase $\gamma$ to $\alpha^*/2$ in the event that $\gamma$ is smaller.

\subsubsection{Formulas~\texorpdfstring{\cref{eqn:green:poisson:nodiv,eqn:green:symmetric}}{(\ref{eqn:green:poisson:nodiv}) and (\ref{eqn:green:symmetric})}}

We observe that the argument for  formulas~\eqref{eqn:2d:nodiv:1} and~\eqref{eqn:2d:nodiv:2} in Section~\ref{sec:green:2:1} is valid with $\beta$ (and~$\beta'$) replaced by~$2$ and with $L$ replaced by $L^*$. Thus, if $f\in C^\infty_0(\Omega)$ then
\begin{equation}\label{eqn:green:poisson:nodiv:reverse}
\text{if } u(y)=\int_\Omega G_y^{L^*}(x) \,f(x)\,dx
\text{ then } u\in \widetilde W^{1,2}(\Omega)\text{ and } \int_\Omega f\varphi=\int_\Omega \nabla\varphi\cdot A\nabla u
\end{equation}
for all $\varphi\in \widetilde W^{1,2}(\Omega)$.

We observe that formula~\eqref{eqn:green:poisson:nodiv} is very similar to formula~\eqref{eqn:green:poisson:nodiv:reverse}; in fact, formula~\eqref{eqn:green:poisson:nodiv} will follow immediately from formula~\eqref{eqn:green:poisson:nodiv:reverse} once we have established formula~\eqref{eqn:green:symmetric}.

We may follow the argument for formula~\eqref{eqn:green:symmetric} in Section~\ref{sec:Green}. In particular, we need only the known formula~\eqref{eqn:green:poisson:nodiv:reverse} (as a replacement for formula~\eqref{eqn:green:poisson:nodiv}) and continuity of $G_y^L(x)$ on $\Omega\times\Omega\setminus D$.

Continuity of $G_y^L(x)$ in $x$ follows from the De Giorgi-Nash estimate~\eqref{eqn:interior:DGN} and the above estimates on $G^L_y$. In particular, we may bound the modulus of continuity in terms of $x$ and~$y$. We are left with continuity in~$y$.

By the bound~\eqref{eqn:2d:cts}, if $|y-z|<\delta(y)/2$ then
\begin{equation}
|u^*_{\vec H}(y)-u^*_{\vec H}(z)|
\leq C|y-z|^{1-\pdmn/\beta} \|\vec H\|_{L^{\beta}(\Omega)}
\end{equation}
and so
\begin{equation*}\|\nabla G_y^L-\nabla G_z^L\|_{L^{\beta'}(\Omega)}
\leq C|y-z|^{1-\pdmn/\beta}.\end{equation*}
Applying Lemma~\ref{lem:boundary:Poincare} and the interior Moser estimate as usual yields continuity of $G_y^L(x)$ in~$y$, and a suitable estimate on the modulus of continuity. Combining these estimates yields continuity of $G_y^L(x)$ on $\Omega\times\Omega\setminus D$, and thus formula~\eqref{eqn:green:symmetric}.

\subsubsection{Formulas~\texorpdfstring{\cref{eqn:green:Poisson,eqn:green:positive}}{(\ref{eqn:green:Poisson}) and (\ref{eqn:green:positive})}}

The relation \eqref{eqn:green:Poisson} follows immediately from formulas~\eqref{eqn:green:W12}, \eqref{eqn:2d:green:Poisson}, the relation $\vec g_y^L=\nabla G_y^L$, the compatible well posedness of the $L^{\beta}(\Omega)$-Poisson problem for~$L^*$, and a density argument.

Finally, formula~\eqref{eqn:green:positive} follows from formula~\eqref{eqn:green:poisson:nodiv} as in Section~\ref{sec:Green}.

\subsection{Harmonic measure}\label{sec:harmonic:measure}

In this paper we will make extensive use of the harmonic measure for the operator~$L$.

If $x\in\Omega$, then the mapping $f\mapsto u_f(x)$, where $u_f$ is the solution to the Dirichlet problem~\eqref{eqn:Dirichlet:cts} with boundary data~$f$ guaranteed by Lemma~\ref{lem:cts}, is a linear operator and satisfies $|u_f(x)|\leq \|f\|_{L^\infty(\partial\Omega)}$. Thus it extends to a bounded linear operator on the set of compactly supported continuous functions. Note also that if $f\geq 0$ then $u_f(x)\geq 0$. Thus by the Riesz-Markov theorem, there is a unique Radon measure $\omega_L^x$ on $\partial\Omega$ that satisfies
\begin{equation*}u_f(x)=\int_{\partial\Omega} f\,d\omega_L^x\end{equation*}
for all compactly supported Lipschitz functions $f:\partial\Omega\to\R$.

We have that $\omega^x_L(E)\leq 1$ for all $E\subset\partial\Omega$. Furthermore, if $\xi\in\partial\Omega$ and $r>0$, then
 by Lemma~\ref{lem:boundary:DGN} applied to $u=1-u_f$ for $f$ a Lipschitz function supported in $\Delta(\xi,r)$ and $1$ in $\Delta(\xi,r/2)$, we have that
\begin{equation*}0\leq 1-\omega^x_L(\Delta(\xi,r)) =1-\int_{\Delta(\xi,r)} 1\,d\omega_L^x
\leq 1-\int_{\partial\Omega} f\,d\omega_L^x = 1-u_f(x)\leq C\biggl(\frac{|x-\xi|}{r}\biggr)^\alpha
\end{equation*}
if $|x-\xi|<r/8$. In particular, letting $r\to\infty$ yields that $\omega_L^x(\partial\Omega)=1$.

This also provides a proof of Bourgain's estimate, which we now state.
\begin{lem}[Bourgain's estimate]
\label{lem:harmonic:from:below}
Let $\Omega\subset\R^\dmn$ be as in Theorem~\ref{thm:Poisson}.
Then there is a $c\in (0,1)$ depending only on $\bdmn$, $\dmn$, $\lambda$, $\Lambda$, and the Ahlfors constant, such that if $\xi\in\partial\Omega$, $0<r\leq\diam\Omega$, and $x\in\Omega\cap B(\xi,cr)$, then $\omega_L^x(\Delta(\xi,r))\geq \frac{1}{2}$.
\end{lem}

Let $E\subseteq\partial\Omega$ be a $\sigma$-measurable set. We have the following result:
\begin{equation}\label{eqn:harmonic:measure:solution}
\text{If }w(x)=\omega^x_{L}(E),\text{ then } Lw=0\text{ in }\Omega.
\end{equation}
This is proven in \cite[Lemma 1.2.7]{Ken94}. (The monograph \cite{Ken94} considers only symmetric coefficients~$A$, but the proof of \cite[Lemma 1.2.7]{Ken94} is elementary and relies only on Harnack's inequality, which also holds for non-symmetric equations, see for example \cite[Theorem~4.17]{HanL11}.)

\subsection{The Green's function and harmonic measure}

We will need to be able to bound the Green's function by the harmonic measure.
The following lemma is fairly well known in the case  where $y=A_r(\xi)$ is the corkscrew point for $\xi$ and~$r$; see for example \cite[Lemma~11.11]{DavFM21} in the higher codimensional case.

\begin{lem}\label{lem:harmonic:controls:Green}Let $\Omega\subset\R^\dmn$ be as in Theorem~\ref{thm:Poisson}.  Let $\xi\in\partial\Omega$, $0<r<\diam\Omega$,  $y\in\Omega\cap B(\xi,2r)$, and $x\in \Omega\setminus \overline{B(y,r)}$. Then
\begin{equation*}
G^L_y(x)\leq C\frac{\omega_L^x(\Delta(\xi,8r))}{r^{\bdmn-1}}.
\end{equation*}
\end{lem}

\begin{proof}
Let $f$ be a nonnegative Lipschitz function that is $1$ on $\Delta(\xi,7r)$ and $0$ outside of $\Delta(\xi,8r)$. Let $u_f(x)=\int_{\partial\Omega} f\,d\omega_L^x$ be the unique solution to the continuous Dirichlet problem~\eqref{eqn:Dirichlet:cts} guaranteed by Lemma~\ref{lem:cts}. Observe that
\begin{equation*}0\leq\omega_L^x(\Delta(\xi,7r))\leq u_f(x)\leq \omega_L^x(\Delta(\xi,8r))\leq 1\end{equation*}
for all $x\in\Omega$.

Let $z\in B(y,r)$. Then $z\in B(\xi,3r)$ and so $\delta(z)<3r$. Let $\zeta\in\partial\Omega$ satisfy $\delta(z)=|z-\zeta|$. Then $|\zeta-\xi|\leq |\zeta-z|+|z-\xi|<6r$, and so $\Delta(\zeta, r)\subset \Delta(\xi,7r)$. Let $c$ be as in Lemma~\ref{lem:harmonic:from:below}. Then $u_f\geq 1/2$ in $B(\zeta,cr)\cap\Omega$. Because $B(z,|z-\zeta|)\subset\Omega$, we may connect $B(\zeta,cr)\cap\Omega$ to $z$ by at most $C(c)$ Harnack balls. Applying Harnack's inequality (see for example \cite[Theorem~4.17]{HanL11}) we see that $u_f(z)\geq 1/C$.

Recall $y\in B(\xi,2r)$ and so $\delta(y)<2r$, that is, $r/2>\delta(y)/4$.
If $z\in B(y,r)\setminus\overline B(y,r/2)$, then by the bound~\eqref{eqn:green:upper:bound}, $G_y^L(z)\leq Cr^{1-\pbdmn}$. Thus there is a constant $K=C^2$ such that, if
\begin{equation*}w(z)=G^L_y(z)-Kr^{1-\pbdmn}u_f(z),\end{equation*}
then $w\leq 0$ in $B(y,r)\setminus\overline B(y,r/2)$.

Our goal is to show that $w\leq 0$ in $\Omega\setminus \overline B(y,r)$.
Define $\varphi$ as follows:
\begin{equation*}\varphi(z)=\begin{cases} 0, & z\in B(y,r/2) \text{ or } w(z)\leq 0,\\
w(z), & w(z)\geq 0 \text{ and } z\notin B(y,r/2).\end{cases}\end{equation*}
We claim that $\varphi\in \widetilde W^{1,2}_{av,2,1/2}(\Omega)$. To see this, recall that the solution $u_f$ to the continuous Dirichlet problem~\eqref{eqn:Dirichlet:cts} is guaranteed to lie in $\dot W^{1,2}_{av,2,1/2}(\Omega)$. By the bound~\eqref{eqn:green:W12} we have that $\nabla G_y^L\in L^2_{av,2,1/2}(\Omega\setminus B(y,(2/3)r);\Omega)$. Thus $w$ is a (local) Sobolev function in $\Omega\setminus B(y,(2/3)r)$. It is well known that $\max(w,0)$ is also a local Sobolev function with gradient equal to $\nabla w$ if $w\geq 0$ and 0 if $w\leq 0$. But by the above remarks $w(z)\leq 0$ in $B(y,r)\setminus\overline B(y,r/2)$, and so we have that $\varphi$ is a local Sobolev function in all of~$\Omega$ with $\nabla\varphi\in L^2_{av,2,1/2}(\Omega)$, that is, $\varphi\in \dot W^{1,2}_{av,2,1/2}(\Omega)$. To see that $\varphi\in\widetilde W^{1,2}_{av,2,1/2}(\Omega)$, recall that by formulas \eqref{eqn:Dirichlet:cts} and~\eqref{eqn:green:boundary}, $w$ is continuous up to the boundary and satisfies $w=-f\leq 0$ on $\partial\Omega$. Then $\varphi$ is also continuous and is zero on the boundary. Thus $\varphi\in \widetilde  W^{1,2}_{av,2,1/2}(\Omega)$ by Lemma~\ref{lem:boundary:Poincare}.

By Proposition~\ref{prp:dense}, the definition~\eqref{eqn:L} of $Lu_f=0$, and formula~\eqref{eqn:green:fundamental}, we have that
\begin{equation*}0=\varphi(z)=\int_\Omega A\nabla w\cdot\nabla\varphi.\end{equation*}
But $\nabla\varphi$ is either zero or equal to~$\nabla w$, and so by the ellipticity condition~\eqref{eqn:elliptic}
\begin{equation*}0\geq \int_{\Omega\setminus B(y,r)}
\lambda\delta^{\bdmn+1-\pdmn}|\nabla \varphi|^2.\end{equation*}
The right hand side is clearly nonnegative and so must be zero. Thus $\varphi$ is constant and so must be uniformly zero; thus we have that $w\leq 0$ and so $G_y^L\leq Kr^{1-\pbdmn}u_f$ in $\Omega\setminus B(y,r)$. Recalling that $u_f(x)\leq \omega^x_L(\Delta(\xi,8r))$ completes the proof.
\end{proof}

\section{The case \texorpdfstring{$p\leq 1$}{p≤1}}
\label{sec:atom}

In this section we will prove the $p\leq 1$, $s>\bdmn/p+1-\bdmn-\alpha^*$ case of  Theorem~\ref{thm:Poisson}. (We will also lay the groundwork for the case $\bdmn<1-\alpha^*$, $p\leq 1$, $\bdmn/p< s \leq\bdmn/p+1-\bdmn-\alpha^*$.)  Most of the work of this section will be in the proof of the following proposition; we will pass from Proposition~\ref{prp:atom:Poisson} to Theorem~\ref{thm:Poisson} in Section~\ref{sec:Dirichlet:atom}.

\begin{prp}
\label{prp:atom:Poisson}
Let $\Omega$, $\dmn$, $\bdmn$, and~$L$ be as in Theorem~\ref{thm:Poisson}.

Let $\beta$ satisfy $p^-_{L^*}< \beta<p^+_L$, where $p^+_L$ is as in the bound~\eqref{eqn:interior:Meyers} and $p^-_{L^*}$ is as in Lemma~\ref{lem:beta}.

Let $\mathcal{G}=\mathcal{G}_1$ be the grid of Whitney cubes in $\Omega$ given by Definition~\ref{dfn:Whitney}, let $Q\in \mathcal{G}_1$, and let $\vec H:\Omega\to\R^\dmn$ satisfy
\begin{equation*}
\supp\vec H\subseteq \overline{Q},
\qquad
\vec H\in L^\beta(Q),\cap L^2(Q),
\qquad
\biggl(\int_{Q} |\vec H|^\beta\biggr)^{1/\beta}\leq \diam(Q)^{\pdmn/\beta}.
\end{equation*}
By Lemma~\ref{lem:averaged:Whitney}, $\vec H\in L^2_{av,\beta,1/2}(\Omega)$.
Let $u\in \widetilde W^{1,2}_{av,\beta,1/2}(\Omega)\cap \widetilde W^{1,2}_{av,2,1/2}(\Omega)
$ be the solution to the $L^{2}_{av,\beta,1/2}(\Omega)$-Poisson problem with data~$\vec H$ guaranteed by Lemma~\ref{lem:beta}.

Let $c\in (0,1)$ and let $C_0$, ${\mathfrak{a}^*}$, ${\mathfrak{q}}$, and~$\theta$ satisfy one of the following conditions.
\begin{itemize}
\item ${\mathfrak{q}}=1$ and $1-\theta={\mathfrak{a}^*}=\alpha^*$, where $\alpha^*$ is the number in the boundary De Giorgi-Nash condition (Lemma~\ref{lem:boundary:DGN}) for~$L^*$; that is,
if $\xi\in\partial\Omega$, $r>0$, $w\in \widetilde W^{1,2}_{av,2,1/2}(B(\xi,r)\cap\Omega;\Omega)$ and $L^*w= 0$ in $B(\xi,r)\cap\Omega$, then for all $x\in B(\xi,r/2)$ we have that
\begin{equation}\label{eqn:p<1:DGN}
|w(x)|\leq C_0\frac{|x-\xi|^{{\mathfrak{a}^*}}}{r^{{\mathfrak{a}^*}}} \frac{1}{\int_{B(\xi,r)\cap\Omega}\delta^{1+\bdmn-\pdmn}} \int_{B(\xi,r)\cap\Omega}|w|\,\delta^{1+\bdmn-\pdmn}.\end{equation}
\item $0<{\mathfrak{q}}<\infty$, $0<\theta<1$, ${\mathfrak{a}^*}=1-\theta-\bdmn(1-1/\mathfrak{q})\in (0,1)$, and the function $u$ (as defined above) satisfies the estimate
\begin{equation}\label{eqn:p<1:extrapolation}
\|u\|_{\dot W^{1,\mathfrak q}_{av,c,\beta,\theta}(\Omega)}
\leq C_0 \diam(Q)^{\bdmn/\mathfrak{q}+1-\theta}.
\end{equation}
\end{itemize}

Let $s\in (1-\mathfrak{a}^*,1)$ and let $p$ satisfy
\begin{equation*}\bdmn\leq\frac{\bdmn}{p}<\bdmn-1+{\mathfrak{a}^*}+s.\end{equation*}
Then
\begin{equation}\label{eqn:atom:Poisson:conclusion}
\|u\|_{\dot W^{1,p}_{av,c,\beta,s}(\Omega)}
\leq C \diam(Q)^{\bdmn/p+1-s}
\end{equation}
where $C$ depends only on $c$, $p$, $s$, ${\mathfrak{a}^*}$, $\mathfrak{q}$, $C_0$, and the standard parameters.
\end{prp}

The allowable values of $s$ and $p$ are illustrated in Figure~\ref{fig:p<1:DGN} (the case~\eqref{eqn:p<1:DGN}) and Figure~\ref{fig:p<1:extrapolation} (the case~\eqref{eqn:p<1:extrapolation}). We observe that by Corollary~\ref{cor:weighted:embedding}, if the bound~\eqref{eqn:p<1:extrapolation} is valid for some $\mathfrak{q}=\mathfrak{q}_0$ and some $\theta=\theta_0$, then it is valid for all $\mathfrak{q}>\mathfrak{q}_0$ with the choice $\theta=\theta_0-\bdmn/\mathfrak{q}_0+\bdmn/\mathfrak{q}$. This is the diagonal line in the $(\theta,1/\mathfrak{q})$-plane of slope $1/\bdmn$ that passes through the point $(\theta_0,1/\mathfrak{q}_0)$. Thus we may (and will) assume without loss of generality that $\mathfrak{q}>1$ in the case~\eqref{eqn:p<1:extrapolation}.

The De Giorgi-Nash estimates are known to be valid, and so there always exist $q$, $\mathfrak{a}^*$ as in the case~\eqref{eqn:p<1:DGN}.
Our main use for the case~\eqref{eqn:p<1:extrapolation} will be to ensure that the numbers $\mathfrak{a}$ and $\mathfrak{b}$ in Theorems~\ref{thm:Poisson} and~\ref{thm:Poisson:Lq} satisfy the estimates $\mathfrak{a}\geq 1-\bdmn$ and $\mathfrak{b}\geq 1+\frac{\bdmn}{q^*}-\bdmn$.

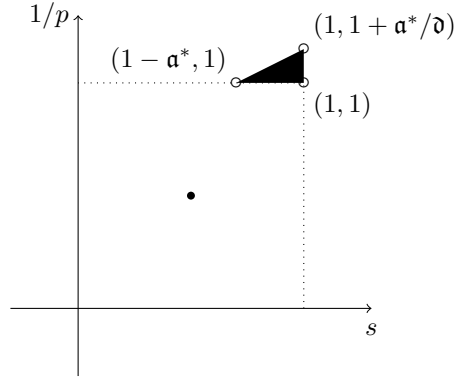
\begin{figure}
\begin{center}
\begin{tikzpicture}[scale=3]
\begin{scope}[shift={(0,0)}]
\draw [dotted] (1,0)--(1,1)--(0,1);
\draw [->] (-0.3,0)--(1.3,0) node [below ] {$\vphantom{1}s$};
\draw [->] (0,-0.3)--(0,1.3) node [ left] {$1/p$};
\fill  (1,1)  -- (1-\figurealpha,1) --(1,1+\figurealpha/\figuredimen) -- cycle;
\node at (1,1) {$\circ$};
\node [below right] at (1,1) {$(1,1)$};
\node at (1,1+\figurealpha/\figuredimen) {$\circ$};
\node [above right] at (1,1+\figurealpha/\figuredimen) {$(1,1+\mathfrak{a}^*/\bdmn)$};
\node at (1-\figurealpha,1) {$\circ$};
\node [above left] at (1-\figurealpha,1) {$(1-\mathfrak{a}^*,1)$};
\fill (1/2,1/2) circle (0.5pt);
\node [left] at (0,1) {$\phantom{(0,1)}$};
\end{scope}
\end{tikzpicture}
\end{center}
\caption{If the condition~\eqref{eqn:p<1:DGN} is valid, then the conclusion~\eqref{eqn:atom:Poisson:conclusion} is valid whenever $(s,1/p)$ lies in the indicated triangle. The $s=1/2$, $p=2$ case of the bound~\eqref{eqn:atom:Poisson:conclusion} follows from Lemma~\ref{lem:beta}.}
\label{fig:p<1:DGN}
\end{figure}

\def\figureq{0.85}
\def\figuretheta{0.25}
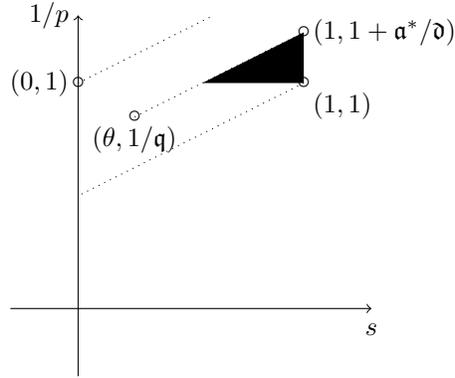
\begin{figure}
\begin{center}
\begin{tikzpicture}[scale=3]
\begin{scope}[shift={(0,0)}]
\fill (1,1)-- (1,\figureq+1/\figuredimen-\figuretheta/\figuredimen) -- (\figuretheta+\figuredimen-\figuredimen*\figureq,1) -- cycle;
\draw [dotted] (1,1) -- (0,1-1/\figuredimen);
\draw [dotted] (0,1) -- (0.3*\figuredimen,1.3);
\draw [dotted] (\figuretheta,\figureq) -- (1,\figureq+1/\figuredimen-\figuretheta/\figuredimen) node {$\circ$} node [right] {$(1,1+\mathfrak{a}^*/\pbdmn)$};
\draw [->] (-0.3,0)--(1.3,0) node [below ] {$\vphantom{1}s$};
\draw [->] (0,-0.3)--(0,1.3) node [ left] {$1/p$};
\node at (1,1) {$\circ$};
\node [below right] at (1,1) {$(1,1)$};
\node at (0,1) {$\circ$};
\node [left] at (0,1) {$(0,1)$};
\node at (\figuretheta,\figureq) {$\circ$};
\node [below] at (\figuretheta,\figureq) {$(\theta,1/\mathfrak{q})$};
\end{scope}
\end{tikzpicture}
\end{center}
\caption{If the condition~\eqref{eqn:p<1:extrapolation} is valid, then the conclusion~\eqref{eqn:atom:Poisson:conclusion} is valid whenever $(s,1/p)$ lies in the indicated region. All three of the dotted lines have slope $1/\bdmn$, and the region is nonempty (and $\mathfrak{a}^*\in (0,1)$) provided $\theta\in (0,1)$ and the point $(\theta,1/q)$ lies between the upper and lower dotted lines. }
\label{fig:p<1:extrapolation}
\end{figure}

The rest of this section will be devoted to the proof of
Proposition~\ref{prp:atom:Poisson} (Sections~\ref{sec:Poisson:proof:start}--\ref{sec:Poisson:proof:end}) and its immediate consequence (the $p\leq 1$ case of Theorem~\ref{thm:Poisson}, to be proven in Section~\ref{sec:Dirichlet:atom}).

\subsection{A preliminary estimate \texorpdfstring{on~$u$}{on u}}
\label{sec:Poisson:proof:start}

Let $u$ be as in Proposition~\ref{prp:atom:Poisson}.
We begin by recalling that the quantity we seek to bound is
\begin{equation*}\|u\|_{\dot W^{1,p}_{av,c,\beta,s}(\Omega)}
=\biggl(\int_\Omega \mathcal{W}_{c,\beta} (\nabla u)^p \delta^{\bdmn-\pdmn+p-ps}\biggr)^{1/p}\end{equation*}
where $\mathcal{W}_{c,\beta}$ is the Whitney averaging operator of Definition~\ref{dfn:weighted:averaged}.

By Lemma~\ref{lem:beta}, if $u$ is as in Proposition~\ref{prp:atom:Poisson} then
\begin{equation*}\int_\Omega \mathcal{W}_{c,\beta}(\nabla u)^2\delta^{\bdmn-\pdmn+1}
=\|u\|_{\dot W^{1,2}_{av,c,\beta,1/2}(\Omega)}^2
\leq
C
\|\vec H\|_{L^{2}_{av,c,\beta,1/2}(\Omega)}^2
.\end{equation*}
But by Lemma~\ref{lem:averaged:Whitney} and the assumption on~$\vec H$ in Proposition~\ref{prp:atom:Poisson}, $\|\vec H\|_{L^{2}_{av,\beta,1/2}(\Omega)} \leq C\diam(Q)^{\bdmn/2+1/2}$, and so
\begin{equation}\label{eqn:p<1:basic}
\int_\Omega \mathcal{W}_{c,\beta}(\nabla u)^2\delta^{\bdmn-\pdmn+1}
=\|u\|_{\dot W^{1,2}_{av,c,\beta,1/2}(\Omega)}^2
\leq
C\diam(Q)^{\bdmn+1}
.\end{equation}

\subsection{Control near the support of the data}

We now give a relatively straightforward estimate of the integral of $\mathcal{W}_{c,\beta}(\nabla u)^p\delta^{\bdmn-\pdmn+p-ps}$ over a certain subset of~$\Omega$.
\begin{lem}\label{lem:local:Poisson}
Let $u$, $\beta$, and $Q$ be as in Proposition~\ref{prp:atom:Poisson}. Let $s\in\R$, $0<p<2$, and $0<c<1$. Additionally let $0<b<1/2$ and $\vartheta>1$.
Define
\begin{align*}
\Omega_{b,\vartheta}
&= \Omega\cap \vartheta Q \setminus \{y\in\Omega:\delta(y)<b\dist(y,Q)\}
\\&= \{y\in \Omega\cap \vartheta Q: \delta(y)\geq b\dist(y,Q)\}.\end{align*}
Then
\begin{equation*}
\int_{\Omega_{b,\vartheta}} \mathcal{W}_{c,\beta}(\nabla u)^p \delta^{\bdmn-\pdmn+p-ps}
\leq
C\diam(Q)^{\bdmn+p-ps}
.\end{equation*}
\end{lem}

We will of course only use this lemma in the case $\frac{\bdmn}{\bdmn+{\mathfrak{a}^*}}<p\leq 1$.

\begin{proof}[Proof of Lemma~\ref{lem:local:Poisson}]
Because $p<2$, we may apply Hölder's inequality to see that
\begin{multline*}
\int_{\Omega_{b,\vartheta}} \mathcal{W}_{c,\beta}(\nabla u)^p \delta^{\bdmn-\pdmn+p-ps}
\\\noalign{\vskip-3pt}
\leq
\biggl(\int_{\Omega_{b,\vartheta}} \mathcal{W}_{c,\beta}(\nabla u)^2
\,\delta^{\bdmn-\pdmn+1}
\biggr)^{p/2}
\biggl(\int_{\Omega_{b,\vartheta}} \delta^{\bdmn-\pdmn+p(1-2s)/(2-p)}
\biggr)^{1-p/2}
.\end{multline*}
We estimate the rightmost integral. The region of integration $\Omega_{b,\vartheta}$ is contained in~$\vartheta Q$ and so has volume at most $\vartheta^\dmn (\ell(Q))^\dmn$.
If $y\in \Omega_{b,\vartheta}\subset \vartheta Q$, then $\delta(y)\leq \dist(Q,\partial\Omega)+\diam(\vartheta Q)$ and so by the bound~\eqref{eqn:Whitney} $\delta(y)\leq (3+\vartheta)\diam Q$. Conversely, $\dist(Q,\partial\Omega)\leq \dist(Q,y)+\dist(y,\partial\Omega)\leq \bigl(\frac{1}{b}+1\bigr)\delta(y)$.
Thus $\delta\approx \diam(Q)$ in $\Omega_{b,\vartheta}$ with comparability constants depending only on $b$ and~$\vartheta$, and so
\begin{multline*}
\int_{\Omega_{b,\vartheta}} \mathcal{W}_{c,\beta}(\nabla u)^p\delta^{\bdmn-\pdmn+p-ps}
\\\noalign{\vskip-5pt}
\leq C_{\dmn,\vartheta,b}
\biggl(\int_{\Omega_{b,\vartheta}} \mathcal{W}_{c,\beta}(\nabla u)^2
\,\delta^{\bdmn-\pdmn+1}
\biggr)^{p/2}
\diam(Q)^{\bdmn-\bdmn p/2+p/2-ps}
.\end{multline*}
By the bound~\eqref{eqn:p<1:basic},
\begin{equation*}
\int_{\Omega_{b,\vartheta}} \mathcal{W}_{c,\beta}(\nabla u)^p\delta^{\bdmn-\pdmn+p-ps}
\leq C
\diam(Q)^{\bdmn+p-ps}
\end{equation*}
as desired.
\end{proof}

\subsection{Pointwise estimates}\label{sec:pointwise}

In order to prove Proposition~\ref{prp:atom:Poisson} we will need some pointwise estimates on the function~$u$.
This section incorporates many ideas from \cite{KenP93}.

Recall that formula~\eqref{eqn:green:fundamental} lets us write $\varphi(y)$ in terms of an integral against the gradient $\nabla G^L_y$ of the Green's function in the special case where $\varphi$ is smooth and compactly supported. We would like to produce a similar formula for $u(y)$, where $u$ is as in Proposition~\ref{prp:atom:Poisson}. However, we must contend with the fact that $u$ may not be either smooth or compactly supported.

\begin{lem}\label{lem:green:solution}
Let $\Omega$ and~$L$ be as in Theorem~\ref{thm:Poisson}.

Let $\varphi:\R^\dmn\to [0,1]$ be smooth, let $y\in\Omega$, and suppose that $\varphi\equiv 1$ in an open neighborhood of~$y$. Suppose further that $\nabla\varphi$ (although not necessarily~$\varphi$) is compactly supported.

Let $W\subseteq\Omega$ be open and contain $\Omega\cap\supp\varphi$.
Let $u\in \widetilde W^{1,2}_{av,2,1/2}(\Omega)$ be such that $Lu=0$ in~$W$. Then
\begin{align*}
u(y)
&=
\int_\Omega u \nabla \varphi\cdot A^*\nabla G^{L^*}_y
-\int_\Omega G^{L^*}_y\nabla u\cdot A^*\nabla \varphi
.\end{align*}
\end{lem}

The assumption that $u\in \dot W^{1,2}_{av,2,1/2}(W;\partial\Omega)$, rather than $\dot W^{1,2}_{av,\beta,1/2}(W;\partial\Omega)$, is used to ensure that the De Giorgi-Nash condition is valid in a neighborhood of~$y$. Thus, because
$Lu=0$ in a neighborhood of~$y$, by the bound~\eqref{eqn:interior:DGN} we have that $u$ is continuous at~$y$, and so $u(y)$ is meaningful.

\begin{proof}[Proof of Lemma~\ref{lem:green:solution}]
Let $\varepsilon>0$ be a sufficiently small positive number and let $\eta_\varepsilon$ be a smooth cutoff function supported in $B(y,2\varepsilon)$ and identically equal to~$1$ in $B(y,\varepsilon)$.

By Lemma~\ref{lem:boundary:Poincare}, $u(\varphi-\eta_\varepsilon) \in \widetilde  W^{1,2}_{av,1/2,2}(\Omega)$.
By the bound~\eqref{eqn:green:W12},
$G^{L^*}_y\in \dot W^{1,2}_{av,2,1/2}(\Omega\setminus B(y,\varepsilon);\Omega)$. By formula~\eqref{eqn:green:solution}, $L^*G^{L^*}_y=0$ in $\Omega\setminus B(y,\varepsilon)$. Furthermore, by Proposition~\ref{prp:dense}, we have that $C^\infty_0(\Omega)$ is dense in $\widetilde  W^{1,2}_{av,c,1/2,2}(\Omega)$. Thus, by definition of~$L^*$, we have that
\begin{align*}
0&=\int_\Omega \nabla(u(\varphi-\eta_\varepsilon))\cdot A^*\nabla G^{L^*}_y
.\end{align*}
By formula~\eqref{eqn:green:fundamental},
\begin{equation*}
u(y)=u(y)\eta_\varepsilon(y)=u(y)\int_\Omega\nabla \eta_\varepsilon\cdot A^*\nabla G_y^{L^*}
\end{equation*}
and so
\begin{align*}
u(y)
&=\int_\Omega \nabla(u(\varphi-\eta_\varepsilon)+u(y)\eta_\varepsilon)\cdot A^*\nabla G^{L^*}_y
.\end{align*}
Several applications of the product rule yield that
\begin{align*}
u(y)
&=\int_\Omega
A\nabla u\cdot \nabla ((\varphi-\eta_\varepsilon)G^{L^*}_y)
-\int_\Omega G^{L^*}_y\nabla u\cdot A^*\nabla \varphi
+\int_\Omega u \nabla \varphi\cdot A^*\nabla G^{L^*}_y
\\&\qquad
+\int_\Omega G^{L^*}_y \nabla u\cdot A^*\nabla \eta_\varepsilon
+\int_\Omega ( u(y)-u)\nabla \eta_\varepsilon\cdot A^*\nabla G^{L^*}_y
\\&=I-II+III+IV+V
.
\end{align*}
To complete the proof we must show that the terms $I$, $IV$ and~$V$ are zero, at least in the limit as $\varepsilon\to 0^+$.

We claim that
\begin{equation*}(\varphi-\eta_\varepsilon)G^{L^*}_y\in \widetilde W^{1,2}_{av,c,2,1/2}(\Omega).\end{equation*}
Because $Lu=0$ in a neighborhood of $\supp \varphi$, this suffices to show that the term $I$ is zero. To prove the claim, observe that by the bound~\eqref{eqn:green:W12} and the Poincar\'e inequality in an annulus, $G_y^{L^*} (1-\eta_{\varepsilon})\in \dot W^{1,2}_{av,2,1/2}(\Omega)$. By formula~\eqref{eqn:green:boundary}, $\Tr G_y^{L^*} (1-\eta_{\varepsilon}) = \Tr G_y^{L^*}=0$. Thus by Lemma~\ref{lem:boundary:Poincare}, $(1-\eta_\varepsilon)G^{L^*}_y\in \widetilde W^{1,2}_{av,2,1/2}(\Omega)$. Observe that $\varphi(1-\eta_\varepsilon)=\varphi-\eta_{\varepsilon}$ if $\varepsilon$ is small enough because $\varphi\equiv 1$ in a neighborhood of~$y$ and $\eta_\varepsilon$ is supported in $B(y,2\varepsilon)$. Because $\varphi$ is bounded and $\nabla\varphi$ is compactly supported, this proves the claim.

The terms $IV$ and $V$ satisfy
\begin{align*}
|IV+V|&=
\biggl|\int_\Omega G^{L^*}_y \nabla u\cdot A^*\nabla \eta_\varepsilon
+\int_\Omega ( u(y)-u)\nabla \eta_\varepsilon\cdot A^*\nabla G^{L^*}_y\biggr|
\\&\leq
\frac{C\delta(y)^{1+\bdmn-\pdmn}}{\varepsilon}
\biggl(\int_{B(y,2\varepsilon)\setminus B(y,\varepsilon)} G^{L^*}_y |\nabla u|
+\int_{B(y,2\varepsilon)\setminus B(y,\varepsilon)}  |u(y)-u| |\nabla G^{L^*}_y|\biggr)
.\end{align*}
By Hölder's inequality
\begin{align*}
|IV+V|
&\leq
\frac{C\delta(y)^{1+\bdmn-\pdmn}}{\varepsilon}
\biggl(\int_{B(y,2\varepsilon)} |\nabla u|^2\biggr)^{1/2}
\biggl(\int_{B(y,2\varepsilon)\setminus B(y,\varepsilon)} |G^{L^*}_y|^2\biggr)^{1/2}
\\&\qquad
+\frac{C\delta(y)^{1+\bdmn-\pdmn}}{\varepsilon}
\biggl(\int_{B(y,2\varepsilon)} |u-u(y)|^2\biggr)^{1/2}
\biggl(\int_{B(y,2\varepsilon)\setminus B(y,\varepsilon)}
|\nabla G^{L^*}_y|^2\biggr)^{1/2}
\end{align*}
and by the Caccioppoli inequality
\begin{align*}
|IV+V|
&\leq
C\delta(y)^{1+\bdmn-\pdmn}\varepsilon^{\dmn-2}
\biggl(\fint_{B(y,3\varepsilon)} |u-u(y)|^2\biggr)^{1/2}
\biggl(\fint_{B(y,3\varepsilon)\setminus B(y,\varepsilon/2)} |G^{L^*}_y|^2\biggr)^{1/2}
.\end{align*}
Let $r=\delta(y)/2$.
By the bound~\eqref{eqn:interior:DGN}, if $0<\varepsilon<r/6$, then
\begin{align*}
|IV+V|
&\leq
C\varepsilon^{\dmn-2+\alpha} r^{1+\bdmn-\pdmn-\alpha}
\biggl(\fint_{B(y,r)} |u|^2\biggr)^{1/2}
\biggl(\fint_{B(y,3\varepsilon)\setminus B(y,\varepsilon/2)} |G^{L^*}_y|^2\biggr)^{1/2}
.\end{align*}
By the bound~\eqref{eqn:green:near:size},
\begin{equation*}
\biggl(\fint_{B(y,3\varepsilon)\setminus B(y,\varepsilon/2)} |G^{L^*}_y|^2\biggr)^{1/2}
\leq
C\delta(y)^{\dmn-\pbdmn-1} \varepsilon^{2-\pdmn}
\biggl(\frac{\delta(y)}{\varepsilon}\biggr)^{\alpha/2}
\end{equation*}
Thus
\begin{align*}
|IV+V|
&\leq
C\varepsilon^{\alpha/2}
r^{-\alpha/2}
\biggl(\fint_{B(y,r)} |u|^2\biggr)^{1/2}
.\end{align*}
Thus, taking the limit as $\varepsilon\to 0^+$, we have that
\begin{align*}
u(y)
&=II-III=
\int_\Omega u \nabla \varphi\cdot A^*\nabla G^{L^*}_y
-\int_\Omega G^{L^*}_y\nabla u\cdot A^*\nabla \varphi
\end{align*}
as desired.
\end{proof}

We will now apply Lemma~\ref{lem:green:solution} for some specific choices of~$\varphi$.
\begin{lem}\label{lem:green:solution:2}
Let $\Omega$, $\dmn$, $\bdmn$, and~$L$ be as in Theorem~\ref{thm:Poisson}. Let $p_{L^*}^-<\beta<p_L^+$.
Let $u\in \widetilde W^{1,2}_{av,2,1/2}(\Omega)\cap  \widetilde W^{1,2}_{av,\beta,1/2}(\Omega)$.

Let $\xi\in \partial\Omega$ and let $r>0$.
Suppose that either
\begin{itemize}
\item
$y\in\Omega\cap B(\xi,r)$ and $Lu=0$ in $\Omega\cap B(\xi,6r)$, or
\item $y\in \Omega\setminus\overline{B(\xi,8r)}$ and $Lu=0$ in $\Omega\setminus \overline{B(\xi,3r)}$.
\end{itemize}
Let $0<c\leq 2/11$. Then
\begin{align*}
|u(y)|
&\leq
Cr^{-\bdmn/2-3/2}
\|u\|_{\dot W^{1,2}_{av,c,\beta,1/2}(\Omega\cap B(\xi,C_2r); \Omega)}
\int_{\Omega\cap B(\xi,7r)\setminus B(\xi,2r)}  G^{L^*}_y\,\delta^{\bdmn+1-\pdmn}
\end{align*}
where $C_2=5C_1/(1-c/2)$, and where $C_1$ is the constant in Lemma~\ref{lem:boundary:Poincare}.
\end{lem}

\begin{proof}Let $\varphi$ be smooth and satisfy $\supp\nabla\varphi\subset B(\xi,5r)\setminus\overline{B(\xi,4r)}$. We further require that $\varphi\equiv 1$ near~$y$ (so either $\varphi\equiv 1$ in a neighborhood of $\overline {B(\xi,4r)}$ or $\varphi\equiv 1$ in an open set containing the complement of~$B(\xi,5r)$), and that $\varphi\equiv 0$ in the other component, that is, in the region of~$\Omega$ where $Lu$ need not equal zero. We may require $|\nabla\varphi|\leq C/r$.

Then by Lemma~\ref{lem:green:solution} and the bound~\eqref{eqn:elliptic:bounded:introduction},
\begin{align*}
|u(y)|
&\leq
\frac{C}{r}
\int_{\Omega\cap\supp\nabla\varphi} |u| \,|\nabla G^{L^*}_y|\,\delta^{\bdmn+1-\pdmn}
+\frac{C}{r}\int_{\Omega\cap\supp\nabla\varphi} |\nabla u|\,G^{L^*}_y\,\delta^{\bdmn+1-\pdmn}
.\end{align*}
By Hölder's inequality,
\begin{align*}
|u(y)|
&\leq
\frac{C}{r}
\biggl(\int_{\Omega\cap\supp\nabla\varphi} |u|^2\,\delta^{\bdmn+1-\pdmn}
\biggr)^{1/2}
\biggl(\int_{\Omega\cap\supp\nabla\varphi} |\nabla G^{L^*}_y|^2\,\delta^{\bdmn+1-\pdmn}
\biggr)^{1/2}
\\&\qquad
+\frac{C}{r}
\biggl(\int_{\Omega\cap\supp\nabla\varphi} |\nabla u|^2\,\delta^{\bdmn+1-\pdmn}
\biggr)^{1/2}
\biggl(\int_{\Omega\cap\supp\nabla\varphi} |G^{L^*}_y|^2\,\delta^{\bdmn+1-\pdmn}
\biggr)^{1/2}
.\end{align*}
By Lemma~\ref{lem:averaged:to:unaveraged},
\begin{equation*}\int_{\Omega\cap\supp\varphi} |\nabla u|^2\delta^{\bdmn+1-\pdmn}
\leq
C\int_{\Omega\cap B(\xi,\frac{5}{1-c/2}r)\setminus B(\xi,\frac{4}{1+c/2}r)} \mathcal{W}_{c/2,2}(\nabla u)^2\delta^{\bdmn+1-\pdmn}
.\end{equation*}
If $z\in \Omega\cap B(\xi,\frac{5}{1-c/2}r)\setminus B(\xi,\frac{4}{1+c/2}r)$ then $B(z,c\delta(z))\subset B(\xi,5\frac{1+c}{1-c/2}r)\setminus B(\xi,4\frac{1-c}{1+c/2}r)$. If $c\leq 2/11$ then $B(z,c\delta(z))\subset \Omega \cap B(\xi,6r)\setminus B(\xi,3r)$ and so $Lu=0$ in $B(z,c\delta(z))$. Thus by the Caccioppoli inequality, the local bound~\eqref{eqn:interior:Moser}, and the Poincar\'e inequality, we have that
\begin{equation*}\mathcal{W}_{c/2,2}(\nabla u)\leq C \mathcal{W}_{c,\beta}(\nabla u).\end{equation*}
Similarly, by Lemma~\ref{lem:averaged:to:unaveraged} and the local bound~\eqref{eqn:interior:Moser},
\begin{equation*}\int_{\Omega\cap\supp\varphi} | u|^2\delta^{\bdmn+1-\pdmn}
\leq
C\int_{\Omega\cap B(\xi,\frac{5}{1-c/4}r)\setminus B(\xi,\frac{4}{1+c/4}r)} \mathcal{W}_{c/2,\beta} ( u)^2 \delta^{\bdmn-\pdmn+1}
.\end{equation*}
Because $\delta(z)\leq |z-\xi|\leq Cr$ for all $z\in B(\xi,\frac{5}{1-c/2}r)$,
\begin{equation*}\int_{\Omega\cap\supp\varphi} | u|^2\delta^{\bdmn+1-\pdmn}
\leq
Cr^2\int_{\Omega\cap B(\xi,\frac{5}{1-c/4}r)\setminus B(\xi,\frac{4}{1+c/4}r)} \mathcal{W}_{c/2,\beta} ( u)^2 \delta^{\bdmn-\pdmn-1}
.\end{equation*}

By Lemma~\ref{lem:boundary:Poincare},
\begin{align*}
\int_{\Omega\cap\supp\nabla\varphi} |u|^2\,\delta^{\bdmn+1-\pdmn}
&\leq
Cr^2\int_{\Omega\cap B(\xi,C_1\frac{5}{1-c/4}r)} \mathcal{W}_{c,\beta} (\nabla u)^2\delta^{\bdmn-\pdmn+1}
\end{align*}
and so
\begin{align*}
|u(y)|
&\leq
C
\|u\|_{\dot W^{1,2}_{av,c,\beta,1/2}(\Omega\cap B(\xi,C_2r); \Omega)}
\biggl(\int_{\Omega\cap\supp\nabla\varphi} |\nabla G^{L^*}_y|^2\,\delta^{\bdmn+1-\pdmn}
\biggr)^{1/2}
\\&\qquad
+\frac{C}{r}
\|u\|_{\dot W^{1,2}_{av,c,\beta,1/2}(\Omega\cap B(\xi,C_2r); \Omega)}
\biggl(\int_{\Omega\cap\supp\nabla\varphi} |G^{L^*}_y|^2\,\delta^{\bdmn+1-\pdmn}
\biggr)^{1/2}
.\end{align*}

Recall that either $y\in \Omega\cap B(\xi,2r)$ or $y\in \Omega\setminus \overline{B(\xi,7r)}$ and that $\nabla\varphi$ is supported in $B(\xi,5r)\setminus B(\xi,4r)$. By formula~\eqref{eqn:green:solution}, we may bound $\nabla G_y^{L^*}$ using the boundary and interior Caccioppoli inequalities (the bound~\eqref{eqn:interior:Caccioppoli} and Lemma~\ref{lem:Caccioppoli:boundary}). Thus
\begin{align*}
\biggl(\int_{\Omega\cap\supp\nabla\varphi} |\nabla G^{L^*}_y|^2\,\delta^{\bdmn+1-\pdmn}
\biggr)^{1/2}
&\leq
\frac{C}{r}
\biggl(\int_{\Omega\cap B(\xi,6r)\setminus B(\xi,3r)} |G^{L^*}_y|^2\,\delta^{\bdmn+1-\pdmn}
\biggr)^{1/2}
.\end{align*}
Another covering argument, Lemma~\ref{lem:close:to:Ahlfors}, Lemma~\ref{lem:boundary:Moser}, and the nonnegativity of Green's functions (formula~\eqref{eqn:green:positive}) yield that
\begin{equation*}
\biggl(\int_{\Omega\cap B(\xi,6r)\setminus B(\xi,3r)} | G^{L^*}_y|^2\,\delta^{\bdmn+1-\pdmn}
\biggr)^{1/2}
\leq
Cr^{-\bdmn/2-1/2}
\int_{\Omega\cap B(\xi,7r)\setminus B(\xi,2r)} G^{L^*}_y\,\delta^{\bdmn+1-\pdmn}
.\end{equation*}
This completes the proof.
\end{proof}

In the case where $y$ is very far from the support of the data, it is useful to estimate the Green's function and derive an explicit decay bound on~$u$.
\begin{lem}\label{lem:u:7delta}%
Let $\Omega$, $\dmn$, $\bdmn$, and~$L$ be as in Theorem~\ref{thm:Poisson}. Let $c\in(0,1)$ and let $p_{L^*}^-<\beta<p_L^+$.
Let $\xi\in \partial\Omega$ and let $r>0$.

Suppose that
$y\in \Omega\setminus\overline{B(\xi,28r)}$,
$u\in \widetilde W^{1,2}_{av,2,1/2}(\Omega)\cap  \widetilde W^{1,2}_{av,c,\beta,1/2}(\Omega)$, and $Lu=0$ in $\Omega\setminus \overline{B(\xi,3r)}$.
Then
\begin{align*}
|u(y)|
&\leq C\frac{r^{\bdmn/2-1/2+\alpha^*} }{|y-\xi|^{\bdmn-1+\alpha^*}}
\|u\|_{\widetilde W^{1,2}_{av,c,\beta,1/2}(\Omega)}
\end{align*}
where $\alpha^*$ is as in Lemma~\ref{lem:boundary:DGN} with $L$ replaced by~$L^*$.
\end{lem}

\begin{proof}
We apply Lemma~\ref{lem:green:solution:2}. For all $y\in \Omega\setminus \overline{B(\xi,8r)}$ we have that
\begin{align*}
|u(y)|
&\leq
\frac{C\|u\|_{\widetilde W^{1,2}_{av,c,\beta,1/2}(\Omega)}}{r^{\bdmn/2+3/2}}
\int_{\Omega\cap B(\xi,7r)\setminus B(\xi,2r)} \delta^{1+\bdmn-\pdmn} G^{L^*}_y
.
\end{align*}
We strengthen the requirement on $y$ to $y\in \Omega\setminus \overline B(\xi,28r)$. Thus $7r<|\xi-y|/4$.

Clearly
\begin{align*}
|u(y)|
&\leq
\frac{C\|u\|_{\widetilde W^{1,2}_{av,c,\beta,1/2}(\Omega)}}{r^{\bdmn/2+3/2}}
\sup_{\Omega\cap B(\xi,7r)}
G^{L^*}_y
\int_{\Omega\cap B(\xi,7r)} \delta^{1+\bdmn-\pdmn}
\end{align*}
and by Lemma \ref{lem:close:to:Ahlfors}
\begin{align*}
|u(y)|
&\leq
{C\|u\|_{\widetilde W^{1,2}_{av,c,\beta,1/2}(\Omega)}r^{\bdmn/2-1/2}}
\sup_{\Omega\cap B(\xi,7r)}
G^{L^*}_y
.\end{align*}

Recall from Section~\ref{sec:Green} that $L^*(G^{L^*}_y)=0$ in $\Omega\setminus\{y\}\supseteq \Omega\cap B(\xi,|y-\xi|/2)$ and $G^{L^*}_y=0$ on all of~$\partial\Omega$, so by Lemma~\ref{lem:boundary:DGN}
\begin{align*}
|u(y)|
&\leq {C\|u\|_{\widetilde W^{1,2}_{av,c,\beta,1/2}(\Omega)}r^{\bdmn/2-1/2}}
\frac{r^{\alpha^*}}{|\xi-y|^{\alpha^*}} \frac{1}{|\xi-y|^{1+\bdmn}} \int_{B(\xi,|y-\xi|/4)\cap\Omega}  G^{L^*}_y \delta^{1+\bdmn-\pdmn}
.\end{align*}
Again by  Lemma~\ref{lem:boundary:DGN}
\begin{align*}
|u(y)|
&\leq {C\|u\|_{\widetilde W^{1,2}_{av,c,\beta,1/2}(\Omega)}r^{\bdmn/2-1/2}}
\frac{r^{\alpha^*}}{|\xi-y|^{\alpha^*}} \sup_{B(\xi,|y-\xi|/4)\cap\Omega}  G^{L^*}_y
.\end{align*}
Observe that $\delta(y)\leq |y-\xi|$ and so $\frac{1}{4}\delta(y)<\frac{3}{4}|y-\xi|<|y-s|$ for all $s\in B(\xi,|y-\xi|/4)$. We may thus apply the bound~\eqref{eqn:green:upper:bound} on the Green's function, which yields that
\begin{align*}
|u(y)|
&\leq C\|u\|_{\widetilde W^{1,2}_{av,c,\beta,1/2}(\Omega)}\frac{r^{\bdmn/2-1/2+\alpha^*} }{|y-\xi|^{\bdmn-1+\alpha^*}}\end{align*}
as desired.\end{proof}

We can use this result to estimate $u$ when $u$ is as in Proposition~\ref{prp:atom:Poisson}.
\begin{lem}\label{lem:pointwise:Poisson}%
Let $u$, $\Omega$, $p$, $s$, and~$Q$ be as in Proposition~\ref{prp:atom:Poisson}.
If $y\in \Omega$ and $\dist(y,Q)>60\diam(Q)$, then
\begin{equation*}|u(y)|\leq C\frac{(\diam Q)^{\bdmn+\alpha^*}} {\dist(y,Q)^{\bdmn-1+\alpha^*}}
.\end{equation*}
\end{lem}

\begin{proof}
Let $\xi\in\partial\Omega$ satisfy $\dist(\xi,Q)=\dist(Q,\partial\Omega)$; because $\partial\Omega$ is closed such a $\xi$ must exist.
By Lemma~\ref{lem:dilate:Whitney} and because $Q\in\mathcal{G}_1$ we have that $2Q\subset B(\xi,5\diam(Q))$.

We may thus apply Lemma~\ref{lem:u:7delta} with  $r=2\diam(Q)$. We then have that
\begin{equation*}|u(y)|\leq C\frac{(\diam Q)^{\bdmn/2-1/2+\alpha^*}} {|y-\xi|^{\bdmn-1+\alpha^*}}
\|u\|_{\widetilde W^{1,2}_{av,\beta,1/2}(\Omega)}
\end{equation*}
for all $y\in \Omega\setminus B(\xi,56\diam(Q))$.
But $|y-\xi|\geq \dist(y,Q)-\diam(Q)-\dist(Q,\partial\Omega)$.
By the bound~\eqref{eqn:Whitney} $|y-\xi|\geq\dist(y,Q)-4\diam(Q)$, and so if $\dist(y,Q)>60\diam(Q)$ then $y\notin B(\xi,56\diam(Q))$.
Using the bound~\eqref{eqn:p<1:basic} completes the proof.
\end{proof}

The previous lemma allows us to estimate $u$ in cases where ${\mathfrak{a}^*}=\alpha^*$, that is, where the bound~\eqref{eqn:p<1:DGN} is valid. We will also need decay in the case where the bound~\eqref{eqn:p<1:extrapolation} is valid.

\begin{lem}\label{lem:pointwise:extrapolation}%
Let $u$, $\Omega$, $p$, $s$, and~$Q$ be as in Proposition~\ref{prp:atom:Poisson}. Suppose that $u$ satisfies the estimate~\eqref{eqn:p<1:extrapolation}.
If $y\in \Omega$ and $\dist(y,Q)>60\diam(Q)$, and if in addition $\delta(y)>b\dist(y,Q)$ for some $b>0$, then
\begin{equation*}|u(y)|\leq C_b\frac{(\diam Q)^{\bdmn+{\mathfrak{a}^*}}} {\dist(y,Q)^{\bdmn-1+{\mathfrak{a}^*}}}
.\end{equation*}
\end{lem}

\begin{proof} Let $\xi\in\partial\Omega$ with $|y-\xi|=\delta(y)$.
Because $\Tr u=0$ on~$\partial\Omega$, we may apply Lemma~\ref{lem:boundary:Poincare} to see that
\begin{equation*}\int_{B(\xi,2\delta(y))} \mathcal{W}_{c/2,\beta}(u)^{\mathfrak{q}} \delta^{\bdmn-\pdmn-\mathfrak{q}\theta}
\leq C \int_{B(\xi,2C_1\delta(y))\cap\Omega} \mathcal{W}_{c,\beta}(\nabla u)^{\mathfrak{q}}
\delta^{\bdmn-\pdmn+\mathfrak{q}-\mathfrak{q}\theta}.\end{equation*}
By the bound~\eqref{eqn:p<1:extrapolation},
\begin{equation*}\int_{B(\xi,2\delta(y))} \mathcal{W}_{c/2,\beta}(u)^{\mathfrak{q}} \delta^{\bdmn-\pdmn-\mathfrak{q}\theta}
\leq C C_0^{\mathfrak{q}} \diam(Q)^{\bdmn+\mathfrak{q}-\mathfrak{q}\theta}.
\end{equation*}
By Lemma~\ref{lem:averaged:nearby}, for some $\upsilon\in (0,1/4)$ depending only on~$c$,
\begin{equation*}\biggl(\fint_{B(y, c\delta(y)/4)} |u|^\beta \biggr)^{\mathfrak{q}/\beta}
=\mathcal{W}_{c/4,\beta}u(y)^{\mathfrak{q}}
\leq C\fint_{B(y,\upsilon\delta(y))} \mathcal{W}_{c/2,\beta}u(y)^{\mathfrak{q}}
\end{equation*}
and because $\delta\approx\delta(y)$ in $B(y,\upsilon\delta(y))$,
\begin{equation*}\biggl(\fint_{B(y, c\delta(y)/4)} |u|^\beta \biggr)^{\mathfrak{q}/\beta}
\leq C\delta(y)^{-\bdmn+\mathfrak{q}\theta} \int_{B(y,\upsilon\delta(y))} \mathcal{W}_{c/2,\beta}u(y)^{\mathfrak{q}} \delta^{\bdmn-\pdmn-q\theta}
.\end{equation*}
Combining with the above estimate yields that
\begin{equation*}\biggl(\fint_{B(y, c\delta(y)/4)} |u|^\beta \biggr)^{1/\beta}
\leq C\delta(y)^{-\bdmn/\mathfrak{q}+\theta} C_0\diam(Q)^{\bdmn/\mathfrak{q}+1-\theta}
.\end{equation*}
We wish to apply the interior Moser estimate~\eqref{eqn:interior:Moser}. We need only show that $Lu=0$ in $B(y,c\delta(y)/4)$, that is, that $B(y,c\delta(y)/4)\cap Q=\emptyset$.

But $\delta(y)\leq \dist(y,Q)+\diam(Q)+\dist(Q,\partial\Omega)\leq \dist(y,Q)+4\diam(Q)$. Because $\dist(y,Q)>60\diam(Q)$, we have that $\delta(y)< \frac{16}{15}\dist(y,Q)$. Because $c\delta(y)/4<\delta(y)/4<\frac{15}{16}\delta(y)< \dist(y,Q)$, we have that $B(y,c\delta(y)/4)\cap Q=\emptyset$, as desired. Thus we may apply the estimate~\eqref{eqn:interior:Moser} to see that
\begin{equation*}u(y)\leq C\frac{\diam(Q)^{\bdmn/\mathfrak{q}+1-\theta}} {\delta(y)^{\bdmn/\mathfrak{q}-\theta}}.\end{equation*}
Recalling that ${\mathfrak{a}^*}=1-\theta-\bdmn(1-1/\mathfrak{q})$ and that $\delta(y)\approx\dist(y,Q)$ completes the proof.
\end{proof}

\subsection{The region far from the boundary}

In Lemma~\ref{lem:local:Poisson}, we established estimates on the solutions $u$ in Proposition~\ref{prp:atom:Poisson} near the support of the data $\vec H$. In this section we will establish estimates on $u$ far from the data, and also far from the boundary. In the next section we will consider the region near the boundary.

\begin{lem}\label{lem:atom:far:Poisson}%
Let $u$, $\Omega$, $p$, $s$, and~$Q$ be as in Proposition~\ref{prp:atom:Poisson}. Let $0<b<1/2$. Then
\begin{equation*}
\int_{\{y\in\Omega:\delta(y)\geq b\dist(y,Q)\}} \mathcal{W}_{c,\beta}(\nabla u)^p \,\delta^{\bdmn-\pdmn+p-ps}\leq C\diam(Q)^{\bdmn+p-ps}.
\end{equation*}
\end{lem}

\begin{proof}
Let $\Omega_b=\{y\in\Omega:\delta(y)\geq b\dist(y,Q)\}$. We seek to bound an integral over $\Omega_b$. If $0<b<1/2$ and $\Omega_{b,\vartheta}$ is as in Lemma~\ref{lem:local:Poisson}, then by Lemma~\ref{lem:local:Poisson} we have that
\begin{equation*}\int_{\Omega_{b,\vartheta}} \mathcal{W}_{c,\beta}(\nabla u)^p \,\delta^{\bdmn-\pdmn+p-ps}\leq C_{c,\dmn,\bdmn,\beta,\vartheta,b}\diam(Q)^{\bdmn+p-ps}.\end{equation*}
We thus need only estimate the integral over
\begin{equation*}\widetilde \Omega_{b,\vartheta}=
\Omega_b\setminus \Omega_{b,\vartheta}
=
\{y\in\Omega:y\notin\vartheta Q,\>\delta(y)\geq b\dist(y,Q)\}.\end{equation*}

Let $1-\widetilde c=\frac{1}{2}(1-c)$. Choose $\vartheta>0$ large enough (depending on $c$ and~$\dmn$) that $\dist(y,Q)>60\diam(Q)+\widetilde c\delta(y)$ for all $y\in \widetilde \Omega_{b,\vartheta}$. Thus Lemma~\ref{lem:pointwise:Poisson} (or Lemma~\ref{lem:pointwise:extrapolation}) applies not only to~$y$ but to every point in $B(y,\widetilde c\delta(y))$.

We may thus apply Meyers's inequality~\eqref{eqn:interior:Meyers} or Hölder's inequality, and then the Caccioppoli inequality, to see that
\begin{align*}\mathcal{W}_{c,\beta}(\nabla u)(y)
&=\biggl(\fint_{B(y,c\delta(y))} |\nabla u|^\beta\biggr)^{1/\beta}
\leq
\frac{C_c}{\delta(y)}
\biggl(\fint_{B(y,\widetilde c\delta(y))} |u|^2\biggr)^{1/2}
.\end{align*}
Then by Lemma~\ref{lem:pointwise:Poisson} or~\ref{lem:pointwise:extrapolation},
\begin{align*}\mathcal{W}_{c,\beta}(\nabla u)(y)
&\leq
\frac{C}{\delta(y)}
\frac{(\diam Q)^{\bdmn+{\mathfrak{a}^*}}} {\dist(y,Q)^{\bdmn-1+{\mathfrak{a}^*}}}
\end{align*}
If $y\in \widetilde \Omega_{b,\vartheta}$ then $\delta(y)\approx \dist(y,Q)$ and so
\begin{equation*}
\int_{\widetilde \Omega_{b,\vartheta}} \mathcal{W}_{c,2}(\nabla u)^p \,\delta^{\bdmn-\pdmn+p-ps}
\\\leq
C\int_{\R^\dmn\setminus \vartheta Q}
\frac{(\diam Q)^{p\bdmn+p{\mathfrak{a}^*}}} {\dist(y,Q)^{p\bdmn+p{\mathfrak{a}^*}-\bdmn+\pdmn-p+ps}}
\,dy
.\end{equation*}
By assumption,
\begin{equation*}\frac{\bdmn}{p}<\bdmn-1+{\mathfrak{a}^*}+s \end{equation*}
and so
\begin{equation*}p\bdmn+p{\mathfrak{a}^*}-\bdmn-p+ps>0. \end{equation*}
Thus the integral converges and
\begin{equation*}
\int_{\widetilde \Omega_{b,\vartheta}} \mathcal{W}_{c,2}(\nabla u)^p \,\delta^{\bdmn-\pdmn+p-ps}
\leq
C(\diam Q)^{\bdmn+p-ps}
\end{equation*}
as desired.
\end{proof}

\subsection{The region near the boundary}

We will now establish the following estimate on the solution $u$ in Proposition~\ref{prp:atom:Poisson} in small balls centered on the boundary.

\begin{lem}\label{lem:W:near:0}
Let $\Omega$, $\dmn$, $\bdmn$, and $L$ be as in Theorem~\ref{thm:Poisson}.

Let $\zeta\in\partial\Omega$, let $0<r<\diam(\partial\Omega)$, and let $u\in \widetilde W^{1,2}_{av,2,1/2}(\Omega)\cap \widetilde W^{1,2}_{av,\beta,1/2}(\Omega)$ satisfy $Lu=0$ in $B(\zeta,14r)\cap \Omega$ for some $p_{L^*}^-<\beta<p_L^+$.

If $0<c\leq 2/11$, $s<1$, and $0<p\leq 1$, then
\begin{align*}
\biggl(\int_{B(\zeta,r)\cap\Omega} \mathcal{W}_{c,\beta}(\nabla u)^p\,
\delta^{\bdmn-\pdmn+p-ps}\biggr)^{1/p}
&\leq
C
r^{1/2-\bdmn/2+\bdmn/p-s}
\|u\|_{\dot W^{1,2}_{av,c,\beta,1/2}(\Omega\cap  B(\zeta,3C_2r);\Omega)}
\end{align*}
where $C_2$ is as in Lemma~\ref{lem:green:solution:2}.
\end{lem}

\begin{proof}
By the Vitali covering lemma, if $0\leq j\leq \infty$ then we may write
\begin{equation*}\{y\in B(\zeta,r)\cap\Omega: 2^{-j-1}r<\delta(y)\leq 2^{-j}r\}
\subseteq \bigcup_{\ell=1}^{L_j} B(y_{j,\ell}, 2^{-j-2}r)\end{equation*}
where each $y_{j,\ell}$ satisfies $y_{j,\ell}\in \Omega\cap B(\zeta,r)$, $2^{-j-1}r<\delta(y_{j,\ell})\leq 2^{-j}r$, and where the balls $B(y_{j,\ell},2^{-j-2}r/5)$ are pairwise disjoint. By Lemma~\ref{lem:close:to:Ahlfors}, this implies that $L_j\leq C2^{j\bdmn}$.
We then have that
\begin{equation*}B(\zeta,r)\cap\Omega\subset \bigcup_{j=0}^\infty \bigcup_{\ell=1}^{L_j} B(y_{j,\ell},2^{-j-2}r)\end{equation*}
and so
\begin{align*}
\int_{B(\zeta,r)\cap\Omega} \mathcal{W}_{c,\beta}(\nabla u)^p
\delta^{\bdmn-\pdmn+p-ps}
&\leq
\sum_{j=0}^\infty
\sum_{\ell=0}^{L_j}
\int_{B(y_{j,\ell},2^{-j-2}r)\cap\Omega} \mathcal{W}_{c,\beta}(\nabla u)^p
\delta^{\bdmn-\pdmn+p-ps}
\\&\leq
C\sum_{j=0}^\infty
(2^{-j}r)^{\bdmn+p-ps}
\sum_{\ell=0}^{L_j}
\sup_{B(y_{j,\ell},2^{-j-2}r)\cap\Omega} \mathcal{W}_{c,\beta}(\nabla u)^p
.\end{align*}

Suppose that $x\in B(y_{j,\ell},2^{-j-2}r)$ for some such~$j$ and~$\ell$. Then $x\in B(\zeta,(5/4)r)$ and so $B(x,\delta(x))\subset B(\zeta,(5/2)r)\cap \Omega$, and so $Lu=0$ in $B(x,\delta(x))$. We may thus apply Holder's inequality or Meyers's inequality~\eqref{eqn:interior:Meyers}, and then the Caccioppoli inequality, to see that
\begin{align*}\mathcal{W}_{c,\beta}(\nabla u)(x)
=\biggl(\fint_{B(x,c\delta(x))} |\nabla u|^\beta\biggr)^{1/\beta}
&\leq
\frac{C_c}{\delta(x)}
\biggl(\fint_{B(x,2c\delta(x))} | u|^2\biggr)^{1/2}
.\end{align*}
Because $\delta(x)>\delta(y_{j,\ell})-2^{-j-2}r> 2^{-j-1}r-2^{-j-2}r=2^{-j-2}r$, we have that
\begin{align*}
\mathcal{W}_{c,\beta}(\nabla u)(x)
&\leq
\frac{C_c}{2^{-j}r}
\biggl(\fint_{B(x,2c\delta(x))} | u|^2\biggr)^{1/2}
.\end{align*}

Our ultimate goal is to apply Lemma~\ref{lem:green:solution:2} to bound $u$ in the right hand integral. Specifically, we will apply Lemma~\ref{lem:green:solution:2} with $\xi=\xi_{j,\ell}$, where $\xi_{j,\ell}\in\partial\Omega$ and $|y_{j,\ell}-\xi_{j,\ell}|=\delta(y_{j,\ell})$.

If $z\in B(x,2c\delta(x))$ and $x\in B(y_{j,\ell},2^{-j-2}r)$, then by the triangle inequality and because $\delta(x)<\delta(y_{j,\ell})+2^{-j-2}r<(5/4)2^{-j}r$, we have that
\begin{align*}
|z-\xi_{j,\ell}|
&\leq
|z-x|+|x-y_{j,\ell}|+|y_{j,\ell}-\xi_{j,\ell}|
\\&<
(1+2c)(5/4)2^{-j}r
.\end{align*}
Because $c\leq2/11$, we have that $(1+2c)(5/4)\leq75/44<2$. Thus
\begin{equation*}B(x,2c\delta(x))\subset B(\xi_{j,\ell},2^{1-j}r)
\subseteq B(\xi_{j,\ell},2r)\end{equation*}
and so
\begin{equation*}\sup_{B(y_{j,\ell},2^{-j-2}r)}
\mathcal{W}_{c,\beta}(\nabla u)
\leq
\frac{C_c}{2^{-j}r}
\sup_{B(\xi_{j,\ell},2^{1-j}r)} |u|.\end{equation*}
Thus
\begin{align*}
\int_{B(\zeta,r)\cap\Omega} \mathcal{W}_{c,\beta}(\nabla u)^p
\delta^{\bdmn-\pdmn+p-ps}
&\leq
C\sum_{j=0}^\infty
(2^{-j}r)^{\bdmn-ps}
\sum_{\ell=0}^{L_j}
\sup_{y\in B(\xi_{j,\ell},2^{1-j}r)} |u(y)|^p
.\end{align*}

We now bound $u(y)$ for $y$ in some $B(\xi_{j,\ell},2^{1-j}r)$. We wish to apply Lemma~\ref{lem:green:solution:2} with $\xi=\xi_{j,\ell}$ and with $r$ replaced by~$2r$. Observe that $|\xi_{j,\ell}-\zeta|\leq |\xi_{j,\ell}-y_{j,\ell}|+|y_{j,\ell}-\zeta| < \delta(y_{j,\ell}) + r<2r$ and so $B(\xi_{j,\ell},12r)\subset B(\zeta,14r)$. We thus have that $Lu=0$ in $\Omega\cap B(\xi_{j,\ell},12r)$. Recall that $u\in \widetilde W^{1,2}_{av,\beta,1/2}(\Omega)\cap \widetilde W^{1,2}_{av,2,1/2}(\Omega)$. Thus Lemma~\ref{lem:green:solution:2} applies and we see that, if $y\in B(\xi_{j,\ell},2^{1-j}r)\cap \Omega \subseteq B(\xi_{j,\ell},2r)\cap\Omega$, then
\begin{equation*}
|u(y)|
\leq
\frac{C}{r^{\bdmn/2+3/2}}
\|u\|_{\dot W^{1,2}_{av,c,\beta,1/2}(\Omega\cap  B(\xi_{j,\ell},2C_2r);\Omega)}
\int_{\Omega\cap B(\xi_{j,\ell},14r)\setminus B(\xi_{j,\ell},4r)}  G^{L^*}_y\,\delta^{\bdmn+1-\pdmn}
.\end{equation*}
We have that $B(\xi_{j,\ell},2C_2r)\subset B(\zeta,2C_2r+2r)\subset B(\zeta,3C_2r)$.
For conciseness let
\begin{equation*}U_r=\frac{1}{r^{\bdmn/2+3/2}}
\|u\|_{\dot W^{1,2}_{av,c,\beta,1/2}(\Omega\cap  B(\zeta,3C_2r);\Omega)}.\end{equation*}
We apply Lemma~\ref{lem:harmonic:controls:Green} with $r$ replaced by $r'=\frac{3}{4}(2^{1-j}r)<\frac{3}{2}r$. Note that if $x\in \Omega\cap B(\xi_{j,\ell},14r)\setminus B(\xi_{j,\ell},4r)$ and $y\in B(\xi_{j,\ell},2^{1-j}r)\cap\Omega\subset B(\xi_{j,\ell},2r)$, then $|y-x|\geq 2r>r'$ and so Lemma~\ref{lem:harmonic:controls:Green} applies.
Thus
\begin{equation*}|u(y)|\leq CU_r
\int_{\Omega\cap B(\zeta,16r)} \delta(x)^{1+\bdmn-\pdmn} \frac{\omega^x_{L^*}(\Delta(\xi_{j,\ell},12(2^{-j}r)))} {(2^{-j}r)^{\bdmn-1}} \,dx.
\end{equation*}

If $f:\partial\Omega\to\R$ is continuous and compactly supported, let $u_f^*$ be the solution to the Dirichlet problem~\eqref{eqn:Dirichlet:cts} for $L^*$ with boundary data~$f$. Then $u_f^*$ is nonnegative if $f$ is, and
\begin{equation*}\int_{\Omega\cap B(\zeta,16r)} \delta(x)^{1+\bdmn-\pdmn} \,u_f^*(x)\,dx
\leq Cr^{\bdmn+1}\sup_{\partial\Omega} f
\end{equation*}
by Lemma~\ref{lem:close:to:Ahlfors}. Thus by the Riesz-Markov theorem, there is a unique Radon measure $\nu=\nu_{\zeta,r}$ on $\partial\Omega$ that satisfies
\begin{equation*}\int_{\partial\Omega} f\,d\nu = \int_{\Omega\cap B(\zeta,16r)} \delta(x)^{1+\bdmn-\pdmn} \,u_f^*(x)\,dx.\end{equation*}
Then
\begin{equation*}\nu(\partial\Omega)\leq Cr^{\bdmn+1}.\end{equation*}
Letting $f=1$ in $\Delta(\xi_{j,\ell},12(2^{-j}r))$ and $f=0$ outside $\Delta(\xi_{j,\ell},13(2^{-j}r))$, with $f$ continuous and with $0\leq f\leq 1$, we see that
\begin{equation*}\omega^x_{L^*}(\Delta(\xi_{j,\ell},12(2^{-j}r)))
= \int_{\Delta(\xi_{j,\ell},12(2^{-j}r))} 1\,d\omega_{L^*}^x
\leq \int_{\partial\Omega} f\,d\omega_{L^*}^x = u_f^*(x)\end{equation*}
and so if $y\in B(\xi_{j,\ell},2^{1-j}r)\cap\Omega$, then
\begin{align*}|u(y)|
&\leq \frac{CU_r} {(2^{-j}r)^{\bdmn-1}}
\int_{\Omega\cap B(\zeta,16r)} \delta(x)^{1+\bdmn-\pdmn} u_f^*(x) \,dx
=\frac{CU_r} {(2^{-j}r)^{\bdmn-1}}
\int_{\partial\Omega} f\,d\nu
\\&\leq
\frac{CU_r} {(2^{-j}r)^{\bdmn-1}}
\nu(\Delta(\xi_{j,\ell},13(2^{-j}r)))
.\end{align*}

Thus
\begin{align*}
\int_{B(\zeta,r)\cap\Omega} \mathcal{W}_{c,\beta}(\nabla u)^p
\delta^{\bdmn-\pdmn+p-ps}
&\leq
\sum_{j=0}^\infty
\sum_{\ell=0}^{L_j}
\frac{C(U_r)^p(2^{-j}r)^{\bdmn-ps}} {(2^{-j}r)^{p\bdmn-p}}
\bigl(
\nu(\Delta(\xi_{j,\ell},2^{4-j}r))\bigr)^p
.\end{align*}
Recall that $p\leq 1$. We may thus apply Hölder's inequality in sequence spaces to see that
\begin{multline*}
\int_{B(\zeta,r)\cap\Omega} \mathcal{W}_{c,\beta}(\nabla u)^p
\delta^{\bdmn-\pdmn+p-ps}
\\\leq
C(U_r)^p
\sum_{j=0}^\infty
\frac{(2^{-j}r)^{\bdmn-ps}}{(2^{-j}r)^{p\bdmn-p}}
\biggl(\sum_{\ell=0}^{L_j}
\nu(\Delta(\xi_{j,\ell},2^{4-j}r))
\biggr)^p L_j^{1-p}
.\end{multline*}

Let $z\in \partial\Omega$. Suppose that $z\in \Delta(\xi_{j,\ell},2^{4-j}r)$. Then $|z-y_{j,\ell}|\leq |z-\xi_{j,\ell}|+|\xi_{j,\ell}-y_{j,\ell}|<17(2^{-j}r)$. In particular, $B(y_{j,\ell},2^{-j-2}r/5)\subset B(z,18(2^{-j}r))$.
Recall that the balls $B(y_{j,\ell},2^{-j-2}r/5)$ are pairwise disjoint. Thus there can be at most $C_{\dmn}$ values of $\ell$ with $z\in \Delta(\xi_{j,\ell},2^{4-j}r)$. Thus $\sum_{\ell=0}^{L_j}
\nu(\Delta(\xi_{j,\ell},2^{4-j}r))\leq C\nu(\partial\Omega)$, and so
\begin{align*}
\int_{B(\zeta,r)\cap\Omega} \mathcal{W}_{c,\beta}(\nabla u)^p
\delta^{\bdmn-\pdmn+p-ps}
&\leq
C(U_r)^p
\sum_{j=0}^\infty
\frac{(2^{-j}r)^{\bdmn-ps}}{(2^{-j}r)^{p\bdmn-p}}
\bigl(
\nu(\partial\Omega)
\bigr)^p L_j^{1-p}
.\end{align*}
Recall that $L_j\leq C2^{j\bdmn}$ and $\nu(\partial\Omega)\leq Cr^{\bdmn+1}$. Thus
\begin{align*}
\int_{B(\zeta,r)\cap\Omega} \mathcal{W}_{c,\beta}(\nabla u)^p
\delta^{\bdmn-\pdmn+p-ps}
&\leq
C(U_r)^p
\sum_{j=0}^\infty
\frac{(2^{-j}r)^{\bdmn-ps}}{(2^{-j}r)^{p\bdmn-p}}
r^{p(\bdmn+1)} (2^{j\bdmn})^{1-p}
.\end{align*}
Simplifying the right hand side reveals that
\begin{align*}
\int_{B(\zeta,r)\cap\Omega} \mathcal{W}_{c,\beta}(\nabla u)^p
\delta^{\bdmn-\pdmn+p-ps}
&\leq
C(U_r)^p
r^{2p+\bdmn-ps}
\sum_{j=0}^\infty 2^{-j(p-ps)}
.\end{align*}
The sum converges if $s<1$.
Recalling the definition of $U_r$ completes the proof.
\end{proof}

\subsection{The proof of Proposition~\ref{prp:atom:Poisson}}
\label{sec:Poisson:proof:end}

In this section we will prove Proposition~\ref{prp:atom:Poisson}.

Let $\Omega$, $L$, $\dmn$, $\bdmn$, $u$, $\vec H$, $s$, and~$p$ be as in Proposition~\ref{prp:atom:Poisson}. By Lemma~\ref{lem:averaged:Whitney} we may assume without loss of generality that $c=1/7$.

Let $0<b<1$ (to be chosen later) and let
\begin{equation}
\Omega_b=\{y\in \Omega:\delta(y)\geq b\dist(y,Q)\}\end{equation}
be as in the proof of Lemma~\ref{lem:atom:far:Poisson}.
By Lemma~\ref{lem:atom:far:Poisson}, we have that
\begin{align*}
\int_{\Omega_b} \mathcal{W}_{c,\beta}(\nabla u)^p
\delta^{\bdmn-\pdmn+p-ps}
\leq C_b\diam(Q)^{\bdmn+p-ps}.\end{align*}
To complete the proof of Proposition~\ref{prp:atom:Poisson} it suffices to establish the corresponding bound on the integral over~$\Omega\setminus\Omega_b$ for some positive~$b$.

Suppose that $y\in\Omega\setminus\Omega_b$. There is a $\zeta\in\partial\Omega$ with $|y-\zeta|=\delta(y)<b\dist(y,Q)$. By the triangle inequality
$\dist(y,Q)\leq |y-\zeta|+\dist(\zeta,Q)<b\dist(y,Q)+\dist(\zeta,Q)$. Thus $\dist(y,Q)<\frac{1}{1-b}\dist(\zeta,Q)$ and so $|y-\zeta|=\delta(y)<\frac{b}{1-b}\dist(\zeta,Q)$. Thus
\begin{equation}\label{eqn:Poisson:proof:Omega1}
\Omega\setminus\Omega_b=\{y\in \Omega:\delta(y)<b\dist(y,Q)\}
\subset \bigcup_{\zeta\in\partial\Omega } B\bigl(\zeta,{\textstyle\frac{b}{1-b}} \dist(\zeta,Q)\bigr).\end{equation}
Let $\kappa=b/(1-b)$.
By the Vitali covering lemma, there exist at most countably many points $\zeta_k\in\partial\Omega$ with $B(\zeta_j,\kappa\dist(\zeta_j,Q))\cap B(\zeta_k,\kappa\dist(\zeta_k,Q))=\emptyset$ whenever $j\neq k$ and with
\begin{equation*}\Omega\setminus\Omega_b\subseteq \bigcup_k B(\zeta_k,5\kappa\dist(\zeta_k,Q)).\end{equation*}

By formula~\eqref{eqn:Poisson:proof:Omega1},
\begin{align*}
\int_{\Omega\setminus\Omega_b} \mathcal{W}_{c,\beta}(\nabla u)^p
\delta^{\bdmn-\pdmn+p-ps}
&\leq
\sum_k
\int_{\Omega\cap B(\zeta_k,5\kappa\dist(\zeta_k,Q))}
\mathcal{W}_{c,\beta}(\nabla u)^p
\delta^{\bdmn-\pdmn+p-ps}
.\end{align*}
Let $r_k=5\kappa\dist(\zeta_k,Q)$.
Suppose that $b<1/71$. Then $\kappa=b/(1-b)<1/70$ and so $14r_k=70\kappa\dist(\zeta_k,Q)<\dist(\zeta_k,Q)$. Recall that $Lu=0$ in $\Omega\setminus\overline{Q}\supset B(\zeta_k,\dist(\zeta_k,Q))\cap\Omega$ and that $u\in \widetilde W^{1,2}_{av,2,1/2}(\Omega)\cap \widetilde W^{1,2}_{av,\beta,1/2}(\Omega)$. Thus the conditions of Lemma~\ref{lem:W:near:0} are satisfied and we have that
\begin{align*}
\int_{\Omega\setminus\Omega_b} \mathcal{W}_{c,\beta}(\nabla u)^p
\delta^{\bdmn-\pdmn+p-ps}
&\leq
C
\sum_k
\|u\|_{\dot W^{1,2}_{av,c,\beta,1/2}(\Omega\cap  B(\zeta_k,3C_2r_k);\Omega)}^p
r_k^{p/2-p\bdmn/2+\bdmn-ps}
.\end{align*}

By Lemma~\ref{lem:dilate:Whitney}, we have that $3Q\subset\overline\Omega$. Thus $\zeta_k\notin 2Q$ for all~$k$.
If $j\geq 1$ and $\zeta_k \in 2^{j+1}Q\setminus 2^jQ$, then
\begin{equation*}2^{j-2}\ell(Q)\leq (2^{j-1}-1/2)\ell(Q) \leq
\dist(\zeta_k,Q)\leq (2^j-1/2)\diam(Q)
=(2^j-1/2)\sqrt{\dmn}\ell(Q)\end{equation*}
and so
\begin{equation*}B(\zeta_k,\kappa\dist(\zeta_k,Q))\subset C_\dmn 2^jQ\end{equation*}
for some $C_\dmn$ independent of~$j$.
By disjointness of the balls $B(\zeta_k,\kappa\dist(\zeta_k,Q))$, we have that there can be at most $C_\dmn$ such values of~$k$. Let $I_j=\{k:\zeta_k \in 2^{j+1}Q\setminus 2^jQ\}$.
Thus
\begin{equation*}
\int_{\Omega\setminus\Omega_b} \mathcal{W}_{c,\beta}(\nabla u)^p
\delta^{\bdmn-\pdmn+p-ps}
\leq
C
\sum_{j=1}^\infty
\sum_{k\in I_j}
r_k^{p/2-p\bdmn/2+\bdmn-ps}
\|u\|_{\dot W^{1,2}_{av,c,\beta,1/2}(\Omega\cap  B(\zeta_k,3C_2r_k);\Omega)}^p
.\end{equation*}
Observe that $r_k\approx 2^j\diam(Q)$, and so each summand is at most
\begin{equation*}C (2^j\diam(Q))^{p/2-p\bdmn/2+\bdmn-ps}
\|u\|_{\dot W^{1,2}_{av,c,\beta,1/2}(\Omega)}^p\end{equation*}
which by the bound~\eqref{eqn:p<1:basic} is at most
\begin{equation*}C2^{j(p/2-p\bdmn/2+\bdmn-ps)}\diam(Q)^{\bdmn+p-ps}.\end{equation*}
The quantity $p/2-p\bdmn/2+\bdmn-ps$ may be positive, and so the sum in $j$ may not converge; however, we can use this estimate to bound the first few terms. Thus, if $J\geq 0$ is an integer, then
\begin{multline}
\label{eqn:p<1:step1}
\int_{\Omega\setminus\Omega_b} \mathcal{W}_{c,\beta}(\nabla u)^p
\delta^{\bdmn-\pdmn+p-ps}
\leq
C_J\diam(Q)^{\bdmn+p-ps}
\\+
C
\sum_{j=J}^\infty
\sum_{k\in I_j}
r_k^{p/2-p\bdmn/2+\bdmn-ps}
\|u\|_{\dot W^{1,2}_{av,c,\beta,1/2}(\Omega\cap  B(\zeta_k,3C_2r_k);\Omega)}^p
.\end{multline}
We must bound the final term for some~$J$ depending only on the standard constants. Let
\begin{equation*}U_k=r_k^{p/2-p\bdmn/2+\bdmn-ps}
\|u\|_{\dot W^{1,2}_{av,c,\beta,1/2}(\Omega\cap  B(\zeta_k,3C_2r_k);\Omega)}^p
.\end{equation*}
Recall that $r_k=5\kappa\dist(\zeta_k,Q)$.
If $\kappa<1/80C_2$, then $B(\zeta_k,16C_2r_k)\cap Q=\emptyset$. Thus by the bound~\eqref{eqn:interior:Meyers} or Hölder's inequality
\begin{equation*}U_k\leq Cr_k^{p/2-p\bdmn/2+\bdmn-ps}
\|u\|_{\dot W^{1,2}_{av,2c,2,1/2}(\Omega\cap  B(\zeta_k,3C_2r_k);\Omega)}^p .\end{equation*}
If $x\in B(\zeta_k,3C_2r_k)$ and $z\in B(x,2c\delta(x))$, then $|z-\zeta_k|\leq (2c+1)|x-\zeta_k|<(1+2c)3C_2r_k$. Thus by Corollary~\ref{cor:averaged:is:Wp},
\begin{equation*}
U_k\leq Cr_k^{p/2-p\bdmn/2+\bdmn-ps}
\biggl(\int_{\Omega\cap  B(\zeta_k,3(1+2c)C_2r_k)} |\nabla u|^2\delta^{\bdmn+1-\pdmn}\biggr)^{p/2}
.\end{equation*}
Because $c\leq 1/6$ we have that $3(1+2c)\leq 4$. By the boundary Caccioppoli inequality
\begin{equation}
\label{eqn:p<1:step}
U_k\leq Cr_k^{-p/2-p\bdmn/2+\bdmn-ps}
\biggl(\int_{\Omega\cap  B(\zeta_k,8C_2r_k)} |u|^2\delta^{\bdmn+1-\pdmn}\biggr)^{p/2}
.\end{equation}

The proof now diverges based on whether the condition~\eqref{eqn:p<1:DGN} or the condition~\eqref{eqn:p<1:extrapolation} is valid.

\subsubsection{Proof given the condition~\texorpdfstring{\eqref{eqn:p<1:extrapolation}}{(\ref*{eqn:p<1:extrapolation})}} By Lemma~\ref{lem:boundary:DGN} and Lemma~\ref{lem:close:to:Ahlfors},
\begin{align*}
U_k
&\leq
Cr_k^{\bdmn-\bdmn p-ps-p}
\biggl(\int_{\Omega\cap  B(\zeta_k,16C_2r_k)} |u|\,\delta^{\bdmn+1-\pdmn}\biggr)^{p}
.\end{align*}
By Lemma~\ref{lem:averaged:to:unaveraged}, and because $c\leq 1/6$ and so $16/(1-c)< 20$,
\begin{align*}
U_k
&\leq
Cr_k^{\bdmn-\bdmn p-ps-p}
\biggl(\int_{\Omega\cap  B(\zeta_k,20C_2r_k)} \mathcal{W}_{c,1}(u)\,\delta^{\bdmn+1-\pdmn}\biggr)^{p}
.\end{align*}
Recall that we may require $\mathfrak{q}>1$. By Hölder's inequality and Lemma~\ref{lem:close:to:Ahlfors},
\begin{align*}
U_k
&\leq
Cr_k^{\bdmn-ps -\bdmn p/\mathfrak{q}+p\theta}
\biggl(\int_{\Omega\cap  B(\zeta_k,20C_2r_k)} \mathcal{W}_{c,\beta}(u)^{\mathfrak{q}}\,\delta^{\bdmn-\pdmn-\mathfrak{q}\theta}\biggr)^{p/\mathfrak{q}}
.\end{align*}
By Lemma~\ref{lem:boundary:Poincare}, letting $C_3=20C_1C_2$, we have that
\begin{align*}
U_k
&\leq
Cr_k^{\bdmn-ps -\bdmn p/\mathfrak{q}+p\theta}
\biggl(\int_{\Omega\cap  B(\zeta_k,C_3r_k)} \mathcal{W}_{c,\beta}(\nabla u)^{\mathfrak{q}}\, \delta^{\bdmn-\pdmn+\mathfrak{q}-\mathfrak{q}\theta}\biggr)^{p/\mathfrak{q}}
.\end{align*}
By assumption on $\mathfrak{q}$ and~$\theta$,
\begin{align*}
U_k
&\leq
Cr_k^{\bdmn-ps -\bdmn p/\mathfrak{q}+p\theta}
\diam(Q)^{\bdmn p/\mathfrak{q}+p-p\theta}
.\end{align*}
By the bound~\eqref{eqn:p<1:step1} with $J=1$,
\begin{equation*}
\int_{\Omega\setminus\Omega_b} \mathcal{W}_{c,\beta}(\nabla u)^p
\delta^{\bdmn-\pdmn+p-ps}
\leq
C
\sum_{j=1}^\infty
\sum_{k\in I_j}
r_k^{\bdmn-ps -\bdmn p/\mathfrak{q}+p\theta}
\diam(Q)^{\bdmn p/\mathfrak{q}+p-p\theta}
.
\end{equation*}
If $k\in I_j$ then $r_k\approx 2^j\diam(Q)$. Recall that each $I_j$ contains at most $C$ values of~$k$. Thus
\begin{equation*}
\int_{\Omega\setminus\Omega_b} \mathcal{W}_{c,\beta}(\nabla u)^p
\delta^{\bdmn-\pdmn+p-ps}
\leq
C
\diam(Q)^{\bdmn+p-ps}
\sum_{j=1}^\infty
2^{j(\bdmn-ps -\bdmn p/\mathfrak{q}+p\theta)}
.\end{equation*}
Recall that $\mathfrak{a}^*=1-\theta-\bdmn+\bdmn/\mathfrak{q}$.
By assumption
$\bdmn/p<\bdmn-1+\mathfrak{a}^*+s
=\bdmn/\mathfrak{q}-\theta+s$ and so $\bdmn-ps -\bdmn p/\mathfrak{q}+p\theta$ is negative. Thus the geometric series converges. This completes the proof in this case.

\subsubsection{Proof given the condition~\texorpdfstring{\eqref{eqn:p<1:DGN}}{(\ref*{eqn:p<1:DGN})}}
We now turn to the case ${\mathfrak{a}^*}=\alpha^*$, that is, where the condition~\eqref{eqn:p<1:DGN} holds.
We are now in a position to apply Lemma~\ref{lem:u:7delta}.
Recall the bound~\eqref{eqn:p<1:step}:
\begin{equation*}
U_k\leq Cr_k^{-p/2-p\bdmn/2+\bdmn-ps}
\biggl(\int_{\Omega\cap  B(\zeta_k,8C_2r_k)} |u|^2\delta^{\bdmn+1-\pdmn}\biggr)^{p/2}
.\end{equation*}
By Lemma~\ref{lem:close:to:Ahlfors},
\begin{equation*}
U_k\leq Cr_k^{\bdmn-ps}
\sup_{\Omega\cap  B(\zeta_k,8C_2r_k)} |u|^p
.\end{equation*}

Let $\xi\in\partial\Omega$ satisfy $\dist(\xi,Q)=\dist(Q,\partial\Omega)$ and let $r=(4/3)\diam(Q)$. By Lemma~\ref{lem:dilate:Whitney} with $K=1$ we have that $Q\subset \overline{B(\xi,4\diam(Q))}$ and so $Lu=0$ in $\Omega\setminus \overline{B(\xi,4\diam(Q))}=\Omega\setminus \overline{B(\xi,3r)}$.

If $y\in B(\zeta_k,8C_2r_k)$, then
\begin{align*}
|y-\xi|
&\geq \dist(\zeta_k,Q)-4\diam(Q)-|y-\zeta_k|
\geq (1-40\kappa C_2)\dist(\zeta_k,Q)-4\diam(Q)
\\&\geq (1-40\kappa C_2) 2^{j-2}\ell(Q)-4\diam(Q). \end{align*}
We now require $\kappa<1/80C_2$ so that
\begin{equation*}|y-\xi|\geq (2^{j-3}{\dmn}^{-1/2}-4)\diam(Q).\end{equation*}
Choose $J$ large enough that $2^{j-3}{\dmn}^{-1/2}-4>28\cdot4/3$. Then the conditions of Lemma~\ref{lem:u:7delta} are satisfied with $r=(4/3)\diam(Q)$, and so
\begin{equation*}
U_k\leq Cr_k^{\bdmn-ps}
\frac{\diam(Q)^{\bdmn p/2-p/2+p\alpha^*}}
{(2^j\diam(Q))^{\bdmn p-p+p\alpha^*}}
\|u\|_{\widetilde W^{1,2}_{av,c,\beta,1/2}(\Omega)}^p
.\end{equation*}
Recalling the bound~\eqref{eqn:p<1:basic} and that $r_k\approx 2^j\diam(Q)$, we see that
\begin{equation*}
U_k\leq C
\frac{\diam(Q)^{
\bdmn
+p
-ps
}}
{2^{j(\bdmn p-p+p\alpha^*-\bdmn+ps)}}
.\end{equation*}
By the bound~\eqref{eqn:p<1:step1}, and recalling that there are at most $C$ values of $k$ in each~$I_j$,
\begin{equation*}
\int_{\Omega\setminus\Omega_b} \mathcal{W}_{c,\beta}(\nabla u)^p
\delta^{\bdmn-\pdmn+p-ps}
\leq
C_J\diam(Q)^{\bdmn+p-ps}
\\+
C
\sum_{j=J}^\infty
\frac{\diam(Q)^{\bdmn+p-ps}}
{2^{j(\bdmn p-p+p\alpha^*+ps-\bdmn)}}
.\end{equation*}
Recalling that ${\bdmn}/{p}<\bdmn-1+{\alpha^*}+s$, we see that the exponent is positive and so the geometric series converges. This completes the proof.

\subsection{The \texorpdfstring{$p\leq 1$}{p≤1} case of Theorem~\ref{thm:Poisson}}\label{sec:Dirichlet:atom}

\begin{cor}\label{cor:atom}Let $\Omega$ and~$L$ be as in Theorem~\ref{thm:Poisson}. Let $p_{L^*}^-<\beta<p_L^+$.

Suppose that either $\mathfrak{a}^*=\alpha^*$, as in Proposition~\ref{prp:atom:Poisson}, or that there exists a $\mathfrak{q}\in (0,\infty)$ and a $\mathfrak{a}^*\in (0,1)$ such that $\theta=1-\mathfrak{a}^*-\bdmn(1-1/\mathfrak{q})\in (0,1)$ and such that the $L^{\mathfrak{q}}_{av,\beta,\theta}(\Omega)$-Poisson problem for $L$ is compatibly solvable.

Let $s\in (1-\mathfrak{a}^*,1)$ and let $p$ satisfy
\begin{equation*}\bdmn\leq\frac{\bdmn}{p}<\bdmn-1+{\mathfrak{a}^*}+s.\end{equation*}

Then the conclusion of Theorem~\ref{thm:Poisson} is valid; that is, if $\vec H\in L^p_{av,\beta,s}(\Omega)$, then there is a unique $u\in \dot W^{1,p}_{av,\beta,s}(\Omega)$ that is a solution to the Poisson problem~\eqref{eqn:Poisson}, and in addition $u$ satisfies $\|u\|_{\dot W^{1,p}_{av,1/2,\beta,s}(\Omega)} \leq C\|\vec H\|_{L^p_{av,1/2,\beta,s}(\Omega)}$.

Furthermore, the problem is compatibly well posed; that is, if $\vec H\in L^2_{av,2,1/2}(\Omega)\cap L^p_{av,\beta,s}(\Omega)$ then the solution $u$ in $\dot W^{1,p}_{av,\beta,s}(\Omega)$ coincides with the Lax-Milgram solution in~$\dot W^{1,2}_{av,2,1/2}(\Omega)$.
\end{cor}

\begin{proof}Suppose first that $\vec H\in L^p_{av,\beta,s}(\Omega)\cap L^2_{av,2,1/2}(\Omega)$.
Let $\mathcal{G}=\mathcal{G}_1$ be the grid of Whitney cubes in~$\Omega$. If $Q\in\mathcal{G}$, let $\vec H_Q=\1_Q\vec H$, that is, $\vec H_Q=\vec H$ in~$Q$ and $\vec H_Q=0$ in $\Omega\setminus Q$.
Let $u_Q$ be the solution to the Poisson problem with data~$\vec H_Q$ guaranteed by Lemma~\ref{lem:beta}.

If $\mathcal{Q}\subset\mathcal{G}$ is finite, let $u_{\mathcal{Q}}=\sum_{Q\in\mathcal{Q}} u_Q$ and let $\vec H_{\mathcal{Q}}=\sum_{Q\in\mathcal{Q}}\vec H_Q$.
Because $\vec H\in L^2_{av,2,1/2}(\Omega)$ we have that $\vec H_{\mathcal{Q}}\to \vec H$ in $ L^2_{av,2,1/2}(\Omega)$ as $\mathcal{Q}\to\mathcal{G}$, and so $u_{\mathcal{Q}}$ converges in $\widetilde W^{1,2}_{av,2,1/2}(\Omega)$ to the solution $u$ to the Poisson problem with data $\vec H$ given by Definition~\ref{dfn:Lax-Milgram}. In particular by Lemma~\ref{lem:averaged:Whitney} $u_{\mathcal{Q}}\to u$ in $\dot W^{1,2}(K)$ for any compact subset $K\subset\Omega$. By definition of $\widetilde W^{1,2}_{av,2,1/2}(\Omega)$ we have that $u_{\mathcal{Q}}\to u$ in $L^2_{av,2,3/2}(\Omega)$ (and in particular both lie in that space), and so $u_{\mathcal{Q}}\to u$ in $L^2(K)$ as well. By H\"older's inequality (if $\beta<2$) or by Meyers's estimate~\eqref{eqn:interior:Meyers} (if $\beta>2$) we have that $u_{\mathcal{Q}}\to u$ in $\dot W^{1,\beta}(R)$ for all $R\in\mathcal{G}$.

In particular
\begin{equation*}\sum_{R\in\mathcal{O}}
\|\nabla u\|_{L^\beta(R)}^p \ell(R)^{\bdmn-ps+p-\pdmn p/\beta}
=\lim_{\mathcal{Q}\to\mathcal{G}} \sum_{R\in\mathcal{O}}
\|\nabla u_{\mathcal{Q}}\|_{L^\beta(R)}^p \ell(R)^{\bdmn-ps+p-\pdmn p/\beta}
\end{equation*}
for any $\mathcal{O}\subset\mathcal{G}$ finite. Recall from Lemma~\ref{lem:averaged:Whitney} that
\[
\sum_{R\in\mathcal{G}}
\|\nabla u_{\mathcal{Q}}\|_{L^\beta(R)}^p \ell(R)^{\bdmn-ps+p-\pdmn p/\beta}
\approx \|u_{\mathcal{Q}}\|_{\dot W^{1,p}_{av,1/2,\beta,s}(\Omega)}.\]

The given assumptions imply that each pair $(\vec H_Q,u_Q)$ satisfy the conditions of Proposition~\ref{prp:atom:Poisson}, and so for each~$Q$
\begin{equation*}\|u_Q\|_{\dot W^{1,p}_{av,1/2,\beta,s}(\Omega)}
\leq C \diam(Q)^{\bdmn/p-s+1-\pdmn/\beta} \|\vec H\|_{L^\beta(Q)}.\end{equation*}Because $p\leq 1$, we have by the triangle inequality that
\begin{align*}\|u_{\mathcal{Q}}\|_{\dot W^{1,p}_{av,1/2,\beta,s}(\Omega)}^p
\leq \sum_{Q\in\mathcal{Q}}\|u_{Q}\|_{\dot W^{1,p}_{av,1/2,\beta,s}(\Omega)}^p
&\leq C \sum_{Q\in\mathcal{Q}} \|\vec H\|_{L^\beta(Q)}^p \ell(Q)^{\bdmn-ps+p-\pdmn p/\beta}
\end{align*}
whenever $\mathcal{Q}\subset\mathcal{G}$ is finite.

By Lemma~\ref{lem:averaged:Whitney},
\begin{align*}
\sum_{Q\in\mathcal{Q}} \|\vec H\|_{L^\beta(Q)}^p \ell(Q)^{\bdmn-ps+p-\pdmn p/\beta}
&\approx \|\vec H_{\mathcal{Q}}\|_{L^p_{av,1/2,\beta,s}(\Omega)}^p
.\end{align*}
Thus
\begin{align*}\|u_{\mathcal{Q}}\|_{\dot W^{1,p}_{av,1/2,\beta,s}(\Omega)}
&\leq C \|\vec H_{\mathcal{Q}}\|_{L^p_{av,1/2,\beta,s}(\Omega)}
\leq C \|\vec H\|_{L^p_{av,1/2,\beta,s}(\Omega)}.
\end{align*}
Combining the above estimates and applying Lemma~\ref{lem:averaged:Whitney} again yields that
\[\|u\|_{\dot W^{1,p}_{av,1/2,\beta,s}(\Omega)}
=
\sup_{\substack{\mathcal{O}\subset\mathcal{G}\\\mathcal{O}\text{ finite}}} \sum_{R\in\mathcal{O}}
\|\nabla u\|_{L^\beta(R)}^p \ell(R)^{\bdmn-ps+p-\pdmn p/\beta}
\leq C\|\vec H\|_{L^p_{av,1/2,\beta,s}(\Omega)}.\]
Applying Corollary~\ref{cor:weighted:embedding} and Lemma~\ref{lem:boundary:Poincare} yields that $\|u\|_{L^1_{av,1/2,\beta,1+\varsigma}(\Omega)}
\leq C\|u\|_{\dot W^{1,p}_{av,1/2,\beta,s}(\Omega)}$ where $\varsigma=s+\bdmn-\bdmn/p$. Thus $\|u\|_{\widetilde W^{1,p}_{av,1/2,\beta,s}(\Omega)}\leq C \|u\|_{\dot W^{1,p}_{av,1/2,\beta,s}(\Omega)}$ and so the $L^p_{av,\beta,s}(\Omega)$-Poisson problem for $L$ is compatibly solvable.

By Remark~\ref{rmk:compatible:solvable} we have solvability of the $L^p_{av,\beta,s}(\Omega)$-Poisson problem for $L$. We are left with uniqueness. If $p=1$ then uniqueness (and thus compatible well posedness) follows immediately from Lemma~\ref{lem:uniqueness}.

If $p<1$, recall that $\varsigma$ above satisfies $\bdmn-\varsigma=\bdmn/p-s$. Then $\varsigma<s<1$. Because $0<\bdmn-1+\mathfrak{a}^*+s-\bdmn/p$ we have that $0<\bdmn-1+\mathfrak{a}^*+\varsigma-\bdmn/1$ and $1-\mathfrak{a}^*<\varsigma$. Thus the pair $(1,\varsigma)$ satisfies the conditions on $(p,s)$ assumed above, and so by the above analysis the $L^1_{av,\beta,\varsigma}(\Omega)$-Poisson problem is compatibly well posed.

But by Corollary~\ref{cor:weighted:embedding}, if $u\in \widetilde W^{1,p}_{av,\beta,s}(\Omega)$ is the solution to the Poisson problem with data $\vec H\in L^{p}_{av,\beta,s}(\Omega)$, then $u\in \widetilde W^{1,1}_{av,\beta,\varsigma}(\Omega)$ and $\vec H\in L^{1}_{av,\beta,\varsigma}(\Omega)$.
Uniqueness of solutions to the $L^{1}_{av,\beta,\varsigma}(\Omega)$-Poisson problem yields that $u$ is the only solution to the Poisson problem in $\widetilde W^{1,1}_{av,\beta,\varsigma}(\Omega)$, and thus is the only solution to the $L^p_{av,\beta,s}(\Omega)$-Poisson problem. This completes the proof.
\end{proof}

\section{The proof of Theorem~\ref{thm:Poisson}}
\label{sec:thm:Poisson}

We have now proven Theorem~\ref{thm:Poisson} in the case $p\leq 1$ and $s>\bdmn/p+1-\bdmn-\alpha^*$ and the case $p=2$, $s=1/2$. See Figure~\ref{fig:Poisson:final}. We now complete the proof.

\begin{figure}
\begin{tikzpicture}[scale=1.9]
\draw [dotted] (1,0)--(1,1)--(0,1);
\draw [->] (-0.1,0)--(1.3,0) node [below] {$\vphantom{1}s$};
\draw [->] (0,-0.3)--(0,1.4) node [above] {$1/p$};
\fill  (1,1)  -- (1-\figurealpha,1) --(1,1+\figurealpha/\figuredimen) -- cycle;
\fill (1/2,1/2) circle (0.5pt);
\end{tikzpicture}
\hfill
\begin{tikzpicture}[scale=1.9]
\draw [dotted] (1,0)--(1,1)--(0,1);
\draw [->] (-0.1,0)--(1.3,0) node [below] {$\vphantom{1}s$};
\draw [->] (0,-0.3)--(0,1.4) node [above] {$1/p$};
\fill [white!60!black] (1/2,1/2)--(1,1) -- (1-\figurealpha,1) -- node [left] {$T_1^L$} cycle;
\fill  (1,1)  -- (1-\figurealpha,1) --(1,1+\figurealpha/\figuredimen) -- cycle;
\fill (1/2,1/2) circle (0.5pt);
\end{tikzpicture}\hfill
\begin{tikzpicture}[scale=1.9]
\draw [dotted] (1,0)--(1,1)--(0,1);
\draw [->] (-0.1,0)--(1.3,0) node [below] {$\vphantom{1}s$};
\draw [->] (0,-0.3)--(0,1.4) node [above] {$1/p$};
\fill [white!60!black] (1/2,1/2)--(1,1) -- (1-\figurealpha,1) -- cycle;
\fill [white!60!black] (1/2,1/2) -- (0,0) -- (\figurealpha,0) -- node [right] {$T_2^L$} cycle;
\fill  (1,1)  -- (1-\figurealpha,1) --(1,1+\figurealpha/\figuredimen) -- cycle;
\fill (1/2,1/2) circle (0.5pt);
\end{tikzpicture}
\hfill
\begin{tikzpicture}[scale=1.9]
\draw [dotted] (1,0)--(1,1)--(0,1);
\draw [->] (-0.1,0)--(1.3,0) node [below] {$\vphantom{1}s$};
\draw [->] (0,-0.3)--(0,1.4) node [above] {$1/p$};
\fill [white!80!black] (1,1) -- (1-\figurealpha,1) -- (0,0)--(\figurealpha,0)--cycle;
\fill [white!60!black] (1/2,1/2)--(1,1) -- (1-\figurealpha,1) -- cycle;
\fill [white!60!black] (1/2,1/2) -- (0,0) -- (\figurealpha,0) -- cycle;
\fill  (1,1)  -- (1-\figurealpha,1) --(1,1+\figurealpha/\figuredimen) -- cycle;
\fill (1/2,1/2) circle (0.5pt);
\end{tikzpicture}

\caption{The regions mentioned in Section~\ref{sec:thm:Poisson} in the case where $\bdmn\geq1-\alpha$ and $\bdmn\geq1-\alpha^*$}\label{fig:Poisson:final}
\end{figure}

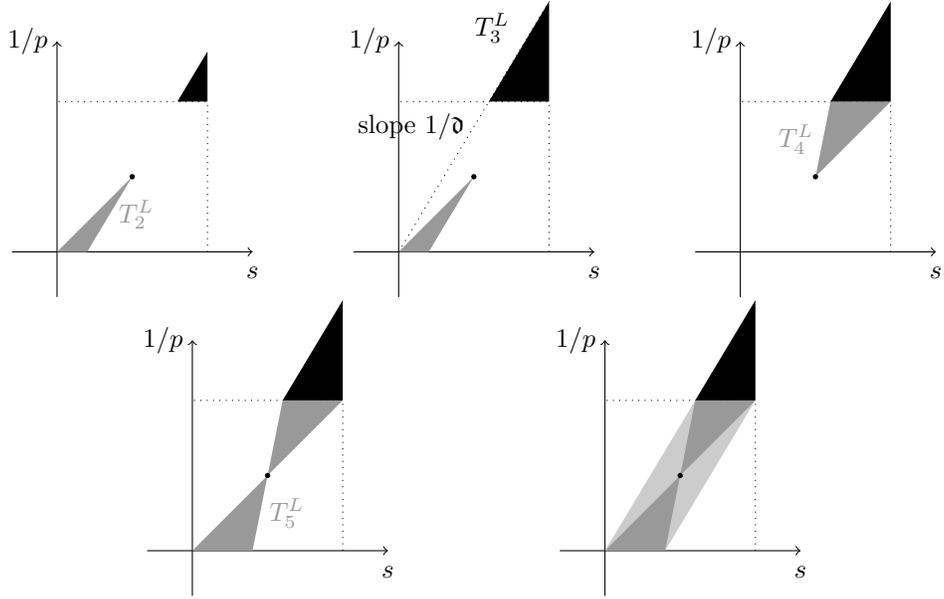
\begin{figure}
\def\figuredimen{0.6}
\def\figurealpha{0.2}
\begin{tikzpicture}[scale=2]
\draw [dotted] (1,0)--(1,1)--(0,1);
\draw [->] (-0.3,0)--(1.3,0) node [below ] {$\vphantom{1}s$};
\draw [->] (0,-0.3)--(0,1.4) node [ left] {$1/p$};
\fill [white!60!black] (1/2,1/2) -- (0,0) -- (\figurealpha,0) -- node [right] {$T_2^L$} cycle;
\fill  (1,1)  -- (1-\figurealpha,1) --(1,1+\figurealpha/\figuredimen) -- cycle;
\fill (1/2,1/2) circle (0.5pt);
\end{tikzpicture}
\hfill
\begin{tikzpicture}[scale=2]
\draw [dotted] (1,0)--(1,1)--(0,1);
\draw [->] (-0.3,0)--(1.3,0) node [below ] {$\vphantom{1}s$};
\draw [->] (0,-0.3)--(0,1.4) node [ left] {$1/p$};
\fill [white!60!black] (1/2,1/2) -- (0,0) -- (\figurealpha,0) -- cycle;
\fill (1/2,1/2) circle (0.5pt);
\draw [dotted] (0,0)--node [left] {slope $1/\bdmn$} (1,1/\figuredimen);
\fill (1,1)--(1,1/\figuredimen)-- node [above left] {$T_3^L$} (\figuredimen,1)--cycle;
\end{tikzpicture}
\hfill
\begin{tikzpicture}[scale=2]
\draw [dotted] (1,0)--(1,1)--(0,1);
\draw [->] (-0.3,0)--(1.3,0) node [below ] {$\vphantom{1}s$};
\draw [->] (0,-0.3)--(0,1.4) node [ left] {$1/p$};
\fill [white!60!black] (1/2,1/2) -- (1,1) -- (\figuredimen,1) -- node [left] {$T_4^L$} cycle;
\fill (1/2,1/2) circle (0.5pt);
\fill (1,1)--(1,1/\figuredimen)-- (\figuredimen,1)--cycle;
\end{tikzpicture}%
\break
\null\hfill
\begin{tikzpicture}[scale=2]
\draw [dotted] (1,0)--(1,1)--(0,1);
\draw [->] (-0.3,0)--(1.3,0) node [below ] {$\vphantom{1}s$};
\draw [->] (0,-0.3)--(0,1.4) node [ left] {$1/p$};
\fill [white!60!black] (1/2,1/2) -- (1,1) -- (\figuredimen,1) -- cycle;
\fill [white!60!black] (1/2,1/2) -- (0,0) -- (1-\figuredimen,0) -- node [right] {$T_5^L$} cycle;
\fill (1/2,1/2) circle (0.5pt);
\fill (1,1)--(1,1/\figuredimen)-- (\figuredimen,1)--cycle;
\end{tikzpicture}
\hfill
\begin{tikzpicture}[scale=2]
\draw [dotted] (1,0)--(1,1)--(0,1);
\draw [->] (-0.3,0)--(1.3,0) node [below ] {$\vphantom{1}s$};
\draw [->] (0,-0.3)--(0,1.4) node [ left] {$1/p$};
\fill [white!80!black] (1,1) -- (\figuredimen,1) -- (0,0)--(1-\figuredimen,0)--cycle;
\fill [white!60!black] (1/2,1/2) -- (1,1) -- (\figuredimen,1) -- cycle;
\fill [white!60!black] (1/2,1/2) -- (0,0) -- (1-\figuredimen,0) -- cycle;
\fill (1/2,1/2) circle (0.5pt);
\fill (1,1)--(1,1/\figuredimen)-- (\figuredimen,1)--cycle;
\end{tikzpicture}\hfill\break

\caption{The regions mentioned in Section~\ref{sec:thm:Poisson} in the case where $\bdmn<1-\alpha$ or $\bdmn<1-\alpha^*$}\label{fig:Poisson:final:no:really}
\end{figure}

Let $p_{L^*}^-<\beta<p_L^+$. By Corollary~\ref{cor:atom}, the $L^{1}_{av,\beta,s}(\Omega)$-Poisson problem for $L$ is compatibly well posed whenever $1-\alpha^*<s<1$. Lemma~\ref{lem:beta} ensures that the $L^{2}_{av,\beta,1/2}(\Omega)$-Poisson problem for $L$ is also compatibly well posed.  We can then use Theorem~\ref{thm:interp:sol} to interpolate between those values, obtaining that the $L^{p}_{av,\beta,s}(\Omega)$-Poisson problem for $L$ is well posed whenever $(s,1/p)$ lies in the open triangle $T^L_1$ with vertices $\{(1-\alpha^*,1), (1,1), (1/2,1/2)\}$.

The same argument with $L^*$ in place of $L$, together with Remark~\ref{rmk:compatible:solvable}, Theorem~\ref{first:step:duality}, and Lemma~\ref{lem:uniqueness}, gives compatible well posedness of the $L^{s}_{av,\beta,p}(\Omega)$-Poisson problem for $L$, whenever $p_{L^*}^-<\beta<p_L^+$ and $(s,1/p)$ lies in the open triangle $T^L_2$ spanned by $\{(\alpha,0), (0,0), (1/2,1/2))\}$, that is the symmetrization of $T^{L^*}_1$ with respect to the point $(1/2,1/2)$. Applying Lemma~\ref{lem:solv:infty} in place of Remark~\ref{rmk:compatible:solvable} yields compatible well posedness of the $L^\infty_{av,\beta,s}(\Omega)$-Poisson problem for~$L$ whenever $1<s<\alpha$, that is, along the bottom edge of $T^L_2$.

If $\alpha^*<1-\bdmn$, then we may apply Corollary~\ref{cor:atom} with $(\theta,1/\mathfrak{q})$ in the triangle $T^L_2$ and close to the corner $(0,0)$. See Figure~\ref{fig:Poisson:final:no:really}. This yields well posedness whenever $(s,1/p)$ lies in the open triangle $T^L_3$ with vertices $\{(\min(\bdmn,1-\alpha^*),1), (1,1), (1,\max(1/\bdmn,1+\alpha^*/\bdmn))\}$.
An interpolation argument as before
yields well posedness whenever $(s,1/p)$ lies in the triangle $T^L_4$ with vertices
$\{(\max(\alpha,1-\bdmn),0), (0,0), (1/2,1/2))\}$.
A duality argument as before
yields well posedness whenever $(s,1/p)$ lies in the triangle $T^L_5$ with vertices $\{(\max(\alpha,1-\bdmn),0), (0,0), (1/2,1/2))\}$.

In order to complete the proof it is sufficient now to use Theorem~\ref{thm:interp:sol} and interpolate between $T^L_4$ and $T^L_5$.

\section{The proof of Theorem~\ref{thm:Poisson:Lq}}
\label{sec:Poisson:Lp}

In this section we will prove Theorem~\ref{thm:Poisson:Lq}.

Let $\Omega$, $L$, and $\beta$ be as in Theorem~\ref{thm:Poisson}, and let $q$ be as in Theorem~\ref{thm:Poisson:Lq}. By Theorem~\ref{thm:Poisson}, the $L^q_{av,\beta,1/q}(\Omega)$-Poisson problem for $L$ is compatibly well posed. We will show that the $L^q_{av,\beta,s}(\Omega)$-Poisson problem for $L$ is compatibly solvable whenever $0<s<1/q$.

Using interpolation and duality results and the case \eqref{eqn:p<1:extrapolation} of Corollary~\ref{cor:atom}, as in Section~\ref{sec:thm:Poisson}, this yields well posedness of the $L^q_{av,\beta,s}(\Omega)$-Poisson problem for the range listed in Theorem~\ref{thm:Poisson:Lq}.

Let $0<s<1/q$. We now establish compatible solvability of the $L^q_{av,\beta,s}(\Omega)$-Poisson problem for~$L$.

Let $\vec H\in L^2(\Omega)\cap L^\beta(\Omega)$ be compactly supported and let $u$ be the solution to the $L^q_{av,\beta,1/q}(\Omega)$-Poisson problem for~$L$ with data~$\vec H$. (By compatible well posedness, $u$ coincides with the solution to the $L^2_{av,2,1/2}(\Omega)$-Poisson problem.)
Our goal is then to show that
\begin{equation*}\|u\|_{\dot W^{1,q}_{av,c,\beta,s}(\Omega)}
\leq C_{q,s} \|\vec H\|_{L^q_{av,c,\beta,s}(\Omega)}\end{equation*}
for some $c\in (0,1)$.
A density argument will then yield the compatible solvability, and thus the compatible well posedness.

By Lemma~\ref{lem:averaged:Whitney}, if $\mathcal{G}=\mathcal{G}_1$ is the grid of dyadic Whitney cubes in~$\Omega$, then
\begin{align*}\|u\|_{\dot W^{1,q}_{av,c,\beta,s}(\Omega)}
&\approx
\biggl(\sum_{Q\in\mathcal{G}} \biggl(\fint_Q |\nabla u|^\beta\biggr)^{q/\beta} \ell(Q)^{\bdmn+q-qs}\biggr)^{1/q}
.\end{align*}
Partitioning $\mathcal{G}$ by side length yields that
\begin{align*}\|u\|_{\dot W^{1,q}_{av,c,\beta,s}(\Omega)}
&\approx
\biggl(\sum_{k=-\infty}^\infty \sum_{\substack{Q\in\mathcal{G}\\\ell(Q)=2^k}} \biggl(\fint_Q |\nabla u|^\beta\biggr)^{q/\beta} \ell(Q)^{\bdmn+q-qs}\biggr)^{1/q}
.\end{align*}

If $k\in\Z$, let
\begin{equation*} \vec H_k= \sum_{\substack{R\in\mathcal{G}\\\clap{$\scriptstyle\ell(R)\leq 2^{k-6}$}}} \1_R \vec H, \qquad u_k\in \widetilde W^{1,q}_{av,\beta,1/q}(\Omega)\cap \widetilde W^{1,2}_{av,2,1/2}(\Omega),\quad Lu_k=\Div (\delta^{1+\bdmn-\pdmn}\vec H_k).\end{equation*}
That is, $u_k$ is the solution to the Poisson problem with data~$\vec H_k$. Then
$\vec H_k=\vec H$ near $\partial\Omega$ and $\vec H_k=0$ far from $\partial\Omega$.
We have that
\begin{align*}\|u\|_{\dot W^{1,q}_{av,c,\beta,s}(\Omega)}
&\leq
C\biggl(\sum_{k=-\infty}^\infty \sum_{\substack{Q\in\mathcal{G}\\\ell(Q)=2^k}} \biggl(\fint_Q |\nabla u_k|^\beta\biggr)^{q/\beta} \ell(Q)^{\bdmn+q-qs}\biggr)^{1/q}
\\&\qquad+
C\biggl(\sum_{k=-\infty}^\infty \sum_{\substack{Q\in\mathcal{G}\\\ell(Q)=2^k}} \biggl(\fint_Q |\nabla (u-u_k)|^\beta\biggr)^{q/\beta} \ell(Q)^{\bdmn+q-qs}\biggr)^{1/q}
.\end{align*}
To complete the proof, it suffices to show that each of the two terms on the right hand side are at most $C_{q,s} \|\vec H\|_{L^q_{av,c,\beta,s}(\Omega)}$.

\subsection{The term involving \texorpdfstring{$u-u_k$}{u-uk}}
We begin with the second term. Let $v_k=u-u_k$, so $v_k\in \widetilde W^{1,q}_{av,\beta,1/q}(\Omega)$ and $Lv_k=-\Div(\delta^{1+\bdmn-\pdmn}(\vec H-\vec H_k))$. We have that
\begin{equation*}\sum_{\substack{Q\in\mathcal{G}\\\ell(Q)=2^k}} \biggl(\fint_Q |\nabla v_k|^\beta\biggr)^{q/\beta} \ell(Q)^{\bdmn+q-qs}
=
2^{k-kqs}\sum_{\substack{Q\in\mathcal{G}\\\ell(Q)=2^k}} \biggl(\fint_Q |\nabla v_k|^\beta\biggr)^{q/\beta} \ell(Q)^{\bdmn+q-1}.\end{equation*}
By the compatible well posedness of the $L^q_{av,\beta,1/q}(\Omega)$-Poisson problem, we have that
\begin{equation*}\sum_{\substack{Q\in\mathcal{G}\\\ell(Q)=2^k}} \biggl(\fint_Q |\nabla v_k|^\beta\biggr)^{q/\beta} \ell(Q)^{\bdmn+q-1}
\leq
C\|\vec H-\vec H_k\|_{L^q_{av,c,\beta,1/q}(\Omega)}^q.\end{equation*}
But if $\vec H-\vec H_k\neq0$ in $R$ then $\ell(R)\geq 2^{k-5}$, and so by Lemma~\ref{lem:averaged:Whitney}
\begin{equation*}\|\vec H-\vec H_k\|_{L^q_{av,c,\beta,1/q}(\Omega)}^q
\approx\sum_{\substack{R\in\mathcal{G}\\\ell(R)\geq 2^{k-5}}}
\biggl(\fint_R |\vec H|^\beta\biggr)^{q/\beta} \ell(R)^{\bdmn+q-1}
.\end{equation*}
Combining all of these estimates yields that
\begin{multline*}
\smash{\sum_{k=-\infty}^\infty
\sum_{\substack{Q\in\mathcal{G}\\\ell(Q)=2^k}}} \biggl(\fint_Q |\nabla v_k|^\beta\biggr)^{q/\beta} \ell(Q)^{\bdmn+q-qs}
\\\leq
C\smash{\sum_{k=-\infty}^\infty}
2^{k-kqs}
\sum_{\substack{R\in\mathcal{G}\\\ell(R)\geq 2^{k-5}}}
\smash{\biggl(\fint_R |\vec H|^\beta\biggr)^{q/\beta} \ell(R)^{\bdmn+q-1}}
.\end{multline*}
Changing the order of summation yields that
\begin{multline*}
\smash{\sum_{k=-\infty}^\infty
\sum_{\substack{Q\in\mathcal{G}\\\ell(Q)=2^k}}} \biggl(\fint_Q |\nabla v_k|^\beta\biggr)^{q/\beta} \ell(Q)^{\bdmn+q-qs}
\\\leq
C
\sum_{R\in\mathcal{G}}
\smash{\biggl(\fint_R |\vec H|^\beta\biggr)^{q/\beta} \ell(R)^{\bdmn+q-1}}
\sum_{\substack{k\in\Z\\ 2^{k-5}\leq \ell(R)}}
2^{k-kqs}
.\end{multline*}
By assumption $qs<1$, and so the geometric series converges. Thus
\begin{multline*}
\smash{\sum_{k=-\infty}^\infty
\sum_{\substack{Q\in\mathcal{G}\\\ell(Q)=2^k}}} \biggl(\fint_Q |\nabla v_k|^\beta\biggr)^{q/\beta} \ell(Q)^{\bdmn+q-qs}
\\\leq
C\sum_{\substack{R\in\mathcal{G}}}
\smash{\biggl(\fint_R |\vec H|^\beta\biggr)^{q/\beta}} \ell(R)^{\bdmn+q-qs}
\approx \|\vec H\|_{L^q_{av,c,\beta,s}(\Omega)}^q
.
\end{multline*}

\subsection{The term involving \texorpdfstring{$u_k$}{uk}}
Let
\begin{equation*}
I_k=\sum_{\substack{Q\in\mathcal{G}\\\ell(Q)=2^k}} \biggl(\fint_Q |\nabla u_k|^\beta\biggr)^{q/\beta} \ell(Q)^{\bdmn+q-qs}
.
\end{equation*}
To complete the proof we need only show that
\begin{equation*}\sum_{k=-\infty}^\infty I_k\leq C\|\vec H\|_{L^q_{av,c,\beta,s}(\Omega)}^q.\end{equation*}

Suppose that $\ell(R)\leq 2^{k-6}$ and $\ell(Q)=2^k$.
By formula~\eqref{eqn:Whitney:dilate:distance} in Lemma~\ref{lem:dilate:Whitney}, we have that
\begin{equation*}
2^{k-2}\sqrt\dmn
<
2^{k-1}\sqrt\dmn-9\cdot 2^{k-7}\sqrt\dmn
\leq
\frac{1}{2}\diam(Q)
-\frac{9}{2}\diam(R)
\leq\dist(2Q,2R)\end{equation*}
and so $2Q$ and $\supp H_k$ are disjoint.
We may thus apply the bound~\eqref{eqn:interior:Meyers} to see that
\begin{equation*}
I_k\leq
C\sum_{\substack{Q\in\mathcal{G}\\\ell(Q)=2^k}} \biggl(\fint_{(3/2)Q} |\nabla u_k|^2\biggr)^{q/2} \ell(Q)^{\bdmn+q-qs}
.\end{equation*}
By the Caccioppoli inequality and the local bound~\eqref{eqn:interior:Moser},
\begin{equation*}
I_k
\leq
C\sum_{\substack{Q\in\mathcal{G}\\\ell(Q)=2^k}} \biggl(\fint_{2Q} | u_k|\biggr)^{q} \ell(Q)^{\bdmn-qs}
.\end{equation*}
But $2Q\cap\supp \vec H_k=\emptyset$, and so by  formula~\eqref{eqn:green:Poisson}
\begin{equation*}u_k(x)=\sum_{\ell(R)\leq 2^{k-6}} \int_R \nabla G_x^{L^*}\cdot \delta^{1+\bdmn-\pdmn}\vec H
\quad\text{for all }x\in 2Q, \>\ell(Q)=2^k.
\end{equation*}
Thus
\begin{equation*}
I_k
\leq
C\sum_{\substack{Q\in\mathcal{G}\\\ell(Q)=2^k}} \biggl(\fint_{2Q}
\sum_{\ell(R)\leq 2^{k-6}} \int_R \nabla G_x^{L^*}(y)\cdot \delta(y)^{1+\bdmn-\pdmn}\,\vec H(y)\,dy
\,dx\biggr)^{q} \ell(Q)^{\bdmn-qs}
.\end{equation*}
Because $\delta\approx \ell(R)$ in~$R$, we may apply Hölder's inequality to see that
\begin{equation*}
I_k
\leq
C\sum_{\substack{Q\in\mathcal{G}\\\ell(Q)=2^k}} \biggl(\fint_{2Q}
\sum_{\ell(R)\leq 2^{k-6}} \ell(R)^{1+\bdmn-\pdmn}
\|\nabla G_x^{L^*}\|_{L^{\beta'}(R)}
\|\vec H\|_{L^{\beta}(R)}
\,dx\biggr)^{q} \ell(Q)^{\bdmn-qs}
.\end{equation*}
But if $x\in 2Q$ then $x\notin \overline{2R}$. Thus $L^* G_x^{L^*}=0$ in $2R$ by formula~\eqref{eqn:green:solution}. Thus we may apply the reverse Hölder estimate~\eqref{eqn:interior:Meyers} and the Caccioppoli inequality to see that
\begin{equation*}\|\nabla G_x^{L^*}\|_{L^{\beta'}(R)}
\leq C\ell(R)^{\pdmn/\beta'-\pdmn/2-1} \|G_x^{L^*}\|_{L^{2}(2R)}
\leq  C\ell(R)^{\pdmn/\beta'-1} \|G_x^{L^*}\|_{L^{\infty}(2R)}
.\end{equation*}
Let $\xi_R\in\partial\Omega$ with $\dist(R,\partial\Omega)=\dist(R,\xi_R)$.
By Lemma~\ref{lem:dilate:Whitney}, if $y\in 2R$ then $|y-\xi_R|\leq (9/2)\sqrt{\dmn}\ell(R)$. Furthermore, $|x-y|\geq \dist(2Q,2R)$ and as noted above $\dist(2Q,2R)\geq 2^{k-2}\sqrt\dmn
\geq 16\sqrt{\dmn}\ell(R)$.

Thus by the symmetry property~\eqref{eqn:green:symmetric} and Lemma~\ref{lem:harmonic:controls:Green} with $2r=5\sqrt{\dmn}\ell(R)$, if $x\in 2Q$ then
\begin{equation*}
\sup_{y\in 2R} G_x^{L^*}(y)=\sup_{y\in 2R} G_y^L(x)\leq \frac{C}{\ell(R)^{\bdmn-1}} \omega_L^x(\Delta(\xi_R,20\sqrt{\dmn}\ell(R)))
.\end{equation*}
Let $\eta_R$ be a nonnegative Lipschitz continuous function defined on~$\partial\Omega$ that is $1$ in $\Delta(\xi_R,20\sqrt{\dmn}\ell(R))$ and supported in $\Delta(\xi_R,30\sqrt{\dmn}\ell(R))$. Because the harmonic measure is an unsigned (nonnegative) measure, we have that
\begin{equation*}\sup_{2R}G_x^{L^*}\leq \frac{C}{\ell(R)^{\bdmn-1}} \omega_L^x(\Delta(\xi_R,20\sqrt{\dmn}\ell(R)))
\leq \frac{C}{\ell(R)^{\bdmn-1}}
\int_{\partial\Omega}\eta_R\, d \omega_L^x.\end{equation*}
Let $v_R$ be the solution to the Dirichlet problem~\eqref{eqn:Dirichlet:cts} with boundary data~$\eta_R$.
By the definition of the harmonic measure (see Section~\ref{sec:harmonic:measure}) we have that
\begin{equation*}\sup_{2R}G_x^{L^*}
\leq \frac{C}{\ell(R)^{\bdmn-1}}
v_R(x).\end{equation*}
Thus
\begin{equation*}
I_k
\leq
C\sum_{\substack{Q\in\mathcal{G}\\\ell(Q)=2^k}} \biggl(\fint_{2Q}
\sum_{\ell(R)\leq 2^{k-6}} \ell(R)^{1-\pdmn/\beta}
\|\vec H\|_{L^{\beta}(R)}
v_R(x)
\,dx\biggr)^{q} \ell(Q)^{\bdmn-qs}
.\end{equation*}
Recall that $\vec H$ is compactly supported and thus $\|\vec H\|_{L^{\beta}(R)}\neq 0$ for only finitely many values of~$R$. Then
\begin{equation*}w_k(x)=\sum_{\ell(R)\leq 2^{k-6}} \ell(R)^{1-\pdmn/\beta}
\|\vec H\|_{L^{\beta}(R)}
v_R(x)
\end{equation*}
is the solution to the continuous Dirichlet problem~\eqref{eqn:Dirichlet:cts} with boundary data
\begin{equation*}f_k=\sum_{\ell(R)\leq 2^{k-6}} \ell(R)^{1-\pdmn/\beta}
\|\vec H\|_{L^{\beta}(R)}
\eta_R.\end{equation*}
If $a$ is large enough, then by the definition~\eqref{eqn:cone} of a nontangential cone and by Lemma~\ref{lem:dilate:Whitney}, we have that
\begin{equation*}2Q\subset \gamma_a(\xi)\quad\text{for all }\xi\in \Delta(\xi_Q,\ell(Q)).\end{equation*}
Thus
\begin{equation*}
I_k
\leq
C\sum_{\substack{Q\in\mathcal{G}\\\ell(Q)=2^k}} \biggl(\fint_{2Q}
w_k(x)
\,dx\biggr)^{q} \ell(Q)^{\bdmn-qs}
\leq
C\sum_{\substack{Q\in\mathcal{G}\\\ell(Q)=2^k}}
\ell(Q)^{\bdmn-qs}
\fint_{\Delta(\xi_Q,\ell(Q))} N_a(w_k)^q\,d\sigma
.\end{equation*}
By Lemma~\ref{lem:Whitney:boundary},
\begin{equation*}
I_k
\leq
C
2^{-kqs}
\int_{\partial\Omega} N_a(w_k)^q\,d\sigma
.\end{equation*}
Recall that Theorem~\ref{thm:Poisson:Lq} assumes validity of the nontangential bound~\eqref{eqn:N:bound}, and so
\begin{equation*}
I_k
\leq
C
2^{-kqs}
\int_{\partial\Omega} f_k^q\,d\sigma
.\end{equation*}
We compute that
\begin{align*}\biggl(\int_{\partial\Omega} f_k^q\,d\sigma\biggr)^{1/q}
&=
\biggl(\int_{\partial\Omega} \Bigl(\sum_{\ell(R)\leq 2^{k-6}} \ell(R)^{1-\pdmn/\beta}
\|\vec H\|_{L^{\beta}(R)}
\eta_R\Bigr)^q\,d\sigma\biggr)^{1/q}
\\&=
\biggl(\int_{\partial\Omega} \Bigl(\sum_{j=-\infty}^{k-6}
\sum_{\ell(R)=2^j} \ell(R)^{1-\pdmn/\beta}
\|\vec H\|_{L^{\beta}(R)}
\eta_R\Bigr)^q\,d\sigma\biggr)^{1/q}
\\&\leq\sum_{j=-\infty}^{k-6}
\biggl(\int_{\partial\Omega} \Bigl(
\sum_{\ell(R)=2^j} \ell(R)^{1-\pdmn/\beta}
\|\vec H\|_{L^{\beta}(R)}
\eta_R\Bigr)^q\,d\sigma\biggr)^{1/q}
.\end{align*}
Again by Lemma~\ref{lem:Whitney:boundary} we have that each $\sum_{\ell(R)=2^j} \ell(R)^{1-\pdmn/\beta}
\|\vec H\|_{L^{\beta}(R)}
\eta_R(\xi)$ is a sum of at most $C$ nonzero terms, and so
\begin{align*}\biggl(\int_{\partial\Omega} f_k^q\,d\sigma\biggr)^{1/q}
&\leq
C_q\sum_{j=-\infty}^{k-6}
\biggl(\int_{\partial\Omega}
\sum_{\ell(R)=2^j} \ell(R)^{q-q\pdmn/\beta}
\|\vec H\|_{L^{\beta}(R)}^q
\eta_R^q\,d\sigma\biggr)^{1/q}
.\end{align*}
By the Ahlfors regularity of the boundary and because $\eta_R$ is supported in a surface ball of radius~$30\diam R$,
\begin{align*}\biggl(\int_{\partial\Omega} f_k^q\,d\sigma\biggr)^{1/q}
&\leq
C_q\sum_{j=-\infty}^{k-6}
\biggl(
\sum_{\ell(R)=2^j} \ell(R)^{q-q\pdmn/\beta+\bdmn}
\|\vec H\|_{L^{\beta}(R)}^q
\biggr)^{1/q}
.\end{align*}
Thus
\begin{equation*}
I_k
\leq
C
2^{-kqs}
\biggl(
C_q\sum_{j=-\infty}^{k-6}
\biggl(
\sum_{\ell(R)=2^j} \ell(R)^{q-q\pdmn/\beta+\bdmn}
\|\vec H\|_{L^{\beta}(R)}^q
\biggr)^{1/q}
\biggr)^q
.\end{equation*}
If $\varepsilon>0$, then by Hölder's inequality in sequence spaces and convergence of geometric series
\begin{equation*}
I_k
\leq
C_q
2^{-kqs}
\sum_{j=-\infty}^{k-6}2^{\varepsilon(k-j)}
\sum_{\ell(R)=2^j} \ell(R)^{q-q\pdmn/\beta+\bdmn}
\|\vec H\|_{L^{\beta}(R)}^q
.\end{equation*}
Thus
\begin{equation*}\sum_{k=-\infty}^\infty I_k
\leq C_q
\sum_{k=-\infty}^\infty
2^{-kqs}
\sum_{j=-\infty}^{k-6}2^{\varepsilon(k-j)}
\sum_{\ell(R)=2^j} \ell(R)^{q-q\pdmn/\beta+\bdmn}
\|\vec H\|_{L^{\beta}(R)}^q
.\end{equation*}
Changing the order of summation yields that
\begin{equation*}\sum_{k=-\infty}^\infty I_k
\leq C_q
\sum_{j=-\infty}^{\infty}
\sum_{\ell(R)=2^j} \ell(R)^{q-q\pdmn/\beta+\bdmn}
\|\vec H\|_{L^{\beta}(R)}^q
\sum_{k=j+6}^\infty
2^{-kqs}
2^{\varepsilon(k-j)}
.\end{equation*}
Choosing $\varepsilon<qs$ yields that
\begin{equation*}\sum_{k=-\infty}^\infty I_k
\leq C_{q,s}
\sum_{R\in\mathcal{G}} \ell(R)^{q-q\pdmn/\beta+\bdmn-qs}
\|\vec H\|_{L^{\beta}(R)}^q
\end{equation*}
which by Lemma~\ref{lem:averaged:Whitney} is comparable to $\|\vec H\|_{L^q_{av,c,\beta,s}(\Omega)}^q$. This completes the proof.

\appendix

\section{The \texorpdfstring{$L^q$}{Lq} Dirichlet problem and small constant \texorpdfstring{$A^\infty$}{A∞} condition}
\label{sec:Lp}

The $L^q$-Dirichlet problem has been a subject of intense study over the past few decades, and so the bound~\eqref{eqn:N:bound} is known to be valid in many cases. A complete survey of such results is beyond the scope of this paper. We refer the reader to our short paper \cite{BarMP25pD} for a brief discussion of a few recent results.

Here, however, we would like to prove a result which seems to be missing and which provides yet another beautiful corollary to our theorems. As is very common in the literature, the authors of \cite{DavLM23} showed not that the $L^q$ Dirichlet problem is well posed, but that the harmonic measure satisfies an $A_\infty$ condition. It has long been known that some forms of this $A_\infty$ condition are equivalent to $L^q$ well posedness for some~$q$; however, it is occasionally somewhat involved to connect the particular form of the $A_\infty$ condition to the $L^q$ Dirichlet problem for any particular value of~$q$. And, in the case of small oscillations and relatively flat domains covered in \cite{DavLM23} we actually expect well-posedness of the Dirichlet problem for arbitrarily small $q>1$. In this section we will provide the connection between the solvability of the Dirichlet problem for all $L^q$ and the small constant $A_\infty$ condition as stated in \cite{DavLM23}. An emerging corollary is that well-posedness is valid in a much bigger range than that drawn in Figure~\ref{fig:extrapolation} as $1/q$ may be taken arbitrarily close to~1.

It is known that well posedness of the $L^q$-Dirichlet problem requires additional conditions on the coefficients and domains, beyond those listed in Theorem~\ref{thm:Poisson}; see \cite{ModM80,CafFK81}.
In this section we will state the conditions on domains and coefficients used in  \cite{DavLM23}.

\begin{defn}[Chord-arc surfaces with small constant]\label{dfn:CASSC} Let $\Gamma\subset\R^{\dmn}$ be closed and unbounded. We say that $\Gamma$ is a chord-arc surface with constant~$\varepsilon$ if, for every $x\in \Gamma$ and $r>0$, there is a Lipschitz function $\varphi_x:\R^\dmnMinusOne\to\R$ with Lipschitz constant at most $\varepsilon$ and a rigid motion $T_x$ such that, if
\begin{equation*} G=\{T_x((y,\varphi(y))):y\in\R^\dmnMinusOne\}, \end{equation*}
then $G\cap B(x,r/2)\neq \emptyset$ and
\begin{equation*}\sigma(B(x,r)\cap \Gamma\setminus G)+\sigma(B(x,r)\cap G\setminus \Gamma)<\varepsilon r^\dmnMinusOne.\end{equation*}
\end{defn}

\begin{rmk}\label{rmk:CASSC}In particular, a chord-arc surface with sufficiently small constant (depending on the dimension~$\dmn$) necessarily has $\pdmnMinusOne$-Ahlfors regular boundary. By \cite[Lemma~5.4]{DavLM23} and the following remarks, if $\Omega\subset\R^{\dmn}$ and $\partial\Omega$ is a chord-arc surface with sufficiently small constant~$\varepsilon$, then $\Omega$ is a chord-arc domain, and therefore satisfies the weak local John condition.\end{rmk}

\begin{defn}[The weak DKP condition]\label{dfn:DKP}
Let $\Omega\subset\R^{\dmn}$ be open and such that $\partial\Omega$ is $\pdmnMinusOne$-Ahlfors regular.
Let $A$ be a matrix-valued function defined on $\Omega$ that satisfies the ellipticity conditions~(\ref{eqn:elliptic:introduction}--\ref{eqn:elliptic:bounded:introduction}) for some $\Lambda>\lambda>0$ with $\bdmn=\dmn-1$.

We say that $A$ satisfies the weak DKP condition with constant $M$ in $\Omega$ if $A=B+C$, where
\begin{equation*}\sup_{x\in\partial\Omega}\sup_{r>0} \frac{1}{r^\dmnMinusOne} \int_{B(x,r)\cap\Omega} |\nabla B|^2\delta + |C|^2\frac{1}{\delta}=M<\infty.\end{equation*}

\end{defn}

An earlier and more restrictive form of this condition was proposed by Dahlberg and investigated by Kenig and Pipher in \cite{KenP01}.

We will conclude this appendix with the following lemma, which was stated without proof in \cite{BarMP25pD}.

\begin{lem}\label{lem:DKP:small}
Let $q\in (1,\infty)$, let $\dmn\geq 3$ be an integer, and let $\Lambda>\lambda>0$. Then there is a $\delta>0$ (depending on $q$, $\dmn$, $\Lambda$, and~$\lambda$) such that if
\begin{enumerate}
\item $\Omega\subset\R^{\dmn}$,
\item $\partial\Omega$ is a chord-arc surface with constant at most~$\delta$ (as defined in Definition~\ref{dfn:CASSC}),
\item $A$ satisfies the ellipticity conditions~(\ref{eqn:elliptic:introduction}--\ref{eqn:elliptic:bounded:introduction}) (with $\bdmn=\dmnMinusOne$), and
\item $A$ satisfies the weak Dahlberg-Kenig-Pipher condition in $\Omega$ with constant at most~$\delta$ (as defined in Definition~\ref{dfn:DKP}),
\end{enumerate}
then the $L^q$-Dirichlet problem is well posed.
\end{lem}

\begin{proof} The proof follows from several known results in the literature. First, by the recent result \cite[Corollary~6.23]{DavLM23}, the given conditions imply a quantitative local $A_\infty$-type condition on the harmonic measure. Second, this $A_\infty$-type condition implies a local reverse Hölder condition on the density (that is, the Radon-Nikodym derivative with respect to the surface measure) of the harmonic measure. Finally, this reverse Hölder estimate yields well posedness of the $L^q$-Dirichlet problem. There are well known arguments for the second and third steps that suffice to establish that there exists \emph{some} $q$ such that the $L^q$-Dirichlet problem is well posed; in the case of the present situation, when we preselect $q$ and seek a setting such that the particular $L^q$-Dirichlet problem is well posed, we must state the relevant estimates with more care.

Specifically, \cite[Corollary~6.23]{DavLM23} is as follows. Suppose that $\varepsilon>0$ and $\kappa>1$. Then there is a $\delta>0$ and $\tau>0$ (depending on $\varepsilon$, $\kappa$, $\dmn$, $\lambda$, and~$\Lambda$) with the following significance. Suppose that the conditions of the lemma are satisfied with the given~$\delta$. Then
for every $x\in \partial\Omega$,
for every $r>0$,
and
for every $y\in\Omega$ with $\dist(y,\partial\Omega)\geq r/\tau$ and with $\Delta=B(x,r)\cap\partial\Omega\subset B(y,\kappa\dist(y,\partial\Omega))$,
we have that $\omega^y_L$, the harmonic measure for $-\Div A\nabla$ with pole at~$y$, is absolutely continuous with respect to the surface measure on $\Delta=B(x,r)\cap\partial\Omega$, and furthermore if $k^y_L$ denotes the Radon-Nikodym derivative of $\omega_L^y$ with respect to~$\sigma$, then \begin{equation*}\sup_{\Delta(y,\varrho)\subseteq \Delta} \fint_{\Delta(y,\varrho)} \bigl|\log k_L^y(z)-{\textstyle \fint_{\Delta(y,\varrho)} \log k_L^y}\bigr|\,d\sigma(z)
=\|\log k_L^y\|_{BMO(\Delta)}\leq\varepsilon.\end{equation*}

A form of the John-Nirenberg inequality in doubling spaces that is known to be true (see \cite[Theorem~2.2]{Buc99}) with this form of $BMO$ in an open set (that is, in $\Delta(x,r)$, which is relatively open in~$\partial\Omega$), is that
\begin{equation*}\fint_{\Delta(x,r/2)} \exp \biggl( \frac{|\log k^y-c_{x,r}
|}{C_1 \|\log k^y\|_{BMO(\Delta(x,r))}} \biggr) \,d\sigma \leq 16\end{equation*}
where $c_{x,r}=\fint_{\Delta(x,r/2)}\log k^y\,d\sigma$ and where $C_1$ depends on $\bdmn$ and the Ahlfors constant of~$\partial\Omega$ (or more precisely on the doubling constant of the doubling measure $\sigma$ on~$\partial\Omega$).

Let $1<p<\infty$. If $c>0$, then
\begin{align*}
\fint_{\Delta(x,r/2)} (k^y)^p\,d\sigma
&=\fint_{\Delta(x,r/2)} \exp(p\log k^y)\,d\sigma
\\&=c^p\fint_{\Delta(x,r/2)} \exp(p(\log k^y-\log c))\,d\sigma.
\end{align*}
If we choose $c=\exp c_{x,r}$, and if $\varepsilon<1/C_1p$, then
\begin{align*}
\fint_{\Delta(x,r/2)} (k^y)^p\,d\sigma
&\leq 16c^p = 16 \biggl(\exp \fint_{\Delta(x,r/2)}\log k^y\,d\sigma\biggr)^p.
\end{align*}
Because $\exp$ is a convex function, by Jensen's inequality
\begin{equation*}\exp \fint_{\Delta(x,r/2)}\log k^y\,d\sigma
\leq \fint_{\Delta(x,r/2)}\exp \log k^y\,d\sigma
=\fint_{\Delta(x,r/2)}k^y\,d\sigma\end{equation*}
and so
\begin{align*}
\fint_{\Delta(x,r/2)} (k^y)^p\,d\sigma
&\leq 16 \biggl(\fint_{\Delta(x,r/2)}k^y\,d\sigma\biggr)^p.
\end{align*}

To summarize, if $1<p<\infty$ and $\kappa>1$, then there is a $\delta>0$ and a $\tau>0$ such that if the conditions of the lemma are satisfied, then
for every $x\in \partial\Omega$,
for every $r>0$,
and every $y\in\Omega$ with $\dist(y,\partial\Omega)\geq r/\tau$ and with $\Delta=B(x,r)\cap\partial\Omega\subset B(y,\kappa\dist(y,\partial\Omega))$, we have that
\begin{align*}
\fint_{\Delta(x,r/2)} (k^y)^p\,d\sigma
&\leq 16 \biggl(\fint_{\Delta(x,r/2)}k^y\,d\sigma\biggr)^p.
\end{align*}
We choose $\kappa=2/\corkscrew$, where $\corkscrew$ is the constant in the interior corkscrew condition (Definition~\ref{dfn:iCS}). By Remark~\ref{rmk:CASSC}, $\corkscrew$ exists and is positive. It is straighforward to establish that if $c\geq 2/\corkscrew>1$ and $c\geq 1/\corkscrew\tau$, then the corkscrew point $y=A_{cr}(x)$ given by Definition~\ref{dfn:iCS} satisfies the above conditions.

Now, recall that $k^y$ is the density of the elliptic measure with pole at~$y$. Thus if $E\subset \Delta(x,r)$, then $\int_E k^y\,d\sigma=\omega^y(E)$.

By the well known comparison principle (a precise statement in the generality we need may be found in \cite[Lemma~15.47]{DavFM20p}) we have that if $z\in \Omega\setminus B(x,2r)$ and $\Delta=\Delta(x,r)$, then
\begin{equation*}\frac{\omega^z(E)}{\omega^z(\Delta)} \approx \frac{\omega^{A_r(x)}(E)}{\omega^{A_r(x)}(\Delta)}\end{equation*}
where in particular the comparability constant does not depend on~$z$. We may clearly replace $A_r(x)$ by $A_{cr}(x)$ (that is, our value of~$y$ from above).

In particular, if \begin{equation*}E=\{x\in\Delta:k^z(x)\geq R k^y(x)\omega^z(\Delta)/\omega^y(\Delta)\},\end{equation*}
then
\begin{equation*}\omega^z(E)=\int_E k^z\,d\sigma \geq R\int_E k^y\omega^z(\Delta)/\omega^y(\Delta)\,d\sigma
= R\omega^y(E)\omega^z(\Delta)/\omega^y(\Delta)\end{equation*}
and so $E$ must be empty for sufficiently large~$R$. The same is true with the roles of $y$ and $z$ reversed. Thus $k^z\approx k^y$ in $\Delta$ and so we have the estimate
\begin{align*}
\fint_{\Delta(x,r/2)} (k^z)^p\,d\sigma
&\leq C \biggl(\fint_{\Delta(x,r/2)}k^z\,d\sigma\biggr)^p
\end{align*}
for all $z\in \Omega\setminus B(x,2r)$, where $C$ does not depend on~$z$.
By \cite[Proposition~4.5]{HofL18}, this implies well posedness of the $L^q$-Dirichlet problem in~$\Omega$.
\end{proof}

\def\cprime{'}
\bibliographystyle{amsalpha}
\bibliography{bibli}
\end{document}